\documentclass[a4paper]{article}
\usepackage[utf8]{inputenc}
\usepackage{amsmath,amsthm, amssymb, amsfonts}
\usepackage{fullpage}
\usepackage[pdfencoding=auto, hidelinks]{hyperref}
\usepackage[dvipsnames]{xcolor}
\usepackage{graphicx}
\graphicspath{{./figures/}}

\usepackage{subcaption}
\usepackage{wasysym}
\usepackage{todonotes}
\usepackage{comment}
\usepackage{mathtools}
\usepackage{wrapfig}
\usepackage{microtype}
\usepackage{tabularx} 
\usepackage{placeins}
\usepackage[]{algorithm2e}
\usepackage{bm}
\usepackage{float} % for figure alignment
\usepackage{tgpagella}
\usepackage[most]{tcolorbox}

\allowdisplaybreaks[1]
\usepackage{sistyle} % or siunitx

\newcommand{\Tcal}{\mathcal{T}}

\newcommand{\N}{\mathbb{N}}

\newcommand{\R}{\mathbb{R}}

\newcommand{\bbP}{\mathbb{P}}

\newcommand{\eps}{\varepsilon}

\newcommand{\sconc}{\odot}

\let\emptyset\varnothing

\let\epsilon\varepsilon

\newtheorem{theorem}{Theorem}[section]
\newtheorem*{theorem*}{Theorem}
\newtheorem{remark}[theorem]{Remark}
\newtheorem{definition}[theorem]{Definition}
\newtheorem{proposition}[theorem]{Proposition}
\newtheorem{lemma}[theorem]{Lemma}
\newtheorem{corollary}[theorem]{Corollary}

\newtheorem*{remark*}{Remark}
\newtheorem*{proposition*}{Proposition}

\numberwithin{equation}{section}

\DeclareMathOperator*{\interior}{int\,}

\definecolor{purple}{rgb}{0.6, 0.0, 1.0}

\newcommand{\rev}[1]{#1}%{{\color{green!50!black}{#1}}} %Revision
\definecolor{dkgreen}{rgb}{0.0, 0.5, 0.0}
\newcommand{\normc}[2][]{\left\| #2 \right\|_{#1}} % norm
\newcommand{\snormc}[2][]{\left| #2 \right|_{#1}} % seminorm

\DeclareMathOperator{\FPar}{FP}
\DeclareMathOperator{\Par}{P}
\DeclareMathOperator{\depth}{L}
\DeclareMathOperator{\size}{M}
\DeclareMathOperator{\Realiz}{R}
\newcommand{\Dop}{{\mathcal{L}}} % differential operator
\newcommand{\Bop}{{\mathcal{B}}} % boundary operator
\newtheorem{setting}[theorem]{Setting}
\newcommand{\bbc}{\boldsymbol c} % corner point.

\DeclareMathOperator{\Id}{Id}

\DeclareMathOperator{\spn}{span}

\DeclareMathOperator{\supp}{supp}
\DeclareMathOperator{\dist}{dist}

\DeclareMathOperator{\vvec}{vec}
\DeclareMathOperator{\ptch}{patch}

\newcommand{\Trace}{T}
\newcommand{\lnorm}{\left\vert\kern-0.25ex\left\vert\kern-0.25ex\left\vert}
\newcommand{\rnorm}{\right\vert\kern-0.25ex\right\vert\kern-0.25ex\right\vert}

\newcommand{\hu}{\hat{u}}

\newcommand{\hv}{\hat{v}}

\newcommand{\hK}{{\widehat{K}}}
\newcommand{\Hmix}{H_{\mathrm{mix}}}
\newcommand{\Wmix}{W_{\mathrm{mix}}}
\newcommand{\hPi}{{\widehat{\Pi}}}
\newcommand{\hpi}{{\widehat{\pi}}}

\newcommand{\tu}{{\tilde{u}}}
\newcommand{\tv}{{\tilde{v}}}

\newcommand{\tgamma}{{\widetilde{\gamma}}}

\newcommand{\ugamma}{{\underline{\gamma}}}
\newcommand{\utgamma}{{\underline{\tgamma}}}

\newcommand{\xpar}{x_\parallel}
\newcommand{\hpar}{h_\parallel}

\newcommand{\hperp}{h_\bot}
\newcommand{\hperpone}{h_{\bot,1}}
\newcommand{\hperptwo}{h_{\bot,2}}
\newcommand{\xperpone}{x_{\bot,1}}
\newcommand{\xperptwo}{x_{\bot,2}}
\newcommand{\pipar}{\pi_\parallel}
\newcommand{\piperp}{\pi_\bot}
\newcommand{\alphapar}{{\alpha_\parallel}}

\newcommand{\alphaperp}{{\alpha_\bot}}
\newcommand{\alphaperpm}{{|\alpha_\bot|}}
\newcommand{\alphaparm}{{|\alpha_\parallel|}}
\newcommand{\dalpha}{\partial^\alpha}
\newcommand{\dalphaperp}{\partial^{\alphaperp}}
\newcommand{\dalphapar}{\partial^{\alphapar}}
\newcommand{\dpar}{\partial_{\parallel}}
\newcommand{\dperpone}{\partial_{\bot,1}}
\newcommand{\dperptwo}{\partial_{\bot,2}}

\newcommand{\alpham}{{|\alpha|}}
\newcommand{\alphainf}{{|\alpha|_\infty}}

\newcommand{\tc}{\widetilde{c}}

\newcommand{\fK}{{\mathfrak{K}}}

\newcommand{\cN}{\mathcal{N}}

\newcommand{\cG}{\mathcal{G}}

\newcommand{\tcG}{\widetilde{\mathcal{G}}}

\newcommand{\Eset}{\mathcal{E}}
\newcommand{\Cset}{\mathcal{C}}

\newcommand{\cT}{\mathcal{T}}

\newcommand{\cJ}{\mathcal{J}}
\newcommand{\bigtimesdim}{\bigtimes_{i=1}^d}

\newcommand{\Xhpelldim}{X_{\mathsf{hp}, d}^{\ell, p}}
\newcommand{\tXhpelldim}{\widetilde{X}_{\mathsf{hp}, d}^{\ell, p}}
\newcommand{\XhpellCpdim}{X_{\mathsf{hp}, d}^{\ell, c_p\ell}}
\newcommand{\Xhpellone}{X_{\mathsf{hp}, 1}^{\ell, p}}
\newcommand{\tXhpellone}{\widetilde{X}_{\mathsf{hp}, 1}^{\ell, p}}

\newcommand{\Pihpelldim}{\Pi_{\mathsf{hp}, d}^{\ell, p}}
\newcommand{\PihpellCpdim}{\Pi_{\mathsf{hp}, d}^{\ell, c_p\ell}}

\newcommand{\tPihpelldim}{\widetilde{\Pi}_{\mathsf{hp}, d}^{\ell, p}}
\newcommand{\tPihpellCpdim}{\widetilde{\Pi}_{\mathsf{hp}, d}^{\ell, c_p\ell}}

\newcommand{\pihpell}{\pi_{\mathsf{hp}}^{\ell, p}}
\newcommand{\tpihpell}{\widetilde{\pi}_{\mathsf{hp}}^{\ell, p}}

\newcommand{\Cappxi}{C_{\mathrm{appx}1}}
\newcommand{\Cappxii}{C_{\mathrm{appx}2}}

\newcommand{\Ndof}{N_{\mathrm{dof}}}

\newcommand{\hx}{\hat{x}}

\newcommand{\fbf}{\bm{f}}
\newcommand{\hp}{\mathsf{hp}}
\newcommand{\hepsilon}{{\widehat{\epsilon}}}
\newcommand{\hvarepsilon}{{\widehat{\varepsilon}}}
\newcommand{\Noned}{N_{\mathrm{1d}}}
\newcommand{\tNoned}{\widetilde{N}_{\mathrm{1d}}}
\newcommand{\Phibasis}{\Phi_{\mathrm{basis}}}

\newcommand{\sumis}{\sum_{i_1, \dots, i_d=1}^{\Noned}}
\newcommand{\is}{{i_1\dots i_d}}
\newcommand{\ps}{{p_1\dots p_d}}
\newcommand{\iscomma}{{i_1,\dots, i_d}}

\newcommand{\corn}{c}

\newcommand{\Omegac}{\Omega_c}
\newcommand{\Omegace}{\Omega_{ce}}
\newcommand{\Omegae}{\Omega_{e}}

\newcommand{\Qc}{Q_c}
\newcommand{\Qce}{Q_{ce}}

\newcommand{\oned}{{\mathrm{1d}}}
\newcommand{\Nint}{N_{\mathrm{int}}}
\newcommand{\tNint}{\widetilde{N}_{\mathrm{int}}}
\newcommand{\sigmaratio}{\tau_\sigma}

\newcommand{\epsprod}{\epsilon_{\times}}
\newcommand{\Mprod}{M_{\times}}
\newcommand{\range}[1]{\{1, \dots, #1\}}

\newcommand{\hcN}{\widehat{\mathcal{N}}}
\newcommand{\distomega}{\dist_{\Omega}}
\newcommand{\hOmega}{\widehat{\Omega}}
\newcommand{\tOmega}{\widetilde{\Omega}}
\newcommand{\cTOmega}{{\cT_\Omega}}
\newcommand{\cNOmega}{{\cN_\Omega}}

\title{Exponential ReLU Neural Network Approximation Rates\\ for Point and Edge Singularities}
\author{\hspace{-1cm}Carlo Marcati\thanks{Seminar for Applied Mathematics, ETH Z\"urich, CH8092 Z\"urich, Switzerland,\newline \texttt{email:} \texttt{\{carlo.marcati, joost.opschoor, christoph.schwab\}@sam.math.ethz.ch}}
\and Joost A. A. Opschoor\footnotemark[1]
\and Philipp C. Petersen\thanks{Faculty of Mathematics and Research Platform Data Science @ Uni Vienna,
  University of Vienna, 1090, Vienna, Austria,\newline \texttt{email:} \texttt{Philipp.Petersen@univie.ac.at}}
\and Christoph Schwab\footnotemark[1]\hspace{-1cm}
}

\date{\today}

\begin{document}
\maketitle
% \footnotetext[1]{
% Seminar for Applied Mathematics, ETH Z\"urich, CH8092 Z\"urich, Switzerland, \\ \texttt{email:}
% \texttt{\{carlo.marcati, joost.opschoor, christoph.schwab\}@sam.math.ethz.ch}} 
% \footnotetext[2]{
% Faculty of Mathematics and Research Platform Data Science @ Uni Vienna, University of Vienna, 1090, Vienna, Austria, \texttt{email:} \texttt{Philipp.Petersen@univie.ac.at}}

\abstract{
We prove exponential expressivity with stable ReLU Neural Networks (ReLU NNs)
in $H^1(\Omega)$ for weighted analytic function classes
in certain polytopal domains $\Omega$, 
in space dimension $d=2,3$.
Functions in these classes are 
locally analytic on open subdomains $D\subset \Omega$,
but may exhibit isolated point singularities in the interior of $\Omega$ 
or
corner and edge singularities at the boundary $\partial \Omega$.
The exponential expression rate bounds proved here
imply uniform exponential expressivity 
by ReLU NNs of solution families 
for several elliptic boundary and eigenvalue problems with analytic data. 
The exponential approximation rates are 
shown to hold in space dimension $d = 2$ on 
Lipschitz polygons with straight sides, and 
in space dimension $d=3$ on 
Fichera-type polyhedral domains with plane faces.
The constructive proofs indicate in particular that 
NN depth and size increase poly-logarithmically 
with respect to the target NN approximation accuracy $\varepsilon>0$ in $H^1(\Omega)$.
The results cover in particular 
solution sets of linear, second order elliptic PDEs
with analytic data and certain nonlinear elliptic eigenvalue problems with 
analytic nonlinearities and singular, weighted analytic potentials 
as arise in electron structure models.
In the latter case, the functions correspond to electron densities
that exhibit isolated point singularities at the positions of the nuclei. 
Our findings provide in particular 
mathematical foundation of recently reported, 
successful uses of deep neural networks in 
variational electron structure algorithms.

\medskip
\noindent
\textbf{Keywords:} 
Neural networks, finite element methods, exponential convergence, analytic regularity, singularities

\smallskip 
\noindent
\textbf{Subject Classification:} 35Q40, 41A25, 41A46, 65N30
}

{\small
\tableofcontents
}%
\section{Introduction}
\label{sec:Intro}
The application of neural networks (NNs) as approximation architecture
in numerical solution methods of partial differential equations (PDEs),
possibly on high-dimensional parameter- and state-spaces,
has attracted significant and increasing attention in recent years.
We mention only
\cite{SirSpil2018,EDeepRitz,RaisKarnHidPhys,RaisKarnPINN,sheng2020pfnn}
and the references therein.
In these works,
the solution of elliptic and parabolic boundary value problems
is approximated by NNs which are found by
minimization of a residual of the NN in the PDE.

A necessary condition for the performance of
the mentioned NN-based numerical approximation methods
is a high rate of approximation which is to hold
uniformly over the solution set associated with the PDE under consideration.
For elliptic boundary and eigenvalue problems,
the function classes that weak solutions of the problems belong to 
are well known.
Moreover, in many cases, representation systems such as splines or polynomials that
achieve optimal linear or nonlinear approximation rates for
functions belonging to these function classes have been identified.
For functions belonging to a Sobolev or Besov type smoothness space
of finite differentiation order
such as continuously differentiable 
or Sobolev-regular functions on a bounded domain,
upper bounds for the approximation rate by 
NNs were established for example in 
\cite{YAROTSKY2017103,guhring2019error,yarotsky2018optimal,lu2020deep,suzuki2018adaptivity}.
Here, we only mentioned results that use the ReLU activation function.
Besides, for PDEs, the solutions of which have a particular structure,
approximation rates of the solution that go beyond classical smoothness-based results can be achieved,
such as in \cite{EGJS2018,schwab2017deep,laakmann2020efficient,berner2018analysis,jentzen2018proof}.
Again, we confine the list to publications with approximation rates for
NNs with the ReLU activation function (referred to as ReLU NNs below).

In the present paper, we analyze approximation rates provided by ReLU NNs
for solution classes of linear and nonlinear elliptic source and eigenvalue problems
on polygonal and polyhedral domains.
Mathematical results on weighted analytic regularity
\cite{GuoBab1,GuoBab2,GuoBab3,GuoBabCurv,BDC85,GuoScStokes,Mazya2010,CoDaNi12,MadMarc2019,MS19_2743}
imply that these classes consist of functions that are
\emph{analytic with possible corner, edge, and corner-edge singularities}.

Our analysis provides, for the aforementioned functions, 
approximation errors in Sobolev norms that decay exponentially
in terms of the number of parameters $M$ of the ReLU NNs.
\subsection{Contribution}
\label{S:Contrib}
The principal contribution of this work is threefold:
\begin{enumerate}
\item
We prove, in Theorem \ref{th:ReLUapprox},
a general result on the approximation by ReLU NNs 
of weighted analytic function classes on $Q \coloneqq (0,1)^d$, 
where $d = 2,3$.
The analytic regularity of solutions is quantified via
countably normed, analytic classes, 
based on weighted Sobolev spaces of Kondrat'ev type in $Q$,
which admit corner and, in space dimension $d=3$, also edge singularities.
Such classes were introduced, e.g.,
in \cite{BDC85,GuoBabCurv,GuoBab3,GuoBab1,GuoBab2,CoDaNi12} and in the references there.
We prove exponential expression rates by ReLU NNs
in the sense that for a number $M$ of free parameters of the NNs,
the approximation error is bounded, in the $H^1$-norm,
by $C\exp(-bM^{1/(2d+1)})$ for constants $b,C > 0$.
\item
Based on the ReLU NN approximation rate bound of Theorem \ref{th:ReLUapprox},
we establish, in Section \ref{sec:applications}, approximation results for
solutions of different types of PDEs by NNs with ReLU activation.
Concretely, in Section \ref{sec:NonlSchrEq}, we study the reapproximation of
solutions of nonlinear Schr\"odinger equations with singular potentials in space dimension $d=2,3$.
We prove that for solutions which are contained in weighted, analytic classes in $\Omega$, 
ReLU NNs (whose realizations are continuous, piecewise affine)
with at most $M$ free parameters yield an approximation 
with accuracy of the order $\exp(-bM^{1/(2d+1)})$ for some $b>0$. 
Importantly, this convergence is in the $H^1(\Omega)$-norm.
In Section \ref{sec:HF}, 
we establish the same exponential approximation rates for the eigenstates
of the Hartree-Fock model with singular potential in $\R^3$.
This result constitutes the first, to our knowledge, mathematical underpinning
of the recently reported, high efficiency of
various NN-based approaches in variational electron structure computations,
e.g., \cite{pfau2019abinitio,hermann2019deep,ESchroed2019}.

In Section \ref{sec:polygonal},
we demonstrate the same approximation rates also for
elliptic boundary value problems
with analytic coefficients and analytic right-hand sides, in plane, polygonal domains $\Omega$.
The approximation error of the NNs is, again, bound in the $H^1(\Omega)$-norm.
We also infer an exponential NN expression rate bound for corresponding traces,
in $H^{1/2}(\partial\Omega)$ and for viscous, incompressible flow.

Finally, in Section \ref{sec:EllPDEFichera},
we obtain the same asymptotic exponential rates for the
approximation of solutions to elliptic boundary value problems,
with analytic data,
on so-called Fichera-type domains $\Omega\subset {\mathbb R}^3$
(being, roughly speaking, finite unions of tensorized hexahedra).
These solutions exhibit corner, edge and corner-edge singularities.
\item
The exponential approximation rates of the ReLU NNs established
here are based on emulating corresponding variable grid and degree (``$hp$'')
piecewise polynomial approximations.
In particular, our construction comprises %on
tensor product $hp$-approximations on
Cartesian products of geometric partitions of intervals.
In Theorem \ref{prop:internal}, we establish
novel \emph{tensor product $hp$-approximation results}
for weighted analytic functions on $Q$ of the form
$\| u - u_{\hp} \|_{H^1(Q)} \leq C \exp(-b\sqrt[2d]{N})$ for $d=1,2,3$,
where $N$ is the number of degrees of freedom in the representation of $u_{\hp}$ and 
$C,b>0$ are independent of $N$ (but depend on $u$).
The geometric partitions employed in $Q$ and the architectures of the constructed ReLU NNs
are of tensor product structure, and generalize to space dimension $d>3$.

We note that $hp$-approximations based on non-tensor-product, geometric partitions
of $Q$ into hexahedra
have been studied before e.g. 
in \cite{SSS15_2016,SchSch2018} in space dimension $d=3$.
There, the rate of $\| u - u_{\hp} \|_{H^1(Q)} \lesssim \exp(-b\sqrt[5]{N})$ was found.
Being based on tensorization,
the present construction of exponentially convergent,
tensorized $hp$-approximations in Appendix \ref{sec:hp-analysis} does not invoke the rather
involved polynomial trace liftings in \cite{SSS15_2016,SchSch2018}, and is interesting in its
own right: the geometric and mathematical simplification comes
at the expense of a slightly smaller (still exponential) rate of approximation.
Moreover, we expect that this construction of $u_{\hp}$ will allow a
rather direct derivation of rank bounds for tensor structured function approximation of $u$ in $Q$,
generalizing results in \cite{KORS17_2264,KS18_2116} 
and extending \cite{MRS19_872} from point to edge and corner-edge singularities.
\end{enumerate}
\subsection{Neural network approximation of weighted analytic function classes}
\label{sec:ReapprNNs}
The proof strategy that yields the main result, 
Theorem \ref{th:ReLUapprox}, is as follows.
We first establish exponential approximation rates in
the Sobolev space $H^1$ for tensor-product, so-called
$hp$-finite elements for weighted analytic functions.
Then, 
we re-approximate the corresponding quasi-interpolants
by ReLU NNs.

The emulation of $hp$-finite element approximation by ReLU NNs
is fundamentally based on the approximate
multiplication network formalized in \cite{YAROTSKY2017103}.
Based on the approximate multiplication
operation and an extension thereof to errors measured with respect to
$W^{1,q}$-norms, for $q \in [1,\infty]$,
we established already in \cite{OPS19_811} a reapproximation theorem
of univariate splines of order $p\in \N$ on arbitrary meshes
with $N\in \N$ cells.
There,
we observed that there exists a NN that reapproximates
a variable-order, free-knot spline $u$ in the $H^1$-norm
up to an error of $\epsilon>0$ with a number of free parameters that scales logarithmically
in $\epsilon$ and $|u|_{H^1}$, linearly in $N$ and quadratically in $p$.
We recall this result in Proposition \ref{prop:relupwpolynom} below.

From this, it is apparent by the triangle inequality that, in univariate
approximation problems where $hp$-finite elements yield exponential
approximation rates, also ReLU NNs achieve exponential approximation rates,
(albeit with a possibly smaller exponent, because of the quadratic
dependence on $p$, see \cite[Theorem 5.12]{OPS19_811}).

The extension of this result to higher dimensions for high-order
finite elements that are built from univariate finite elements by tensorization
is based on the underlying compositionality of NNs.
Because of that, it is possible to compose a NN implementing
a multiplication of $d$ inputs with $d$ approximations of univariate finite elements.
We introduce a formal framework describing these operations in Section \ref{sec:ReLUCalc}.

We remark that for high-dimensional functions with a radial structure,
of which the univariate radial profile allows an exponentially convergent $hp$-approximation,
exponential convergence was obtained in \cite[Section 6]{OPS19_811}
by composing ReLU approximations of univariate splines with an
exponentially convergent approximation of the Euclidean norm,
obtaining exponential convergence without the curse of dimensionality.
\subsection{Outline}
\label{sec:outline}
The manuscript is structured as follows:
in Section~\ref{sec:setting},
in particular Section~\ref{sec:WgtSpcNonHomNrm},
we review the weighted function spaces which will be used to describe the weighted analytic function
classes in polytopes $\Omega$ that underlie our approximation results.
In Section~\ref{sec:hpTP},
we present an approximation result by tensor-product $hp$-finite elements for functions
from the weighted analytic class.
A proof of this result is provided in Appendix~\ref{sec:hp-analysis}.
In Section~\ref{sec:ReLUCalc} we review definitions of NNs and a ``ReLU calculus''
from \cite{EGJS2018,PETERSEN2018296} 
whose operations will be required in the ensuing NN approximation results.

In Section~\ref{sec:hpReapproxReLU},
we state and prove the key results of the present paper.
In Section \ref{sec:applications},
we illustrate our results by deducing novel NN expression rate bounds
for solution classes of
several concrete examples of elliptic boundary-value and eigenvalue problems
where solutions belong to the weighted analytic function classes
introduced in Section~\ref{sec:setting}.
Some of the more technical proofs of Section \ref{sec:applications} are
deferred to Appendix \ref{sec:proofs-appendix}.
In Section~\ref{sec:ConclExt}, we briefly recapitulate the principal mathematical
results of this paper and indicate possible consequences and further directions.

\section{Setting and functional spaces}
\label{sec:setting}
We start by recalling some general notation that will be used throughout the
paper. We also introduce some tools that are required to describe two and three dimensional
domains as well as the associated weighted Sobolev spaces.
\subsection{Notation}
\label{sec:Notat}
For $\alpha\in \mathbb{N}^d_0$, 
define
$\alpham \coloneqq \alpha_1+\dots+\alpha_d$ and $\alphainf \coloneqq \max\{\alpha_1, \dots, \alpha_d\}$. 
When we indicate a relation on $\alpham$ or $\alphainf$ in the subscript of a
sum, we mean the sum over all multi-indices that fulfill that relation:
e.g., for a $k\in \mathbb{N}_0$
\begin{equation*}
\sum_{\alpham \leq k} = \sum_{\alpha\in \mathbb{N}^d_0:\alpham\leq k}.
\end{equation*}

For a domain $\Omega\subset\mathbb{R}^d$, $k\in\mathbb{N}_0$ and for $1\leq p\leq \infty$, 
we indicate by $W^{k,p}(\Omega)$ the classical $L^p(\Omega)$-based Sobolev space of order $k$. 
We write $H^k(\Omega) = W^{k,2}(\Omega)$. 
We introduce the norms $\| \cdot \|_{\Wmix^{1,p}(\Omega)}$ as
\begin{equation*}
\| v\|_{\Wmix^{1,p}(\Omega)}^{p} \coloneqq \sum_{\alphainf\leq 1} \| \dalpha v\|^p_{L^p(\Omega)},
\end{equation*}
with associated spaces
\begin{equation*}
\Wmix^{1,p}(\Omega) \coloneqq \left\{ v\in L^p(\Omega): \|v\|_{\Wmix^{1,p}(\Omega)} < \infty \right\}.
\end{equation*}
We denote $\Hmix^1(\Omega) = \Wmix^{1,2}(\Omega)$. 
For $\Omega = I_1\times\dots\times I_d$, 
with bounded intervals $I_j\subset\mathbb{R}$, $j=1, \dots, d$, 
$\Hmix^{1}(\Omega) = H^1(I_1)\otimes\dots\otimes H^1(I_d)$
with Hilbertian tensor products.
Throughout, 
$C$ will denote a generic positive constant whose value may change at each appearance, 
even within an equation. 

The $\ell^p$ norm, $1\leq p\leq \infty$, on $\R^n$ is denoted by $\normc[p]{x}$. The number of nonzero entries of a vector or matrix $x$ is denoted by $\|x\|_0$.

\paragraph{Three dimensional domain.}

Let $\Omega \subset \mathbb{R}^3$ be a bounded, polygonal/polyhedral domain. 
Let $\Cset $ denote a set of isolated points, situated either at the corners of $\Omega$
or in the interior of $\Omega$ (that we refer to as the singular corners
in either case, for simplicity), 
and let $\Eset$ be a subset of the edges of $\Omega$ (the singular edges).
Furthermore, denote by $\Eset_\corn \subset \Eset$ 
the set of singular edges abutting at a corner $\corn$.
For each $\corn\in \Cset$ and each $e\in \Eset$, 
we introduce the following weights:
\begin{equation*}
  r_c(x) \coloneqq |x-\corn| = \dist(x, \corn),\qquad r_e(x) \coloneqq \dist(x, e),\qquad \rho_{ce}(x) \coloneqq \frac{r_e(x)}{r_c(x)}\quad \text{ for }x \in \Omega.
\end{equation*}
For $\varepsilon>0$, we define edge-, corner-, and corner-edge neighborhoods: 
\begin{align*}
\Omegae^\varepsilon &\coloneqq \bigg\{ x\in \Omega: r_e(x)< \varepsilon \text{ and }r_c(x)>\varepsilon, \forall \corn\in\Cset\bigg\},\\
  \Omegac^\varepsilon &\coloneqq \bigg\{ x\in \Omega: r_c(x)< \varepsilon \text{ and }\rho_{ce}(x)>\varepsilon, \forall e\in \Eset\bigg\},\quad
    \Omegace^\varepsilon \coloneqq \bigg\{ x\in \Omega: r_c(x)< \varepsilon \text{ and }\rho_{ce}(x)<\varepsilon\bigg\}.
\end{align*}
We fix a value $\hvarepsilon>0$ small enough so that $\Omega_c^{\hvarepsilon}\cap \Omega^{\hvarepsilon}_{c'} = \emptyset$
for all $\corn\neq \corn'\in \Cset$ and
$\Qce^\hvarepsilon \cap \Omega^\hvarepsilon_{ce'} = \Omega^\hvarepsilon_e \cap \Omega^\hvarepsilon_{e'}=\emptyset$ 
for all singular edges $e\neq e'$. In the sequel, we 
omit the dependence of the subdomains on $\hvarepsilon$. 
Let also 
\begin{equation*}
  \Omega_{\Cset}\coloneqq \bigcup_{c\in\Cset}\Omega_c,\qquad
  \Omega_{\Eset}\coloneqq\bigcup_{e\in\Eset}\Omega_e,\qquad
  \Omega_{\Cset\Eset} \coloneqq \bigcup_{c\in\Cset}\bigcup_{e\in\Eset_c}\Omega_{ce},
\end{equation*}
and
\begin{equation*}
  \Omega_0 \coloneqq \Omega\setminus\overline{(\Omega_{\Cset}\cup \Omega_{\Eset}\cup \Omega_{\Cset\Eset})}.
\end{equation*}
In each subdomain $\Omega_{ce}$ and $\Omega_e$, for any multi-index $\alpha\in
\mathbb{N}_0^3$, we denote by $\alphapar$ the multi-index whose component
in the coordinate direction parallel to $e$ is equal to the component of
$\alpha$ in the same direction, and which is zero in every other component.
Moreover, we set $\alphaperp \coloneqq \alpha -\alphapar$.

\paragraph{Two dimensional domain.}
Let $\Omega \subset \mathbb{R}^2$ be a polygon. We adopt the convention that
$\Eset \coloneqq \emptyset$.
For $\corn\in\Cset$, we define
\begin{equation*}
  \Qc^\varepsilon \coloneqq \bigg\{ x\in \Omega: r_c(x)< \varepsilon \bigg\} \;.
\end{equation*}
As in the three dimensional case, we fix a sufficiently small 
$\hvarepsilon>0$ so that $\Omega^{\hvarepsilon}_{c}\cap \Omega^{\hvarepsilon}_{c'}=\emptyset$ for $\corn\neq \corn'\in
\Cset$ and write $\Omegac = \Omegac^\hvarepsilon$. 
Furthermore, $\Omega_{\Cset}$ is defined as for $d=3$, 
and $\Omega_0 \coloneqq \Omega\setminus \overline{\Omega_{\Cset}}$.

\subsection{Weighted spaces with nonhomogeneous norms}
\label{sec:WgtSpcNonHomNrm}
We introduce classes of weighted, analytic functions in space dimension $d = 3$,
as arise in analytic regularity theory for linear, elliptic boundary value problems
in polyhedra, in the particular form introduced in \cite{CoDaNi12}. 
There, the 
structure of the weights is in terms of Cartesian coordinates which is particularly
suited for the presently adopted, tensorized approximation architectures.

The definition of the corresponding classes when $d=2$ is analogous.
For a \emph{weight exponent vector}
$\ugamma = \{\gamma_c, \gamma_e, \, c\in \Cset, e\in \Eset\}$, 
we introduce the \emph{nonhomogeneous, weighted Sobolev norms}
\begin{align*}
\|v\|_{\cJ^{k}_\ugamma(\Omega)} \coloneqq 
\sum_{\alpham\leq k}\|\dalpha v\|_{L^2(\Omega_0)}
  & + \sum_{c\in\Cset}\sum_{\alpham\leq k}\|r_c^{(\alpham -\gamma_c)_+}\dalpha v \|_{L^2(\Omega_c)}\\
  & + \sum_{e\in \Eset}\sum_{\alpham\leq k}\|r_e^{(\alphaperpm -\gamma_e)_+}\dalpha v \|_{L^2(\Omega_e)}\\
  &+ \sum_{c\in\Cset}\sum_{e\in\Eset_c}
   \sum_{\alpham\leq k}\|r_c^{(\alpham -\gamma_c)_+}\rho_{ce}^{(\alphaperpm-\gamma_e)_+}\dalpha v\|_{L^2(\Omega_{ce})}
\end{align*}
where $(x)_+ = \max\{0, x\}$. 
Moreover, we define the associated function space by
\begin{equation*}
  \cJ^k_\ugamma (\Omega; \Cset, \Eset) \coloneqq \bigg\{
  v\in L^2(\Omega): \| v\|_{\cJ^k_\ugamma(\Omega)}< \infty\bigg\}.
  \end{equation*}
Furthermore,
\begin{equation*}
  \cJ^\infty_\ugamma (\Omega;\Cset, \Eset) = \bigcap_{k\in \mathbb{N}_0} \cJ^k_\ugamma(\Omega;\Cset, \Eset).
\end{equation*}
For $A, C>0$, we define the space of weighted analytic functions with nonhomogeneous norm as 
\begin{equation}
    \label{eq:analytic}
    \begin{aligned}
    \cJ^{\varpi}_\ugamma(\Omega;\Cset, \Eset;C, A) \coloneqq \bigg\{v\in \cJ^\infty_\ugamma(\Omega;\Cset, \Eset):
{}&\sum_{\alpham=k}\|\dalpha v\|_{L^2(\Omega_0)}\leq CA^kk!,\\ 
  &\sum_{\alpham=k}\|r_c^{(\alpham -\gamma_c)_+}\dalpha v \|_{L^2(\Omega_c)}\leq CA^kk!\quad \forall c\in \Cset,\\
  &\sum_{\alpham=k}\|r_e^{(\alphaperpm -\gamma_e)_+}\dalpha v \|_{L^2(\Omega_e)}\leq CA^kk!\quad \forall e\in \Eset,\\
  &\sum_{\alpham=k}\|r_c^{(\alpham -\gamma_c)_+}\rho_{ce}^{(\alphaperpm-\gamma_e)_+}\dalpha v\|_{L^2(\Omega_{ce})}\leq CA^kk!\quad \\ 
  &\forall c\in \Cset\text{ and }\forall e\in \Eset_c,
  \forall k\in \mathbb{N}_0
  \bigg\}.
    \end{aligned}
  \end{equation}
  Finally, we denote
  \begin{equation*}
    \cJ^\varpi_\ugamma(\Omega;\Cset, \Eset) \coloneqq \bigcup_{C, A>0} \cJ^\varpi_\ugamma(\Omega;\Cset, \Eset; C, A).
  \end{equation*}

\subsection{Approximation of weighted analytic functions on tensor product geometric meshes}
\label{sec:hpTP}
The approximation result of weighted analytic functions via NNs that we present below is based on emulating an approximation strategy of tensor product $hp$-finite elements.
In this section, we present this $hp$-finite element approximation. Let $I \subset \R$ be an interval. 
A \emph{partition of $I$ into $N \in \N$ intervals} is a set $\cG$ such that $|\cG|= N$, 
all elements of $\cG$ are disjoint, connected, and open subsets of $I$, 
and 
$$
\bigcup_{U \in \cG} \overline{U} = \overline{I}. 
$$
We denote, for all $p\in \mathbb{N}_0$, by
$\mathbb{Q}_p(\cG)$ the piecewise polynomials of degree $p$ on $\cG$.
	% [{If this deviates from the usual definition of $\mathbb{Q}_p$, 
	% then we may need a different notation.}]

One specific partition of $I= [0,1]$ is given by the 
\emph{one dimensional geometrically graded grid}, 
which for $\sigma\in (0, 1/2]$ and $\ell\in \mathbb{N}$, is given by
\begin{equation} \label{eq:1dmesh}
\cG^\ell_{1} \coloneqq \left\{J^\ell_k, \, k=0, \dots, \ell\right\},\quad \text{where }  
J_0^\ell \coloneqq (0, \sigma^\ell)\quad\text{and} \quad J_k^{\ell} \coloneqq (\sigma^{\ell-k+1}, \sigma^{\ell-k}), \, k=1, \dots, \ell.
\end{equation}
\begin{theorem}\label{thm:Interface}
Let $d \in \{2,3\}$ and $Q \coloneqq (0,1)^d$. Let $\Cset =\{\corn\}$ where $\corn$ is
  one of the corners of $Q$ and
let $\Eset = \Eset_\corn$ contain the edges adjacent to $c$ when $d=3$,
$\Eset=\emptyset$ when $d=2$. 
Further assume given constants $C_f, A_f>0$, 
and  
\begin{alignat*}{3}
&\ugamma = \{\gamma_\corn: \corn\in \Cset\}, &&\text{with } \gamma_\corn>1,\;
\text{for all } \corn\in\Cset &&\text{ if } d = 2,\\ %
&\ugamma = \{\gamma_\corn, \gamma_e: \corn\in \Cset, e\in \Eset\}, \quad&&\text{with
} \gamma_\corn>3/2\text{ and } \gamma_e>1,\; \text{for all }\corn\in\Cset\text{
  and }e\in \Eset\quad &&\text{ if } d = 3. %
\end{alignat*}
 Then, there exist $C_p>0$, $C_L>0$ such that, for every $0< \epsilon<1$, 
 there exist $p, L \in \N$ with
 \begin{equation*}
  p \leq C_p(1+\left|\log(\epsilon)\right|),\quad L \leq C_L (1+\left|\log(\epsilon)\right|), 
 \end{equation*}
 so that there exist piecewise polynomials
 \begin{equation*}
 v_{i}\in \mathbb{Q}_p(\cG^L_1)\cap H^1(I),\qquad   i=1, \dots, \Noned,
 \end{equation*}
 with $\Noned = (L+1)p + 1$, and, for all $f\in \cJ^{\varpi}_\ugamma(Q;\Cset, \Eset;C_f,A_f)$
there exists a $d$-dimensional array of coefficients
\[ 
	c  = \left\{c_{\is}: (\iscomma) \in \{1, \dots, \Noned\}^d\right\}
 \]
 such that
\begin{enumerate}
\item \label{item:vli}For every $i = 1, \dots \Noned$,
    $\supp(v_{i})$ intersects either a single interval or 
   two neighboring subintervals of $\cG^L_1$. 
  Furthermore, there exist constants $C_v$, $b_v$ depending only on $C_f$,
  $A_f$, $\sigma$ such that
  \begin{equation}
    \label{eq:vli}
  \|v_{ i}\|_\infty \leq 1, \qquad  \|v_{i}\|_{H^1(I)} \leq  C_v \epsilon^{-b_v}, \qquad  \forall i=1, \dots, \Noned.
  \end{equation}
  \item\label{item:appx-eps} There holds
  \begin{equation}
    \label{eq:appx-eps}
    \|f - \sum_{\iscomma = 1}^{\Noned} c_{\is} \phi_{\is}\|_{H^1(Q)} \leq \epsilon \qquad \text{with}\qquad\phi_{\is} = \bigotimes_{j=1}^d v_{ i_j} ,\,\forall \iscomma=1, \dots, \Noned.
  \end{equation}
    \item\label{item:c}  $\|c\|_\infty \leq C_2 (1+\left| \log(\epsilon) \right|)^d$ and $\|c\|_1 \leq C_c (1+\left| \log(\epsilon) \right|)^{2d}$, for $C_2, C_c>0$ independent of $p$, $L$, $\epsilon$.
\end{enumerate}
\end{theorem}
We present the proof in Subsection \ref{subsec:ProofOfInterface} after developing an appropriate framework of $hp$-approximation in Section \ref{sec:hp-analysis}.

\section{Basic ReLU neural network calculus}
\label{sec:ReLUCalc}
 
In the sequel,
we distinguish between a neural network, as a collection of weights, 
and the associated \emph{realization of the NN}.
This is a function that is determined through the weights and an activation function. 
In this paper, we only consider the so-called ReLU activation:
\begin{equation*}%\label{eq:ReLU}
\varrho: \R \to \R: x \mapsto \max\{0, x\}.
\end{equation*}
\begin{definition}[{\cite[Definition 2.1]{PETERSEN2018296}}] 
\label{def:NeuralNetworks}
Let $d, L\in \N$. 
A \emph{neural network $\Phi$ with input dimension $d$ and $L$ layers} 
is a sequence of matrix-vector tuples 
\[
    \Phi = \big((A_1,b_1),  (A_2,b_2),  \dots, (A_L, b_L)\big), 
\]
where $N_0 \coloneqq d$ and $N_1, \dots, N_{L} \in \N$, and 
where $A_\ell \in \R^{N_\ell\times N_{\ell-1}}$ and $b_\ell \in \R^{N_\ell}$
for $\ell =1,...,L$.

For a NN $\Phi$, 
we define the associated
\emph{realization of the NN $\Phi$} as 
\[
 \mathrm{R}(\Phi): \R^d \to \R^{N_L} : x\mapsto x_L \rev{\eqqcolon} \mathrm{R}(\Phi)(x),
\]
where the output $x_L \in \R^{N_L}$ results from 
\begin{equation}
    \label{eq:NetworkScheme}
    \begin{split}
        x_0 &\coloneqq x, \\
        x_{\ell} &\coloneqq \varrho(A_{\ell} \, x_{\ell-1} + b_\ell) \quad \text{ for } \ell = 1, \dots, L-1,\\
        x_L &\coloneqq A_{L} \, x_{L-1} + b_{L}.
    \end{split}
\end{equation}
Here $\varrho$ is understood to act component-wise on vector-valued inputs, 
i.e., for $y = (y^1, \dots, y^m) \in \R^m$,  $\varrho(y) := (\varrho(y^1), \dots, \varrho(y^m))$.
We call $N(\Phi) \coloneqq d + \sum_{j = 1}^L N_j$ the \emph{number of neurons of the NN} 
$\Phi$, $\depth(\Phi)\coloneqq L$ the \emph{number of layers} or 
\emph{depth}, $\size_j(\Phi)\coloneqq \| A_j\|_{0} + \| b_j \|_{0}$ 
the \emph{number of nonzero weights in the $j$-th layer}, 
and
$\size(\Phi) \coloneqq \sum_{j=1}^L \size_j(\Phi)$ the \emph{number of nonzero weights of $\Phi$}, 
also referred to as its \emph{size}. 
We refer to $N_L$ as the \emph{dimension of the output layer of $\Phi$}.
\end{definition}

\subsection{Concatenation, parallelization, emulation of identity}
\label{S:ConcParEm}
An essential component in the ensuing proofs is to 
    construct NNs out of simpler building blocks. 
For instance, given two NNs, 
we would like to identify another NN 
so that the realization of it equals the sum or the composition of the first two NNs. 
To describe these operations precisely, 
we introduce a formalism of operations on NNs below. 
The first of these operations is the concatenation.

\begin{proposition}[NN concatenation, {{\cite[Remark 2.6]{PETERSEN2018296}}}]
\label{prop:conc}
% Remark 2.6
Let $L_1, L_2 \in \N$, and let 
$\Phi^1, \Phi^2$ 
be two NNs of respective depths $L_1$ and $L_2$ such that $N^1_0 = N^2_{L_2}\eqqcolon d$, i.e.,
the input layer of $\Phi^1$ has the same dimension as the output layer of $\Phi^2$. 

Then, there exists a NN $\Phi^1 \sconc \Phi^2$, called 
the \emph{sparse concatenation of $\Phi^1$ and $\Phi^2$}, 
such that $\Phi^1 \sconc \Phi^2$ has $L_1+L_2$ layers,   
$\mathrm{R}(\Phi^1 \sconc \Phi^2) = \mathrm{R}(\Phi^1) \circ \mathrm{R}(\Phi^2)$ %,
and 
$\size\left(\Phi^1 \sconc \Phi^2\right) \leq 2\size\left(\Phi^1\right)  + 2\size\left(\Phi^2\right)$.
\end{proposition}
The second fundamental operation on NNs is parallelization, 
achieved with the following construction.
\begin{proposition}[NN parallelization, {\cite[Definition 2.7]{PETERSEN2018296}}]\label{prop:parall}
% Definition 2.7
Let $L, d \in \N$ and let 
$\Phi^1, \Phi^2$ 
be two NNs with $L$ layers and with $d$-dimensional input each.
Then there exists a NN 
$\mathrm{P}(\Phi^1, \Phi^2)$ with $d$-dimensional input and $L$ layers, 
which we call the \emph{parallelization of $\Phi^1$ and $\Phi^2$}, 
such that 
\begin{equation*}
\mathrm{R}\left(\mathrm{P}\left(\Phi^1,\Phi^2\right)\right) (x) 
= 
\left(\mathrm{R}\left(\Phi^1\right)(x), \mathrm{R}\left(\Phi^2\right)(x)\right), 
\text{ for all } x \in \R^d
%\label{eq:ParallelizationDoesTheRightThing}
\end{equation*}
and
$\size(\mathrm{P}(\Phi^1, \Phi^2)) = \size(\Phi^1) + \size(\Phi^2)$.
\end{proposition}

Proposition \ref{prop:parall} requires two NNs to have the same depth. 
If two NNs have different depth, then we can artificially enlarge one of them by concatenating 
with a NN that implements the identity. One possible construction of such a NN is presented next. 

\begin{proposition}[NN emulation of $\mathrm{Id}$, {{\cite[Remark 2.4]{PETERSEN2018296}}}]\label{prop:Id}
% Remark 2.4
For every $d,L\in \N$ there exists a NN 
$\Phi^{\mathrm{Id}}_{d,L}$ with $\depth(\Phi^{\mathrm{Id}}_{d,L}) = L$
and
$\size(\Phi^{\mathrm{Id}}_{d,L}) \leq 2 d L$,
% [ for $L=1$ it holds that $M(\Phi^{\mathrm{Id}}_{d,1}) = d$ ... ]
%$\sizefirst(\Phi^{\mathrm{Id}}_{d,L})\leq 2 \rev{d}$ and 
%$\sizelast(\Phi^{\mathrm{Id}}_{d,L})\leq 2 \rev{d}$ 
such that 
$\mathrm{R} (\Phi^{\mathrm{Id}}_{d,L}) = \mathrm{Id}_{\R^d}$.
\end{proposition}

Finally, we sometimes require a parallelization of NNs that do not share inputs. 

\begin{proposition}[Full parallelization of NNs with distinct inputs, {\cite[Setting 5.2]{EGJS2018}}] 
\label{prop:parallSep}
% Setting 5.2 
Let $L \in \N$ and let
$$
\Phi^1 = \left(\left(A_1^1,b_1^1\right), \dots, \left(A_{L}^1,b_{L}^1\right)\right), 
\quad 
\Phi^2 = \left(\left(A_1^2,b_1^2\right), \dots, \left(A_{L}^2,b_{L}^2\right)\right)
$$
be two NNs with $L$ layers each and with 
input dimensions $N^1_0=d_1$ and $N^2_0=d_2$, respectively. 

Then there exists a NN, denoted by $\mathrm{FP}(\Phi^1, \Phi^2)$, 
with $d$-dimensional input where $d = (d_1+d_2)$ and $L$ layers, 
which we call the \emph{full parallelization of $\Phi^1$ and $\Phi^2$}, such that 
%$M(\mathrm{FP}(\Phi^1, \Phi^2)) = M(\Phi^1) + M(\Phi^2)$, and 
for all $x = (x_1,x_2) \in \R^d$ with $x_i \in \R^{d_i}, i = 1,2$ 
\begin{equation*}
%M(\mathrm{FP}(\Phi^1, \Phi^2)) = M(\Phi^1) + M(\Phi^2)\;,\quad 
\mathrm{R}\left(\mathrm{FP}\left(\Phi^1,\Phi^2\right)\right) (x_1,x_2) 
= 
  \left(\mathrm{R}\left(\Phi^1\right)(x_1), \mathrm{R}\left(\Phi^2\right)(x_2)\right)
\end{equation*}
and
$\size(\mathrm{FP}(\Phi^1, \Phi^2)) = \size(\Phi^1) + \size(\Phi^2)$.
%$\sizefirst(\mathrm{FP}(\Phi^1, \Phi^2)) = \sizefirst(\Phi^1) + \sizefirst(\Phi^2)$ 
%and
%$\sizelast(\mathrm{FP}(\Phi^1, \Phi^2)) = \sizelast(\Phi^1) + \sizelast(\Phi^2)$.
%
\end{proposition}
\begin{proof}
Set 
$\mathrm{FP}\left(\Phi^1,\Phi^2\right) \coloneqq \left(\left(A_1^3, b_1^3\right), \dots, \left(A_L^3, b_L^3\right)\right)$ 
where, for $j = 1, \dots, L$, we define
\begin{align*}
    A_{j}^3 \coloneqq \left(\begin{array}{cc}
        A_j^1 & 0 \\
        0 & A_j^2
    \end{array}\right) \text{ and } 
    b_{j}^3 \coloneqq \left(\begin{array}{c}
        b_j^1 \\
        b_j^2
    \end{array}\right).
\end{align*}
All properties of $\mathrm{FP}\left(\Phi^1,\Phi^2\right)$ claimed in the statement of the proposition follow immediately from the construction.
\end{proof}
\subsection{Emulation of multiplication and piecewise polynomials}
\label{S:EmMult}
In addition to the basic operations above, 
we use two types of functions that we can approximate especially efficiently with NNs. 
These are high dimensional multiplication functions and univariate piecewise polynomials. 
We first give the result of an emulation of a multiplication in arbitrary dimension.

\begin{proposition}[{\cite[Lemma C.5]{guhring2019error}, \cite[Proposition 2.6]{OSZ19_839}}]
\label{prop:Multiplication}
There exists a constant $C>0$ such that, for every $0<\eps< 1$, $d \in \N$ and $M \geq 1$ 
there is a NN $\Pi_{\epsilon, M}^{d}$ with $d$-dimensional input- and one-dimensional output, 
so that 
\begin{align*}
&\left|\prod_{\ell = 1}^d x_\ell  - \Realiz(\Pi_{\epsilon, M}^{d})(x)\right| \leq \epsilon, 
           \text{ for all } x=(x_1, \dots, x_d) \in [-M,M]^d, 
%\\
%&\left|\frac{\partial}{\partial x_j} 
% \prod_{\ell = 1}^d x_\ell  - \frac{\partial}{\partial x_j} 
%  \Realiz(\Pi_{\epsilon, M}^{d})(x)\right| \leq \epsilon, 
%  \text{ for almost every }x =(x_1, \dots, x_d) \in [-M,M]^{d}
%  \text{ and all } j = 1, \dots, d,
\\
&\left|\frac{\partial}{\partial x_j} 
 \prod_{\ell = 1}^d x_\ell  - \frac{\partial}{\partial x_j} 
  \Realiz(\Pi_{\epsilon, M}^{d})(x)\right| \leq \epsilon, 
  {\begin{aligned}\text{ for almost every }x =(x_1, \dots, x_d) \in [-M,M]^{d} \\
  \text{ and all } j = 1, \dots, d,\end{aligned}}
\end{align*}
and $\Realiz(\Pi_{\epsilon, M}^{d})(x) = 0$ if $\prod_{\ell=1}^dx_\ell = 0$,
for all $x = (x_1, \dots, x_d)\in \mathbb{R}^d$.
Additionally,  $\Pi_{\epsilon, M}^{d}$ satisfies
\begin{align*} 
\max\left\{ \depth\left(\Pi_{\epsilon, M}^{d}\right), \size\left(\Pi_{\epsilon, M}^{d}\right)\right\} 
\leq 
C \left( 1+ d \log(d M^d/\epsilon)\right).
\end{align*}
\end{proposition}

In addition to the high-dimensional multiplication, 
we can efficiently approximate univariate continuous, piecewise polynomial functions 
by realizations of NNs with the ReLU activation function.
\begin{proposition}[{\cite[Proposition 5.1]{OPS19_811}}]
\label{prop:relupwpolynom}
There exists a constant $C>0$ such that, 
for all $\bm p = (p_i)_{i\in\{1,\ldots,{N_{\mathrm{int}}}\}} \subset \N$,
for all partitions $\Tcal$ of $I=(0,1)$ into ${N_{\mathrm{int}}}$ open, disjoint, 
connected subintervals $I_i$, $i=1,\ldots,N_{\mathrm{int}}$, 
for all $v\in {S_{\bm p} (I,\Tcal)} \coloneqq \{v\in H^1(I): v|_{I_i} \in \bbP_{p_i}(I_i), i=1,\ldots,N_{\mathrm{int}}\}$, 
and for every $0<\eps< 1$, 
there exist NNs $\{\Phi^{v,\Tcal,\bm p}_{\eps}\}_{\eps\in(0,1)}$ 
such that for all $1\leq q'\leq \infty$ it holds that
\begin{align*}
\normc[W^{1,q'}(I)]{v-\mathrm{R}\left(\Phi^{v,\Tcal,\bm p}_{\eps}\right)}
	\leq &\, \eps \snormc[W^{1,q'}(I)]{v},
	\\
\depth\left(\Phi^{v,\Tcal,\bm p}_{\eps}\right)
	\leq &\, C (1+\log(p_{\max})) 
        \left( p_{\max} + \left|\log\eps\right| \right),
	\\
\size\left(\Phi^{v,\Tcal,\bm p}_{\eps}\right)
	\leq &\, C{N_{\mathrm{int}}} (1+\log(p_{\max})) \left( p_{\max} + \left|\log\eps\right| \right) 
                + C \sum_{i=1}^{N_{\mathrm{int}}} p_i\left(p_i + |\log\eps| \right) 
,
\end{align*}
where $p_{\max} \coloneqq \max \{p_i \colon i = 1, \dots, N_{\mathrm{int}}\}$.
In addition, 
$\mathrm{R}\left( \Phi^{v,\Tcal,\bm p}_{\eps} \right)(x_j)=v(x_j)$ 
for all $j\in\{0,\ldots,{N_{\mathrm{int}}}\}$, 
where $\{x_j\}_{j=0}^{N_{\mathrm{int}}}$ are the nodes of $\Tcal$. 
\end{proposition}

\begin{remark}\label{rem:RemarkGeneralIntervals}
It is not hard to see that the result holds also for $I = (a,b)$, 
where $a,b\in \R$, with $C>0$ depending on $(b-a)$. 
Indeed, 
for any $v \in H^1((a,b))$ the concatenation of $v$ with the invertible, 
affine map $T \colon x \mapsto (x-a)/(b-a)$ is in $H^1((0,1))$. 
Applying Proposition \ref{prop:relupwpolynom} yields 
NNs $\{\Phi^{v,\Tcal,\bm p}_{\eps}\}_{\eps\in(0,1)}$ approximating $v \circ T$ to an appropriate accuracy. 
Concatenating these networks with the $1$-layer NN $(A_1,b_1)$, 
where $A_1x + b_1 = T^{-1}x$ yields the result.
The explicit dependence of $C>0$ on $(b-a)$ 
can be deduced from the error bounds in $(0,1)$ by affine transformation.
\end{remark}

\section{Exponential approximation rates by realizations of NNs}
\label{sec:hpReapproxReLU}
We now establish several technical results on the \emph{exponentially consistent} 
approximation by realizations of NNs with ReLU activation 
of univariate and multivariate tensorized polynomials.
These results will be used to establish Theorem \ref{th:ReLUapprox},
which yields exponential approximation rates of NNs for functions 
in the weighted, analytic classes introduced in 
Section \ref{sec:WgtSpcNonHomNrm}. 
They are of independent interest, as they imply that 
spectral and pseudospectral methods can, in principle, 
be emulated by realizations of NNs with ReLU activation.
%%%%%%%%%%%%%%%%%%%%%%%%%%%%%%%%%%%%%%%%%%%%%%%%%%%%%%%%%%%%%%%%%%%%%%%%
\subsection{NN-based approximation of univariate, piecewise polynomial functions}
\label{S:1dBasFct}
%%%%%%%%%%%%%%%%%%%%%%%%%%%%%%%%%%%%%%%%%%%%%%%%%%%%%%%%%%%%%%%%%%%%%%%%
We start with the following corollary to Proposition \ref{prop:relupwpolynom}. 
It quantifies stability and consistency of realizations of NNs with ReLU
activation for the
emulation of the univariate, piecewise polynomial basis functions
in Theorem \ref{thm:Interface}.
\begin{corollary} \label{cor:basis-NN}
Let $I=(a,b)\subset \mathbb{R}$ be a bounded interval. 
Fix $C_p>0$, $C_v>0$, and $b_v>0$. 
Let $0<\epsilon_{\hp} < 1$ and $p, \Noned, N_{\mathrm{int}}\in \mathbb{N}$ 
be such that 
   $p \leq C_p(1+\left| \log\epsilon_{\hp} \right|)$ and let $\cG_\oned$ 
be a partition of $I$ into $\Nint$ open, disjoint, connected subintervals and, 
for $i\in\{1, \dots, \Noned\}$, let $v_i\in \mathbb{Q}_p(\cG_{\oned})
\cap H^1(I)$ be such that $\supp(v_i)$ intersects either a single interval 
or two adjacent intervals in $\cG_\oned$ and $ \|v_i\|_{H^1(I)}\leq C_v \epsilon_{\hp}^{-b_v}$, 
for all $i\in \{1, \dots, \Noned\}$.  

Then, 
for every $0 < \epsilon_1 \leq  \epsilon_{\hp}$, 
and 
for every $i\in\{1, \dots, \Noned\}$, 
there exists a NN $\Phi^{v_{i}}_{\epsilon_1}$ 
such that
\begin{align}
\normc[H^1(I)]{v_{i}-\Realiz\left(\Phi^{v_{i}}_{\epsilon_1}\right)}
	\leq{} & \epsilon_1 |v_i|_{H^1(I)} , \label{eq:Corboundvjk1}
	\\
\depth\left(\Phi^{v_{i}}_{\epsilon_1}\right)
	\leq{} &  C_4  (1 + \left|\log(\epsilon_1)\right|)(1 + \log(1+\left|\log(\epsilon_1)\right|))  ,\label{eq:Corboundvjk2}	\\
\size\left(\Phi^{v_{i}}_{\epsilon_1}\right) \leq{} & C_5 (1 +  \left|\log(\epsilon_1)\right|^2) 	, \label{eq:Corboundvjk3}
%\\
%\sizefirst\left(\Phi^{v_{i}}_{\epsilon_1}\right)
%	\leq{} & 24
%	,\label{eq:Corboundvjk4}
%	\\
%\sizelast\left(\Phi^{v_{i}}_{\epsilon_1}\right) \leq{} & 10, \label{eq:Corboundvjk5}
\end{align}
for constants $C_4, C_5>0$ depending on $C_p>0$, $C_v>0$, $b_v>0$ and $(b-a)$ only. 
%$\epsilon_1$ and $\epsilon_{\hp}$.
In addition, 
$\Realiz\left( \Phi^{v_i}_{\eps_1} \right)(x_j)=v_i(x_j)$ 
for all $i\in\{1,\ldots,\Noned\}$ and $j\in\{0,\ldots,{N_{\mathrm{int}}}\}$, 
where $\{x_j\}_{j=0}^{N_{\mathrm{int}}}$ are the nodes of $\cG_{\oned}$.
\end{corollary}
\begin{proof}
Let $i=1, \dots, \Noned$. 
For $v_{i}$ as in the assumption of the corollary, 
we have that 
either $\supp(v_{i}) = \overline{J}$ for a unique $J\in \cG_\oned$ 
or 
$\supp(v_{i}) =\overline{J \cup J'}$ for two neighboring intervals $J, J'\in \cG_\oned$.
Hence, there exists a partition $\Tcal_{i}$ of $I$ of at most four subintervals so that 
$v_{i} \in S_{\bm p} (I,\Tcal_{i})$,  where $\bm p = (p)_{i\in\{1,\ldots,4\}}$. %{N_{\mathrm{int}}}
% with $p$ as in Theorem \ref{thm:Interface}.

Because of this, 
an application of Proposition \ref{prop:relupwpolynom} with $q' = 2$ 
and Remark \ref{rem:RemarkGeneralIntervals}  yields that 
for every  $0<\epsilon_1 \leq \epsilon_{\hp}< 1$ 
there exists a 
NN $\Phi^{v_{i}}_{\epsilon_1} := \Phi^{v_{i},\Tcal_{i},\bm p}_{\epsilon_1}$ 
such that 
\eqref{eq:Corboundvjk1} %, \eqref{eq:Corboundvjk4}, and \eqref{eq:Corboundvjk5}
holds. 
In addition, by invoking 
$p \lesssim 1+\left|\log(\epsilon_\hp)  \right|\leq 1+\left|\log(\epsilon_1)  \right|$, 
we observe that
\begin{align*}
\depth\left(\Phi^{v_{i}}_{\epsilon_1}\right)
%	&\leq \, C_{\depth} (1+\log(p)) 
%        \left( 2p + \left|\log\left({\epsilon_1}\right)\right| \right) 
%		+ C_{\depth} \left|\log\left({\epsilon_1}\right)\right| + C\left(1+\log^3(p) \right)\\
	&\leq \, C (1+\log(p)) \left( p + \left|\log\left({\epsilon_1}\right)\right| \right) 
%	&\lesssim (1 + \log(1+\left|\log(\epsilon_1)\right|)) \left|\log(\epsilon_1)\right| + \left|\log(\epsilon_1)\right| + 1 + \log^3(1+\left|\log(\epsilon_1)\right|)\\
\lesssim 1 + \left|\log(\epsilon_1)\right|(1 + \log(1+\left|\log(\epsilon_1)\right|)). 
\end{align*}
Therefore, 
there exists $C_4 >0$ such that \eqref{eq:Corboundvjk2} holds.
Furthermore, 
\begin{align*}
\size\left(\Phi^{v_{i}}_{\epsilon_1}\right)
%	\leq &\, \rev{8}C_{\size} \sum_{i=1}^{4} p^2 
%		+ \rev{4}C_{\size} \left|\log\left({\epsilon_1}\right)\right| \sum_{i=1}^{4} p 
%		+ \left|\log\left({\epsilon_1}\right)\right| C\left(1+ \sum_{i=1}^{4} \log^{\rev{2}}(p) \right) 
%	\\
%	&\, + C\left(1+ \sum_{i=1}^{4} p \log^{\rev{2}}(p) \right) 
%	\\
%	&\, + 8 \left( C_{\depth} (1+\log(p)) \left( 2p + \left|\log\left(\epsilon_1\right)\right| \right) 
%		+ C\left(1+\log^3(p)\right)\right)\\
		\leq &\, 4C(1+\log(p)) \left( p + \left|\log\left(\epsilon_1\right)\right| \right) 
		+C \sum_{i=1}^{4} p(p+\left|\log\left({\epsilon_1}\right)\right| )\\
%\lesssim &\, p^2 +  \left|\log\left({\epsilon_1}\right)\right|  p  
%            + \left|\log\left(\epsilon_1\right)\right|\left(1+ \log^{\rev{2}}(p) \right) 
%            + \left(1+ p \log^{\rev{2}}(p) \right) 
%\\
%& +  (1+\log(p)) \left( p + \left|\log\left({\epsilon_1}\right)\right|\right)+1+\log^3(p).
\lesssim &\, p^2 +  \left|\log\left({\epsilon_1}\right)\right| p
	+  (1+\log(p)) \left( p + \left|\log\left({\epsilon_1}\right)\right|\right).
\end{align*}
We use 
$p \lesssim 1+\left|\log(\epsilon_1)\right|$ 
and obtain that 
there exists $C_5 >0$ such that \eqref{eq:Corboundvjk3} holds.
\end{proof}
%%%%%%%%%%%%%%%%%%%%%%%%%%%%%%%%%%%%%%%%%%%%%%%%%%%%%%%%%%%%%%%%%%%%%%%%%%%%%%%%%%%%%%%%%%%%%%%%%%%%%%%%%
\subsection{Emulation of functions with singularities in cubic domains by NNs}
\label{S:ApprCub}
Below we state a result describing the efficiency of re-approximating continuous, 
piecewise tensor product polynomial functions 
in a cubic domain, as introduced in Theorem \ref{thm:Interface},
by realizations of  NNs with the ReLU activation function.

\begin{theorem} \label{th:ReLU-hp}
Let $d\in \{2,3\}$, let $I = (a,b)\subset \mathbb{R}$ be a bounded interval, and let $Q=I^d$.
Suppose that there exist constants
   $C_p>0$, $C_{\Noned}>0$, $C_v>0$, $C_c>0$, $b_v>0$,
   and, for $0< \epsilon \leq 1$,
assume there exist $p, \Noned, N_{\mathrm{int}}\in \mathbb{N}$, 
and $c\in \mathbb{R}^{\Noned\times\dots\Noned}$, such that
   \begin{equation*}
       \Noned \leq C_{\Noned}(1+\left| \log\epsilon \right|^2),\quad \|c\|_{1} \leq C_{c}(1+\left| \log\epsilon \right|^{2d}),\quad
       p \leq C_p(1+\left| \log\epsilon \right|).
     \end{equation*}
Further, let $\cG_\oned$ be a partition of $I$ into $\Nint$ open, disjoint, connected subintervals and let, 
   for all $i\in\{1, \dots, \Noned\}$, $v_i\in\mathbb{Q}_p(\cG_{\oned}) \cap H^1(I)$ be 
   such that $\supp(v_i)$ intersects either a single interval or 
   two neighboring subintervals of $\cG_\oned$ and
     \begin{equation*}
      \|v_i\|_{H^1(I)}\leq C_v \epsilon^{-b_v}, \qquad \|v_i\|_{L^\infty(I)}\leq 1,\qquad \forall i\in \{1, \dots, \Noned\}.
   \end{equation*}

Then, 
%for every $0< \epsilon\leq 1$ 
there exists a NN $\Phi_{\epsilon, c}$ 
such that
\begin{align}\label{eq:ReLuhp-approx}
\left \| \sum_{\iscomma=1}^{\Noned} c_{\is}\bigotimes_{j=1}^dv_{i_j} - \Realiz\left(\Phi_{\epsilon, c}\right) \right\|_{H^1(Q)} 
\leq \epsilon.
\end{align}
Furthermore, 
%as $\epsilon \to 0$ 
there holds
$
\left\|\Realiz\left(\Phi_{\epsilon, c}\right) \right\|_{L^\infty(Q)} 
\leq  
(2^d+1)C_c (1 + \left| \log\epsilon \right|^{2d}),
$
\begin{equation*}
\size(\Phi_{\epsilon, c}) \leq C (1+\left|\log\epsilon\right|^{2d+1}),
\;
\depth(\Phi_{\epsilon, c}) \leq C (1+ \left|\log\epsilon\right|\log(\left|\log\epsilon\right|)),
\end{equation*}
where $C>0$ depends on $C_p$, $C_{\Noned}$, $C_v$, $C_c$, $b_v$, $d$, and $(b-a)$ only.
\end{theorem}
 \begin{proof}
Assume $I \neq \emptyset$ as otherwise there is nothing to show. Let $C_I\geq1$  be such that $C_I^{-1}\leq (b-a) \leq C_I$.
Let $c_{v, \mathrm{max}} \coloneqq \max\{\|v_i\|_{H^1(I)}\colon  i \in \{1, \dots, \Noned\}\} \leq C_v  \epsilon^{-b_v}$, 
let $\epsilon_1 \coloneqq \min \{\epsilon/ (2 \cdot d \cdot (c_{v, \mathrm{max}}+1)^{d} \cdot \|c\|_1), 1/2,
C_I^{-1/2}C_v^{-1}\epsilon^{b_v}\}$, 
and let $\epsilon_2 \coloneqq \min\{\epsilon /(2 \cdot (\sqrt{d}+1) \cdot (c_{v, \mathrm{max}}+1) \cdot \|c\|_1), 1/2 \}$. 

\paragraph{Construction of the neural network.}
Invoking Corollary \ref{cor:basis-NN} we choose, for $i=1, \dots, \Noned$, NNs 
$\Phi_{\epsilon_1}^{v_{i}}$ so that
\[
\left\|\Realiz(\Phi_{\epsilon_1}^{v_{i}}) - v_{ i}
\right\|_{H^1(I)} \leq  C_v \epsilon_1 \epsilon^{-b_v} \leq 1. 
\]
It follows that for all  $i \in \{1, \dots, \Noned\}$
\begin{align}
\left\|\Realiz\left(\Phi_{\epsilon_1}^{v_{i}}\right) \right\|_{H^1(I)} 
\leq & \left\|\Realiz\left(\Phi_{\epsilon_1}^{v_{i}}\right) - v_{i}\right\|_{H^1(I)} 
+ \left\|v_{i}\right\|_{H^1(I)} 
%\leq C_1 \eps_1 \eps^{-b_1} + c_{v,\max}
\leq 1+c_{v,\max}
\label{eq:vjiH1err}
\end{align}
and that, by Sobolev imbedding,
\begin{equation}
  \begin{aligned}
\left\|\Realiz\left(\Phi_{\epsilon_1}^{v_{i}}\right) \right\|_{\infty} 
&\leq \left\|\Realiz\left(\Phi_{\epsilon_1}^{v_{i}}\right) - v_{i}\right\|_{\infty} 
+ \left\|v_{i}\right\|_\infty 
\leq C_I^{1/2} \left\|\Realiz\left(\Phi_{\epsilon_1}^{v_{i}}\right) - v_{i}\right\|_{H^1(I)} + 1 
\\ &
\leq C_I^{1/2}C_v \epsilon_1 \epsilon^{-b_v} + 1 \leq 2.
  \end{aligned}
\label{eq:vjipointwise}
\end{equation}
Then, let $\Phibasis$ be the NN defined as
\begin{equation}
  \label{eq:Phibasis}
  \Phibasis \coloneqq \FPar\left( \Par(\Phi^{v_{ 1}}_{\epsilon_1}, \dots, \Phi^{v_{ \Noned}}_{\epsilon_1}),
    \dots,
    \Par(\Phi^{v_{ 1}}_{\epsilon_1}, \dots, \Phi^{v_{\Noned}}_{\epsilon_1})
  \right),
\end{equation}
where the full parallelization is of $d$ copies of $\Par(\Phi^{v_{ 1}}_{\epsilon_1}, \dots, \Phi^{v_{ \Noned}}_{\epsilon_1})$.
Note that $\Phibasis$ is a NN with $d$-dimensional input and
$d\Noned$-dimensional output.
Subsequently, we introduce the $\Noned^d$ matrices %linear operators 
$E^{(i_1, \dots, i_d)} \in \{0,1\}^{d\times d\Noned }$ such that, 
for all $(i_1, \dots, i_d)\in \{1, \dots,\Noned\}^d$,
\begin{equation*}
 E^{(i_1, \dots, i_d)} a = \{a_{(j-1)\Noned+i_j} : j=1, \dots,d\}
\qquad\text{for all }a = (a_1, \dots, a_{d\Noned})\in \mathbb{R}^{d\Noned}.
\end{equation*}
Note that, for all $(i_1, \dots, i_d)\in \{1, \dots, \Noned\}^d$,
\begin{equation*}
\Realiz((E^{(i_1, \dots, i_d)}, 0)\sconc \Phibasis) : (x_1, \dots, x_d)\mapsto \left\{ \Realiz(\Phi^{v_{i_j}}_{\epsilon_1} )(x_j): j=1, \dots, d \right\}.
\end{equation*}
Then, we set
\begin{align}\label{eq:ConstructionOfPhiEpsilon}
  \Phi_{\epsilon} \coloneqq \Par\left(\Pi_{\epsilon_2, 2}^d \sconc(E^{(i_1, \dots, i_d)}, 0) :(\iscomma)\in \{1, \dots, \Noned\}^d \right)\sconc \Phibasis,
\end{align}
where $\Pi_{\epsilon_2, 2}^d$ is according to Proposition \ref{prop:Multiplication}. Note that, by \eqref{eq:vjipointwise}, the inputs of $\Pi_{\epsilon_2, 2}^d$
are bounded in absolute value by $2$.
Finally, we define 
\[
  \Phi_{\epsilon, c} \coloneqq (( \vvec( c )^\top, 0 )) \sconc \Phi_\epsilon,
\]
where $\vvec(c) \in \R^{\Noned^d}$ is the reshaping such that, 
for all $(\iscomma)\in \{1, \dots, \Noned\}^d$
\begin{equation}\label{def:vec}
 ( \vvec(c))_{i}  = c_{\is}, \qquad \text{with } i = 1+\sum_{j=1}^d (i_j-1) \Noned^{j-1}.
 %i= \sum_{j=1}^d(j-1)\Noned + i_j.
\end{equation}
See Figure \ref{fig:NN-appx} for a schematic representation of the NN
$\Phi_{\epsilon, c}$.
\begin{figure}
  \centering
  \includegraphics[width=.8\textwidth]{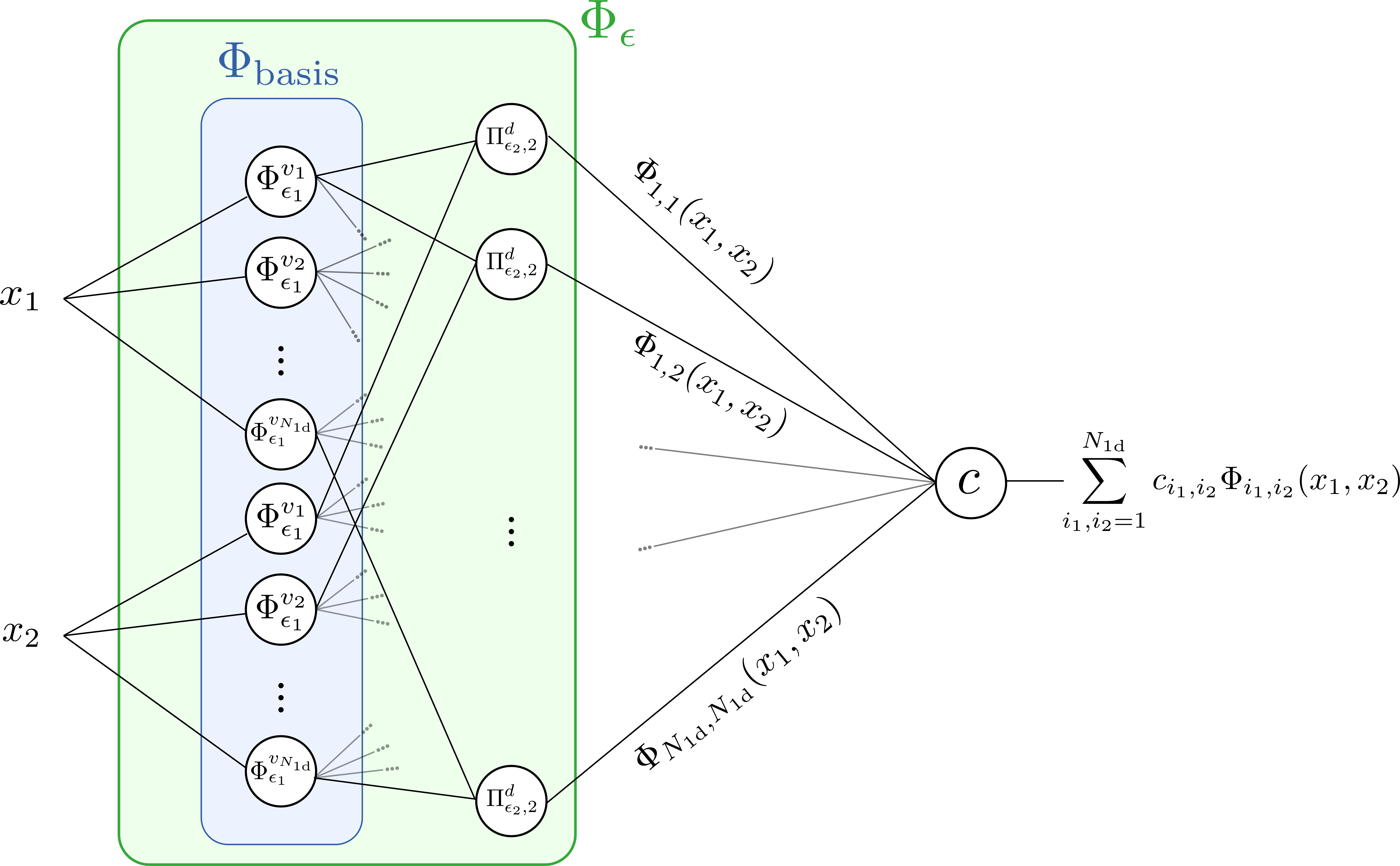}
  \caption{Schematic representation of the neural network $\Phi_{\epsilon, c}$,
    for the case $d=2$ constructed in the proof
    of Theorem \ref{th:ReLU-hp}. The circles represent subnetworks (i.e., the
    neural networks $\Phi^{v_{ i}}_{\epsilon_1}$, $\Pi^d_{\epsilon_2, 2}$, and $((\vvec(c)^\top,0))$).
    Along some branches, we indicate
    $\Phi_{i, k}(x_1, x_2) = \Realiz\left(\Pi^2_{\epsilon_2,2}\sconc ((E^{(i,k)}, 0))
    \sconc \Phibasis\right)(x_1, x_2)$.}
  \label{fig:NN-appx}
\end{figure}

\paragraph{Approximation accuracy.}
Let us now analyze if $\Phi_{\epsilon, c}$ has the asserted approximation
accuracy.
Define, for all $(\iscomma)\in \{1, \dots, \Noned\}^d$ 
\[
  \phi_{i_1\dots i_d} = \bigotimes_{j = 1}^d v_{ i_j} , 
\]
Furthermore, for each $(\iscomma)\in \{1, \dots, \Noned\}^d$, 
let $\Phi_{\is}$ denote the NNs 
\begin{equation*}
\Phi_{\is} = \Pi_{\epsilon_2, 2}^d \sconc((E^{(i_1, \dots, i_d)}, 0)) \sconc \Phibasis.
\end{equation*}
We estimate by the triangle inequality that
\begin{equation}
  \begin{aligned}
  \left\|\sumis c_\is \phi_\is - \Realiz(\Phi_{\epsilon, c}) \right\|_{H^1(Q)}
  &= \left\|\sumis c_\is \phi_\is - \sumis c_\is \Realiz(\Phi_{\is}) \right\|_{H^1(Q)} \\
  & \leq \sumis |c_\is| \left\| \phi_\is - \Realiz(\Phi_{\is}) \right\|_{H^1(Q)}.\label{eq:SecondTriangleInequality}
  \end{aligned}
\end{equation}
We have that \[
\left\| \phi_\is - \Realiz(\Phi_{\is}) \right\|_{H^1(Q)} = \left\|\bigotimes_{j = 1}^d v_{ i_j} - \Realiz\left(\Pi_{\epsilon_2, 2}^d\right) \circ \left[\Realiz\left(\Phi_{\epsilon_1}^{v_{i_1}}\right), \dots, \Realiz\left(\Phi_{\epsilon_1}^{v_{i_d}}\right)\right] \right\|_{H^1(Q)}
\]
and, by another application of the triangle inequality, we have that
\begin{align}
\left\| \phi_\is - \Realiz\left(\Phi_{\is}\right) \right\|_{H^1(Q)}  \leq & \left\|\bigotimes_{j = 1}^d v_{ i_j} - \bigotimes_{j = 1}^d \Realiz\left(\Phi_{\epsilon_1}^{v_{i_j}}\right)\right\|_{H^1(Q)}
\nonumber\\
& \qquad +  \left\|\bigotimes_{j = 1}^d \Realiz\left(\Phi_{\epsilon_1}^{v_{i_j}}\right) -\Realiz\left(\Pi_{\epsilon_2, 2}^d\right) \circ \left[\Realiz\left(\Phi_{\epsilon_1}^{v_{i_1}}\right), \dots, \Realiz\left(\Phi_{\epsilon_1}^{v_{i_d}}\right)\right] \right\|_{H^1(Q)} 
\nonumber\\
\leq & \left\|\bigotimes_{j = 1}^d v_{i_j} - \bigotimes_{j = 1}^d \Realiz\left(\Phi_{\epsilon_1}^{v_{i_j}}\right)\right\|_{H^1(Q)} + (\sqrt{d}+1)\epsilon_2(c_{v,\max}+1),
\label{eq:ThirdTriangleInequality}
\end{align}
where the last estimate follows from Proposition \ref{prop:Multiplication} and
the chain rule:
\begin{equation*}
    \left\|\bigotimes_{j = 1}^d \Realiz\left(\Phi_{\epsilon_1}^{v_{i_j}}\right) 
      -\Realiz\left(\Pi_{\epsilon_2, 2}^d\right) \circ \left[\Realiz\left(\Phi_{\epsilon_1}^{v_{i_1}}\right), \dots, \Realiz\left(\Phi_{\epsilon_1}^{v_{i_d}}\right)\right] \right\|_{L^2(Q)} 
    \leq \epsilon_2
  \end{equation*}
and
  \begin{align*}
    & \left|\bigotimes_{j = 1}^d \Realiz\left(\Phi_{\epsilon_1}^{v_{i_j}}\right) 
-\Realiz\left(\Pi_{\epsilon_2, 2}^d\right) \circ \left[\Realiz\left(\Phi_{\epsilon_1}^{v_{i_1}}\right), \dots, \Realiz\left(\Phi_{\epsilon_1}^{v_{i_d}}\right)\right] \right|_{H^1(Q)}^2
\\ 
&\qquad =  
\sum_{k=1}^d 
 \left\| \frac{\partial}{\partial x_k} \bigotimes_{j = 1}^d \Realiz\left(\Phi_{\epsilon_1}^{v_{i_j}}\right) 
 - \frac{\partial}{\partial x_k} \Realiz\left(\Pi_{\epsilon_2, 2}^d\right) \circ \left[\Realiz\left(\Phi_{\epsilon_1}^{v_{i_1}}\right), \dots, \Realiz\left(\Phi_{\epsilon_1}^{v_{i_d}}\right)\right] \right\|_{L^2(Q)}^2
 \\
& \qquad =  
\sum_{k=1}^d
\left\| \left( \bigotimes_{\substack{j = 1 \\ j\neq k}}^d \Realiz\left(\Phi_{\epsilon_1}^{v_{i_j}}\right) 
 - \left( \frac{\partial}{\partial x_k} \Realiz\left(\Pi_{\epsilon_2, 2}^d\right) \right) \circ \left[\Realiz\left(\Phi_{\epsilon_1}^{v_{i_1}}\right), \dots, \Realiz\left(\Phi_{\epsilon_1}^{v_{i_d}}\right)\right] \right)
\left( \frac{\partial}{\partial x} \Realiz\left(\Phi_{\epsilon_1}^{v_{i_k}} \right) \right)
\right\|_{L^2(Q)}^2
\\
& \qquad \leq  \sum_{k=1}^d \epsilon_2^2
\left\| \frac{\partial}{\partial x} \Realiz\left(\Phi_{\epsilon_1}^{v_{i_k}} \right) \right\|_{L^2(I)}^2
\leq d \epsilon_2^2 (c_{v,\max}+1)^2,
\end{align*}%
where we used \eqref{eq:vjiH1err}.
We now use \eqref{eq:vjipointwise} to bound the first term in
\eqref{eq:ThirdTriangleInequality}: for $d = 3$, 
we have that, for all $(\iscomma) \in \{1, \dots, \Noned\}^d$, 
\begin{align*}
\left\|\bigotimes_{j = 1}^d v_{ i_j} - \bigotimes_{j = 1}^d \Realiz\left(\Phi_{\epsilon_1}^{v_{i_j}}\right)\right\|_{H^1(Q)} &\leq \left\|(v_{ i_1} - \Realiz(\Phi_{\epsilon_1}^{v_{i_1}}))\otimes \bigotimes_{j = 2}^d v_{i_j}\right\|_{H^1(Q)}\\
&\qquad  + \left\|\Realiz\left(\Phi_{\epsilon_1}^{v_{i_j}}\right) \otimes \left(v_{ i_2} - \Realiz\left(\Phi_{\epsilon_1}^{v_{i_2}}\right)\right) \otimes v_{ i_d}\right\|_{H^1(Q)}\\
&\qquad \qquad + \left\|\bigotimes_{j = 1}^{d-1}\Realiz(\Phi_{\epsilon_1}^{v_{i_j}}) \otimes (v_{ i_d} - \Realiz(\Phi_{\epsilon_1}^{v_{i_d}}))\right\|_{H^1(Q)} \eqqcolon \mathrm{(I)}.
\end{align*}
For $d = 2$, we end up with a similar estimate with only two terms.  
By the tensor product structure, it is clear that $\mathrm{(I)} \leq d \epsilon_1 (c_{v, \mathrm{max}}+1)^{d}$. We have from \eqref{eq:SecondTriangleInequality} and the considerations above that 
\begin{align*}
\left\|\sumis c_\is \phi_\is - \Realiz\left(\Phi_{\epsilon, c}\right) \right\|_{H^1(Q)} \leq \|c\|_1 \left( d \epsilon_1 (c_{v, \mathrm{max}}+1)^{d} + (\sqrt{d}+1) \epsilon_2 (c_{v,\max}+1) \right) \leq \epsilon.
\end{align*}
This yields \eqref{eq:ApproximationPartOfReLUStatement}.

\paragraph{Bound on the $L^\infty$ norm of the neural network.}
  As we have already shown, 
  $\left\|\Realiz\left(\Phi_{\epsilon_1}^{v_{i}}\right) \right\|_{\infty} \leq 2$.
  Therefore, by Proposition \ref{prop:Multiplication},
  $  \left\|\Realiz\left(\Phi_\epsilon\right) \right\|_{\infty} \leq 2^d + \epsilon_2$.
  It follows that
  \begin{equation*}
  \left\|\Realiz\left(\Phi_{\epsilon, c}\right) \right\|_{\infty} 
  \leq 
  \|c\|_1\left(2^d + \epsilon_2  \right) \leq  (2^d+1)C_c (1 + \left| \log\epsilon \right|^{2d}).
  \end{equation*}
\paragraph{Size of the neural network.} 
Bounds on the size and depth of $\Phi_{\epsilon, c}$ follow 
from Proposition \ref{prop:Multiplication} and Corollary \ref{cor:basis-NN}. 
Specifically, we start by remarking that 
there exists a constant $C_1 > 0$ depending on $C_v$,
$b_v$, $C_I$ and $d$ only, such that 
$\left|\log(\epsilon_1)\right|\leq C_1 (1+\left| \log\epsilon  \right|)$. 
Then, by Corollary \ref{cor:basis-NN}, there exist constants 
$C_{4}$, $C_{5}>0$ depending on $C_p, C_v, b_v, (b-a),$ and $d$ only
such that for all $i=1, \dots, \Noned$,
\begin{equation*}
\depth\left(\Phi^{v_{i}}_{\epsilon_1}\right)
	\leq{}   C_{4}  (1 + \left|\log\epsilon\right|)(1 + \log(1+\left|\log\epsilon\right|)) \text{ and }
 \size\left(\Phi^{v_{i}}_{\epsilon_1}\right) \leq{}  C_{5} (1 +  \left|\log\epsilon\right|^2).
\end{equation*}
Hence, by Propositions \ref{prop:parallSep} and \ref{prop:parall}, 
there exist $C_{6}$, $C_{7}>0$ depending on $C_p, C_v, b_v, (b-a),$ and $d$ only
such that
\begin{equation*} 
\depth(\Phibasis)\leq C_{6}  (1 + \left|\log\epsilon\right|)(1 + \log(1+\left|\log\epsilon\right|)) \text{ and }
  \size(\Phibasis)\leq C_{7} d \Noned (1 +  \left|\log\epsilon\right|^2).
\end{equation*}
Then, 
remarking that for all $(\iscomma)\in \{1, \dots, \Noned\}^d$ there holds
$\|E^{(\iscomma)}\|_0 =d$ and, 
by Propositions \ref{prop:conc}, \ref{prop:Multiplication}, and \ref{prop:parall}, 
we have
\begin{equation*}
  \depth(\Phi_{\epsilon})\leq C_{8} (1 + \left|\log\epsilon\right|)(1 + \log(1+\left|\log\epsilon\right|)) , 
\;\;
%\end{equation*}
%and
%\begin{equation*}
  \size(\Phi_{\epsilon})\leq C_{9}\left( \Noned^d(1+\left| \log\epsilon \right|)  + \size(\Phibasis)\right).
\end{equation*}
For $C_{8}, C_{9}>0$ depending on $C_p, C_v, b_v, (b-a), d$ and $C_c$ only.
Finally, we conclude that there exists a
constant $C_{10}>0$ depending on $C_p, C_v, b_v, (b-a), d$ and $C_c$ only 
such that
\begin{equation*}
\depth( \Phi_{\epsilon, c} ) 
\leq 
C_{10}(1 + \left|\log\epsilon\right|)(1 + \log(1+\left|\log\epsilon\right|)).
\end{equation*}
Using also the fact that $\Noned \leq C
(1+\left| \log\epsilon \right|^2)$ for $C>0$ independent of $\epsilon$ and since $d\geq 2$,
\begin{equation*}
\size(\Phi_{\epsilon, c})\leq C_{11} (1+\left|\log\epsilon\right|)^{2d+1},
\end{equation*}
for a constant $C_{11}>0$ depending on $C_p, C_v, b_v, (b-a), d$ and $C_c$ only. 
\end{proof}

Next, we state our main approximation result, which describes the approximation of singular functions in
$(0,1)^d$ by realizations of NNs.
\begin{theorem}
  \label{th:ReLUapprox}
Let $d \in \{2,3\}$ and $Q \coloneqq (0,1)^d$. Let $\Cset =\{\corn\}$ where $\corn$ is
  one of the corners of $Q$ and
let $\Eset = \Eset_\corn$ contain the edges adjacent to $c$ when $d=3$,
$\Eset=\emptyset$ when $d=2$.
Assume furthermore that $C_f, A_f>0$, and  
\begin{alignat*}{3}
&\ugamma = \{\gamma_\corn: \corn\in \Cset\}, &&\text{with } \gamma_\corn>1,\;
\text{for all } \corn\in\Cset &&\text{ if } d = 2,\\ %
&\ugamma = \{\gamma_\corn, \gamma_e: \corn\in \Cset, e\in \Eset\}, \quad&&\text{with
} \gamma_\corn>3/2\text{ and } \gamma_e>1,\; \text{for all }\corn\in\Cset\text{
  and }e\in \Eset\quad &&\text{ if } d = 3. %
\end{alignat*}
%
%%%(i.e. $\gamma_e$ is the same for all singular edges).
 Then, for every $f\in \cJ^{\varpi}_\ugamma(Q;\Cset, \Eset;C_f,A_f)$ and every $0< \epsilon <1$, there exists a NN $\Phi_{\epsilon, f}$ so that 
\begin{equation}\label{eq:ApproximationPartOfReLUStatement}
\left \| f - \Realiz\left(\Phi_{\epsilon, f}\right) \right\|_{H^1(Q)} \leq \epsilon.
\end{equation}
In addition, $\|\Realiz\left(\Phi_{\epsilon, f}\right)\|_{L^\infty(Q)} = \mathcal{O}(\left| \log\epsilon \right|^{2d})$ for $\epsilon \to 0$. Also, $\size(\Phi_{\epsilon, f}) = \mathcal{O}(\left|\log\epsilon\right|^{2d+1})$ 
and 
$\depth(\Phi_{\epsilon, f}) = \mathcal{O}(\left|\log\epsilon\right|\log(\left|\log\epsilon\right|))$, 
for $\epsilon \to 0$.
\end{theorem}
\begin{proof}
Denote $I\coloneqq (0,1)$ and let  $f\in \mathcal{J}^{\varpi}_\ugamma(Q;\Cset, \Eset;
  C_f,A_f)$ and $0<\epsilon<1$.
Then, by Theorem \ref{thm:Interface} (applied with $\epsilon/2$ instead of $\epsilon$)
there exists $\Noned\in \N$ so that $\Noned =
\mathcal{O}((1+\left|\log\epsilon\right|)^{2})$, $c \in
\R^{\Noned\times\dots\times\Noned}$ with $\|c\|_1 \leq C_c (1+\left|\log\epsilon\right|^{2d})$, and, for all $(i_1,
  \dots, i_d)\in\{1, \dots, \Noned\}^d$,
\[
  \phi_{i_1\dots i_d} = \bigotimes_{j = 1}^d v_{ i_j} , 
\]
such that the hypotheses of Theorem \ref{th:ReLU-hp} are met, and 
\[
	\left\|f - \sum_{i_1, \dots i_d=1}^{\Noned} c_{i_1\dots i_d} \phi_{i_1\dots i_d} \right\|_{H^1(Q)} \leq \frac\epsilon2.
\]
We have, by Theorem \ref{thm:Interface} and the triangle inequality, that for $\Phi_{\epsilon, f} := \Phi_{\epsilon/2, c}$
\begin{align*}
\left\|f - \Realiz(\Phi_{\epsilon, f}) \right\|_{H^1(Q)} 
	\leq \frac\epsilon2 + \left\|\sumis c_\is \phi_\is - \Realiz(\Phi_{\epsilon/2, c}) \right\|_{H^1(Q)}.
\end{align*}
Then, the application of Theorem \ref{th:ReLU-hp} (with $\epsilon/2$ instead of
$\epsilon$) concludes the proof of
  \eqref{eq:ApproximationPartOfReLUStatement}. Finally, the bounds on
  $\depth(\Phi_{\epsilon, f}) = \depth(\Phi_{\epsilon/2, c})$,
  $\size(\Phi_{\epsilon, f}) = \size(\Phi_{\epsilon/2, c})$, and on
  $\|\Realiz(\Phi_{\epsilon, f})\|_{L^\infty(Q)} = \|\Realiz(\Phi_{\epsilon/2,
    c})\|_{L^\infty(Q)} $ follow from the corresponding estimates of Theorem \ref{th:ReLU-hp}.
\end{proof}

Theorem \ref{th:ReLUapprox} admits a straightforward generalization to functions with multivariate output, so that each coordinate is a weighted analytic function with the same regularity.
Here, we denote for a NN $\Phi$ with $N$-dimensional output, $N\in \mathbb{N}$, by $\Realiz(\Phi)_n$ the $n$-th component of the output (where $n\in \{1, \dots, N\}$).
\begin{corollary}
  \label{cor:ReLUapprox-vector}
Let $d \in \{2,3\}$ and $Q \coloneqq (0,1)^d$. Let $\Cset =\{\corn\}$ where $\corn$ is
  one of the corners of $Q$ and
let $\Eset = \Eset_\corn$ contain the edges adjacent to $\corn$ when $d=3$;
$\Eset=\emptyset$ when $d=2$. Let $N_f\in \mathbb{N}$. Further assume that $C_f,
  A_f>0$, and  
\begin{alignat*}{3}
&\ugamma = \{\gamma_\corn: \corn\in \Cset\}, &&\text{with } \gamma_\corn>1,\;
\text{for all } \corn\in\Cset &&\text{ if } d = 2,\\ %
&\ugamma = \{\gamma_\corn, \gamma_e: \corn\in \Cset, e\in \Eset\}, \quad&&\text{with
} \gamma_\corn>3/2\text{ and } \gamma_e>1,\; \text{for all }\corn\in\Cset\text{
  and }e\in \Eset\quad &&\text{ if } d = 3. %
\end{alignat*}
% \begin{alignat*}{3}
% &\ugamma = \{\gamma_\corn: \corn\in \Cset\}, &&\text{with }\forall \corn\in\Cset: \gamma_c>1, &&\text{ if } d = 2,\\ %
% &\ugamma = \{\gamma_\corn, \gamma_e: \corn\in \Cset, e\in \Eset\}, \quad&&\text{with
% }\forall \corn\in\Cset: \gamma_c>3/2 \text{ and }\forall e\in\Eset: \gamma_e>1,\quad &&\text{ if } d = 3. %
% \end{alignat*}%
 Then, for all $\fbf = (f_1, \dots, f_{N_f}) \in
 \left[\cJ^{\varpi}_\ugamma(Q;\Cset, \Eset;C_f,A_f)  \right]^{N_f}$ and every $0< \epsilon
 <1$, there exists a NN $\Phi_{\epsilon, \fbf}$ with $d$-dimensional
 input and $N_f$-dimensional output such that, for all $ n=1, \dots, N_f$,
\begin{equation}\label{eq:ReLUapprox-vector}
\left \| f_n - \Realiz\left(\Phi_{\epsilon, \fbf}\right)_n \right\|_{H^1(Q)} \leq \epsilon.
\end{equation}
In addition,  
$\|\Realiz(\Phi_{\epsilon, \fbf})_n\|_{L^\infty(Q)} = \mathcal{O}(\left| \log\epsilon \right|^{2d})$ 
for every $n = \{1, \dots, N_f\}$,  
$\size(\Phi_{\epsilon, f}) 
 = 
 \mathcal{O}(\left|\log\epsilon\right|^{2d+1} + N_f\left|\log\epsilon\right|^{2d})$ 
and 
$\depth(\Phi_{\epsilon, f}) = \mathcal{O}(\left|\log\epsilon\right|\log(\left|\log\epsilon\right|))$, 
for $\epsilon \to 0$.
\end{corollary}
\begin{proof}
Let $\Phi_\epsilon$ be as in \eqref{eq:ConstructionOfPhiEpsilon} 
and 
let $c^{(n)} \in\R^{\Noned\times\cdots\times\Noned}$, $n=1, \dots, N_f$ 
be the matrices of coefficients such that, in the notation of the proof of 
Theorems \ref{th:ReLU-hp} and \ref{th:ReLUapprox}, 
for all $n\in \{1, \dots, N_f\}$,
\[
	\left\|f_n - \sum_{i_1, \dots i_d=1}^{\Noned} c^{(n)}_{i_1\dots i_d} \phi_{i_1\dots i_d} \right\|_{H^1(Q)} \leq \frac\epsilon2.
\]
We define, for $\vvec$ as defined in \eqref{def:vec}, the NN $\Phi_{\epsilon, \fbf}$ as
\begin{equation*}
   \Phi_{\epsilon, \fbf}  \coloneqq \Par\left(((\vvec(c^{(1)}  )^\top, 0)), \dots, ((\vvec(c^{(N_f)}  )^\top,0))\right) \sconc \Phi_\epsilon.
  \end{equation*}%
  The estimate \eqref{eq:ReLUapprox-vector} and the $L^\infty$-bound then follow from Theorem
  \ref{th:ReLU-hp}. The bound on $\depth(\Phi_{\epsilon, \fbf})$ follows
  directly from Theorem \ref{th:ReLU-hp} and Proposition \ref{prop:conc}. Finally, the bound on $\size(\Phi_{\epsilon,
    \fbf})$ follows by Theorem \ref{th:ReLU-hp} and Proposition \ref{prop:conc}, as well as, from the observation that
  \begin{equation*}
    \size\left( \Par\left(((\vvec(c^{(1)}  )^\top,0)), \dots, ((\vvec(c^{(N_f)}  )^\top, 0 ))\right) \right)\leq N_f \Noned^{d}\leq C N_f (1+\left| \log\epsilon \right|^{2d}),
  \end{equation*}
  for a constant $C>0$ independent of $N_f$ and $\epsilon$.
\end{proof}

\section{Exponential %neural network 
         expression rates for solution classes of PDEs}
\label{sec:applications}
In this section, 
we develop Theorem \ref{th:ReLUapprox} into several exponentially decreasing 
upper bounds for the rates of approximation, by realizations of NNs
with ReLU activation, for
solution classes to elliptic PDEs with singular data (such as 
singular coefficients or domains with nonsmooth boundary).
In particular,
we consider elliptic PDEs in two-dimensional \emph{general} polygonal
domains, in three-dimensional domains that are a union of cubes,
and elliptic eigenvalue problems with isolated point singularities 
in the potential which arise in models of electron structure in
quantum mechanics.

In each class of examples, 
the solution sets belong to the class of weighted analytic functions 
introduced in Subsection \ref{sec:WgtSpcNonHomNrm}. However, the approximation
rates established in Section \ref{sec:hpReapproxReLU} only hold on tensor
product domains with singularities on the boundary. 
Therefore, we will first extend the exponential NN approximation rates to 
functions which exhibit singularities on a set of isolated
points internal to the domain, arising from singular
potentials of nonlinear Schrödinger operators.
In Section \ref{sec:polygonal}, we demonstrate,
using an argument based on a partition of unity, that the approximation problem 
on general polygonal domains can be reduced to 
that on tensor product domains and Fichera-type domains,
and establish exponential NN expression rates for 
linear elliptic source and eigenvalue problems.
In Section \ref{sec:EllPDEFichera}, 
we show  exponential NN expression rates for classes of
weighted analytic functions on two- and three-dimensional Fichera-type domains. 
%%%%%%%%%%%%%%%%%%%%%%%%%%%%%%%%%%%%%%%%%%%%%%%%%%%%%%%%%%%%%%%%%%%%%%%%%%%%%%%%%%%%%
\subsection{Nonlinear eigenvalue problems with isolated point singularities}
\label{sec:EVPPtSing}
%%%%%%%%%%%%%%%%%%%%%%%%%%%%%%%%%%%%%%%%%%%%%%%%%%%%%%%%%%%%%%%%%%%%%%%%%%%%%%%%%%%%
Point singularities emerge in the solutions of elliptic eigenvalue problems,
as arise, for example, for electrostatic interactions between charged particles 
that are modelled mathematically as point sources in $\R^3$.
Other problems that exhibit point singularities appear in general relativity, 
and for electron structure models in quantum mechanics.
We concentrate here on the expression rate of ``ab initio'' 
NN approximation of the electron density near isolated singularities of the
nuclear potential.
Via a ReLU-based partition of unity argument, an exponential approximation rate bound
for a single, isolated point singularity in Theorem \ref{prop:Schrodinger}
is extended in Corollary \ref{coro:multnucl} to
electron densities corresponding to 
potentials with multiple point singularities
at a priori known locations, 
modeling (static) molecules.

The numerical approximation in ab initio 
electron structure computations with NNs has been
recently reported to be competitive with other established, methodologies 
(e.g. \cite{pfau2019abinitio,hermann2019deep} and the references there). 
The exponential ReLU expression rate bounds obtained here can, in part, underpin
competitive performances of NNs in (static) electron structure computations.

We recall that all NNs are realized with the ReLU activation
function, see \eqref{eq:NetworkScheme}.
%%%%%%%%%%%%%%%%%%%%%%%%%%%%%%%%%%%%%%%%%%%%%%%%%%%%%%%%%%%
\subsubsection{Nonlinear Schr\"{o}dinger equations}
\label{sec:NonlSchrEq}
%%%%%%%%%%%%%%%%%%%%%%%%%%%%%%%%%%%%%%%%%%%%%%%%%%%%%%%%%%%
Let $\Omega = \mathbb{R}^d/(2\mathbb{Z})^d$, 
where $d \in \{2,3\}$, be a flat torus and 
let $V:\Omega\to \mathbb{R}$ be a potential such that 
$V(x)\geq V_0>0$ for all $x\in \Omega$ and 
there exists $\delta>0$ and $A_V>0$ such that
\begin{equation}
  \label{eq:V}
  \|r^{2+\alpham-\delta} \dalpha V\|_{L^\infty(\Omega)} 
  \leq 
  A_V^{\alpham+1}\alpham!\qquad \forall \alpha\in \mathbb{N}_0^d,
\end{equation}
where $r(x) = \dist(x, (0, \dots, 0))$.
For $k \in \{0, 1, 2\}$, we introduce the Schr\"{o}dinger eigenproblem that consists in finding
the smallest eigenvalue $\lambda \in \R$ and an associated eigenfunction $u \in
H^1(\Omega)$ such that
\begin{equation}
  \label{eq:Schrodinger}
  \begin{aligned}
  (-\Delta +V +|u|^k) u = \lambda u \quad \text{in }\Omega, \quad 
  \|u\|_{L^2(\Omega)} = 1.
  \end{aligned}
\end{equation}

There holds the following approximation result.
%%%%%%%%%%%%%%%%%%%%%%%%%%%%%%%%%%%%%%%%%%%%%%%%%%%%%%%%%%%%%%%%%%%%%%%%%%%%%%%%%%%%%%%%%%%%%
\begin{theorem}
\label{prop:Schrodinger}
Let $k \in \{0,1,2\}$ and $(\lambda, u)\in \mathbb{R}\times H^1(\Omega)\backslash \{ 0 \}$ 
be a solution of the eigenvalue problem \eqref{eq:Schrodinger}
with minimal $\lambda$, where $V$ satisfies \eqref{eq:V}. 

Then, for every $0< \epsilon \leq 1$ 
there exists a NN $\Phi_{\epsilon, u}$ %as per our discussion last week, the phrase "ReLU NN" should not be used due to its ambiguity and mathematical inaccuray.
such that
\begin{equation}\label{eq:SchrodingerNN}
\left \| u - \Realiz\left(\Phi_{\epsilon, u}\right) \right\|_{H^1(Q)} \leq \epsilon.
\end{equation}
In addition, as $\epsilon \to 0$, 
$$
\size(\Phi_{\epsilon, u}) = \mathcal{O}(|\log(\epsilon)|^{2d+1}),
\;
\depth(\Phi_{\epsilon, u}) = \mathcal{O}(|\log(\epsilon)|\log(|\log(\epsilon)|))
\;.
$$
\end{theorem}
\begin{proof}
Let $\Cset = \{(0, \dots, 0)\}$ and $\Eset=\emptyset$. 
The regularity of $u$
is a consequence of \cite[Theorem 2]{Maday2019b} (see also \cite[Corollary 3.2]{MadMarc2019}
for the linear case $k=0$):
  there exists $\gamma_\corn> d/2$ and $C_u, A_u>0$ such that $u\in
  \cJ^\varpi_{\gamma_\corn}(\Omega; \Cset, \Eset; C_u, A_u)$.  Here, $\gamma_\corn$
  and the constants $C_u$ and $A_u$ depend only on, $V_0$, $A_V$ and $\delta$ in
  \eqref{eq:V}, and on $k$ in \eqref{eq:Schrodinger}.

  Then, for all $0 < \epsilon \leq 1$,  
  by Theorem \ref{th:ReLU-hp} and Proposition \ref{prop:internal},  
  there exists a NN $\Phi_{\epsilon, u}$ 
  such that \eqref{eq:SchrodingerNN} holds. 
  Furthermore,
  there exist constants $C_1$, $C_2 > 0$
  dependent only on $V_0$, $A_V$, $\delta$, and $k$,
  such that
  \begin{equation*}
    \size(\Phi_{\epsilon, u}) \leq C_1(1 + |\log(\epsilon)|^{2d+1}) \text{ and } 
    \depth(\Phi_{\epsilon, u}) \leq C_2 \big(1+|\log(\epsilon)|\big)\big(1+\log(1+|\log(\epsilon)|)\big).
  \end{equation*}
\end{proof}
%%%%%%%%%%%%%%%%%%%%%%%%%%%%%%%%%%%%%%%%%%%%%%%%%%%%%%%%%%%%%5
\subsubsection{Hartree-Fock model}
\label{sec:HF}
%%%%%%%%%%%%%%%%%%%%%%%%%%%%%%%%%%%%%%%%%%%%%%%%%%%%%%%%%%%%%%
The Hartree-Fock model is an approximation of the full many-body representation
of a quantum system under the Born-Oppenheimer approximation, where the many-body wave function is replaced
by a sum of Slater determinants. Under this hypothesis, for $M, N \in \N$, the
Hartree-Fock energy of a system with $N$ electrons and $M$ nuclei with positive charges $Z_i$ at
isolated locations $R_i\in \mathbb{R}^3$, reads
\begin{multline}
  \label{eq:EHF}
  E_{\mathrm{HF}} = \inf\bigg\{
    \sum_{i=1}^N\int_{\mathbb{R}^3}\left(|\nabla\varphi_i|^2 + V|\varphi_i|^2   \right)
    +\frac{1}{2} \int_{\mathbb{R}^3}\int_{\mathbb{R}^3} \frac{\rho(x)\rho(y)}{|x-y|}dxdy
    -\frac{1}{2} \int_{\mathbb{R}^3}\int_{\mathbb{R}^3} \frac{\tau(x,y)^2}{|x-y|}dxdy
    :\\
(\varphi_1, \dots, \varphi_N)\in
H^1(\mathbb{R}^3)^N\text{ such that }\int_{\mathbb{R}^3} \varphi_i\varphi_j = \delta_{ij}
  \bigg\},
\end{multline}
where $\delta_{ij}$ is the Kronecker delta, 
$V(x) = -\sum_{i=1}^{M}
Z_i/|x-R_i|$, $\tau(x, y) = \sum_{i=1}^N\varphi_i(x)\varphi_i(y)$, and $\rho(x)
= \tau(x,x)$, see, e.g., \cite{Lieb1977,Lions1987}. The Euler-Lagrange equations of \eqref{eq:EHF} read
\begin{equation}
  \label{eq:HF}
  (-\Delta+V(x))\varphi_i(x)  +  \int_{\mathbb{R}^3}\frac{\rho(y)}{|x-y|}dy\varphi_i(x) - \int_{\mathbb{R}^3} \frac{\tau(x, y)}{|x-y|}\varphi_i(y) dy= \lambda_i \varphi_i(x), \qquad i=1, \dots, N, \, \text{and } x\in\mathbb{R}^3
\end{equation}
with $\int_{\mathbb{R}^3}\varphi_i\varphi_j=\delta_{ij}$.
  \begin{remark} \label{rmk:ExGrndStat}
    It has been shown in \cite{Lieb1977} that, if $\sum_{k=1}^MZ_k>N-1$, %then 
    there exists a ground state $\varphi_1,\dots, \varphi_N $ of \eqref{eq:EHF}, solution to \eqref{eq:HF}.
  \end{remark}
The following statement gives exponential expression rate bounds of the 
NN-based approximation of electronic wave functions in the vicinity 
of one singularity (corresponding to the location of a nucleus) of the potential.
\begin{theorem}
  \label{prop:HF}
  Assume that \eqref{eq:HF} has $N$ real eigenvalues $\lambda_1, \dots,
  \lambda_N$ with associated eigenfunctions $\varphi_1, \dots, \varphi_N$, 
  such that $\int_{\R^3}\varphi_i\varphi_j = \delta_{ij}$.
  Fix
  $k\in\{1, \dots, M\}$, let $R_k$
  be one of the singularities of $V$ and let $a>0$ such that $|R_j-R_k|>2a$ for
  all $j\in \{1, \dots, M\}\setminus \{k\}$. Let $\Omega_k$ be the cube 
    $\Omega_k = \left\{ x\in \mathbb{R}^3:\|x - R_k\|_{\infty}\leq a \right\}$.

  Then there exists a NN $\Phi_{\epsilon, \varphi}$ 
  such that
  $\Realiz(\Phi_{\epsilon, \varphi}) : \mathbb{R}^3\to \mathbb{R}^N$, 
  satisfies
\begin{equation}\label{eq:ReLUapprox-HF}
\left \| \varphi_i - \Realiz(\Phi_{\epsilon, \varphi})_i \right\|_{H^1(\Omega_k)} \leq \epsilon,\qquad\forall i\in\range{N}.
\end{equation}
In addition, as $\epsilon \to 0$, 
$\|\Realiz(\Phi_{\epsilon, \varphi})_i\|_{L^\infty(\Omega_k)} 
   = \mathcal{O}(\left| \log\epsilon \right|^{6})$ for every $i = \{1, \dots, N\}$, 
$$
\size(\Phi_{\epsilon, \varphi}) 
= 
\mathcal{O}(\left|\log(\epsilon)\right|^{7} 
+
N\left|\log(\epsilon)\right|^{6}),
\;\; 
\depth(\Phi_{\epsilon, \varphi}) 
 = 
 \mathcal{O}(\left|\log(\epsilon)\right|\log(\left|\log(\epsilon)\right|)).
$$
\end{theorem}
\begin{proof}
  Let $\Cset = \{(0, 0,0)\}$ and $\Eset = \emptyset$ and fix $k\in\{1 ,\dots, M\}$.
  From the regularity result in \cite[Corollary 1]{MadayMarcati2020}, see also \cite{Flad2008,Fournais2009},
  there exist $C_\varphi$, $A_\varphi$, and $\gamma_\corn>3/2$ such that $(\varphi_1, \dots, \varphi_N)
  \in \left[\cJ^\varpi_{\gamma_\corn}(\Omega_k; \Cset, \Eset; C_\varphi, A_\varphi)\right]^N$. 
Then, \eqref{eq:ReLUapprox-HF}, the $L^\infty$ bound and the depth and size bounds 
on the NN $\Phi_{\epsilon, \varphi}$ follow 
from the $hp$ approximation result in Proposition \ref{prop:internal} (centered in $R_k$ by translation), 
from Theorem \ref{th:ReLU-hp}, 
as in Corollary \ref{cor:ReLUapprox-vector}.
\end{proof}
The arguments in the preceding subsections applied to wave functions
for a single, isolated nucleus modelled by the singular potential $V$ as in \eqref{eq:V}
can then be extended to give upper bounds on the approximation rates achieved by
realizations of NNs of the wave functions in a bounded, sufficiently large
domain containing all singularities of the nuclear potential in \eqref{eq:EHF}.
\begin{corollary} \label{coro:multnucl}
   % Let $Z_k$, $k=1, \dots, M$ be such that \eqref{eq:HF} has $N$ smallest eigenvalues with
   % associated eigenfunctions $(\lambda_i, \varphi_i)$, $i=1, \dots, N$
Assume that \eqref{eq:HF} has $N$ real eigenvalues $\lambda_1, \dots, \lambda_N$ 
with associated eigenfunctions $\varphi_1, \dots, \varphi_N$, such
that $\int_{\R^3}\varphi_i\varphi_j = \delta_{ij}$.
Let $a_i, b_i\in\mathbb{R}$, $i=1,2,3$, and $\Omega = \bigtimes_{i=1}^d(a_i, b_i)$ 
such that $\{R_j\}_{j=1}^M\subset \Omega$.
Then, for every $0 < \epsilon<1$, 
there exists a NN $\Phi_{\epsilon, \varphi}$ such that
  $\Realiz(\Phi_{\epsilon, \varphi}) : \mathbb{R}^3\to \mathbb{R}^N$ and
\begin{equation}\label{eq:ReLUapprox-HF-multising}
\left \| \varphi_i - \Realiz(\Phi_{\epsilon, \varphi})_i \right\|_{H^1(\Omega)} 
\leq \epsilon,\qquad\forall i=1, \dots, N.
\end{equation}
Furthermore, as $\epsilon \to 0$
$\size(\Phi_{\epsilon, \varphi}) 
= \mathcal{O}(\left|\log(\epsilon)\right|^{7} + N\left|\log(\epsilon)\right|^{6})$ 
and 
$\depth(\Phi_{\epsilon, \varphi}) 
= 
\mathcal{O}(\left|\log(\epsilon)\right|\log(\left|\log(\epsilon)\right|))$.
\end{corollary}
\begin{proof}
The proof is based on a partition of unity argument. 
We only sketch it at this point, but will develop it in detail in the proof of
Theorem \ref{th:polygon}.
Let $\cT$ be a tetrahedral, regular triangulation of $\Omega$, and let 
$\{\kappa_k \}_{k=1}^{N_\kappa}$ be the hat-basis functions associated to it. 
We suppose that the triangulation is sufficiently refined to ensure that, 
for all $k\in\range{N_\kappa}$, exists a cube $\tOmega_{k}\subset \Omega$ 
such that
$\supp(\kappa_k) \subset \tOmega_{k}$ and that there exists at most
one $j\in\range{M}$ such that $\overline{\tOmega}_k\cap R_j \neq \emptyset$.

For all $k\in \range{N_\kappa}$, %we denote 
by \cite[Theorem 5.2]{he2020}, which is based on \cite{TM1999},
there exists a NN $\Phi^{\kappa_k}$ such that
 \begin{equation*}
  \Realiz (\Phi^{\kappa_k}) (x) = \kappa_k(x), \qquad \forall x\in \Omega.
 \end{equation*}
 For all $0<\epsilon<1$, let
\begin{equation*}
       \epsilon_1 \coloneqq \frac{\epsilon}{
         2 N_\kappa\left(\max_{k\in\{1, \dots, N_\kappa\}}\|\kappa_k \|_{W^{1,\infty}(\Omega)}  \right) }.
     \end{equation*}
     For all $k\in\range{N_\kappa}$ and $i\in\{1,\ldots,N\}$, there holds 
     $\varphi_i|_{\tOmega_k} \in \cJ^{\varpi}_\gamma(\tOmega_k; \{R_1, \dots, R_M\}\cap
       \overline{\tOmega}_k, \emptyset)$. 
     Then there exists a NN $\Phi_{\epsilon_1, \varphi}^{k}$, 
     as defined in Theorem \ref{prop:HF}, such that
\begin{equation}
\label{eq:HF-multising-element}
\| \varphi_i - \Realiz(\Phi_{\epsilon_1, \varphi}^k)_i \|_{H^1(\tOmega_k)} \leq \epsilon_1,\qquad\forall i\in\range{N}.
\end{equation}
Let 
\begin{equation*}
  C_\infty\coloneqq 
  \max_{k \in \range{N_\kappa}}  
  \sup_{\hat{\epsilon} \in (0,1)} \frac{\|\Realiz(\Phi^{k}_{\hat{\epsilon}, \varphi}) \|_{L^\infty(\tOmega_k)}}{1+\left| \log\hat{\epsilon} \right|^6} < \infty
\end{equation*}
where the finiteness is due to Theorem \ref{prop:HF}. 
Then, we denote
\begin{equation*}
     \epsprod\coloneqq \frac{\epsilon}{2N_\kappa({|\Omega|^{1/2}+} 1+\max_{i=1, \dots, N}|\varphi_i|_{H^1(\Omega)} 
      + \max_{k=1, \dots, N_\kappa}\|\kappa_k\|_{W^{1, \infty}(\Omega)}|\Omega|^{1/2})}
   \end{equation*} 
      and $\Mprod(\epsilon_1) \coloneqq C_\infty (1+\left| \log\epsilon_1 \right|^6)$.
As detailed in the proof of Theorem \ref{th:polygon} below,
after concatenating with identity NNs and possibly after increasing the constants,
we assume that $\depth(\Phi_{\epsilon_1, \varphi}^k)$ is independent of $k$ 
and that the bound on $\size(\Phi_{\epsilon_1, \varphi}^k)$ is independent of $k$,
and that the same holds for $\Phi^{\kappa_k}$, $k=1,\ldots,N_\kappa$.

Let now, for $i\in\range{N}$, $E_i:\mathbb{R}^{N+1}\to\mathbb{R}^2$ be the
matrices such that, for all $x = (x_1, \dots, x_{N+1})$, $E_i x = (x_i,
x_{N+1})$. Let also $A \in \mathbb{R}^{N\times N_\kappa}$ be a matrix of ones.
Then, we introduce the NN
\begin{equation} \label{eq:NN-HF-multising-def}
  \Phi_{\epsilon, \varphi} 
 = 
  (A, 0)\sconc \Par\left(\left\{\Par\left( \left\{ \Pi^2_{\epsprod, \Mprod(\epsilon_1)} \sconc (E_i, 0) \right\}_{i=1}^N\right)\sconc \Par(\Phi^{k}_{\epsilon_1, \varphi}, \Phi^{\Id}_{1,L} \sconc 
       \Phi^{\kappa_k})\right\}_{k=1}^{N_\kappa} \right),
\end{equation}
where $L\in\N$ is such that 
$\depth(\Phi^{\Id}_{1,L} \sconc  \Phi^{\kappa_k})
= \depth(\Phi^{k}_{\epsilon_1, \varphi})$,
from which it follows that 
$%\depth(\Phi^{\Id}_{1,L}),
\size(\Phi^{\Id}_{1,L}) \leq C\depth(\Phi^{k}_{\epsilon_1, \varphi})$.
There holds, for all $i\in \range{N}$,
\begin{equation*}
  \Realiz(\Phi_{\epsilon, \varphi})(x)_i = \sum_{k=1}^{N_\kappa} \Realiz(\Pi^2_{\epsprod, \Mprod(\epsilon_1) })  (\Realiz(\Phi^k_{\epsilon_1, \varphi})(x)_i, \kappa_k(x)), \qquad \forall x\in \Omega.
\end{equation*}
By the triangle inequality, \cite[Theorem 2.1]{Melenk1996},
\eqref{eq:HF-multising-element}, and Proposition \ref{prop:Multiplication},
%it then follows that, 
for all $i\in \range{N}$,
\begin{equation*}
  \begin{aligned}
    &\| \varphi_i - \Realiz(\Phi_{\epsilon, \varphi})_i \|_{H^1(\Omega)}  
    \\
    &\quad \leq
    \| \varphi_i - \sum_{i=1}^{N_\kappa}\kappa_k \Realiz(\Phi^{k}_{\epsilon_1, \varphi})_i \|_{H^1(\Omega)}
    % \\ &\qquad + 
    +
    \sum_{k=1}^{N_\kappa}
    \| \Realiz(\Pi^2_{\epsprod, \Mprod(\epsilon_1)})
    \left(\Realiz(\Phi^{k}_{\epsilon_1, \varphi})_i, \kappa_k \right)  
    - 
    \kappa_k\Realiz(\Phi^{k}_{\epsilon_1, \varphi})_i \|_{H^1(\Omega_k)}
    % \\ & = (I) + (II).
    \\ & \quad
    \leq N_\kappa \left(\max_{k\in \{1, \dots, N_\kappa\}}\|\kappa_k\|_{W^{1, \infty}(\Omega)} \right)\epsilon_1\\
    &\quad \qquad + N_\kappa
    (|\Omega|^{1/2}+ 1+\max_{i=1, \dots, N}|\varphi_i|_{H^1(\Omega)} + \max_{k=1,
         \dots, N_\kappa}\| \kappa_k \|_{W^{1, \infty}(\Omega)}|\Omega|^{1/2})
    \epsprod
    \\ &\quad \leq \epsilon.
  \end{aligned}
     \end{equation*}
     The asymptotic bounds on the size and depth of $\Phi_{\epsilon, \varphi}$
     can then be derived from \eqref{eq:NN-HF-multising-def}, 
using Theorem \ref{prop:HF},
%     and the same procedure as in the proof of Theorem \ref{th:polygon}.
as developed in more detail in the proof of Theorem \ref{th:polygon} below.
\end{proof}

\subsection{Elliptic PDEs in polygonal domains}
\label{sec:polygonal}
%%%%%%%%%%%%%%%%%%%%%%%%%%%%%%%%%%%%%%%%%%%%%%%%%%%%%%%%%%%
We establish exponential expressivity for realizations of NNs with ReLU activation
of solution classes to elliptic PDEs in polygonal domains $\Omega$,
the boundaries $\partial \Omega$ of which are Lipschitz and 
consist of a finite number of straight line segments. 
Notably, $\Omega \subset \R^2$ need not be a finite union of axiparallel rectangles.

In the following lemma, we construct a partition of unity in $\Omega$ 
subordinate to an open covering,
of which each element is the affine image of one out of three \emph{canonical patches}.
Remark that we admit corners with associate angle of aperture $\pi$; this
  will be instrumental, in Corollaries \ref{cor:polygon-BVP} and
  \ref{cor:Eigen}, for the imposition of different boundary
  conditions on $\partial \Omega$.
The three canonical patches that we consider are listed in Lemma \ref{lemma:exist-PU}, item [P2].
Affine images of $(0,1)^2$ are used away from corners of $\partial\Omega$
and when the internal angle of a corner is smaller than $\pi$.
Affine images of $(-1,1)\times(0,1)$ are used near corners with internal angle $\pi$.
PDE solutions may exhibit point singularities near such corners
e.g. if the two neighboring edges have different types of boundary conditions.
Affine images of $ (-1,1)^2\setminus (-1, 0]^2$ are used near corners 
with internal angle larger than $\pi$. 
In the proof of Theorem \ref{th:polygon}, we use on each patch Theorem \ref{th:ReLUapprox} 
or a result from Subsection \ref{sec:EllPDEFichera} below.

A triangulation $\cT$ of $\Omega$ is defined as 
a finite partition of $\Omega$ into
open triangles $K$ such that $\bigcup_{K\in\cT} \overline{K} = \overline{\Omega}$.
A \emph{regular triangulation} of $\Omega$ is, additionally, a triangulation $\cT$ of
$\Omega$ such that, for any two neighboring elements $K_1, K_2\in \cT$, 
$\overline{K}_1\cap \overline{K}_2$ is either a corner of both $K_1$ and $K_2$ or
an entire edge of both $K_1$ and $K_2$.
For a regular triangulation $\cT$ of $\Omega$, 
we denote by $S_1(\Omega, \cT)$
the space of functions $v\in C(\Omega)$ such that for every $K \in \cT$, $v|_{K} \in \mathbb{P}_1$.
%
		% \begin{Note}{JO} Both in the lemma, its proof and the theorem,
		% the number of functions $\phi_i$ is in general larger 
		% than the number of unique parallelograms/affinely mapped Fichera-type domains 
		% (denoted by $\Omega_n$ in the proof of the lemma).
		% If I understand it correctly, we sometimes use the same $\Omega_n$ for multiple nodes $n$.
		% \end{Note}
%    

We postpone the proof of Lemma \ref{lemma:exist-PU} to Appendix \ref{sec:triangulationpolygon}.
\begin{lemma}
  \label{lemma:exist-PU}
  Let $\Omega\subset\mathbb{R}^2$ 
  be a %simply connected 
  polygon with Lipschitz boundary, consisting of straight sides, 
  and with a finite set $\Cset$ of corners.
Then, 
there exists $N_p\in \mathbb{N}$, a regular triangulation $\cT$ of $\R^2$, 
such that for all $K\in\Tcal$ either $K\subset\Omega$ or $K\subset\Omega^c$. 
Moreover, there exists a partition of unity
$\{\phi_i\}_{i=1}^{N_p}\subset\left[S_1(\Omega,\cT)\right]^{N_p}$ 
%subordinate to a cover $\{\Omega_i\}_{i=1}^{N_p}$ of $\Omega$
%in the sense that $\supp(\phi_i)\cap\Omega\subset\Omega_i$ for all $i=1,\ldots,N_p$,
such that
%  \begin{enumerate}
%  \item for each $i\in\range{N_p}$, there exists an affine map $\psi_i$ such
%    that either $\psi_i^{-1}(\Omega_i) = (0,1)^2$ or $\psi_i^{-1}(\Omega_i) =
%    (-1,1)^2\setminus (-1, 0]^2$,
%  \item it holds that $\psi_i^{-1}(c) = (0,0)$ for every $c\in \Cset \cap \overline{\Omega}_i$. In particular, $\#(\Cset \cap \overline{\Omega}_i) \leq 1$.
%  \end{enumerate}
  	\begin{itemize}
  		\item[\rm{[P1]}] $\supp(\phi_i)\cap\Omega\subset\Omega_i$ for all $i=1,\ldots,N_p$,
  		\item[\rm{[P2]}] for each $i\in\range{N_p}$, 
                there exists an affine map $\psi_i \colon \R^2 \to \R^2$ 
                such that $\psi_i^{-1}(\Omega_i) = \hOmega_i$ for 
  		$$
  		\hOmega_i \in \{(0,1)^2, \Omega_{DN}, \Omega\}, 
  		\quad  \text{ with } \quad \Omega_{DN} := (-1,1)\times(0,1),
  		\quad \Omega := (-1,1)^2\setminus (-1, 0]^2;
  		$$
  		\item[\rm{[P3]}] $\mathcal{C} \cap\overline{\Omega}_i \subset \psi_i(\{(0,0)\})$ for all $i\in \{1, \dots, N_p\}$.
  	\end{itemize}
\end{lemma}
The following statement, then, provides expression rates for the NN
approximation of functions in weighted analytic classes in polygonal domains.

We recall that all NNs are realized with the ReLU activation
function, see \eqref{eq:NetworkScheme}.
\begin{theorem}
  \label{th:polygon} 
  Let $\Omega\subset\mathbb{R}^2$ be a polygon with Lipschitz boundary 
  consisting of straight sides and with a finite set $\Cset$ of corners.
  Let $\ugamma = \{\gamma_c: c\in \Cset\}$ such that $\min\ugamma >1$.
Then, for all $u\in \cJ^\varpi_\ugamma(\Omega; \Cset, \emptyset)$ and
for every $0  < \epsilon < 1$, there exists a NN $\Phi_{\epsilon, u}$
such that
  \begin{equation}
    \label{eq:polygon}
    \| u - \Realiz(\Phi_{\epsilon, u}) \|_{H^1(\Omega)}\leq \epsilon.
  \end{equation}
In addition, as $\epsilon \to 0$,
$$
\size(\Phi_{\epsilon, u}) = \mathcal{O}(\left|\log(\epsilon)\right|^{5}),
\;\;
\depth(\Phi_{\epsilon,u}) 
= 
\mathcal{O}(\left|\log(\epsilon)\right|\log(\left|\log(\epsilon)\right|)).
$$
\end{theorem}
\begin{proof}
We introduce, using Lemma \ref{lemma:exist-PU}, a regular triangulation $\cT$ of $\R^2$, 
an open cover $\{\Omega_i\}_{i=1}^{N_p}$ of $\Omega$, and 
a partition of unity $\{\phi_i\}_{i=1}^{N_p} \in \left[ S_1(\Omega,\cT) \right]^{N_p}$ 
such that the properties \textrm{[P1]} -- \textrm{[P3]} of Lemma \ref{lemma:exist-PU} hold.

We define $\hu_i \coloneqq u_{|_{\Omega_i}}\circ \psi_i : \hOmega_i\to \mathbb{R}$.
Since $u\in \cJ^\varpi_{\ugamma}(\Omega;  \Cset, \emptyset)$ with
$\min\ugamma>1$ and since the maps $\psi_i$ are affine, 
we observe that for every
$i\in\{1, \dots, N_p\}$, there exists $\ugamma$ such that $\min\ugamma>1$
and $\hu_i\in \cJ^\varpi_{\ugamma}(\hOmega_i, \{(0,0)\}, \emptyset)$, 
because of \textrm{[P2]} and \textrm{[P3]}. 
Let
\begin{equation*}
\epsilon_1 \coloneqq 
\frac{\epsilon}{2 N_p\max_{i\in\{1, \dots, N_p\}} 
\|\phi_i\|_{W^{1,\infty}(\Omega)} 
\left(\| \det J_{\psi_i}\|_{L^\infty((0,1)^2)}  
\left( 1 + \|\| J_{\psi^{-1}_i}\|_2\|_{L^\infty(\Omega_i)}^{2} \right) \right)^{1/2} }.
\end{equation*}
By Theorem \ref{th:ReLUapprox} and by 
Lemma \ref{lem:straightcorner} and 
Theorem \ref{prop:Fichera-appx} 
in the forthcoming Subsection \ref{sec:EllPDEFichera}, 
there exist $N_p$ NNs
$\Phi^{\hu_i}_{\epsilon_1}$, $i\in \{1,\dots, N_p\}$, 
such that
     \begin{equation}
      \label{eq:Phi-hu}
      \|\hu_i - \Realiz(\Phi^{\hu_i}_{\epsilon_1}) \|_{H^1(\hOmega_i)}\leq \epsilon_1, \qquad \forall i\in \{1, \dots, N_p\},
     \end{equation}
     and there exists $C_\infty>0$ independent of $\epsilon_1$ such that, 
     for all $i \in \{1, \dots, N_p\}$ and all $\hat{\epsilon} \in (0,1)$
     \begin{equation*}
     \|\Realiz(\Phi^{\hu_i}_{\hat{\epsilon}}) \|_{L^\infty(\hOmega_i)} 
     \leq 
     C_\infty (1+\left| \log\hat{\epsilon} \right|^4).
     \end{equation*}
     The NNs given by Theorem \ref{th:ReLUapprox}, 
     Lemma \ref{lem:straightcorner} and Theorem \ref{prop:Fichera-appx},
     which we here denote by $\widetilde{\Phi}^{\hu_i}_{\epsilon_1}$ for $i = 1,\ldots,N_p$,
     may not have equal depth. 
     Therefore, for all $i=1,\ldots,N_p$ and suitable $L_i\in\N$ we define
     $\Phi^{\hu_i}_{\epsilon_1} := \Phi^{\Id}_{1,L^1_i} \sconc \widetilde{\Phi}^{\hu_i}_{\epsilon_1}$,
     so that the depth is the same for all $i=1,\ldots,N_p$.
     To estimate the size of the enlarged NNs, we use the fact that the size of a 
     NN is not smaller than the depth unless the associated realization is constant. 
     In the latter case, we could replace the NN by a NN with one non-zero weight without changing the realization.
     By this argument, we obtain for all $i=1,\ldots,N_p$ that
     $\size(\Phi^{\hu_i}_{\epsilon_1}) 
     \leq 2\size(\Phi^{\Id}_{1,L_i}) + 2\size(\widetilde{\Phi}^{\hu_i}_{\epsilon_1})
     \leq C \max_{j=1,\ldots,N_p} \depth(\widetilde{\Phi}^{\hu_j}_{\epsilon_1}) 
     	+ C \size(\widetilde{\Phi}^{\hu_i}_{\epsilon_1})
     \leq C \max_{j=1,\ldots,N_p} \size(\widetilde{\Phi}^{\hu_j}_{\epsilon_1})$.
     Furthermore, as shown in \cite{he2020}, there exist NNs
     $\Phi^{\phi_i}$, $i\in\range{N_p}$, such that 
     \begin{equation*}
       \Realiz(\Phi^{\phi_i})(x) = \phi_i(x), \qquad \forall x\in \Omega,\; \forall i\in\{1, \dots, N_p\}.
     \end{equation*}
     Here we use that $\cT$ is a partition $\R^2$, 
     so that $\phi_i$ is defined on all of $\R^2$ and \cite[Theorem 5.2]{he2020} applies,
     which itself is based on \cite{TM1999}.
     Similarly to the previously handled case of $\Phi^{\hu_i}_{\epsilon_1}$, 
     we can assume that $\Phi^{\phi_i}$ for $i=1,\ldots,N_p$ all have equal depth 
     and that the size of $\Phi^{\phi_i}$ is bounded independent of $i$.
     
     Since by \textrm{[P2]} the mappings $\psi_i$ are affine and invertible, it follows that 
     $\psi_i^{-1}$ is affine for every $i \in \{1, \dots, N_p\}$. 
     Thus, there exist NNs $\Phi^{\psi^{-1}_i}$, $i\in\range{N_p}$, of depth $1$,
     such that
     \begin{equation}
       \label{eq:Phi-psi}
       \Realiz(\Phi^{\psi_i^{-1}})(x) = \psi_i^{-1}(x), \qquad \forall x\in \Omega_i,\; \forall i\in\{1, \dots, N_p\}.
     \end{equation}
     Next, we define
     \begin{equation*}
     \epsprod\coloneqq \frac{\epsilon}{2N_p(|\Omega|^{1/2}+ 1+|u|_{H^1(\Omega)} + \max_{i=1,
         \dots, N_p}\|\phi_i\|_{W^{1, \infty}(\Omega)}|\Omega|^{1/2})}
   \end{equation*} 
   and $\Mprod(\epsilon_1)\coloneqq C_\infty (1+\left| \log\epsilon_1 \right|^4)$.
     Finally, we set
     \begin{equation}
       \label{eq:NN-polygon-def}
       \Phi_{\epsilon, u} \coloneqq ((\underbrace{1, \dots, 1}_{N_p\text{ times}}), 0)
       \sconc \Par\left(\left\{\Pi^2_{\epsprod, \Mprod(\epsilon_1)}
       \sconc \Par(\Phi^{\hu_i}_{\epsilon_1}\sconc \Phi^{\psi^{-1}_i}, \Phi^{\Id}_{1,L}  
       \sconc \Phi^{\phi_i})\right\}_{i=1}^{N_p} \right),
     \end{equation}
     where $L\in\N$ is such that 
     $\depth(\Phi^{\hu_1}_{\epsilon_1}\sconc \Phi^{\psi^{-1}_1}) 
     = \depth(\Phi^{\Id}_{1,L} \sconc \Phi^{\phi_1})$,
     which yields that 
     $%\depth(\Phi^{\Id}_{1,L_i}),
     \size(\Phi^{\Id}_{1,L})
     \leq C \depth(\Phi^{\hu_1}_{\epsilon_1}\sconc \Phi^{\psi^{-1}_1})$.
     \paragraph{Approximation accuracy.}
     By \eqref{eq:NN-polygon-def}, 
     we have for all $x\in \Omega$,
     \begin{equation*}
\Realiz(\Phi_{\epsilon, u})(x) 
= 
\sum_{i=1}^{N_p} \Realiz(\Pi^2_{\epsprod, \Mprod(\epsilon_1)})
   \left(  \Realiz(\Phi^{\hu_i}_{\epsilon_1}\sconc\Phi^{\psi^{-1}_i})(x), \Realiz(\Phi^{\phi_i})(x)\right).
     \end{equation*}
     Therefore,
     \begin{equation}
       \label{eq:polygon-appx-triangle}
     \begin{aligned}
       \| u - \Realiz(\Phi_{\epsilon,u}) \|_{H^1(\Omega)}
       &\leq
       \| u - \sum_{i=1}^{N_p}\phi_i\Realiz(\Phi^{\hu_i}_{\epsilon_1}\sconc \Phi^{\psi^{-1}_i}) \|_{H^1(\Omega)}
         \\ &\qquad + 
 \sum_{i=1}^{N_p}
       \| \Realiz(\Pi^2_{\epsprod, \Mprod(\epsilon_1)})
 \left(\Realiz(\Phi^{\hu_i}_{\epsilon_1}\sconc \Phi^{\psi^{-1}_i}), \phi_i \right)  
  - 
   \phi_i\Realiz(\Phi^{\hu_i}_{\epsilon_1}\sconc \Phi^{\psi_i^{-1}}) \|_{H^1(\Omega)}
              \\ & = (I) + (II).
     \end{aligned}
     \end{equation}
We start by considering term $(I)$.
     For each $i\in \{1, \dots, N_p\}$,
     thanks to \eqref{eq:Phi-hu}, there holds,
     with $\| J_{\psi^{-1}_i} \|_2^2$  
         denoting the square of the matrix $2$-norm of the Jacobian of $\psi_i^{-1}$,
     \begin{equation}
       \label{eq:appx-omegai}
\begin{aligned}
& \| u - \Realiz(\Phi^{\hu_i}_{\epsilon_1}\sconc \Phi^{\psi^{-1}_i}) \|_{H^1(\Omega_i)}
\\
&\qquad =  
\| \hu_i\circ \psi^{-1}_i - \Realiz(\Phi^{\hu_i}_{\epsilon_1})\circ \psi^{-1}_i \|_{H^1(\Omega_i)}
%          +
% \| \Realiz(\Phi^{\hu_i}_{\epsilon_1})\circ \psi^{-1}_i - \Realiz(\Phi^{\hu_i}_{\epsilon_1}\sconc \Phi^{\psi^{-1}_i}_{\epsilon_\psi}) \|_{H^1(\Omega_i)}
       \\ & \qquad
       = \left( \int_{\hOmega_i} \left(\snormc{\hu_i}^2 + \normc[2]{J_{\psi_i^{-1}} \nabla 
     	\left(\hu_i - \Realiz(\Phi^{\hu_i}_{\epsilon_1}) \right)}^2 
    	 \right) \det J_{\psi_i} dx \right)^{1/2} 
     \\ &\qquad
     \leq \epsilon_1 
         \left( \| \det J_{\psi_i}\|_{L^\infty(\hOmega_i)} 
              + \| \det J_{\psi_i}\|_{L^\infty(\hOmega_i)} \| \| J_{\psi^{-1}_i} \|_2^2 \|_{L^\infty(\Omega_i)} \right)^{1/2} 
\\ & \qquad
\leq 
\epsilon_2:= \epsilon_1 
\max_i   
\left( 
\| \det J_{\psi_i}\|_{L^\infty(\hOmega_i)} + \| \det J_{\psi_i}\|_{L^\infty(\hOmega_i)} 
                          \| \| J_{\psi^{-1}_i} \|_2^2 \|_{L^\infty(\Omega_i)} 
\right)^{1/2}.
\end{aligned}
\end{equation}
     By \cite[Theorem 2.1]{Melenk1996},
     \begin{equation}
       \label{eq:polygon-appx-I}
       (I) \leq N_p \epsilon_2\max_{i\in \{1, \dots, N_p\}}\|\phi_i\|_{W^{1, \infty}(\Omega)}  \leq \frac\epsilon2.
     \end{equation}
     We now consider term $(II)$ in \eqref{eq:polygon-appx-triangle}.
     There holds, by Theorem \ref{th:ReLUapprox} and \eqref{eq:Phi-psi},
     \begin{equation*}
       \|\Realiz(\Phi^{\hu_i}_{\epsilon_1}\sconc \Phi^{\psi^{-1}_i})\|_{L^\infty(\Omega_i)}
         = \|\Realiz(\Phi^{\hu_i}_{\epsilon_1})\|_{L^\infty(\hOmega_i)} % + % \|\Realiz(\Phi^{\hu_i}_{\epsilon_1})\circ\Realiz( \Phi^{\psi^{-1}_i}_{\epsilon_\psi})
% -
       % \Realiz(\Phi^{\hu_i}_{\epsilon_1})\circ\psi^{-1}_i
       % \|_{L^\infty(\Omega_i)}
       \leq C_{\infty} (1 + \left| \log\epsilon_1 \right|^4)
         % + {\color{red} \epsilon_\psi \|\Realiz(\Phi^{\hu_i}_{\epsilon_1})\|_{W^{1, \infty}((0,1)^2)}}
     \end{equation*}
     for all $i\in \{1, \dots, N_p\}$. 
     Furthermore, by [P1], $\phi_i(x) = 0$ for all $x\in \Omega\setminus\Omega_i$ and, 
     by Proposition \ref{prop:Multiplication},
       \begin{equation*}
         \Realiz(\Pi^2_{\epsprod, \Mprod(\epsilon_1)})
       \left(\Realiz(\Phi^{\hu_i}_{\epsilon_1}\sconc \Phi^{\psi^{-1}_i})(x), \phi_i (x)\right) = 0, 
           \qquad \forall x\in \Omega\setminus\Omega_i.
       \end{equation*}
       From \eqref{eq:appx-omegai}, we also have 
         \begin{equation*}
           | \Realiz(\Phi^{\hu_i}_{\epsilon_1}\sconc \Phi^{\psi^{-1}_i}) |_{H^1(\Omega_i)}  \leq
           | u |_{H^1(\Omega_i)} + \| u - \Realiz(\Phi^{\hu_i}_{\epsilon_1}\sconc \Phi^{\psi^{-1}_i}) \|_{H^1(\Omega_i)} 
           \leq 1+ | u |_{H^1(\Omega_i)}.
         \end{equation*}
     Hence,
     \begin{equation}
       \label{eq:polygon-appx-II}
       \begin{aligned}
         (II) & = \sum_{i=1}^{N_p}
       \| \Realiz(\Pi^2_{\epsprod, \Mprod(\epsilon_1)})
      \left(\Realiz(\Phi^{\hu_i}_{\epsilon_1}\sconc \Phi^{\psi^{-1}_i}), \phi_i \right) 
       - \phi_i\Realiz(\Phi^{\hu_i}_{\epsilon_1}\sconc \Phi^{\psi_i^{-1}}) \|_{H^1(\Omega_i)}
       \\
       &\leq \sum_{i=1}^{N_p}\bigg( \| \Realiz(\Pi^{2}_{\epsprod, \Mprod(\epsilon_1)}) (a,b)  
                    - ab\|_{W^{1, \infty}([-\Mprod(\epsilon_1), \Mprod(\epsilon_1)]^2)}
       \\  &\qquad \qquad\cdot
\left(|\Omega|^{1/2}+ | \Realiz(\Phi^{\hu_i}_{\epsilon_1}\sconc \Phi^{\psi^{-1}_i}) |_{H^1(\Omega_i)} +| \phi_i |_{H^1(\Omega_i)}  \right) \bigg)
       \\
       &\leq N_p \epsprod \left(|\Omega|^{1/2}+ 1 + | u |_{H^1(\Omega_i)}  + |\Omega|^{1/2}\max_{i=1,\dots, N_p} \|\phi_i\|_{W^{1, \infty}(\Omega)}\right)
       \\ & \leq \frac{\epsilon}{2}.
       \end{aligned}
     \end{equation}
     The asserted approximation accuracy follows by combining \eqref{eq:polygon-appx-triangle}, \eqref{eq:polygon-appx-I}, 
     and \eqref{eq:polygon-appx-II}.
     \paragraph{Size of the neural network.} 
     To bound the size of the NN, we
     remark that $N_p$ and the sizes of $\Phi^{\psi_i^{-1}}$ and of
     $\Phi^{\phi_i}$ only depend on the domain $\Omega$. 
     Furthermore, there
     exist constants $C_{\Omega, i}$, $i=1,2,3$, that depend only on
     $\Omega$ and $u$ such that
     % we could probably make this more precise by writing that it depends on $u$ 
     % only through $C_u, A_u, \ugamma$.
     \begin{equation}
       \label{eq:epsilons-polygon}
       \begin{gathered}
       \left| \log\epsilon_1 \right| \leq C_{\Omega, 1} (1+\left| \log\epsilon \right|),\qquad \qquad
       \left| \log\epsilon_\times \right| \leq C_{\Omega, 2}(1+ \left| \log\epsilon \right|), \\
       \left| \log M_\times(\epsilon_1) \right| \leq C_{\Omega, 3}(1+  \log(1+\left|\log\epsilon \right|)).
       \end{gathered}
     \end{equation}
     From Theorem \ref{th:ReLUapprox} and Proposition \ref{prop:Multiplication},
     in addition, there exist constants $C^L_{\hu}, C^M_{\hu}, C_{\times}>0$
     such that, for all $0< \epsilon_1 , \epsilon_\times\leq 1$,
     \begin{equation}
       \label{eq:sizes-polygon}
       \begin{gathered}
       \depth(\Phi^{\hu_i}_{\epsilon_1})  \leq C^L_{\hu} (1+\left|
         \log\epsilon_1 \right|)(1 + \log(1+\left| \log\epsilon_1 \right|)),
       \qquad \qquad
       \size(\Phi^{\hu_i}_{\epsilon_1})  \leq C^M_{\hu} (1+\left|
         \log\epsilon_1 \right|^{5}),\\
       \max(\size(\Pi^2_{\epsprod, \Mprod(\epsilon_1)}), \depth(\Pi^2_{\epsprod,
         \Mprod(\epsilon_1)}))
         \leq C_\times (1+\log(\Mprod(\epsilon_1)^2/\epsprod)).
     \end{gathered}
   \end{equation}
   Then, by \eqref{eq:NN-polygon-def}, we have
   \begin{equation}
     \label{eq:depth-size-polygon}
     \begin{aligned}
     &
       \depth(\Phi_{\epsilon, u}) = 1 + \depth(\Pi^2_{\epsprod, \Mprod(\epsilon_1)}) 
       +\max_{i=1, \dots, N_p} \left(\depth(\Phi^{\hu_i}_{\epsilon_1}) 
       + \depth(\Phi^{\psi_i^{-1}}) %+ \depth(\Phi^{\phi_i})  
       \right),
       \\ & 
       \size(\Phi_{\epsilon, u}) 
         \leq {C}\left( N_p + \size(\Pi^2_{\epsprod, \Mprod(\epsilon_1)}) +\sum_{i=1}^{N_p} \left(\size(\Phi^{\hu_i}_{\epsilon_1}) 
       + \size(\Phi^{\psi_i^{-1}}) + \size(\Phi^{\Id}_{1,L})
       + \size(\Phi^{\phi_i})  \right) \right).
     \end{aligned}
   \end{equation}
   The desired depth and size bounds follow from \eqref{eq:epsilons-polygon}, \eqref{eq:sizes-polygon},
     and \eqref{eq:depth-size-polygon}. This concludes the proof.
\end{proof}
The exponential expression rate for the class of weighted, analytic functions in $\Omega$ 
by realizations of NNs with ReLU activation in the $H^1(\Omega)$-norm
established in Theorem \ref{th:polygon}
implies an exponential expression rate bound on $\partial\Omega$, 
via the trace map and the fact that 
$\partial\Omega$ can be \emph{exactly parametrized by the realization 
of a shallow NN with ReLU activation}.
This is relevant for NN-based solution of boundary integral equations.
%as considered in \cite{deepBEM}.

%%%%%%%%%%%%%%%%%%%%%%%%%%%%%%%%%%%%%%%%%%%%%%%%%
\begin{corollary}{(NN expression of Dirichlet traces)}\label{cor:polygontrace} 
Let $\Omega\subset \R^2$ be a polygon with Lipschitz boundary 
and a finite set $\Cset$ of corners.
  Let $\ugamma = \{\gamma_c: c\in \Cset\}$ such that $\min\ugamma >1$.
% let $L$, $B$ and $f$ be as in Theorem \ref{th:polygon}.
For any connected component $\Gamma$ of $\partial\Omega$,
let $\ell_\Gamma>0$ be the length of $\Gamma$, 
such that there exists a continuous, piecewise affine parametrization 
$\theta:[0,\ell_\Gamma]\to\R^2:t\mapsto\theta(t)$ of $\Gamma$
with finitely many affine linear pieces and $\normc[2]{\tfrac{d}{dt}\theta} = 1$ 
for almost all $t\in[0,\ell_\Gamma]$.

Then, for all $u\in \cJ^\varpi_\ugamma(\Omega; \Cset, \emptyset)$ and
for all $0  < \epsilon < 1$, there exists a NN $\Phi_{\epsilon, u, \theta}$
approximating the trace $\Trace u := u_{|_{\Gamma}}$ 
such that
  \begin{equation}
    \label{eq:polygontrace}
    \| \Trace u- \Realiz(\Phi_{\epsilon, u, \theta})\circ\theta^{-1} \|_{H^{1/2}(\Gamma)}\leq \epsilon.
  \end{equation}
In addition, as $\epsilon \to 0$,
$$
\size(\Phi_{\epsilon, u, \theta}) 
= 
\mathcal{O}(\left|\log(\epsilon)\right|^{5}),
\;\;
\depth(\Phi_{\epsilon, u, \theta}) 
= 
\mathcal{O}(\left|\log(\epsilon)\right|\log(\left|\log(\epsilon)\right|)). 
$$
\end{corollary}
\newcommand{\Ctrace}{C_{\mathrm{\Gamma}}}
\begin{proof}
We note that both components of $\theta$ are continuous, piecewise affine functions on $[0,\ell_\Gamma]$,
thus they can be represented exactly as realization of a NN of depth two, 
with the ReLU activation function.  
Moreover, the number of 
weights of these NNs is of the order of the number of affine linear pieces of $\theta$.
We denote the parallelization of the NNs emulating exactly the two components of $\theta$ by $\Phi^{\theta}$. 

By continuity of the trace operator
$\Trace: H^1(\Omega) \to H^{1/2}(\partial\Omega)$ 
(e.g. \cite{Gagliardo1957,Brenner2008}),
there exists a constant $\Ctrace>0$ such that for all $v\in H^{1}(\Omega)$
it holds
$
\normc[H^{1/2}(\Gamma)]{\Trace v }
	\leq \Ctrace \normc[H^{1}(\Omega)]{ v },
$
and without loss of generality we may assume $\Ctrace\geq1$. 

Next, for any $\eps\in(0,1)$, let $\Phi_{\epsilon/\Ctrace, u}$ 
be as given by Theorem \ref{th:polygon}.
Define $\Phi_{\epsilon, u, \theta} := \Phi_{\epsilon/\Ctrace, u} \sconc \Phi^{\theta}$.
It follows that
\begin{align*}
\normc[H^{1/2}(\Gamma)]{\Trace u -  \Realiz(\Phi_{\epsilon, u, \theta})\circ\theta^{-1} }
   = \normc[H^{1/2}(\Gamma)]{\Trace \left(u - \Realiz(\Phi_{\epsilon/\Ctrace, u})  \right) }
  % & \leq \normc[H^{1/2}(\partial\Omega)]{\Trace \left(u - \Realiz(\Phi_{\epsilon/\Ctrace, u})  \right) }\\
   \leq \Ctrace \normc[H^{1}(\Omega)]{ u - \Realiz(\Phi_{\epsilon/\Ctrace, u}) }
	\leq \eps.
\end{align*}
The bounds on its depth and size follow directly from 
Proposition \ref{prop:conc}, Theorem \ref{th:polygon},
and the fact that the depth and size of $\Phi^{\theta}$ are independent of $\eps$.
This finishes the proof.
\end{proof}
\begin{remark}
The exponent $5$ in the bound on the NN size 
$\size(\Phi_{\epsilon, u, \theta})$ 
in Corollary \ref{cor:polygontrace} is likely not optimal, 
due to it being transferred from the NN rate in $\Omega$.
\end{remark}

The proof of Theorem \ref{th:polygon} established exponential expressivity of
realizations of NNs with ReLU activation
for the analytic class $\cJ^\varpi_{\ugamma}(\Omega;  \Cset, \emptyset)$ in $\Omega$.
% Similar results can, therefore, be expected in other settings, provided that a corresponding
% analytic regularity result is available. We indicate two particular cases:
This implies that realizations of NNs can approximate, with exponential
expressivity, solution classes of elliptic PDEs in polygonal domains 
$\Omega$. 
We illustrate this by formulating concrete results for 
three problem classes: 
second order, linear, elliptic source and eigenvalue problems in $\Omega$,
and viscous, incompressible flow.
To formulate the results, we specify the assumptions on $\Omega$.
\begin{definition}[Linear, second order, elliptic divergence-form
differential operator with analytic coefficients]
\label{def:Dop}
Let $d\in\{2, 3\}$ and let $\Omega\subset\R^d$ be a bounded domain.
Let the coefficient functions
$a_{ij},b_i, c:\overline{\Omega}\to \R$ be real analytic in $\overline{\Omega}$,
and such that the matrix function $A = (a_{ij})_{1\leq i,j\leq d}:\Omega \to \R^{d \times d}$
is symmetric and uniformly positive definite in $\Omega$. 
With these functions, 
we define the linear, second order, elliptic divergence-form
differential operator 
$\Dop$ acting on $w\in C^\infty_0(\Omega)$ via
(summation over repeated indices $i,j\in \{1,\dots, d\}$)
$$
(\Dop w)(x) := -\partial_i(a_{ij}(x)\partial_j w(x)) + b_j(x)\partial_j w(x) + c(x)w(x) \;,
\quad 
x\in \Omega\;.
$$
\end{definition}
\begin{setting}
  \label{setting:polygon}
We assume that $\Omega\subset \R^2$ is an open, bounded polygon 
with 
boundary $\partial\Omega$ that is Lipschitz and connected.
In addition, $\partial\Omega$ is the closure of a finite number
$J \geq 3$ of straight, open sides $\Gamma_j$, i.e., 
$\Gamma_{i} \cap \Gamma_j =\emptyset$ for $i\ne j$ and 
$\partial\Omega = \bigcup_{1\leq j \leq J} \overline{\Gamma_j}$.
We assume the sides are enumerated cyclically, according to arc length, i.e.
$\Gamma_{J+1} = \Gamma_1$.
By $n_j$, 
we denote the exterior unit normal vector to $\Omega$
on $\Gamma_j$ and by 
$\bbc_j := \overline{\Gamma_{j-1}} \cap \overline{\Gamma_j}$ 
the corner $j$ of $\Omega$.

%We assume that $\Dop$ is a linear, second order, divergence-form elliptic
%operator with analytic coefficients as in Definition \ref{def:Dop}.
With $\Dop$ as in Definition \ref{def:Dop}, 
we associate on boundary segment $\Gamma_j$ 
a boundary operator $\Bop_j \in \{\gamma^j_0,\gamma^j_1\}$,
i.e. 
either the Dirichlet trace $\gamma_0$ 
or the distributional (co-)normal derivative operator $\gamma_1$, 
acting on 
$w\in C^1(\overline{\Omega})$ 
via
\begin{equation}
  \label{eq:boundarycond}
\gamma^j_0 w := w|_{\Gamma_j}
\, ,\qquad 
\gamma^j_1 w := (A\nabla w) \cdot n_j|_{\Gamma_j},
\quad j=1,...,J\;.
\end{equation}
We collect the boundary operators $\Bop_j$ in $\Bop := \{ \Bop_j \}_{j=1}^J$.
\end{setting}
The first corollary addresses exponential ReLU expressibility of solutions
of the source problem corresponding to $(\Dop,\Bop)$.
\begin{corollary}
  \label{cor:polygon-BVP}
  % Let $L$ denote a linear, second-order, divergence form elliptic operator
  % with coefficients that are assumed analytic in $\overline{\Omega}$ and let $B$ be a
  % Dirichlet or Neumann boundary operator with coefficients that are analytic in $\overline{\Omega}$
  % as in the setting of \cite[Section 3]{GuoBabCurv}.
  Let $\Omega$, $\Dop$, and $\Bop$ be as in Setting \ref{setting:polygon} with $d=2$.
For $f$ analytic in $\overline{\Omega}$, 
  let $u$ denote a solution to the boundary value problem
  \begin{equation}\label{eq:HomDiri}
    \Dop u = f \;\text{ in }\;\Omega, \qquad \Bop u = 0\;\text{ on }\;\partial \Omega \;.
  \end{equation}
Then, 
for every $0  < \epsilon < 1$, there exists a 
NN $\Phi_{\epsilon, u}$ such that
  \begin{equation}
    \label{eq:polygon-BVP}
    \| u - \Realiz(\Phi_{\epsilon, u}) \|_{H^1(\Omega)}\leq \epsilon.
  \end{equation}
In addition, $\size(\Phi_{\epsilon, u}) = \mathcal{O}(\left|\log(\epsilon)\right|^{5})$ 
and 
$\depth(\Phi_{\epsilon,u}) 
= \mathcal{O}(\left|\log(\epsilon)\right|\log(\left|\log(\epsilon)\right|))$, 
as $\epsilon \to 0$.
\end{corollary}
\begin{proof}
The proof is obtained by verifying weighted, analytic regularity of solutions.
By
\cite[Theorem 3.1]{GuoBab4}
there
exists $\ugamma$ such that $\min\ugamma>1$ 
and $u\in \cJ^\varpi_\ugamma(\Omega; \Cset, \emptyset)$. 
Then, the application of Theorem \ref{th:polygon} concludes the proof.
\end{proof}
Next, we address NN expression rates for eigenfunctions of $(\Dop,\Bop)$.
\begin{corollary}\label{cor:Eigen}
Let $\Omega$, $\Dop$, $\Bop$ be as in Setting \ref{setting:polygon} with $d=2$,
and $b_i = 0$ in Definition \ref{def:Dop},% to ensure self-adjoint operator.
and let $ 0 \ne w\in H^1(\Omega)$ be an eigenfunction 
of the elliptic eigenvalue problem
\begin{equation}\label{eq:EllEVP}
    \Dop w = \lambda w \text{ in }\Omega, \qquad \Bop w = 0 \text{ on }\partial \Omega.
\end{equation}
Then, 
for every $0  < \epsilon < 1$, there exists a 
NN $\Phi_{\epsilon, w}$ such that
  \begin{equation}
    \label{eq:polygon-EVP}
    \| w - \Realiz(\Phi_{\epsilon, w}) \|_{H^1(\Omega)}\leq \epsilon.
  \end{equation}
In addition, $\size(\Phi_{\epsilon, w}) = \mathcal{O}(\left|\log(\epsilon)\right|^{5})$ 
and 
$\depth(\Phi_{\epsilon,w}) = \mathcal{O}(\left|\log(\epsilon)\right|\log(\left|\log(\epsilon)\right|))$, 
as $\epsilon \to 0$.
\end{corollary}
\begin{proof}
The statement follows from 
the regularity result \cite[Theorem 3.1]{Babuska1989},
and Theorem \ref{th:polygon} as in Corollary \ref{cor:polygon-BVP}.
\end{proof}

The analytic regularity of solutions $u$ in the proof of Theorem \ref{th:polygon}
also holds for certain nonlinear, elliptic PDEs. 
We illustrate it for the velocity field of viscous, incompressible flow in $\Omega$.
\begin{corollary}\label{cor:NSE}
Let $\Omega\subset \R^2$ be as in Setting \ref{setting:polygon}.
Let $\nu>0$ and let $\bm{u}\in  H^1_0(\Omega)^2$ 
be
the velocity field of the Leray solutions of the 
viscous, incompressible Navier-Stokes equations in $\Omega$,
with homogeneous Dirichlet (``no slip'') boundary conditions
\begin{equation}
  \label{eq:NSE}
  -\nu\Delta \bm{u} + (\bm{u}\cdot\nabla) \bm{u} + \nabla p = \bm{f} \text{ in }\Omega,\qquad
  \nabla\cdot\bm{u} = 0\text{ in }\Omega,\qquad
  \bm{u} =\bm{0}\text{ on }\partial\Omega,
\end{equation}
where the components of $\bm{f}$ are analytic in $\overline{\Omega}$ 
and 
such that
$\|\bm{f}\|_{H^{-1}(\Omega)}/ \nu^2$ is small enough so that $\bm{u}$ is unique.

Then, 
for every $0  < \epsilon < 1$, 
there exists a NN $\Phi_{\epsilon, \bm{u}}$ 
with two-dimensional output
such that
  \begin{equation}
    \label{eq:polygon-NSE}
    \| \bm{u} - \Realiz(\Phi_{\epsilon, \bm{u}}) \|_{H^1(\Omega)}\leq \epsilon.
  \end{equation}
In addition, $\size(\Phi_{\epsilon, \bm{u}})= \mathcal{O}(\left|\log(\epsilon)\right|^{5})$ 
and 
$\depth(\Phi_{\epsilon,\bm{u}})  = \mathcal{O}(\left|\log(\epsilon)\right|\log(\left|\log(\epsilon)\right|))$, 
as $\epsilon \to 0$.
\end{corollary}
\begin{proof}
The velocity fields of Leray solutions of the Navier-Stokes equations in $\Omega$
satisfy the weighted, analytic regularity 
$\bm{u} \in \big[\cJ^\varpi_{\ugamma}(\Omega; \Cset,\emptyset)\big]^2$, 
with $\min\ugamma>1$, see \cite{MS19_2743}.
Then, the application of Theorem \ref{th:polygon} concludes the proof.
\end{proof}
%
%%%%%%%%%%%%%%%%%%%%%%%%%%%%%%%%%%%%%%%%%%%%%%%%%%%%%%%%%%%%%%%%%%%%%%%%%
\subsection{Elliptic PDEs in Fichera-type polyhedral domains}
\label{sec:EllPDEFichera}
%%%%%%%%%%%%%%%%%%%%%%%%%%%%%%%%%%%%%%%%%%%%%%%%%%%%%%%%%%%%%%%%%%%%%%%%%
Fichera-type polyhedral domains $\Omega\subset \R^3$ are, loosely speaking, 
closures of finite, disjoint unions of (possibly affinely mapped)
axiparallel hexahedra with $\partial\Omega$ Lipschitz. 
In Fichera-type domains, analytic regularity of solutions
of linear, elliptic boundary value problems from acoustics
and linear elasticity in displacement formulation has been
established in \cite{CoDaNi12}.
As an example of a boundary value problem covered by \cite{CoDaNi12} and our theory,
consider $\Omega \coloneqq (-1, 1)^d \setminus (-1, 0]^d$ 
for $d=2,3$, displayed for $d=3$ in Figure \ref{fig:Fichera}.
\begin{figure}
  \centering
  \includegraphics[width=0.3\textwidth]{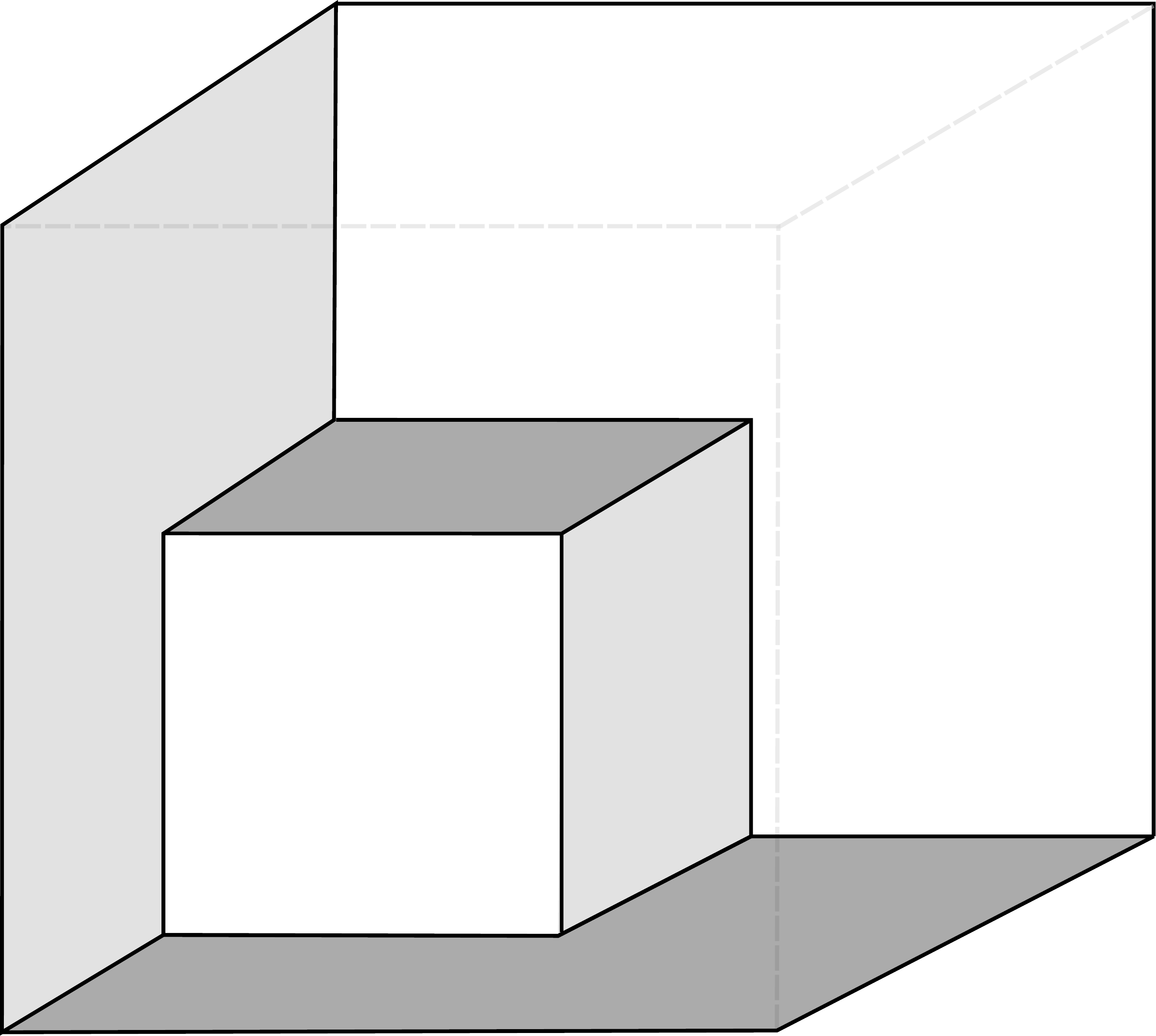}
  \caption{Example of Fichera-type corner domain.}
  \label{fig:Fichera}
\end{figure}

We recall that all NNs are realized with the ReLU activation
function, see \eqref{eq:NetworkScheme}.

We introduce the 
\emph{setting for elliptic problems with analytic coefficients in $\Omega$}.
Note that the boundary of $\Omega$ is composed of $6$ edges when $d=2$ and
of $9$ faces when $d=3$.
\begin{setting}
\label{setting:Fichera}
We assume that $\Dop$ is an elliptic operator as in Definition \ref{def:Dop}.
When $d=3$, we assume furthermore that the 
diffusion coefficient $A\in \R^{3\times 3}$ is a symmectric, positive
    matrix and $b_i = c = 0$.
On each edge (if $d=2$) or face (if $d=3$)  $\Gamma_j\subset \partial \Omega$, $j\in\range{3d}$, 
we introduce the boundary operator $\Bop_j \in
\{\gamma_0,\gamma_1\}$, where $\gamma_0$ and $\gamma_1$ are defined as in \eqref{eq:boundarycond}.
We collect the boundary operators $\Bop_j$ in $\Bop := \{ \Bop_j \}_{j=1}^{3d}$.
\end{setting}
For a right hand side $f$, 
the elliptic boundary value problem we consider in this section is then
\begin{equation} \label{eq:Fichera}
      \begin{aligned}
        \Dop u = f \text{ in }\Omega, \quad 
        \Bop u = 0 \text{ on }\partial\Omega.
      \end{aligned}
    \end{equation}

The following extension lemma will be useful 
for the approximation of the solution to \eqref{eq:Fichera} by NNs. 
We postpone its proof to Appendix \ref{sec:W11ext}.
\begin{lemma}
  \label{lemma:W11ext}
  Let $d\in\{2,3\}$ and $u\in \Wmix^{1,1}(\Omega)$. 
  Then, 
  there exists a function $v\in \Wmix^{1,1}((-1,1)^d)$ such that $v|_{\Omega} = u$. 
  The extension is stable
  with respect to the $\Wmix^{1,1}$ norm.
\end{lemma}
We denote the set containing all corners (including the re-entrant one) of
  $\Omega$ as
  \begin{equation*}
    \Cset  = \{-1,0,1\}^d\setminus (-1, \dots, -1).
  \end{equation*}
When $d=3$, for all $\corn\in\Cset$, then we denote by $\Eset_\corn$ 
the set of edges abutting at $\corn$
and we denote 
$\Eset \coloneqq \bigcup_{\corn\in\Cset}\Eset_\corn$.
\begin{theorem}
  \label{prop:Fichera-appx}
  Let $u \in \cJ^\varpi_\ugamma(\Omega; \Cset, \Eset)$ with
\begin{alignat*}{3}
&\ugamma = \{\gamma_\corn: \corn\in \Cset\}, &&\text{with } \gamma_\corn>1,\;
\text{for all } \corn\in\Cset &&\text{ if } d = 2,\\ %
&\ugamma = \{\gamma_\corn, \gamma_e: \corn\in \Cset, e\in \Eset\}, \quad&&\text{with
} \gamma_\corn>3/2\text{ and } \gamma_e>1,\; \text{for all }\corn\in\Cset\text{
  and }e\in \Eset\quad &&\text{ if } d = 3. %
\end{alignat*}
Then, for any $0< \epsilon <1$ 
there exists a NN $\Phi_{\epsilon, u}$ 
so that 
\begin{equation}\label{eq:FicheraNN-appx}
\left \| u - \Realiz\left(\Phi_{\epsilon, u}\right) \right\|_{H^1(\Omega)} \leq \epsilon.
\end{equation}
In addition, $\|\Realiz\left(\Phi_{\epsilon, u}\right)\|_{L^\infty(\Omega)} 
= 
\mathcal{O}(1+\left| \log\epsilon \right|^{2d})$, as $\epsilon \to 0$. 
Also, $\size(\Phi_{\epsilon, u}) = \mathcal{O}(|\log(\epsilon)|^{2d+1})$
and 
$\depth(\Phi_{\epsilon, u})  
 = \mathcal{O}(|\log(\epsilon)|\log(|\log(\epsilon)|))$, 
as $\epsilon \to 0$.
\end{theorem}
\begin{proof}
    By Lemma \ref{lemma:W11ext}, we extend the function $u$ to a function $\tu$ such that 
  \begin{equation*}
    \tu\in \Wmix^{1,1}((-1, 1)^d) \quad \text{and}\quad\tu|_{\Omega} = u.
  \end{equation*}
  Note that, by the stability of the extension, there exists a constant
    $C_{\mathrm{ext}}>0$ independent of $u$ such that
    \begin{equation}
      \label{eq:ext-fichera-stable}
    \|\tu\|_{\Wmix^{1,1}((-1,1)^d)}\leq C_{\mathrm{ext}} \| u\|_{\Wmix^{1,1}(\Omega)}.
    \end{equation}
    Since $u \in \cJ^\varpi_\ugamma(\Omega; \Cset, \Eset)$, there
      holds $u\in \cJ^\varpi_\ugamma(S; \Cset_S, \Eset_S)$ 
      for all
      \begin{align} \label{eq:insidefichera}
        S \in %\mathcal{S}\coloneqq 
          \left\{ 
          \bigtimes_{j=1}^d(a_j, a_j+1/2): (a_1, \dots, a_d)\in\{-1,-1/2,0,1/2\}^d
        \right\}
        \text{ such that }S\cap \Omega\neq\emptyset
      \end{align}
   with $\Cset_S = \overline{S}\cap \Cset$ and $\Eset_{S} = \{e\in\Eset : e\subset\overline{S}\}$.
   Since $S \subset \Omega$ and $\tu|_{\Omega} = u|_{\Omega}$, 
   we also have 
$$
   \tu\in \cJ^\varpi_\ugamma(S; \Cset_S, \Eset_S) \;\mbox{for all }S 
   \mbox{ satisfying \eqref{eq:insidefichera}}.
$$
By Theorem \ref{prop:internal}  
exist $C_p>0$, $C_{\tNoned}>0$, $C_{\tNint}>0$, $C_{\tv}>0$, $C_{\tc}>0$, and $b_{\tv}>0$ such that, 
for all $0<\epsilon\leq 1$, there exists $p\in\N$, a partition $\cG_\oned$ of $(-1, 1)$
into $\tNint$ open, disjoint, connected subintervals, a $d$-dimensional array
$c\in \R^{\tNoned\times\dots\times\tNoned}$, and piecewise polynomials 
$\tv_i \in \mathbb{Q}_p(\cG_\oned)\cap H^1((-1,1))$, $i=1, \dots, \tNoned$, 
such that
\begin{equation*}
       \tNoned \leq C_{\tNoned}(1+\left| \log\epsilon \right|^2),\quad
       \tNint \leq C_{\tNint}(1+\left| \log\epsilon \right|),\quad
       \|c\|_{1} \leq C_{\tc}(1+\left| \log\epsilon \right|^{2d}),\quad
       p \leq C_p(1+\left| \log\epsilon \right|)
\end{equation*}
and
\begin{equation*}
  \|\tv_i\|_{H^1(I)}\leq C_{\tv} \epsilon^{-b_{\tv}}, \qquad \|{\tv}_i\|_{L^\infty(I)}\leq 1,\qquad \forall i\in \{1, \dots, \tNoned\}.
\end{equation*}
  %     there exists a $c_p>0$ such
  % that for all $0< \epsilon<1$,
  % there exist $L\in \mathbb{N}$ with $L\leq C(1+|\log\epsilon|)$
  % and $v_{\hp}\in \widetilde{X}_{\hp, d}^{L, c_p L}$
   Furthermore,
  \begin{equation*}
     \| u - v_{\hp} \|_{H^1(\Omega)} = \| \tu - v_{\hp} \|_{H^1(\Omega)}\leq \frac\epsilon2, \qquad 
v_{\hp}  = \sum_{\iscomma=1}^{\tNoned} \tc_{\is} \bigotimes_{j=1}^d\tv_{i_j}.
  \end{equation*}
  Due to the stability \eqref{eq:ext-fichera-stable} and to Lemmas
  \ref{lemma:cbound} and \ref{lemma:W11J3}, there holds
  \begin{equation*}
    \|\tc\|_1 \leq C \Nint^{2d} \| u\|_{\cJ^d_\ugamma(\Omega)}, 
    %%% JO: was L, but L was not defined
  \end{equation*}
  i.e., the bound on the coefficients $\tc$ is independent of the extension
  $\tu$ of $u$.
  By Theorem \ref{th:ReLU-hp}, there exists a NN ${\Phi}_{\epsilon, u}$ 
  with the stated approximation properties and asymptotic size bounds. 
  The bound on the $L^\infty(\Omega)$ norm of the realization of $\Phi_{\epsilon, u}$ 
  follows as in the proof of Theorem \ref{th:ReLUapprox}.
\end{proof}
  \begin{remark}
    Arguing as in Corollary \ref{cor:polygontrace}, a NN with ReLU activation and
    two-dimensional input can be constructed so that its realization
    approximates the Dirichlet trace of solutions to \eqref{eq:Fichera} in $H^{1/2}(\partial\Omega)$
    at an exponential rate in terms of the NN size $\size$.
  \end{remark}

The following statement now gives expression rate bounds for the
approximation of solutions to the Fichera problem \eqref{eq:Fichera} 
by realizations of NNs with the ReLU activation function.
\begin{corollary} \label{prop:Fichera}
  Let $f$ be an analytic function on $\overline{\Omega}$
%   that satisfies \eqref{eq:compatibility-Fichera} 
% if the boundary operator $B$ is the Neumann boundary operator. 
  % Further, let $u$ be
  and let $u$ be a solution to \eqref{eq:Fichera} with operators $\Dop$ and
  $\Bop$ as in Setting \ref{setting:Fichera} 
  and with source term $f$.
Then, for any $0< \epsilon <1$ 
there exists a NN $\Phi_{\epsilon, u}$ 
so that 
\begin{equation}\label{eq:FicheraNN}
\left \| u - \Realiz\left(\Phi_{\epsilon, u}\right) \right\|_{H^1(\Omega)} \leq \epsilon.
\end{equation}
In addition, $\size(\Phi_{\epsilon, u}) = \mathcal{O}(|\log(\epsilon)|^{2d+1})$ 
and 
$\depth(\Phi_{\epsilon, u}) = \mathcal{O}(|\log(\epsilon)|\log(|\log(\epsilon)|))$, for $\epsilon \to 0$.
\end{corollary}
\begin{proof}
  By \cite[Corollary 7.1, Theorems 7.3 and 7.4]{CoDaNi12} if $d=3$ and
    \cite[Theorem 3.1]{GuoBab4} if $d=2$, 
  there exists
  $\ugamma$ such that $\gamma_\corn-d/2>0$ for all $\corn\in \Cset$ and
  $\gamma_e>1$ for all $e\in \Eset$ such that 
  $u \in \cJ^\varpi_\ugamma(\Omega; \Cset, \Eset)$. 
  An application of
  Theorem \ref{prop:Fichera-appx} concludes the proof.
\end{proof}

\begin{remark}\label{rem:weightedanalyticf}
By \cite[Corollary 7.1 and Theorem 7.4]{CoDaNi12}, 
Corollary \ref{prop:Fichera} holds verbatim also under the hypothesis that the right-hand side
$f$ is weighted analytic, with singularities at the corners/edges of the domain;
specifically, \eqref{eq:FicheraNN} and the size bounds on the NN
$\Phi_{\epsilon, u}$ hold under
the assumption that there exists $\ugamma$ such that $\gamma_\corn-d/2>0$ 
for all $\corn\in \Cset$ and $\gamma_e>1$ for all $e\in \Eset$ such that 
  \begin{equation*}
    f \in \cJ^\varpi_{\ugamma-2}(\Omega; \Cset, \Eset).
  \end{equation*}
\end{remark}

\begin{remark}\label{rmk:otheractivat}
    The numerical approximation of solutions for \eqref{eq:Fichera} 
    with a NN in two dimensions 
    has been investigated e.g. in \cite{LuMengKarniadakis2019} using 
    the so-called `PINNs' methodology.
    There, the loss function was based on minimization of the residual 
    of the NN approximation in the strong form of the PDE.
    Evidently, 
    a different (smoother) activation than the ReLU activations considered here
    had to be used. 
    Starting from the approximation of products by NNs with
    smoother activation functions introduced in \cite[Sec.3.3]{schwab2017deep}
    and following the same line of reasoning as in the present paper,
    the results we obtain for ReLU-based realizations of NNs can be extended to large classes of NNs 
    with smoother activations and similar %the same 
    architecture.

    Furthermore, in \cite[Section 3.1]{EDeepRitz}, 
    a slightly different elliptic boundary value problem is numerically approximated 
    by realizations of NNs.  
    Its solutions exhibit the same weighted, analytic regularity as considered in this paper. 
    The presently obtained approximation rates by NN realizations 
    extend also to the approximation of solutions for the problem considered in \cite{EDeepRitz}.
\end{remark}
%%%%%%%%%%%%%%%%%%%%%%%%%%%%%%%%%%%%%%%%%%%%%%%%%%%%%%%%%%%

In the proof of Theorem \ref{th:polygon}, we require in particular the approximation of weighted analytic functions on $(-1,1)\times(0,1)$ 
with a corner singularity at the origin. 
For convenient reference, we detail the argument in this case.

\begin{lemma}\label{lem:straightcorner}
Let $d=2$ and $\Omega_{DN} := (-1,1)\times(0,1)$. 
Denote $\Cset_{DN}  = \{-1,0,1\} \times\{0,1\}$.
Let $u \in \cJ^\varpi_\ugamma(\Omega_{DN}; \Cset_{DN}, \emptyset)$ with
$\ugamma = \{\gamma_\corn: \corn\in \Cset_{DN}\}$, 
with $\gamma_\corn>1$ for all $\corn\in\Cset_{DN}$.

Then, for any $0< \epsilon <1$ 
there exists a NN $\Phi_{\epsilon, u}$ 
so that 
\begin{equation}\label{eq:straightcorner}
\left \| u - \Realiz\left(\Phi_{\epsilon, u}\right) \right\|_{H^1(\Omega_{DN})} \leq \epsilon.
\end{equation}
In addition, $\|\Realiz\left(\Phi_{\epsilon, u}\right)\|_{L^\infty(\Omega_{DN})} 
= 
\mathcal{O}(1+\left| \log\epsilon \right|^{4})$ , for $\epsilon \to 0$. 
Also, $\size(\Phi_{\epsilon, u}) = \mathcal{O}(|\log(\epsilon)|^{5})$
and 
$\depth(\Phi_{\epsilon,u}) 
= 
\mathcal{O}(|\log(\epsilon)|\log(|\log(\epsilon)|))$, 
for $\epsilon \to 0$.
\end{lemma}
\begin{proof}
Let $\tu\in \Wmix^{1,1}((-1, 1)^2)$ be defined by 
$$
\begin{cases}
\tu(x_1,x_2) = u(x_1,x_2) & \text{ for all } (x_1,x_2)\in(-1,1)\times[0,1),\\
\tu(x_1,x_2) = u(x_1,0) & \text{ for all } (x_1,x_2)\in(-1,1)\times(-1,0),
\end{cases}
$$
such that $\tu|_{\Omega_{DN}} = u$.
Here we used that there exist continuous imbeddings
$\cJ^\varpi_\ugamma(\Omega_{DN}; \Cset_{DN}, \emptyset)
	\hookrightarrow \Wmix^{1,1}(\Omega_{DN})
	\hookrightarrow C^0(\overline{\Omega_{DN}})$
(see Lemma \ref{lemma:W11J3} for the first imbedding),
i.e. $u$ can be extended to a continuous function on $\overline{\Omega_{DN}}$.

As in the proof of Lemma \ref{lemma:W11ext}, 
this extension is stable, 
i.e. there exists a constant
$C_{\mathrm{ext}}>0$ independent of $u$ 
such that
\begin{equation}
  \label{eq:straightcorner-stable}
\|\tu\|_{\Wmix^{1,1}((-1,1)^d)}\leq C_{\mathrm{ext}} \| u\|_{\Wmix^{1,1}(\Omega_{DN})}.
\end{equation}
Because $u \in \cJ^\varpi_\ugamma(\Omega_{DN}; \Cset_{DN}, \emptyset)$,
it holds with $\Cset_S = \overline{S}\cap \Cset_{DN}$  
that $u\in \cJ^\varpi_\ugamma(S; \Cset_S, \emptyset)$ for all
\[
S\in \left\{
  \bigtimes_{j=1,2}(a_j, a_j+1/2): (a_1, a_2)\in\{-1,-1/2,0,1/2\}\times\{0,1/2\} 
\right\}
\;.
\]
The remaining steps are the same as those in the proof of Theorem \ref{prop:Fichera-appx}.
\end{proof}

\section{Conclusions and extensions}
\label{sec:ConclExt}
We review the main findings of the present paper 
and outline extensions of the present results, and 
perspectives for further research.
%%%%%%%%%%%%%%%%%%%%%%%%%%%%%%%%%%%%%%%%%%%%%%%%%%%%%%%%%%%%%%%5
\subsection{Principal mathematical results}
\label{sec:MainRes}
%%%%%%%%%%%%%%%%%%%%%%%%%%%%%%%%%%%%%%%%%%%%%%%%%%%%%%%%%%%%%%%5
We established exponential expressivity of realizations of NNs
with the ReLU activation function in the Sobolev norm $H^1$ for
functions which belong to certain countably normed, 
weighted analytic function spaces in cubes $Q=(0,1)^d$
of dimension $d=2,3$.
The admissible function classes comprise functions which
are real analytic at points $x\in Q$, and which admit analytic
extensions to the open sides $F\subset \partial Q$,
but may have singularities at corners and
(in space dimension $d=3$) edges of $Q$.
We have also extended this result to cover exponential expressivity of 
realizations of NNs with ReLU activation 
for solution classes of linear, second order elliptic PDEs 
in divergence form in plane, polygonal domains 
and of elliptic, nonlinear eigenvalue
problems with singular potentials in three space dimensions.
Being essentially an approximation result,
the DNN expression rate bound in Theorem \ref{th:polygon} 
will apply to any elliptic boundary value problem in polygonal domains
where weighted, analytic regularity is available.
Apart from the source and eigenvalue problems, 
such regularity is in space dimension $d=2$ 
also available for linearized elastostatics,
Stokes flow and general elliptic systems \cite{GuoBab3,GuoScStokes,CoDaNi12}.

The established approximation rates of realizations of NNs with ReLU activation
are fundamentally based on a novel exponential upper bound on approximation 
of weighted analytic functions via tensorized $hp$ approximations 
on multi-patch configurations
in finite unions of axiparallel rectangles/hexahedra. 
The $hp$ approximation result is presented in Theorem \ref{prop:internal} 
and of independent interest in the numerical analysis of spectral elements. 

The proofs of exponential expressivity of NN realizations are, in principle, constructive. 
They are based on explicit bounds on the coefficients of $hp$ projections
and on corresponding emulation rate bounds for the (re)approximation of modal $hp$ bases.
%%%%%%%%%%%%%%%%%%%%%%%%%%%%%%%%%%%%%%%%%%%%%%%%%%%%%%%%%%%%%
\subsection{Extensions and future work}
\label{sec:ExtFrtWrk}
%%%%%%%%%%%%%%%%%%%%%%%%%%%%%%%%%%%%%%%%%%%%%%%%%%%%%%%%%%%55
The tensor structure of the $hp$ approximation 
considered here 
limited geometries of domains that are admissible for our results.
\emph{Curvilinear, mapped domains} with analytic domain 
maps will allow corresponding approximation rates, with the NN
approximations obtained by composing the present constructions
with NN emulations of the domain maps and the fact that
compositions of NNs are again NNs. 

The only activation function considered in this work is the ReLU.
Following the same proof strategy, 
exponential expression rate bounds can be
obtained for functions with smoother, 
nonlinear activation functions. 
We refer to Remark \ref{rmk:otheractivat} and 
to the discussion in \cite[Sec. 3.3]{schwab2017deep}.

The principal results in Section \ref{sec:EVPPtSing} 
yield exponential expressivity of realizations of NNs with ReLU activation
for singular eigenvalue problems with multiple, isolated 
point singularities as arise in electron-structure computations 
\emph{for static molecules with known loci of the nuclei}.
Inspection of our proofs reveals that the expression rate bounds are
robust with respect to perturbations of the nuclei sites;
only interatomic distances enter the constants 
in the expression rate bounds of Section \ref{sec:HF}. 
Given the \emph{closedness of NNs under composition}, 
obtaining similar expression rates also for 
solutions of the \emph{vibrational Schr\"odinger equation} appears 
in principle possible.

The presently proved deep ReLU NN expression rate bounds 
can, 
in connection with recently proposed, 
residual-based DNN training methodologies 
(e.g., \cite{shin2020error}),
imply exponential convergence rates of numerical NN approximations
of PDE solutions based on machine learning approaches.

% \section*{Acknowledgements}

\appendix
\section{Tensor product $hp$ approximation}
\label{sec:hp-analysis}
In this section, we construct the $hp$ tensor product approximation which
will then be emulated to obtain the NN expression rate estimates.  
We denote $Q=(0,1)^d$, $d\in \{2,3\}$ and introduce the set of corners $\Cset$,
\begin{equation}
\label{eq:Cset}
    \Cset =
    \begin{cases}
      \big\{(0,0)\big\} & \text{if }d=2,\\
      \big\{(0,0,0)\big\} & \text{if }d=3,
    \end{cases}
  \end{equation}
  and the set of edges $\Eset$,
  \begin{equation}
\label{eq:Eset}
    \Eset =
    \begin{cases}
    \emptyset & \text{if }d=2,\\
 \big \{ \{0\}\times\{0\}\times(0,1), \{0\}\times(0,1)\times\{0\}, (0,1)\times \{0\}\times\{0\} \big \}
 & \text{if }d=3.
    \end{cases}
\end{equation}
  The results in this section extend, by rotation or reflection, 
  to the case where $\Cset$ contains any of
  the corners of $Q$ and $\Eset$ is the set of the adjacent edges when $d=3$.
Most of the section addresses the construction of exponentially consistent 
    $hp$-quasiinterpolants in the reference cube $(0,1)^d$; 
    in Section~\ref{sec:patches} the analysis will be extended to domains which
    are specific finite unions of such patches.
%
%%%%%%%%%%%%%%%%%%%%%%%%%%%%%%%%%%%%%%%%%%%%%%%%%%%%%%%%%%%%%%%%
\subsection{Product geometric mesh and tensor product $hp$ space}
\label{sec:hp-prdmshspc}
%%%%%%%%%%%%%%%%%%%%%%%%%%%%%%%%%%%%%%%%%%%%%%%%%%%%%%%%%%%%%%%%
%
We fix a geometric mesh grading factor $\sigma \in (0,1/2]$.
Furthermore, let
\begin{equation*}
J_0^\ell 
= (0, \sigma^\ell)\qquad\text{and} \qquad J_k^\ell 
= (\sigma^{\ell-k+1}, \sigma^{\ell-k}), \, k=1, \dots, \ell.
\end{equation*}
In $(0,1)$, 
    the geometric mesh with $\ell$ layers is 
$\cG^\ell_{1} = \left\{J^\ell_k: k=0, \dots, \ell\right\}$. Moreover, we denote the nodes of $\cG^\ell_{1}$ by 
$x_0^\ell = 0$ and $x_k^\ell = \sigma^{\ell-k+1}$ for $k=1,\ldots,\ell+1$.
%\todo{Add superscript $\ell$ in the remainder of this appendix.}
In $(0,1)^d$, the $d$-dimensional tensor product geometric mesh is\footnote{We assume \emph{isotropic tensorization},
i.e. the same $\sigma$ and the same number of geometric mesh layers in each coordinate direction; 
all approximation results remain valid 
(with possibly better numerical values for the constants in the error bounds)
for anisotropic, co-ordinate dependent choices of $\ell$ and of $\sigma$.}
\begin{equation*}
  \cG^\ell_d = \left\{ \bigtimes_{i=1}^d K_i,\, \text{for all } K_1,\dots, K_d \in  \cG^\ell_{1} \right\}.
\end{equation*}
For an element $K = \bigtimesdim J_{k_i}^{\ell}$, $k_i\in\{0,\ldots,\ell\}$, we denote by
$d^K_\corn$ the distance from the singular corner, and $d^K_e$ the distance from the closest singular edge.
We observe that
\begin{equation}
  \label{eq:dist1}
  d^K_\corn = \left( \sum_{i=1}^{d}\sigma^{2(\ell-k_i+1)} \right)^{1/2}
\end{equation}
and 
\begin{equation}
  \label{eq:dist2}
  d^K_e = \min_{(i_1, i_2)  \in \{1, 2, 3\}^2}   \left( \sum_{i\in \{i_1,i_2\}}\sigma^{2(\ell-k_i+1)} \right)^{1/2}.
\end{equation}

The \emph{$hp$ tensor product space} 
is defined as
\begin{equation*}
  \Xhpelldim  \coloneqq \{ v\in H^1(Q): v_{|_{K}} \in \mathbb{Q}_{p}(K), \text{ for all } K\in \cG^\ell_d\},
\end{equation*}
where $\mathbb{Q}_p(K) \coloneqq \spn\left\{\prod_{i=1}^d (x_i)^{k_i} \colon k_i\leq p, i=1,
  \dots, d\right\}$. Note that, by construction, 
  $\Xhpelldim = \bigotimes_{i=1}^d \Xhpellone$.

For positive integers $p$ and $s$ such that $1\leq s \leq p$, 
we will write
\begin{equation}
  \label{eq:Psi}
  \Psi_{p,s} \coloneqq \frac{(p-s)!}{(p+s)!}.
\end{equation}
Additionally, we will denote, for all $\sigma\in (0, 1/2]$,
  \begin{equation}
    \label{eq:sigmaratio}
    \sigmaratio \coloneqq \frac{1-\sigma}{\sigma} \in [1,\infty).
  \end{equation}
%%%%%%%%%%%%%%%%%%%%%%%%%%%%%%%%%%%%%%%%%%%%%%%%%%%%%%%%%%%%%%%%
\subsection{Local projector}
\label{sec:loc-proj}
%%%%%%%%%%%%%%%%%%%%%%%%%%%%%%%%%%%%%%%%%%%%%%%%%%%%%%%%%%%%%%%%
We denote the reference interval by $I=(-1, 1)$ and the reference cube by $\hK =
(-1, 1)^d$. We also write $\Hmix^1(\hK) = \bigotimes_{i=1}^d H^1(I)\supset H^d(\hK)$. Let $p\geq 1$: 
we introduce the univariate projectors 
$\hpi_p : H^1(I) \to \mathbb{P}_p(I)$
as
\begin{equation}
  \label{eq:ondim-proj-explicit}
  \begin{aligned}
    \left(\hpi_p \hv\right)(x)
   & =  \hv(-1) + \sum_{n=0}^{p-1} \left(\hv', \frac{2n+1}{2}L_n\right)\int_{-1}^x L_n(\xi)d\xi\\
   &=  \hv(-1)\left( \frac{1-x}{2} \right) + \hv(1)\left( \frac{1+x}{2} \right) 
    + \sum_{n=1}^{p-1} \left(\hv', \frac{2n+1}{2}L_n\right)\int_{-1}^x L_n(\xi)d\xi,
  \end{aligned}
\end{equation}
where $L_n$ is the $n$th Legendre polynomial, $L^\infty$ normalized, and
$(\cdot, \cdot)$ is the scalar product of $L^2((-1,1))$.
Note that 
\begin{equation}
  \label{eq:onedim-prop}
 \left( \hpi_p \hv \right) (\pm 1) = \hv (\pm 1), \qquad \forall \hv \in H^1(I).
\end{equation}
For $p\in \mathbb{N}$, we introduce the projection on the reference element
$\hK$ as $ \hPi_{p} = \bigotimes_{i=1}^d\hpi_{p} $.
For all $K\in \cG^\ell_d$, we introduce an affine
transformation from $K$ to the reference element 
\begin{equation}
  \label{eq:Phi}
  \Phi_K :  K\to \hK\quad \text{such that}\quad\Phi_K(K) = \hK.
\end{equation}
Remark that since the elements are axiparallel, the affine transformation can be
written as a $d$-fold product of one dimensional affine transformations $\phi_k
: J_k^\ell \to I$, i.e.,
supposing that $K= \bigtimes_{i=1}^d J_{k_i}^{\ell}$,  there holds
\begin{equation*}
  \Phi_K = \bigotimes_{i=1}^d \phi_{k_i}.
\end{equation*}
Let $K\in \cG^\ell_d$ and let $k_i$, $i=1, \dots, d$ be the indices such that
$K=\bigtimes_{i=1}^d J_{k_i}^{\ell}$. Define, for $w\in H^1(J_{k_i}^{\ell})$,
\begin{equation*}
  \pi^{k_i}_p w = \left( \hpi_p (w\circ \phi_{k_i}^{-1}) \right)\circ \phi_{k_i}. %removed brackets (\phi_{k_i})
\end{equation*}
For $v$ defined on $K$ such that $v\circ \Phi_K^{-1} \in \Hmix^1(\hK)$ 
and for $(p_1, \dots, p_d)\in \mathbb{N}^d$, we introduce the local projection operator
\begin{equation}
  \label{eq:PiK_tens}
  \Pi_{\ps}^K = \bigotimes_{i=1}^d \pi^{k_i}_{p_i}.
\end{equation}
We also write
\begin{equation}
  \label{eq:PiK}
  \Pi_{p}^K v =\Pi_{p\dots p}^K v=  \left(\hPi_{p}  (v\circ \Phi_K^{-1}) \right)\circ \Phi_K.
\end{equation}
For later reference, we note the following property of $\Pi_{p}^K v$:
\begin{lemma}
  \label{lemma:reg-cont}
 Let $K_1, K_2$ be two axiparallel cubes that share one regular face $F$ 
 (i.e., $F$ is an entire face of both $K_1$ and $K_2$). 
 Then, for $v\in \Hmix^1(\interior(\overline{K}_1\cup \overline{K}_2))$,
 the piecewise polynomial
 \begin{equation*}
   \Pi^{K_1\cup K_2}_{p} v = 
   \begin{cases}
     \Pi_p^{K_1} v &\text{in }K_1,\\
     \Pi_p^{K_2} v &\text{in }K_2
   \end{cases}
 \end{equation*}
is continuous across $F$. 
\end{lemma}
\begin{proof}
	This follows directly from \eqref{eq:onedim-prop}.
\end{proof}
%%%%%%%%%%%%%%%%%%%%%%%%%%%%%%%%%%%%%%%%%%%%%%%%%%%%%%%%%%%%%%%%
\subsection{Global projectors}
\label{sec:glob-proj}
%%%%%%%%%%%%%%%%%%%%%%%%%%%%%%%%%%%%%%%%%%%%%%%%%%%%%%%%%%%%%%%%
We introduce, for $\ell, p \in \N$, the univariate projector $\pihpell: H^1((0,1)) \to \Xhpellone$ as
\begin{equation}
  \label{eq:pihpell1d}
  \left(\pihpell u  \right) (x)=
  \begin{cases}
  \left(\pi^0_1 u  \right)(x)    &\text{if } x\in J_0^\ell,\\
\left(    \pi_p^{k}  u\right) (x) &\text{if }x\in J_{k}^\ell,\, k\in \{1, \dots, \ell\}.
  \end{cases}
\end{equation}
Note that for all $\ell\in \N$, for $x\in J_0^\ell$
\begin{equation*}
  \left(\pi^0_1 u  \right)(x) = u(0) + \sigma^{-\ell}\left( u(\sigma^\ell) - u(0) \right)x.
\end{equation*}
The $d$-variate $hp$ quasi-interpolant is then obtained by tensorization, i.e.
\begin{equation}
  \label{eq:Pihpell}
  \Pihpelldim \coloneqq \bigotimes_{i=1}^d \pihpell.
\end{equation}
\begin{remark}
\label{rem:global-continuity}
  By the nodal exactness of the projectors, the operator
  $\Pihpelldim$ is continuous across interelement interfaces
  (see Lemma \ref{lemma:reg-cont}), %\eqref{eq:onedim-prop}  
  hence its image is contained in $H^1((0,1)^d)$.
  The continuity can also be observed from the expansion 
  in terms of continuous, globally defined  basis functions
  given in Proposition \ref{prop:compactbasis}.
\end{remark}
\begin{remark}
\label{rem:bigger-space-hp}
  The projector $\Pihpelldim$ is defined on a larger space than
  $\Hmix^1(Q)$ as specified below (e.g. Remark \ref{rem:hpw11mix}).
\end{remark}
 %%%%%%%%%%%%%%%%%%%%%%%%%%%%%%%%%%%%%%%%%%%%%%%%%%%%%%%%%%%%%%%%
\subsection{Preliminary estimates}
\label{sec:prelim-estimates}
%%%%%%%%%%%%%%%%%%%%%%%%%%%%%%%%%%%%%%%%%%%%%%%%%%%%%%%%%%%%%%%%
The projector on $\hK$ given by
\begin{equation}
  \label{eq:hPi}
  \hPi_{\ps} \coloneqq \bigotimes_{i=1}^d\hpi_{p_i} 
\end{equation}
has the following property.
\begin{lemma}[{{\cite[Propositions 5.2 and 5.3]{SSWII}}}]
  \label{lemma:ref-proj}
  Let $d=3$, $(p_1, p_2, p_3)\in \mathbb{N}^3$, and $(s_1, s_2, s_3)\in
  	\mathbb{N}^3$ with $1\leq s_i\leq p_i$. Then the projector 
  $\hPi_{p_1p_2p_3}:\Hmix^1(\hK) \to \mathbb{Q}_{p_1, p_2, p_3}(\hK)$ satisfies that
  \begin{equation}
    \label{eq:disc-approx}
    \begin{aligned}
    \| v - \hPi_{p_1 p_2 p_3} v\|_{H^1(\hK)}^2 
    \leq \Cappxi \bigg( 
     & \Psi_{p_1,s_1}
      \sum_{\alpha_1, \alpha_2 \leq 1}
      \| \partial^{(s_1+1, \alpha_1, \alpha_2)}v \|^2_{L^2(\hK)} 
    \\ &\qquad +
      \Psi_{p_2,s_2}
      \sum_{\alpha_1, \alpha_2 \leq 1}
      \| \partial^{(\alpha_1, s_2+1, \alpha_2)}v \|^2_{L^2(\hK)} 
    \\ &\qquad +
      \Psi_{p_3,s_3}
      \sum_{\alpha_1, \alpha_2 \leq 1}
      \| \partial^{(\alpha_1, \alpha_2, s_3+1)}v \|^2_{L^2(\hK)} 
    \bigg),
    \end{aligned}
  \end{equation}
 for all $v\in H^{s_1+1}(I)\otimes H^{s_2+1}(I)\otimes H^{s_3+1}(I)$. Here,
 $\Cappxi$ is independent of $(p_1, p_2, p_3)$, $(s_1, s_2, s_3)$ and $v$.
\end{lemma}
\begin{remark}
In space dimension $d=2$, 
a result analogous to Lemma \ref{lemma:ref-proj} holds, see \cite{SSWII}.
\end{remark}
	\begin{lemma}
  \label{lemma:ref-proj2}
  Let $d=3$, $(p_1,p_2,p_3)\in\N^3$, and $(s_1,
  	s_2, s_3)\in
  	\mathbb{N}^3$ with $1\leq s_i\leq p_i$. Further, let $\{i, j,
  	k\}$ be a permutation of $\{1,2,3\}$. Then, the projector 
  $\hPi_{p_1p_2p_3}:\Hmix^1(\hK) \to \mathbb{Q}_{p_1, p_2, p_3}(\hK)$ satisfies
  \begin{equation}
    \label{eq:disc-approx2}
    \begin{aligned}
   \| \partial_{x_i}\left( v - \hPi_{p_1 p_2 p_3} v  \right)\|_{L^2(\hK)}^2 
    \leq \Cappxii \bigg(
&\Psi_{p_i,s_i} 
\sum_{\alpha_1, \alpha_2 \leq 1}
    \| \partial_{x_i}^{s_i+1}\partial_{x_j}^{ \alpha_1}\partial_{x_k}^{\alpha_2}v \|^2_{L^2(\hK)} 
\\    &\qquad +
\Psi_{p_j,s_j} 
    \sum_{\alpha_1\leq 1}\| \partial_{x_i}\partial_{x_j}^{s_j+1}\partial_{x_k}^{\alpha_1}v \|^2_{L^2(\hK)} 
 \\   &\qquad +
\Psi_{p_k,s_k}  
    \sum_{\alpha_1\leq 1}\| \partial_{x_i}\partial_{x_j}^{\alpha_1}\partial_{x_k}^{s_k+1}v \|^2_{L^2(\hK)} 
  \bigg),
    \end{aligned}
      \end{equation}
      for all $v\in
      H^{s_1+1}(I)\otimes H^{s_2+1}(I)\otimes H^{s_3+1}(I)$.
   Here, $\Cappxii>0$ is independent of $(p_1, p_2, p_3)$, $(s_1, s_2, s_3)$, and $v$. 
\end{lemma}
\begin{proof}
 Let $(p_1,p_2,p_3)\in\N^3$, and $(s_1,
 s_2, s_3)\in
 \mathbb{N}^3$, be as in the statement of the lemma. Also, let $i\in \{1,2,3\}$ and $\{j, k\} = \{1,2 ,3\} \setminus \{i\}$. By Lemma \ref{lemma:ref-proj}, there holds
 \begin{equation}
   \label{eq:ref-proj2-proof}
  \begin{aligned}
    \| \partial_{x_i}( v - \hPi_p v  )\|_{L^2(\hK)}^2 
    \leq \Cappxi \bigg( 
     & \Psi_{p_1,s_1} 
      \sum_{\alpha_1, \alpha_2 \leq 1}
      \| \partial^{(s_1+1, \alpha_1, \alpha_2)}v \|^2_{L^2(\hK)} 
    \\ &\qquad +
      \Psi_{p_2,s_2} 
      \sum_{\alpha_1, \alpha_2 \leq 1}
      \| \partial^{(\alpha_1, s_2+1, \alpha_2)}v \|^2_{L^2(\hK)} 
    \\ &\qquad +
      \Psi_{p_3,s_3} 
      \sum_{\alpha_1, \alpha_2 \leq 1}
      \| \partial^{(\alpha_1, \alpha_2, s_3+1)}v \|^2_{L^2(\hK)} 
    \bigg).
    \end{aligned}
    \end{equation}
 With a $\Cappxi>0$ independent of $(p_1, p_2, p_3)$, $(s_1, s_2, s_3)$, and $v$.
 Let now, $\overline{v}_i:I^2 \to \mathbb{R}$ such that
 \begin{equation*}
   \overline{v}_i(x_j,x_k) = \int_{I} v (x_1, x_2, x_3)dx_i .
 \end{equation*}
 We denote $\tv \coloneqq v - \overline{v}_i$ and, remarking that $\partial_{x_i}\overline{v}_i
 =  \partial_{x_i}\hPi_p\overline{v}_i = 0$, we apply \eqref{eq:ref-proj2-proof} to $\tv$, so that
\begin{equation}
   \label{eq:ref-proj2-proof2}
  \begin{aligned}
    \| \partial_{x_i}( v - \hPi_p v )\|_{L^2(\hK)}^2 
    \leq C \bigg( 
     & \Psi_{p_1,s_1} 
      \sum_{\alpha_1, \alpha_2 \leq 1}
      \| \partial^{(s_1+1, \alpha_1, \alpha_2)}\tv \|^2_{L^2(\hK)} 
    \\ &\qquad +
      \Psi_{p_2,s_2} 
      \sum_{\alpha_1, \alpha_2 \leq 1}
      \| \partial^{(\alpha_1, s_2+1, \alpha_2)}\tv \|^2_{L^2(\hK)} 
    \\ &\qquad +
      \Psi_{p_3,s_3} 
      \sum_{\alpha_1, \alpha_2 \leq 1}
      \| \partial^{(\alpha_1, \alpha_2, s_3+1)}\tv \|^2_{L^2(\hK)} 
    \bigg).
    \end{aligned}
    \end{equation}
    By the Poincaré inequality, it holds for all $\alpha_1\in \{0,1\}$ that
 \begin{equation*}
    \| \partial_{x_j}^{s_j+1}\partial_{x_k}^{\alpha_1}\tv \|^2_{L^2(\hK)} 
    \leq C
    \| \partial_{x_i}\partial_{x_j}^{s_j+1}\partial_{x_k}^{\alpha_1}v \|^2_{L^2(\hK)} 
 \quad\text{ and }\quad
    \| \partial_{x_j}^{\alpha_1}\partial_{x_k}^{s_k+1}\tv \|^2_{L^2(\hK)} 
    \leq C
    \| \partial_{x_i}\partial_{x_j}^{\alpha_1}\partial_{x_k}^{s_k+1}v \|^2_{L^2(\hK)}.
 \end{equation*}
 Using the fact that $\partial_{x_i}\tv = \partial_{x_i} v$ in the remaining
 terms of \eqref{eq:ref-proj2-proof2} concludes the proof.
\end{proof}

%%%%%%%%%%%%%%%%%%%%%%%%%%%%%%%%%%%%%%%%%%%%%%%%%%%%%%%%%%%%%%%%
\subsubsection{One dimensional estimate}
\label{sec:oned}
%%%%%%%%%%%%%%%%%%%%%%%%%%%%%%%%%%%%%%%%%%%%%%%%%%%%%%%%%%%%%%%%
The following result is a consequence of, e.g., \cite[Lemma 8.1]{SchSch2018} and scaling.
\begin{lemma}
  \label{lemma:oned}
  There exists $C>0$ such
  that for all $\ell\in\mathbb{N}$, all integer $0<k\leq \ell$, all integers
  $1\leq s \leq p$, all $\gamma>0$, % $0<\gamma<s+1$, 
  and all $v\in H^{s+1}(J^\ell_k)$ 
  \begin{equation}
    \label{eq:oned-appx}
    h^{-2}\| v - \pi^k_p v\|_{L^2(J^\ell_k)}^2 +  \| \nabla (v -\pi^k_p v)\|_{L^2(J^\ell_k)}^2 \leq 
    C  \sigmaratio^{2(s+1)}\Psi_{p, s} h^{2(\min\{\gamma-1,s\})} \| |x|^{(s+1-\gamma)_{+}} v^{(s+1)} \|_{L^2(J^\ell_k)}^2
  \end{equation}
  where $h= |J^\ell_k| \simeq \sigma^{\ell-k}$.
\end{lemma}
\begin{proof}
 From \cite[Lemma 8.1]{SchSch2018}, there exists $C>0$ independent of $p$, $k$,
 $s$, and $v$ such that
 \begin{equation*}
% \label{eq:oned-appx-noweight}
    h^{-2}\| v - \pi^k_p v\|_{L^2(J^\ell_k)}^2 +  \| \nabla (v -\pi^k_p v)\|_{L^2(J^\ell_k)}^2 \leq 
C \Psi_{p,s} h^{2s} \| v^{(s+1)} \|_{L^2(J^\ell_k)}^2.
 \end{equation*}
In addition, for all $k=1, \dots, \ell$, there holds 
$x|_{J^\ell_k}\geq \frac{\sigma}{1-\sigma}h$. 
Hence, for all $\gamma < s+1$,
\begin{equation*}
h^{2s} \| v^{(s+1)} \|_{L^2(J^\ell_k)}^2  
\leq 
\sigmaratio^{2(s+1-\gamma)}h^{2\gamma-2}\| x^{s+1-\gamma} v^{(s+1)}\|_{L^2(J^\ell_k)}^{2}.
\end{equation*}
This concludes the proof.
\end{proof}
%%%%%%%%%%%%%%%%%%%%%%%%%%%%%%%%%%%%%%%%%%%%%%%%%%%%%%%%%%%%%%%%
\subsubsection{Estimate at a corner in dimension $d=2$}
\label{sec:twod}
%%%%%%%%%%%%%%%%%%%%%%%%%%%%%%%%%%%%%%%%%%%%%%%%%%%%%%%%%%%%%%%%
We consider now a setting with a two dimensional corner singularity. 
Let $\beta\in\R$, 
$\fK =J_0^\ell\times J_0^\ell$, $r(x) = |x-x_0|$ with $x_0 = (0,0)$, 
and define the corner-weighted norm $\| v \|_{\cJ^{2}_\beta(\fK)}$ by
\begin{equation*}
\| v \|_{\cJ^{2}_\beta(\fK)}^2 \coloneqq \sum_{\alpham\leq 2} \|r^{(\alpham - \beta)_+}\dalpha v\|^2_{L^2(\fK)}.
\end{equation*}
\begin{lemma}
\label{lemma:2d-square}
Let $d = 2$,  $\beta\in (1,2)$. There exists $C_1, C_2>0$ such that for
all $v\in \cJ^2_\beta(\fK)$ 
\begin{equation} \label{eq:2d-square-stab}
    \sum_{\alpha \in \mathbb{N}^2_0:\alpham\leq 1} \|\dalpha(\pi^0_1 \otimes \pi^0_1) v\|_{L^2(\fK)} \leq C_1
\left(  \|v\|_{H^1(\fK)} + \sum_{\alpha \in \mathbb{N}^2_0 : \alpham =2}\sigma^{(\beta-1)\ell}\| r^{2 - \beta} \dalpha v\|_{L^2(\fK)}  \right)
\;.
\end{equation}
and
  \begin{equation}
    \label{eq:2d-square-approx}
    \sum_{\alpha \in \mathbb{N}^2_0: \alpham \leq 1} \sigma^{-\ell(1-\alpham)}\|\dalpha (v - (\pi^0_1 \otimes \pi^0_1) v)\|_{L^2(\fK)} \leq C_2 \sigma^{\ell (\beta-1)}\sum_{\alpha \in \mathbb{N}^2_0: \alpham =2}\|r^{2-\beta}\dalpha v\|_{L^2(\fK)}.
  \end{equation}
\end{lemma}
\begin{proof}
  Denote by $\corn_i$, $i=1, \dots, 4$ the corners of $\fK$ and by $\psi_i$, $i=1,
  \dots, 4$ the bilinear functions such that $\psi_i(\corn_j) = \delta_{ij}$. Then,
  \begin{equation*}
    (\pi^0_1\otimes \pi_1^0) v = \sum_{i=1}^4 v(\corn_i)\psi_i.
  \end{equation*}
  Therefore, writing $h=\sigma^\ell$, we have 
  \begin{equation}
    \label{eq:stab-proof-1}
   \|  (\pi^0_1\otimes \pi_1^0) v  \|_{L^2(\fK)} \leq \sum_{i=1, \dots, 4}|v(\corn_i)| \|\psi_i \|_{L^2(\fK)} \leq 4 \|v\|_{L^\infty(\fK)} |\fK|^{1/2}\leq 4 h \|v\|_{L^\infty(\fK)}.
  \end{equation}
   With the imbedding $\cJ^2_\beta ((0,1)^2)\hookrightarrow L^\infty((0,1)^2)$ 
   which is valid for $\beta>1$
  (which follows e.g. from Lemma \ref{lemma:W11J3} and 
  $\Wmix^{1,1}((0,1)^2)\hookrightarrow L^\infty((0,1)^2)$), a scaling argument gives
  \begin{equation*}
    h^2 \|v \|_{L^\infty(\fK)}^2 \leq C h^2 \left( h^{-2} \|v\|_{L^2(\fK)}^2 + |v|_{H^1(\fK)}^2 + \sum_{\alpham =2}h^{2\beta-2}\| r^{2 - \beta} \dalpha v\|_{L^2(\fK)}^2 \right),
  \end{equation*}
  so that we obtain 
  \begin{equation}
  \label{eq:2d-square-stab-L2}
  \|  (\pi^0_1\otimes \pi_1^0) v  \|_{L^2(\fK)}^2
  \leq C \left( \|v\|_{L^2(\fK)}^2 + h^2 |v|_{H^1(\fK)}^2 
  	+ \sum_{\alpham =2}h^{2\beta}\| r^{2 - \beta} \dalpha v\|_{L^2(\fK)}^2 \right).
  \end{equation}
  For any $\alpham = 1$, denoting $v_0  = v(0,0)$ and using the
    fact that $(\pi^0_1\otimes \pi_1^0) v_0 = v_0$ hence $\dalpha (\pi^0_1\otimes \pi_1^0) v_0 = 0$,
  \begin{equation}
    \label{eq:stab-proof-2}
    \| \dalpha (\pi^0_1\otimes \pi_1^0) v  \|_{L^2(\fK)}
 	= \| \dalpha (\pi^0_1\otimes \pi_1^0) (v - v_0)  \|_{L^2(\fK)}
 	\leq \sum_{i=1, \dots, 4}|(v-v_0)(\corn_i)| \| \dalpha \psi_i \|_{L^2(\fK)}
    \leq C \|v - v_0\|_{L^\infty(\fK)}.
  \end{equation}
  With the imbedding $\cJ^2_\beta ((0,1)^2)\hookrightarrow L^\infty((0,1)^2)$, Poincaré's inequality, and rescaling
  we obtain
    \begin{align*}
  \| \dalpha (\pi^0_1\otimes \pi_1^0) v  \|_{L^2(\fK)}^2
  \leq C \left( |v|_{H^1(\fK)}^2 
  	+ \sum_{\alpham =2}h^{2\beta - 2}\| r^{2 - \beta} \dalpha v\|_{L^2(\fK)}^2 \right),
  \end{align*}
 which finishes the proof of \eqref{eq:2d-square-stab}.
To prove \eqref{eq:2d-square-approx}, note that by the Sobolev imbedding of
  $W^{2,1}(\fK)$ into $H^1(\fK)$ and by scaling, 
  we have
  \begin{equation*}
\sum_{\alpham \leq 1} h^{\alpham-1}\|\dalpha (v - (\pi^0_1 \otimes \pi^0_1) v)\|_{L^2(\fK)} \leq 
C \sum_{\alpham \leq 2} h^{\alpham-2}\|\dalpha (v - (\pi^0_1 \otimes \pi^0_1) v)\|_{L^1(\fK)} .
  \end{equation*}
  By classical interpolation estimates \cite[Theorem 4.4.4]{Brenner2008}, we additionally conclude that
  \begin{equation*}
    \sum_{\alpham \leq 1} h^{\alpham-2}\|\dalpha (v - (\pi^0_1 \otimes \pi^0_1) v)\|_{L^1(\fK)} 
    \leq C | v |_{W^{2,1}(\fK)}.
  \end{equation*}
  Using the Cauchy-Schwarz inequality,
  \begin{align*}
\sum_{\alpham \leq 1} h^{\alpham-1}\|\dalpha (v - (\pi^0_1 \otimes \pi^0_1) v)\|_{L^2(\fK)}
    &\leq C \sum_{\alpham =2}\| \dalpha v \|_{L^1(\fK)}
    \\ &\leq C \sum_{\alpham =2}\|r^{-2+\beta}\|_{L^2(\fK)}\| r^{2-\beta}\dalpha v \|_{L^2(\fK)}
    \\ &\leq C \sum_{\alpham =2}h^{\beta-1}\| r^{2-\beta}\dalpha v \|_{L^2(\fK)}
  \end{align*}
  where we also have used, in the last step, the facts that $r(x)\leq \sqrt{2}h$ for all $x\in \fK$
  and that
  $\beta >1$.
\end{proof}
%%%%%%%%%%%%%%%%%%%%%%%%%%%%%%%%%%%%%%%%%%%%%%%%%%%%%%%%%%%%%%%%
\subsection{Interior estimates}
\label{sec:internal}
%%%%%%%%%%%%%%%%%%%%%%%%%%%%%%%%%%%%%%%%%%%%%%%%%%%%%%%%%%%%%%%%
%
The following lemmas give the estimate of the approximation error on the elements not
belonging to edge or corner layers.
   For $d=3$, all $\ell\in \mathbb{N}$, all $k_1,k_2,k_3\in\{0,\ldots,\ell\}$ 
   and all $K = J^\ell_{k_1}\times J^\ell_{k_2}\times J^{\ell}_{k_3}$,
   we denote, by $\hpar$ the length of
   $K$ in the direction parallel to the closest singular edge, and by $\hperpone$
   and $\hperptwo$ the lengths of $K$ in the other two directions. 
   If an element has multiple closest singular edges, 
   we choose one of those and consider it as ``closest edge'' for all points in that element.
   When considering functions from $\cJ^d_\ugamma(Q)$, 
   $\gamma_e$ will refer to the weight of this closest edge.
   Similarly, we
   denote by $\dpar$ (resp. $\dperpone$ and $\dperptwo$) the derivatives in the
   direction parallel (resp. perpendicular) to the closest singular edge.
\begin{lemma}
\label{lemma:internal-appx-1}
Let $d=3$, $\ell\in \mathbb{N}$ and $K = J^\ell_{k_1}\times J^\ell_{k_2}\times J^{\ell}_{k_3}$ for
$0< k_1, k_2, k_3\leq \ell$. Let also $v\in
  \cJ^\varpi_\ugamma(Q; \Cset,\Eset; C_v, A_v)$
  with $%\forall\corn\in\Cset:
  \gamma_\corn \in (3/2, 5/2)$, $%\forall e\in\Eset:
  \gamma_e \in (1, 2)$. Then, there exists $C>0$ dependent only on
  $\sigma$, $\Cappxii$, $C_v$ and $A>0$ dependent only on $\sigma$, $A_v$ such that
  for all $1\leq s\leq p$
  \begin{equation}
\label{eq:internal-appx-1}
        \|\dpar ( v - \Pi^K_p v )\|^2_{L^2(K)}
\leq C \Psi_{p,s}  A^{2s+6}\left( (d_\corn^K)^{2} + (d_\corn^K)^{2(\gamma_\corn-1)} \right)((s+3)!)^2,
  \end{equation}
  where $\dpar$ is the derivative in the direction parallel to the closest
  singular edge.
  \end{lemma}
\begin{proof}
  We write $d_a = d_a^K$, $a \in \{\corn, e\}$.
   There holds
  \begin{equation*}
    d_\corn^2 = \left( \frac{\sigma}{1-\sigma} \right)^2(\hpar^2+\hperpone^2+\hperptwo^2),
    \qquad
    d_e^2 = \left( \frac{\sigma}{1-\sigma} \right)^2(\hperpone^2+\hperptwo^2).
  \end{equation*}
  Denoting $\hv = v \circ \Phi_K^{-1}$ and 
  $\hPi_p \hv = \Pi^K_p v \circ \Phi_K^{-1} = \hPi_p (v \circ \Phi_K)$,
  using the result of Lemma \ref{lemma:ref-proj2} and rescaling, we have
  \begin{equation}
    \label{eq:Kappx1}
\begin{aligned}
\| \widehat{\partial}_{\parallel}( \hv - \hPi_p \hv  )\|_{L^2(\hK)}^2 
    &\leq \Cappxii \Psi_{p,s} \frac{\hpar}{\hperpone\hperptwo}\left(\sum_{\alpha_1, \alpha_2 \leq 1}
    \hpar^{2s}\hperpone^{2\alpha_1}\hperptwo^{2\alpha_2}\| \dpar^{s+1}\dperpone^{ \alpha_1}\dperptwo^{\alpha_2}v \|^2_{L^2(K)} 
    \right.
    \\ &\qquad +
    \sum_{\alpha_1\leq 1}\hperpone^{2s+2}\hperptwo^{2\alpha_1}\| \dpar\dperpone^{s+1}\dperptwo^{\alpha_1}v \|^2_{L^2(K)} 
    \\ &\qquad +
    \left.
    \sum_{\alpha_1\leq 1}\hperpone^{2\alpha_1}\hperptwo^{2s+2}\| \dpar\dperpone^{\alpha_1}\dperptwo^{s+1}v \|^2_{L^2(K)} 
  \right)
  \\
  & = \Cappxii \Psi_{p,s}  \frac{\hpar}{\hperpone\hperptwo}\bigg((I) + (II) + (III)  \bigg).
%     \| \hv - \hPi_p \hv\|_{\Hmix^1(\hK)}^2   
% & \leq \Cappxii \frac{(p-s)!}{(p+s-2)!} \left(
%     h_1^{2s+1}h_2h_3\| \partial^{(s+1, 1, 1)}v \|^2_{L^2(K)} \right.
%     \\ &\qquad +
%     h_1h_2^{2s+1}h_3\| \partial^{(1, s+1, 1)}v \|^2_{L^2(K)} \\
%   & \left. \qquad +
%       h_1h_2h_3^{2s+1}\| \partial^{(1, 1, s+1)}v \|^2_{L^2(K)} \right)\\
  \end{aligned}
  \end{equation}
Denote $K_\corn = K\cap Q_\corn$, $K_e=K\cap Q_e$, $K_{\corn e} = K\cap Q_{\corn e}$, and $K_0 = K\cap Q_0$. 
Furthermore, we indicate
  \begin{equation*}
    (I)_{\corn} =   \sum_{\alpha_1, \alpha_2\leq 1}\hpar^{2s}\hperpone^{2\alpha_1}\hperptwo^{2\alpha_2}\| \dpar^{s+1}\dperpone^{ \alpha_1}\dperptwo^{\alpha_2}v \|^2_{L^2(K_\corn)},
  \end{equation*}
  and do similarly for the other terms of the sum $(II)$ and $(III)$ and the other subscripts $e$,
  $\corn e$, $0$.
  Remark also that $r_{i|_K}\geq d_i$, $i\in\{\corn, e\}$, and that for $a, b\in \mathbb{R}$ holds
  $r_\corn^ar_e^b  = r_\corn^{a+b} \rho_{\corn e}^{b}$.

  We will also write $\tgamma = \gamma_\corn-\gamma_e$.
We start by considering the term $(I)_{\corn e}$. Let $\alpha_1= \alpha_2 = 1$; then,
\begin{align*}
\hpar^{2s}\hperpone^{2}\hperptwo^{2}\| \dpar^{s+1}\dperpone\dperptwo v \|^2_{L^2(K_{\corn e})}
   & \leq \sigmaratio^{2s+4} d_\corn^{2s}d_e^{4}\| \dpar^{s+1}\dperpone\dperptwo v \|^2_{L^2(K_{\corn e})}
  \\  &
     \leq \sigmaratio^{2s+4} d_\corn^{2\tgamma-2}d_e^{2\gamma_e}
           \|r_\corn^{s+3-\gamma_\corn}\rho_{\corn e}^{2-\gamma_e} \dpar^{s+1}\dperpone\dperptwo v \|^2_{L^2(K_{\corn e})}\;,
\end{align*}
where $\sigmaratio$ is as in \eqref{eq:sigmaratio}.
Furthermore, 
if $\alpha_1 + \alpha_2\leq 1$ and $s+1+\alpha_1+\alpha_2-\gamma_\corn\geq0$,
\begin{align*}
\hpar^{2s}\hperpone^{2\alpha_1}\hperptwo^{2\alpha_2}\| \dpar^{s+1}\dperpone^{\alpha_1}\dperptwo^{\alpha_2} v \|^2_{L^2(K_{\corn e})}
   & \leq \sigmaratio^{2s+2(\alpha_1+\alpha_2)} d_\corn^{2s}d_e^{2(\alpha_1+\alpha_2)}\| \dpar^{s+1}\dperpone^{\alpha_1}\dperptwo^{\alpha_2} v \|^2_{L^2(K_{\corn e})}
  \\  &
        \leq \sigmaratio^{2s+2(\alpha_1+\alpha_2)} d_\corn^{2\gamma_\corn-2}\|r_\corn^{s+1+\alpha_1+\alpha_2-\gamma_\corn} \dpar^{s+1}\dperpone^{\alpha_1}\dperptwo^{\alpha_2} v \|^2_{L^2(K_{\corn e})},
\end{align*}
where we have also used $d_e\leq d_\corn$. Therefore,
\begin{equation*}
  (I)_{\corn e} \leq \sigmaratio^{2s+4}
%  \max\left( d_\corn^{2\gamma_\corn-2}, d_\corn^{2\tgamma-2}d_e^{2\gamma_e} \right)
  d_\corn^{2\gamma_\corn-2}
  \sum_{\alpha_1, \alpha_2\leq 1}\|r_\corn^{s+1+\alpha_1+\alpha_2-\gamma_\corn}\rho_{\corn e}^{(\alpha_1+\alpha_2-\gamma_e)_+} \dpar^{s+1}\dperpone^{\alpha_1}\dperptwo^{\alpha_2} v \|^2_{L^2(K_{\corn e})}.
\end{equation*}
If $s+1+\alpha_1+\alpha_2-\gamma_\corn<0$, then $s=1$ and $\alpha_1=\alpha_2=0$, thus 
  $$(I)_{\corn e} \leq \sigmaratio^{2s+4} d_\corn^{2} 
  \| r_\corn^{(s+1+\alpha_1+\alpha_2-\gamma_\corn)_+}\rho_{\corn e}^{(\alpha_1+\alpha_2-\gamma_e)_+}\dpar^{s+1}\dperpone^{ \alpha_1}\dperptwo^{\alpha_2}v \|^2_{L^2(K_{\corn e})}. $$
 Then, if $s+1+\alpha_1+\alpha_2-\gamma_\corn\geq0$
\begin{align*}
   (I)_{\corn}
  & = \sum_{\alpha_1, \alpha_2 \leq 1}
    \hpar^{2s}\hperpone^{2\alpha_1}\hperptwo^{2\alpha_2}\| \dpar^{s+1}\dperpone^{ \alpha_1}\dperptwo^{\alpha_2}v \|^2_{L^2(K_\corn)} 
    \\
  & \leq \sigmaratio^{2s+4}\sum_{\alpha_1, \alpha_2 \leq 1}
    d_\corn^{2s}d_e^{2(\alpha_1+\alpha_2)}\| \dpar^{s+1}\dperpone^{ \alpha_1}\dperptwo^{\alpha_2}v \|^2_{L^2(K_\corn)} 
    \\
  & \leq \sigmaratio^{2s+4}
    d_\corn^{2\gamma_\corn-2}\sum_{\alpha_1, \alpha_2 \leq 1}\| r_\corn^{(s+1+\alpha_1+\alpha_2-\gamma_\corn)_+}\dpar^{s+1}\dperpone^{ \alpha_1}\dperptwo^{\alpha_2}v \|^2_{L^2(K_\corn)} 
  \end{align*}
  where the last inequality follows also from $d_e\leq d_\corn$.
If $s+1+\alpha_1+\alpha_2-\gamma_\corn<0$, then the same bound holds 
with $d_\corn^{2\gamma_\corn-2}$ replaced by $d_\corn^2$.
%  If $s+1+\alpha_1+\alpha_2-\gamma_\corn<0$, then $s=1$ and thus 
%  $$(I)_{\corn} \leq \sigmaratio^{2s+4} d_\corn^{2} 
%  \| \dpar^{s+1}\dperpone^{ \alpha_1}\dperptwo^{\alpha_2}v \|^2_{L^2(K_\corn)}. $$
  Similarly,
\begin{align*}
   (I)_{e}
  & = \sum_{\alpha_1, \alpha_2 \leq 1}
    \hpar^{2s}\hperpone^{2\alpha_1}\hperptwo^{2\alpha_2}\| \dpar^{s+1}\dperpone^{ \alpha_1}\dperptwo^{\alpha_2}v \|^2_{L^2(K_e)} 
    \\
  &\leq\sigmaratio^{2s+4}  \sum_{\alpha_1, \alpha_2 \leq 1}
    d_\corn^{2s}d_e^{2\alpha_1+2\alpha_2 - 2(\alpha_1+\alpha_2-\gamma_e)_+}\| r_e^{(\alpha_1+\alpha_2-\gamma_e)_+}\dpar^{s+1}\dperpone^{ \alpha_1}\dperptwo^{\alpha_2}v \|^2_{L^2(K_e)} 
    \\
  &\leq\sigmaratio^{2s+4}  
    d_\corn^{2s}\sum_{\alpha_1, \alpha_2 \leq 1}\| r_e^{(\alpha_1+\alpha_2-\gamma_e)_+}\dpar^{s+1}\dperpone^{ \alpha_1}\dperptwo^{\alpha_2}v \|^2_{L^2(K_e)},
\end{align*}
where we used that $d_e\leq 1$.
The bound on $(I)_0$ follows directly from the definition:
\begin{align*}
   (I)_0
  & = \sum_{\alpha_1, \alpha_2 \leq 1}
    \hpar^{2s}\hperpone^{2\alpha_1}\hperptwo^{2\alpha_2}\| \dpar^{s+1}\dperpone^{ \alpha_1}\dperptwo^{\alpha_2}v \|^2_{L^2(K_0)}
%    \\&
  \leq\sigmaratio^{2s+4}  
    d_\corn^{2s}\sum_{\alpha_1, \alpha_2 \leq 1}\| \dpar^{s+1}\dperpone^{ \alpha_1}\dperptwo^{\alpha_2}v \|^2_{L^2(K_0)}. 
\end{align*}
%jo
  Using \eqref{eq:analytic}, there exists $C>0$ dependent only on $C_v$ and
  $\sigma$ and $A>0$ dependent only on $A_v$ and $\sigma$
  such that
  \begin{equation}
    \label{eq:Ibound}
    (I) \leq C A^{2s+6}((s+3)!)^2\left( %d_\corn^{2\tgamma-2} d_e^{2\gamma_e}
    d_\corn^2 + d_\corn^{2\gamma_\corn-2} \right).
  \end{equation}
   We then apply the same argument to the terms $(II)$ and $(III)$. Indeed,
\begin{align*}
   (II)_{\corn e}
  &=
\sum_{\alpha_1\leq 1}\hperpone^{2s+2}\hperptwo^{2\alpha_1}\| \dpar\dperpone^{s+1}\dperptwo^{\alpha_1}v \|^2_{L^2(K_{\corn e})} 
    \\
  &\leq \sigmaratio^{2s+4}
\sum_{\alpha_1\leq 1}d_e^{2s+2+2\alpha_1}\| \dpar\dperpone^{s+1}\dperptwo^{\alpha_1}v \|^2_{L^2(K_{\corn e})} 
    \\
  &\leq \sigmaratio^{2s+4}
\sum_{\alpha_1\leq 1}d_\corn^{2\tgamma-2}d_e^{2\gamma_e}\| r_\corn^{s+2+\alpha_1-\gamma_\corn}\rho_{\corn e}^{s+1+\alpha_1-\gamma_e}\dpar\dperpone^{s+1}\dperptwo^{\alpha_1}v \|^2_{L^2(K_{\corn e})} 
  \end{align*}
  and the estimate for $(III)_{\corn e}$ follows by exchanging $\hperpone$ and
  $\dperpone$ with $\hperptwo$ and $\dperptwo$ in the inequality above.
  The estimates for $(II)_{\corn,e,0}$ and $(III)_{\corn, e, 0}$ can be obtained as for
  $(I)_{\corn, e, 0}$:
  %[intermediate steps in .tex as comments]
  \begin{align*}
   (II)_{\corn}
%  &=
%\sum_{\alpha_1\leq 1}\hperpone^{2s+2}\hperptwo^{2\alpha_1}\| \dpar\dperpone^{s+1}\dperptwo^{\alpha_1}v \|^2_{L^2(K_{\corn})} 
%    \\
%  &\leq \sigmaratio^{2s+4}
%\sum_{\alpha_1\leq 1}d_e^{2s+2+2\alpha_1}\| \dpar\dperpone^{s+1}\dperptwo^{\alpha_1}v \|^2_{L^2(K_{\corn})} 
%    \\
  &\leq \sigmaratio^{2s+4}
\sum_{\alpha_1\leq 1}d_\corn^{2\gamma_\corn-2}\| r_\corn^{s+2+\alpha_1-\gamma_\corn}\dpar\dperpone^{s+1}\dperptwo^{\alpha_1}v \|^2_{L^2(K_{\corn})},
\\
   (II)_{e}
%  &=
%\sum_{\alpha_1\leq 1}\hperpone^{2s+2}\hperptwo^{2\alpha_1}\| \dpar\dperpone^{s+1}\dperptwo^{\alpha_1}v \|^2_{L^2(K_{e})} 
%    \\
%  &\leq \sigmaratio^{2s+4}
%\sum_{\alpha_1\leq 1}d_e^{2s+2+2\alpha_1}\| \dpar\dperpone^{s+1}\dperptwo^{\alpha_1}v \|^2_{L^2(K_{e})} 
%    \\
  &\leq \sigmaratio^{2s+4}
\sum_{\alpha_1\leq 1}d_e^{2\gamma_e}\| r_e^{s+1+\alpha_1-\gamma_e}\dpar\dperpone^{s+1}\dperptwo^{\alpha_1}v \|^2_{L^2(K_{e})},
\\
   (II)_{0}
%  &=
%\sum_{\alpha_1\leq 1}\hperpone^{2s+2}\hperptwo^{2\alpha_1}\| \dpar\dperpone^{s+1}\dperptwo^{\alpha_1}v \|^2_{L^2(K_{0})} 
%    \\
%  &\leq \sigmaratio^{2s+4}
%\sum_{\alpha_1\leq 1}d_e^{2s+2+2\alpha_1}\| \dpar\dperpone^{s+1}\dperptwo^{\alpha_1}v \|^2_{L^2(K_{0})} 
%    \\
  &\leq \sigmaratio^{2s+4}
\sum_{\alpha_1\leq 1}d_e^{2s+2}\| \dpar\dperpone^{s+1}\dperptwo^{\alpha_1}v \|^2_{L^2(K_{0})}.
  \end{align*}
  Therefore, we have 
  \begin{equation}
    \label{eq:IIbound}
    (II), (III) \leq C A^{2s+6}(d_\corn^2+d_\corn^{2\gamma_\corn-2})((s+3)!)^2. %was d_\corn^{2\gamma_\corn-2}
  \end{equation}
%  and
 % \begin{equation}
  %  \label{eq:IIIbound}
   % (III) \leq C A^{2s+6}(d_\corn^2+d_\corn^{2\gamma_\corn-2})((s+3)!)^2.
 % \end{equation}
  We obtain, from \eqref{eq:Kappx1}, \eqref{eq:Ibound}, and \eqref{eq:IIbound}
  that there exists $C>0$ (dependent only on $\sigma$, $\Cappxii$, $C_v$ and 
  $A>0$ (dependent only on $\sigma$, $A_v$) 
  such that
  \begin{equation*}
  \| \widehat{\partial}_\parallel(\hv - \hPi_p \hv)\|_{L^2(\hK)}^2  
    \leq 
   C \frac{\hpar}{\hperpone\hperptwo}\Psi_{p,s} A^{2s+6}(d_\corn^2+d_\corn^{2\gamma_\corn-2})((s+3)!)^2.
  \end{equation*}
  Considering that
  \begin{equation*}
  \|\dpar (v-\Pi_p v)\|^2_{L^2(K)}\leq \frac{\hperpone\hperptwo}{\hpar} \|\widehat{\partial}_\parallel (\hv - \hPi_p\hv)\|^2_{L^2(\hK)}
  \end{equation*}
  completes the proof.
  \end{proof}
\begin{lemma}
\label{lemma:internal-appx-2}
Let $d=3$, $\ell\in \mathbb{N}$ and $K = J^\ell_{k_1}\times J^\ell_{k_2}\times J^{\ell}_{k_3}$ for
$0< k_1, k_2, k_3\leq \ell$.
 Let also $v\in
  \cJ^\varpi_\ugamma(Q; \Cset, \Eset; C_v, A_v)$
  with $\gamma_\corn \in (3/2, 5/2)$, $\gamma_e \in (1, 2)$. Then, there exists $C>0$ dependent only on
  $\sigma$, $\Cappxii$, $C_v$ and $A>0$ dependent only on $\sigma$, $A_v$ such that
  for all $p\in\N$ and all $1\leq s \leq p$
  \begin{multline}
\label{eq:internal-appx-2}
 \|\dperpone ( v - \Pi^K_p v )\|^2_{L^2(K)}+
\|\dperptwo ( v - \Pi^K_p v )\|^2_{L^2(K)}
\\
\leq  C \Psi_{p,s} A^{2s+6}\left((d_\corn^K)^{2(\gamma_\corn-1)} + (d_\corn^K)^{2(\gamma_e-1)}\right)((s+3)!)^2,
  \end{multline}
  where $\dperpone$, $\dperptwo$ are the derivatives in the directions perpendicular to the closest
  singular edge.
  \end{lemma}
\begin{proof}
  The proof follows closely that of Lemma \ref{lemma:internal-appx-1} and we use
  the same notation.
   From Lemma \ref{lemma:ref-proj2} and rescaling, we have
  \begin{equation}
    \label{eq:Kappx1-2}
\begin{aligned}
\| \widehat{\partial}_{\bot,1}( \hv - \hPi_p \hv  )\|_{L^2(\hK)}^2 
    &\leq \Cappxii \Psi_{p,s} \frac{\hperpone}{\hpar\hperptwo}\left(\sum_{\alpha_1 \leq 1}
    \hpar^{2s+2}\hperptwo^{2\alpha_1}\| \dpar^{s+1}\dperpone\dperptwo^{\alpha_1}v \|^2_{L^2(K)} 
    \right.
    \\ &\qquad +
    \sum_{\alpha_1, \alpha_2\leq 1}\hpar^{2\alpha_1}\hperpone^{2s}\hperptwo^{2\alpha_2}\| \dpar^{\alpha_1}\dperpone^{s+1}\dperptwo^{\alpha_2}v \|^2_{L^2(K)} 
    \\ &\qquad +
    \left.
    \sum_{\alpha_1\leq 1}\hpar^{2\alpha_1}\hperptwo^{2s+2}\| \dpar^{\alpha_1}\dperpone\dperptwo^{s+1}v \|^2_{L^2(K)} 
  \right)
  \\
  & = \Cappxii \Psi_{p,s} \frac{\hperpone}{\hpar\hperptwo}\bigg((I) + (II) + (III)  \bigg).
  \end{aligned}
  \end{equation}
      As before, we will write $\tgamma = \gamma_\corn-\gamma_e$.
We start by considering the term $(I)_{\corn e}$. When $\alpha_1 = 1$,
\begin{align*}
\hpar^{2s+2}\hperptwo^{2}\| \dpar^{s+1}\dperpone\dperptwo v \|^2_{L^2(K_{\corn e})}
   & \leq \sigmaratio^{2s+4} d_\corn^{2s+2}d_e^{2}\| \dpar^{s+1}\dperpone\dperptwo v \|^2_{L^2(K_{\corn e})}
  \\  &
        \leq \sigmaratio^{2s+4} d_\corn^{2\tgamma}d_e^{2\gamma_e-2}\|r_\corn^{s+3-\gamma_\corn}\rho_{\corn e}^{2-\gamma_e} \dpar^{s+1}\dperpone\dperptwo v \|^2_{L^2(K_{\corn e})},
\end{align*}
where $d_\corn^{2\tgamma}d_e^{2\gamma_e-2} \leq d_\corn^{2\gamma_\corn-2}$.
Furthermore, if $\alpha_1 =0$,
\begin{align*}
\hpar^{2s+2}\| \dpar^{s+1}\dperpone v \|^2_{L^2(K_{\corn e})}
   & \leq \sigmaratio^{2s+2} d_\corn^{2s+2}\| \dpar^{s+1}\dperpone v \|^2_{L^2(K_{\corn e})}
   %%% JO:removed d_e^{2} on RHS
  \\  &
        \leq \sigmaratio^{2s+2} d_\corn^{2\gamma_\corn-2}\|r_\corn^{s+2-\gamma_\corn} \dpar^{s+1}\dperpone v \|^2_{L^2(K_{\corn e})}.
\end{align*}
Therefore,
\begin{equation*}
  (I)_{\corn e} \leq \left(  \frac{1-\sigma}{\sigma}\right)^{2s+4} d_\corn^{2\gamma_\corn-2}\sum_{\alpha_1\leq 1}\|r_\corn^{s+2+\alpha_1-\gamma_\corn}\rho_{\corn e}^{(1+\alpha_1-\gamma_e)_+} \dpar^{s+1}\dperpone\dperptwo^{\alpha_1} v \|^2_{L^2(K_{\corn e})}.
\end{equation*}
The estimates for $(I)_{\corn, e, 0}$ follow from the same technique:
%[intermediate steps in .tex as comments]
\begin{align*}
%\hpar^{2s+2}\hperptwo^{2}\| \dpar^{s+1}\dperpone\dperptwo v \|^2_{L^2(K_{e})}
%   & \leq \sigmaratio^{2s+4} d_\corn^{2s+2}d_e^{2}\| \dpar^{s+1}\dperpone\dperptwo v \|^2_{L^2(K_{e})}
%  \\  &
%        \leq \sigmaratio^{2s+4} d_\corn^{2s+2}d_e^{2\gamma_e-2}\|r_e^{(1+\alpha_1-\gamma_e)_+} \dpar^{s+1}\dperpone\dperptwo v \|^2_{L^2(K_{e})},
%\\
%\hpar^{2s+2}\| \dpar^{s+1}\dperpone v \|^2_{L^2(K_{e})}
%   & \leq \sigmaratio^{2s+2} d_\corn^{2s+2}\| \dpar^{s+1}\dperpone v \|^2_{L^2(K_{e})},
%\\
(I)_{e}
%	& \leq \sum_{\alpha_1 \leq 1} \hpar^{2s+2}\hperptwo^{2\alpha_1}\| \dpar^{s+1}\dperpone\dperptwo^{\alpha_1} v \|^2_{L^2(K_{e})}
%   \\
   & \leq \sum_{\alpha_1 \leq 1} \sigmaratio^{2s+4} d_\corn^{2s+2}\|r_e^{(1+\alpha_1-\gamma_e)_+} \dpar^{s+1}\dperpone\dperptwo^{\alpha_1} v \|^2_{L^2(K_{e})},
\\
(I)_{\corn}
%	& \leq \sum_{\alpha_1 \leq 1} \hpar^{2s+2}\hperptwo^{2\alpha_1}\| \dpar^{s+1}\dperpone\dperptwo^{\alpha_1} v \|^2_{L^2(K_{\corn})}
%	\\
%   & \leq \sum_{\alpha_1 \leq 1} \sigmaratio^{2s+4} d_\corn^{2s+2}d_e^{2\alpha_1}\| \dpar^{s+1}\dperpone\dperptwo^{\alpha_1} v \|^2_{L^2(K_{\corn})}
%  \\  
  & \leq \sum_{\alpha_1 \leq 1} \sigmaratio^{2s+4} d_\corn^{2\gamma_\corn-2}\|r_\corn^{s+2+\alpha_1-\gamma_\corn} \dpar^{s+1}\dperpone\dperptwo^{\alpha_1} v \|^2_{L^2(K_{\corn})}, %%% d_e^{2\alpha_1}
\\
(I)_{0}
%  & \leq \sum_{\alpha_1 \leq 1}
%\hpar^{2s+2}\hperptwo^{2\alpha_1}\| \dpar^{s+1}\dperpone\dperptwo^{\alpha_1} v \|^2_{L^2(K_{0})}
%\\
   & \leq \sum_{\alpha_1 \leq 1} \sigmaratio^{2s+4} d_\corn^{2s+2}\| \dpar^{s+1}\dperpone\dperptwo^{\alpha_1} v \|^2_{L^2(K_{0})}. %%% d_e^{2\alpha_1}
\end{align*}%jo
  Hence, from \eqref{eq:analytic}, there exists $C>0$ dependent only on $C_v$ and
  $\sigma$ and $A>0$ dependent only on $A_v$ and $\sigma$
  such that
  \begin{equation}
    \label{eq:Ibound-2}
    (I) \leq C A^{2s+6}((s+3)!)^2 d_\corn^{2\gamma_\corn-2} .
  \end{equation}
   We then apply the same argument to the terms $(II)$ and $(III)$. Indeed, 
   if $s+1+\alpha_1+\alpha_2-\gamma_\corn\geq0$ 
\begin{align*}
   (II)_{\corn e}
  &=
\sum_{\alpha_1, \alpha_2\leq 1}\hpar^{2\alpha_1}\hperpone^{2s}\hperptwo^{2\alpha_2}\| \dpar^{\alpha_1}\dperpone^{s+1}\dperptwo^{\alpha_2}v \|^2_{L^2(K_{\corn e})} 
    \\
  &\leq \sigmaratio^{2s+4}
\sum_{\alpha_1, \alpha_2\leq 1}d_\corn^{2\alpha_1}d_e^{2s+2\alpha_2}\| \dpar^{\alpha_1}\dperpone^{s+1}\dperptwo^{\alpha_2}v \|^2_{L^2(K_{\corn e})} 
    \\
  &\leq \sigmaratio^{2s+4}
\sum_{\alpha_1\leq 1}d_\corn^{2\tgamma}d_e^{2\gamma_e-2}\| r_\corn^{s+1+\alpha_1+\alpha_2-\gamma_\corn}\rho_{\corn e}^{s+1+\alpha_2-\gamma_e}\dpar^{\alpha_1}\dperpone^{s+1}\dperptwo^{\alpha_2}v \|^2_{L^2(K_{\corn e})} 
    \\
  &\leq \sigmaratio^{2s+4}
\sum_{\alpha_1\leq 1}d_\corn^{2\gamma_\corn-2}\| r_\corn^{s+1+\alpha_1+\alpha_2-\gamma_\corn}\rho_{\corn e}^{s+1+\alpha_2-\gamma_e}\dpar^{\alpha_1}\dperpone^{s+1}\dperptwo^{\alpha_2}v \|^2_{L^2(K_{\corn e})},
  \end{align*}
  where at the last step we have used that $\gamma_e>1$ and $d_e\leq d_\corn$.
  If $s+1+\alpha_1+\alpha_2-\gamma_\corn<0$, then
\begin{align*}
   (II)_{\corn e}
  &=
\sum_{\alpha_1, \alpha_2\leq 1}\hpar^{2\alpha_1}\hperpone^{2s}\hperptwo^{2\alpha_2}\| \dpar^{\alpha_1}\dperpone^{s+1}\dperptwo^{\alpha_2}v \|^2_{L^2(K_{\corn e})} 
    \\
  &\leq \sigmaratio^{2s+4}
\sum_{\alpha_1, \alpha_2\leq 1}d_\corn^{2\alpha_1}d_e^{2s+2\alpha_2}\| \dpar^{\alpha_1}\dperpone^{s+1}\dperptwo^{\alpha_2}v \|^2_{L^2(K_{\corn e})} 
    \\
  &\leq \sigmaratio^{2s+4}
\sum_{\alpha_1\leq 1}d_\corn^{2\alpha_1}d_e^{2s+2\alpha_2}(d_e/d_\corn)^{-2s-2-2\alpha_2+2\gamma_e}\| \rho_{\corn e}^{s+1+\alpha_2-\gamma_e}\dpar^{\alpha_1}\dperpone^{s+1}\dperptwo^{\alpha_2}v \|^2_{L^2(K_{\corn e})} 
    \\
  &\leq \sigmaratio^{2s+4}
\sum_{\alpha_1\leq 1}d_\corn^{2s+2-2\gamma_e}d_e^{2\gamma_e-2}\| \rho_{\corn e}^{s+1+\alpha_2-\gamma_e}\dpar^{\alpha_1}\dperpone^{s+1}\dperptwo^{\alpha_2}v \|^2_{L^2(K_{\corn e})}. 
%%% d_\corn^{2\alpha_1+2\alpha_2}
  \end{align*}  
  Thus, using $d_e\leq d_\corn$,
  \begin{align*}
   (II)_{\corn e}
  &\leq \sigmaratio^{2s+4}
\sum_{\alpha_1\leq 1}(d_\corn^{2s}+d_\corn^{2\gamma_\corn-2})\| r_\corn^{(s+1+\alpha_1+\alpha_2-\gamma_\corn)_+}\rho_{\corn e}^{s+1+\alpha_2-\gamma_e}\dpar^{\alpha_1}\dperpone^{s+1}\dperptwo^{\alpha_2}v \|^2_{L^2(K_{\corn e})}. 
  \end{align*} 
  The estimates for $(II)_{\corn,e,0}$ and $(III)_{\corn e, \corn, e, 0}$ can be obtained as
  above:
  %[intermediate steps in .tex as comments]
  \begin{align*}
     (II)_{e}
%  &=
%\sum_{\alpha_1, \alpha_2\leq 1}\hpar^{2\alpha_1}\hperpone^{2s}\hperptwo^{2\alpha_2}\| \dpar^{\alpha_1}\dperpone^{s+1}\dperptwo^{\alpha_2}v \|^2_{L^2(K_{e})} 
%    \\
%  &\leq \sigmaratio^{2s+4}
%\sum_{\alpha_1, \alpha_2\leq 1}d_\corn^{2\alpha_1}d_e^{2s+2\alpha_2}\| \dpar^{\alpha_1}\dperpone^{s+1}\dperptwo^{\alpha_2}v \|^2_{L^2(K_{e})} 
%    \\
  &\leq \sigmaratio^{2s+4}
\sum_{\alpha_1\leq 1}d_e^{2\gamma_e-2}\| r_e^{s+1+\alpha_2-\gamma_e}\dpar^{\alpha_1}\dperpone^{s+1}\dperptwo^{\alpha_2}v \|^2_{L^2(K_{e})},
%\\
\end{align*}
if $s+1+\alpha_1+\alpha_2-\gamma_\corn\geq0$, then
\begin{align*}
   (II)_{\corn}
%  &=
%\sum_{\alpha_1, \alpha_2\leq 1}\hpar^{2\alpha_1}\hperpone^{2s}\hperptwo^{2\alpha_2}\| \dpar^{\alpha_1}\dperpone^{s+1}\dperptwo^{\alpha_2}v \|^2_{L^2(K_{\corn})} 
%    \\
%  &\leq \sigmaratio^{2s+4}
%\sum_{\alpha_1, \alpha_2\leq 1}d_\corn^{2\alpha_1}d_e^{2s+2\alpha_2}\| \dpar^{\alpha_1}\dperpone^{s+1}\dperptwo^{\alpha_2}v \|^2_{L^2(K_{\corn})} 
%    \\
  &\leq \sigmaratio^{2s+4}
\sum_{\alpha_1\leq 1}d_\corn^{2\gamma_\corn-2}\| r_\corn^{s+1+\alpha_1+\alpha_2-\gamma_\corn}\dpar^{\alpha_1}\dperpone^{s+1}\dperptwo^{\alpha_2}v \|^2_{L^2(K_{\corn})},
  \end{align*}
if $s+1+\alpha_1+\alpha_2-\gamma_\corn<0$, then
\begin{align*}
   (II)_{\corn}
%  &=
%\sum_{\alpha_1, \alpha_2\leq 1}\hpar^{2\alpha_1}\hperpone^{2s}\hperptwo^{2\alpha_2}\| \dpar^{\alpha_1}\dperpone^{s+1}\dperptwo^{\alpha_2}v \|^2_{L^2(K_{\corn})} 
%    \\
%  &\leq \sigmaratio^{2s+4}
%\sum_{\alpha_1, \alpha_2\leq 1}d_\corn^{2\alpha_1}d_e^{2s+2\alpha_2}\| \dpar^{\alpha_1}\dperpone^{s+1}\dperptwo^{\alpha_2}v \|^2_{L^2(K_{\corn})} 
%    \\
  &\leq \sigmaratio^{2s+4}
\sum_{\alpha_1\leq 1}d_\corn^{2s}\| \dpar^{\alpha_1}\dperpone^{s+1}\dperptwo^{\alpha_2}v \|^2_{L^2(K_{\corn})},
  \end{align*}
so that
\begin{align*}
   (II)_{\corn}
  &\leq \sigmaratio^{2s+4}
\sum_{\alpha_1\leq 1}(d_\corn^{2s}+d_\corn^{2\gamma_\corn-2})\| \dpar^{\alpha_1}\dperpone^{s+1}\dperptwo^{\alpha_2}v \|^2_{L^2(K_{\corn})},
\\
   (II)_{0}
%  &=
%\sum_{\alpha_1, \alpha_2\leq 1}\hpar^{2\alpha_1}\hperpone^{2s}\hperptwo^{2\alpha_2}\| \dpar^{\alpha_1}\dperpone^{s+1}\dperptwo^{\alpha_2}v \|^2_{L^2(K_{0})} 
%    \\
%  &\leq \sigmaratio^{2s+4}
%\sum_{\alpha_1, \alpha_2\leq 1}d_\corn^{2\alpha_1}d_e^{2s+2\alpha_2}\| \dpar^{\alpha_1}\dperpone^{s+1}\dperptwo^{\alpha_2}v \|^2_{L^2(K_{0})} 
%    \\
  &\leq \sigmaratio^{2s+4}
\sum_{\alpha_1\leq 1}d_\corn^{2s}\| \dpar^{\alpha_1}\dperpone^{s+1}\dperptwo^{\alpha_2}v \|^2_{L^2(K_{0})},
\\
   (III)_{\corn e}
%  &=
%\sum_{\alpha_1\leq 1}\hpar^{2\alpha_1}\hperptwo^{2s+2}\| \dpar^{\alpha_1}\dperpone\dperptwo^{s+1} v \|^2_{L^2(K_{\corn e})} 
%    \\
%  &\leq \sigmaratio^{2s+4}
%\sum_{\alpha_1\leq 1}d_\corn^{2\alpha_1}d_e^{2s+2}\| \dpar^{\alpha_1}\dperpone\dperptwo^{s+1} v \|^2_{L^2(K_{\corn e})} 
%    \\
%  &\leq \sigmaratio^{2s+4}
%\sum_{\alpha_1\leq 1}d_\corn^{2\tgamma}d_e^{2\gamma_e-2}\| r_\corn^{s+2+\alpha_1-\gamma_\corn}\rho_{\corn e}^{s+2-\gamma_e} \dpar^{\alpha_1}\dperpone\dperptwo^{s+1} v \|^2_{L^2(K_{\corn e})} 
%    \\
  &\leq \sigmaratio^{2s+4}
\sum_{\alpha_1\leq 1}d_\corn^{2\gamma_\corn-2}\| r_\corn^{s+2+\alpha_1-\gamma_\corn}\rho_{\corn e}^{s+2-\gamma_e} \dpar^{\alpha_1}\dperpone\dperptwo^{s+1} v \|^2_{L^2(K_{\corn e})},
\\
   (III)_{e}
%  &=
%\sum_{\alpha_1\leq 1}\hpar^{2\alpha_1}\hperptwo^{2s+2}\| \dpar^{\alpha_1}\dperpone\dperptwo^{s+1} v \|^2_{L^2(K_{e})} 
%    \\
%  &\leq \sigmaratio^{2s+4}
%\sum_{\alpha_1\leq 1}d_\corn^{2\alpha_1}d_e^{2s+2}\| \dpar^{\alpha_1}\dperpone\dperptwo^{s+1} v \|^2_{L^2(K_{e})} 
%    \\
  &\leq \sigmaratio^{2s+4}
\sum_{\alpha_1\leq 1}d_e^{2\gamma_e-2}\| r_e^{s+2-\gamma_e} \dpar^{\alpha_1}\dperpone\dperptwo^{s+1} v \|^2_{L^2(K_{e})},
\\
   (III)_{\corn}
%  &=
%\sum_{\alpha_1\leq 1}\hpar^{2\alpha_1}\hperptwo^{2s+2}\| \dpar^{\alpha_1}\dperpone\dperptwo^{s+1} v \|^2_{L^2(K_{\corn})} 
%    \\
%  &\leq \sigmaratio^{2s+4}
%\sum_{\alpha_1\leq 1}d_\corn^{2\alpha_1}d_e^{2s+2}\| \dpar^{\alpha_1}\dperpone\dperptwo^{s+1} v \|^2_{L^2(K_{\corn})} 
%    \\
  &\leq \sigmaratio^{2s+4}
\sum_{\alpha_1\leq 1}d_\corn^{2\gamma_\corn-2}\| r_\corn^{s+2+\alpha_1-\gamma_\corn} \dpar^{\alpha_1}\dperpone\dperptwo^{s+1} v \|^2_{L^2(K_{\corn})},
\\
   (III)_{0}
%  &=
%\sum_{\alpha_1\leq 1}\hpar^{2\alpha_1}\hperptwo^{2s+2}\| \dpar^{\alpha_1}\dperpone\dperptwo^{s+1} v \|^2_{L^2(K_{0})} 
%    \\
  &\leq \sigmaratio^{2s+4}
\sum_{\alpha_1\leq 1}d_e^{2s+2}\| \dpar^{\alpha_1}\dperpone\dperptwo^{s+1} v \|^2_{L^2(K_{0})}.
\end{align*}%jo
  Therefore, we have 
  \begin{equation}
    \label{eq:IIbound-2}
    (II)  + (III) 
    \leq C A^{2s+6}(d_\corn^{2\gamma_\corn-2} + d_\corn^{2\gamma_e-2})((s+3)!)^2.
  \end{equation}
  We obtain, from \eqref{eq:Kappx1-2}, \eqref{eq:Ibound-2}, and \eqref{eq:IIbound-2}
  that there exists $C>0$ dependent only on $\sigma$, $\Cappxii$, $C_v$ and $A>0$
  dependent only on $\sigma$, $A_v$ such that
  \begin{equation*}
    \| \widehat{\partial}_{\bot, 1}(\hv - \hPi_p \hv)\|_{L^2(\hK)}^2  \leq C \frac{\hperpone}{\hpar\hperptwo}\Psi_{p,s} A^{2s+6}\left(d_\corn^{2(\gamma_\corn-1)} + d_\corn^{2(\gamma_e-1)}\right)((s+3)!)^2.
  \end{equation*}
  Considering that
  \begin{equation*}
     \|\dperpone (v-\Pi_p v)\|^2_{L^2(K)}\leq \frac{\hpar\hperptwo}{\hperpone} \|\widehat{\partial}_{\bot,1} (\hv - \hPi_p\hv)\|^2_{L^2(\hK)}
  \end{equation*}
  and considering that the estimate for the other term at the left-hand side of
  \eqref{eq:internal-appx-2} is obtained by exchanging
  $\{h, \partial\}_{\bot,1}$ with $\{h, \partial\}_{\bot,2}$ completes the proof.
  \end{proof}
  \begin{lemma}
\label{lemma:internal-appx-3}
Let $d=3$, $\ell\in \mathbb{N}$ and $K = J^\ell_{k_1}\times J^\ell_{k_2}\times J^{\ell}_{k_3}$ for
$0< k_1, k_2, k_3\leq \ell$.
 Let also $v\in
  \cJ^\varpi_\ugamma(Q; \Cset, \Eset; C_v, A_v)$
  with $\gamma_\corn \in (3/2, 5/2)$, $\gamma_e \in (1, 2)$. Then, there exists $C>0$ dependent only on
  $\sigma$, $\Cappxi$, $C_v$ and $A>0$ dependent only on $\sigma$, $A_v$ such that
  for all $p\in\N$ and all $1\leq s \leq p$
  \begin{equation}
\label{eq:internal-appx-3}
        \|v - \Pi^K_p v \|^2_{L^2(K)}
        \leq C \Psi_{p,s} A^{2s+6}\left(d_\corn^{2(\gamma_\corn-1)} + d_\corn^{2(\gamma_e-1)}\right)((s+3)!)^2.
  \end{equation}
  \end{lemma}
\begin{proof}
  The proof follows closely that of Lemmas \ref{lemma:internal-appx-1} and
  \ref{lemma:internal-appx-2}; we use the same notation.
   From Lemma \ref{lemma:ref-proj} and rescaling, we have
  \begin{equation}
    \label{eq:Kappx1-3}
\begin{aligned}
\| \hv - \hPi_p \hv  \|_{L^2(\hK)}^2 
    &\leq \Cappxi \Psi_{p,s} \frac{1}{\hpar\hperpone\hperptwo}\left(\sum_{\alpha_1, \alpha_2 \leq 1}
    \hpar^{2s+2}\hperpone^{2\alpha_1}\hperptwo^{2\alpha_2}\| \dpar^{s+1}\dperpone^{\alpha_1}\dperptwo^{\alpha_2}v \|^2_{L^2(K)} 
    \right.
    \\ &\qquad +
    \sum_{\alpha_1, \alpha_2\leq 1}\hpar^{2\alpha_1}\hperpone^{2s+2}\hperptwo^{2\alpha_2}\| \dpar^{\alpha_1}\dperpone^{s+1}\dperptwo^{\alpha_2}v \|^2_{L^2(K)} 
    \\ &\qquad +
    \left.
    \sum_{\alpha_1, \alpha_2\leq 1}\hpar^{2\alpha_1}\hperpone^{2\alpha_2}\hperptwo^{2s+2}\| \dpar^{\alpha_1}\dperpone^{\alpha_2}\dperptwo^{s+1}v \|^2_{L^2(K)} 
  \right)
  % \\
  % & = \Cappxii \Psi_{p,s} \frac{1}{\hpar\hperpone\hperptwo}\bigg((I) + (II) + (III)  \bigg).
  .
  \end{aligned}
  \end{equation}
  Most terms at the right-hand side above have already been considered in the proofs of Lemmas
  \ref{lemma:internal-appx-1} and \ref{lemma:internal-appx-2}, 
  and the terms with $\alpha_1 = \alpha_2 = 0$ can be estimated similarly; the observation that
  \begin{equation*}
    \| v - \Pi_p v\|^2_{L^2(K)}\leq \hpar\hperpone\hperptwo
\| \hv - \hPi_p \hv  \|_{L^2(\hK)}^2 
  \end{equation*}
  concludes the proof.
\end{proof}
We summarize Lemmas \ref{lemma:internal-appx-1}  to
\ref{lemma:internal-appx-3} in the following result.
\begin{lemma}
  Let $d=3$, $\ell\in \mathbb{N}$ and $K=J^\ell_{k_1}\times J^\ell_{k_2}\times
  J^\ell_{k_3}$ such that $0<k_1,k_2,k_3\leq \ell$. Let also $v\in
  \cJ^\varpi_\ugamma(Q; \Cset, \Eset;  C_v, A_v)$
  with $\gamma_\corn \in (3/2, 5/2)$, $\gamma_e \in (1, 2)$. Then, there exists $C>0$ dependent only on
  $\sigma$, $\Cappxi$, $\Cappxii$, $C_v$ and $A>0$ dependent only on $\sigma$, $A_v$ such that
  for all $p\in\N$ and all $1\leq s \leq p$
  \begin{equation}
\label{eq:internal-appx}
        \|v - \Pi^K_p v\|^2_{H^1(K)} 
\leq C \Psi_{p,s} A^{2s+6}\left(d_\corn^{2(\gamma_\corn-1)} + d_\corn^{2(\gamma_e-1)}\right)((s+3)!)^2.
  \end{equation}
  \end{lemma}
  We then consider elements on the faces (but not abutting edges) of $Q$.
\begin{lemma}
  \label{lemma:internal-appx-face}
Let $d=3$, $\ell\in \mathbb{N}$ and $K = J^\ell_{k_1}\times J^\ell_{k_2}\times
J^{\ell}_{k_3}$ such that $k_j =0$ for one $j\in\{1,2,3\}$ and $0<k_i\leq \ell$
for $i\neq j$. For all $p\in\N$ and all $1\leq s \leq p$, 
let $p_j= 1$ and  $p_i = p\in\mathbb{N}$ for $i\neq j$.
Let also $v\in \cJ^\varpi_\ugamma(Q; \Cset, \Eset;  C_v, A_v)$
  with $\gamma_\corn \in (3/2, 5/2)$, $\gamma_e \in (1, 2)$. Then, there exists $C>0$ dependent only on
  $\sigma$, $\Cappxi$, $\Cappxii$, $C_v$ and $A>0$ dependent only on $\sigma$, $A_v$ such that
  \begin{equation}
\label{eq:internal-appx-face}
        \|v - \Pi^K_{p_1 p_2 p_3} v\|^2_{H^1(K)} 
% \leq C \left(\frac{(p-s)!}{(p+s)!} A^{2s+6}((s+3)!)^2 + A^{8} \right) 
% \left((d_\corn^K)^{2(\gamma_\corn-1)} + (d_\corn^K)^{2(\gamma_e-1)}\right).
%% Old:
C \left(\Psi_{p,s} A^{2s+6}(d^K_\corn)^{2(\min(\gamma_\corn, \gamma_e)-1)}((s+3)!)^2 
+ 
         (d_e^K)^{2(\min(\gamma_\corn, \gamma_e)-2)}\sigma^{2\ell}A^8
 \right).
  \end{equation}
\end{lemma}
\begin{proof}
  We write $d_a = d_a^K$, $a \in \{\corn, e\}$.
 Suppose, for ease of notation, that $j=3$, i.e. $k_3=0$. 
 The projector is then given by 
 $\Pi^K_{p p 1} = \pi^{k_1}_p\otimes \pi^{k_2}_p\otimes \pi^0_1$. 
 Also, we denote 
 $\hperptwo = \sigma^\ell$ and $\dperptwo = \partial_{x_3}$. 
 By \eqref{eq:disc-approx2},
 \begin{align*}
   \| \dpar( v - \Pi^K_{pp1} v  )\|_{L^2(K)}^2 
    &\leq \Cappxii\left( \Psi_{p,s} \bigg(\sum_{\alpha_1, \alpha_2 \leq 1}
    \hpar^{2s}\hperpone^{2\alpha_1}\hperptwo^{2\alpha_2}\| \dpar^{s+1}\dperpone^{ \alpha_1}\dperptwo^{\alpha_2}v \|^2_{L^2(K)} 
    \right.
    \\ &\qquad +
    \sum_{\alpha_1\leq 1}\hperpone^{2s+2}\hperptwo^{2\alpha_1}\| \dpar\dperpone^{s+1}\dperptwo^{\alpha_1}v \|^2_{L^2(K)} \bigg)
    \\ &\qquad +
    \left.
    \sum_{\alpha_1\leq 1}\hperpone^{2\alpha_1}\hperptwo^{4}\| \dpar\dperpone^{\alpha_1}\dperptwo^{2}v \|^2_{L^2(K)} 
  \right)
  \\
  & = \Cappxii \bigg((I) + (II) + (III)  \bigg).
 \end{align*}
 The bounds on the terms $(I)$ and $(II)$ can be derived as in Lemma
 \ref{lemma:internal-appx-1}, and give
 \begin{equation*}
   (I) + (II) \leq C \Psi_{p,s} A^{2s+6}\left( (d_\corn^K)^{2} + (d_\corn^K)^{2(\gamma_\corn-1)}\right)((s+3)!)^2. 
 \end{equation*}
 We consider then term $(III)$: with the usual notation, writing $\tgamma = \gamma_c-\gamma_e$,
 \begin{equation}
   \label{eq:IIIface-a}
   \begin{aligned}
   (III)_{\corn e} &= \sum_{\alpha_1\leq 1}\hperpone^{2\alpha_1}\hperptwo^{4}\| \dpar\dperpone^{\alpha_1}\dperptwo^{2}v \|^2_{L^2(K_{\corn e})} \\
   & \leq\sum_{\alpha_1\leq 1} \sigmaratio^{4+2\alpha_1}
          d_\corn^{2\tgamma-2}d_e^{2\gamma_e-4}\sigma^{4\ell}
       \| r_\corn^{3+\alpha_1-\gamma_\corn}\rho_{\corn e}^{2+\alpha_1-\gamma_e}\dpar\dperpone^{\alpha_1}\dperptwo^{2}v \|^2_{L^2(K_{\corn e})} 
   \\ &\leq C
        \sigmaratio^{6}d_\corn^{2\tgamma-2}d_e^{2\gamma_e-4}\sigma^{4\ell}A^{8}.
   \end{aligned}
 \end{equation}
 Note that $d_\corn\geq d_e$ and
 \begin{equation}
   \label{eq:tgamma-gammae}
   d_\corn^{\tgamma}d_e^{\gamma_e}\leq
  \begin{cases}
     1^{\tgamma}d_e^{\gamma_e} & \text{if }\tgamma\geq 0\\
     d_e^{\tgamma}d_e^{\gamma_e} & \text{if }\tgamma\geq 0
   \end{cases}
\leq d_e^{\min(\gamma_\corn, \gamma_e)},
 \end{equation}
 where we have also used that $d_c\leq 1$.
 Hence,
 \begin{equation}
   \label{eq:IIIface}
   (III)_{\corn e}\leq C
        \sigmaratio^{6}d_e^{2\min(\gamma_e, \gamma_\corn)-6}\sigma^{4\ell}A^{8}
        \leq C
        \sigmaratio^{6}d_e^{2\min(\gamma_e, \gamma_\corn)-4}\sigma^{2\ell}A^{8}.
 \end{equation}
 The bounds on the terms $(III)_{\corn, e, 0}$ follow by the same argument:
 \begin{align*}
   (III)_{e} 
%   &= \sum_{\alpha_1\leq 1}\hperpone^{2\alpha_1}\hperptwo^{4}\| \dpar\dperpone^{\alpha_1}\dperptwo^{2}v \|^2_{L^2(K_{e})} \\
%              & \leq\sum_{\alpha_1\leq 1} \sigmaratio^{4+2\alpha_1}d_e^{2\gamma_e-4}\sigma^{4\ell}\| r_e^{2+\alpha_1-\gamma_e}\dpar\dperpone^{\alpha_1}\dperptwo^{2}v \|^2_{L^2(K_{e})} 
%   \\ 
    &\leq C \sigmaratio^{6}d_e^{2\gamma_e-4}\sigma^{4\ell} A^{8},
\\
   (III)_{\corn} 
%   &= \sum_{\alpha_1\leq 1}\hperpone^{2\alpha_1}\hperptwo^{4}\| \dpar\dperpone^{\alpha_1}\dperptwo^{2}v \|^2_{L^2(K_{\corn})} \\
%              & \leq\sum_{\alpha_1\leq 1} \sigmaratio^{4+2\alpha_1}d_\corn^{2\gamma_\corn-6}\sigma^{4\ell}\| r_\corn^{3+\alpha_1-\gamma_\corn}\dpar\dperpone^{\alpha_1}\dperptwo^{2}v \|^2_{L^2(K_{\corn})} 
%   \\ 
   &\leq C \sigmaratio^{6}d_\corn^{2\gamma_\corn-6}\sigma^{4\ell} A^{8}
      \leq C\sigmaratio^{6} d_e^{2\gamma_\corn -4}\sigma^{2\ell}A^8,
\\
   (III)_{0} 
%   &= \sum_{\alpha_1\leq 1}\hperpone^{2\alpha_1}\hperptwo^{4}\| \dpar\dperpone^{\alpha_1}\dperptwo^{2}v \|^2_{L^2(K_{0})} \\
%  & \leq\sum_{\alpha_1\leq 1} \sigmaratio^{6}\sigma^{4\ell}\| \dpar\dperpone^{\alpha_1}\dperptwo^{2}v \|^2_{L^2(K_{0})}
%   \\ 
   &\leq C \sigmaratio^{6}\sigma^{4\ell} A^{8}.
  \end{align*}
 Then,
 \begin{align*}
   \| \dperpone( v - \Pi^K_{pp1} v  )\|_{L^2(K)}^2 
    &\leq \Cappxii\left(\frac{(p-s)!}{(p+s)!} \bigg(
\sum_{\alpha_1\leq 1} \hpar^{2s+2}\hperptwo^{2\alpha_1}\| \dpar^{s+1}\dperpone\dperptwo^{\alpha_1}v \|^2_{L^2(K)} 
    \right.
    \\ &\qquad +
\sum_{\alpha_1, \alpha_2 \leq 1}
    \hpar^{2\alpha_1}\hperpone^{2s}\hperptwo^{2\alpha_2}\| \dpar^{\alpha_1}\dperpone^{s+1}\dperptwo^{\alpha_2}v \|^2_{L^2(K)} \bigg)
    \\ &\qquad +
    \left.
    \sum_{\alpha_1\leq 1}\hpar^{2\alpha_1}\hperptwo^{4}\| \dpar^{\alpha_1}\dperpone\dperptwo^{2}v \|^2_{L^2(K)} 
  \right)
  \\
  & \leq \Cappxii \Big( (I) + (II) + (III) \Big).
 \end{align*}
The bounds on the first two terms at the right-hand side above can be obtained
as in Lemma \ref{lemma:internal-appx-2}:
%\begin{multline*}
\begin{align*}
%  \sum_{\alpha_1\leq 1} \hpar^{2s+2}\hperptwo^{2\alpha_1}\| \dpar^{s+1}\dperpone\dperptwo^{\alpha_1}v \|^2_{L^2(K)} 
% +
%\sum_{\alpha_1, \alpha_2 \leq 1}
%    \hpar^{2\alpha_1}\hperpone^{2s}\hperptwo^{2\alpha_2}\| \dpar^{\alpha_1}\dperpone^{s+1}\dperptwo^{\alpha_2}v \|^2_{L^2(K)} 
%\\
(I) + (II)
& \leq C \Psi_{p,s} A^{2s+6}\left((d_\corn^K)^{2(\gamma_\corn-1)} + (d_\corn^K)^{2(\gamma_e-1)}\right)((s+3)!)^2 ,
\end{align*}
%\end{multline*}
 while the last term can be bounded as in \eqref{eq:IIIface}, 
\begin{align*}
(III)_{\corn e}
%	& = \sum_{\alpha_1\leq 1}\hpar^{2\alpha_1}\hperptwo^{4}\| \dpar^{\alpha_1}\dperpone\dperptwo^{2}v \|^2_{L^2(K_{\corn e})} 
%	\\
%	& \leq \sum_{\alpha_1\leq 1} \sigmaratio^{6} 
%		d_\corn^{2\alpha_1}\sigma^{4\ell} \| \dpar^{\alpha_1}\dperpone\dperptwo^{2}v \|^2_{L^2(K_{\corn e})} 
%	\\
%	& \leq \sum_{\alpha_1\leq 1} \sigmaratio^{6} 
%		d_\corn^{2\tgamma}d_e^{2\gamma_e-6}\sigma^{4\ell} \| r_\corn^{3+\alpha_1-\gamma_\corn} \rho_{\corn e}^{3-\gamma_e} \dpar^{\alpha_1}\dperpone\dperptwo^{2}v \|^2_{L^2(K_{\corn e})} 
%	\\
	& \leq \sigmaratio^{6} 
		d_\corn^{2\tgamma}d_e^{2\gamma_e-6}\sigma^{4\ell} A^8
      \leq C\sigmaratio^{6} d_e^{2\min(\gamma_\corn, \gamma_e) -4}\sigma^{2\ell}A^8,
\\
(III)_{e}
%	& = \sum_{\alpha_1\leq 1}\hpar^{2\alpha_1}\hperptwo^{4}\| \dpar^{\alpha_1}\dperpone\dperptwo^{2}v \|^2_{L^2(K_{e})} 
%	\\
%	& \leq \sum_{\alpha_1\leq 1} \sigmaratio^{6} 
%		d_\corn^{2\alpha_1}\sigma^{4\ell} \| \dpar^{\alpha_1}\dperpone\dperptwo^{2}v \|^2_{L^2(K_{e})} 
%	\\
%	& \leq \sum_{\alpha_1\leq 1} \sigmaratio^{6} 
%		d_\corn^{2\alpha_1}d_e^{2\gamma_e-6}\sigma^{4\ell} \| r_e^{3-\gamma_e} \dpar^{\alpha_1}\dperpone\dperptwo^{2}v \|^2_{L^2(K_{e})} 
%	\\
	& \leq \sigmaratio^{6} 
		d_e^{2\gamma_e-6}\sigma^{4\ell} A^8
      \leq C\sigmaratio^{6} d_e^{2\gamma_e -4}\sigma^{2\ell}A^8,
\\
(III)_{\corn}
%	& = \sum_{\alpha_1\leq 1}\hpar^{2\alpha_1}\hperptwo^{4}\| \dpar^{\alpha_1}\dperpone\dperptwo^{2}v \|^2_{L^2(K_{\corn})} 
%	\\
%	& \leq \sum_{\alpha_1\leq 1} \sigmaratio^{6} 
%		d_\corn^{2\alpha_1}\sigma^{4\ell} \| \dpar^{\alpha_1}\dperpone\dperptwo^{2}v \|^2_{L^2(K_{\corn})} 
%	\\
%	& \leq \sum_{\alpha_1\leq 1} \sigmaratio^{6} 
%		d_\corn^{2\gamma_\corn - 6}\sigma^{4\ell} \| r_\corn^{3+\alpha_1-\gamma_\corn} \dpar^{\alpha_1}\dperpone\dperptwo^{2}v \|^2_{L^2(K_{\corn})} 
%	\\
	& \leq \sigmaratio^{6} 
		d_\corn^{2\gamma_\corn - 6}\sigma^{4\ell} A^8
      \leq C\sigmaratio^{6} d_e^{2\gamma_\corn -4}\sigma^{2\ell}A^8,
\\
(III)_{0}
%	& = \sum_{\alpha_1\leq 1}\hpar^{2\alpha_1}\hperptwo^{4}\| \dpar^{\alpha_1}\dperpone\dperptwo^{2}v \|^2_{L^2(K_{0})} 
%	\\
%	& \leq \sum_{\alpha_1\leq 1} \sigmaratio^{6} 
%		d_\corn^{2\alpha_1}\sigma^{4\ell} \| \dpar^{\alpha_1}\dperpone\dperptwo^{2}v \|^2_{L^2(K_{0})} 
%	\\
	& \leq \sigmaratio^{6} 
		\sigma^{4\ell} A^8,
\end{align*}
%jo
so that
\begin{equation*}
    \sum_{\alpha_1\leq 1}\hpar^{2\alpha_1}\hperptwo^{4}\| \dpar^{\alpha_1}\dperpone\dperptwo^{2}v \|^2_{L^2(K)} 
    \leq C
  d_e^{2\min(\gamma_\corn, \gamma_e) -4}\sigma^{2\ell}A^{8}.
\end{equation*}
The same holds true for the last term of the gradient of
the approximation error, given by
\begin{align*}
   \| \dperptwo( v - \Pi^K_{pp1} v  )\|_{L^2(K)}^2 
    &\leq \Cappxii\left( \Psi_{p,s} \bigg(
\sum_{\alpha_1\leq 1} \hpar^{2s+2}\hperpone^{2\alpha_1}\| \dpar^{s+1}\dperpone^{\alpha_1}\dperptwo v \|^2_{L^2(K)} 
    \right.
    \\ &\qquad +
\sum_{\alpha_1 \leq 1}
    \hpar^{2\alpha_1}\hperpone^{2s+2}\| \dpar^{\alpha_1}\dperpone^{s+1}\dperptwo v \|^2_{L^2(K)} \bigg)
    \\ &\qquad +
    \left.
    \sum_{\alpha_1, \alpha_2\leq 1}\hpar^{2\alpha_1}\hperpone^{2\alpha_2}\hperptwo^{2}\| \dpar^{\alpha_1}\dperpone^{\alpha_2}\dperptwo^{2}v \|^2_{L^2(K)} 
         \right)
 \\      
 & \leq \Cappxii \Big( (I) + (II) + (III) \Big).
%& \leq C \left(\frac{(p-s)!}{(p+s)!} A^{2s+6}(d^K_\corn)^{2\gamma_\corn-1}((s+3)!)^2 + d_\corn^{2(\gamma_\corn-\gamma_e)}d_e^{2(\gamma_e-2)}\sigma^{2\ell}A^4
%           \right).
 \end{align*}
From Lemma \ref{lemma:internal-appx-2}, we obtain 
\begin{align*}
(I)+(II) & \leq C \Psi_{p,s} A^{2s+6}\left((d_\corn^K)^{2(\gamma_\corn-1)} + (d_\corn^K)^{2(\gamma_e-1)}\right)((s+3)!)^2,
\end{align*}
whereas for the third term, it holds that if $\alpha_1+\alpha_2+2-\gamma_\corn \geq0$
\begin{align*}
(III)_{\corn e}
%  & = \sum_{\alpha_1, \alpha_2\leq 1}\hpar^{2\alpha_1}\hperpone^{2\alpha_2}\hperptwo^{2}\| \dpar^{\alpha_1}\dperpone^{\alpha_2}\dperptwo^{2}v \|^2_{L^2(K_{\corn e})} 
%  \\
%  & \leq \sum_{\alpha_1, \alpha_2\leq 1} \sigmaratio^{6} 
%  	d_\corn^{2\alpha_1}d_e^{2\alpha_2}\sigma^{2\ell}\| \dpar^{\alpha_1}\dperpone^{\alpha_2}\dperptwo^{2}v \|^2_{L^2(K_{\corn e})}
%  \\
%  & \leq  \sum_{\alpha_1, \alpha_2\leq 1} \sigmaratio^{6} 
%    d_\corn^{2\tgamma}d_e^{2\gamma_e-4}\sigma^{2\ell}\| r_\corn^{\alpha_1+\alpha_2+2-\gamma_\corn} \rho_{\corn e}^{\alpha_2+2-\gamma_e} \dpar^{\alpha_1}\dperpone^{\alpha_2}\dperptwo^{2}v \|^2_{L^2(K_{\corn e})}
%    \\
    & \leq \sigmaratio^{6} d_\corn^{2\tgamma}d_e^{2\gamma_e-4}\sigma^{2\ell} A^8
      \leq C\sigmaratio^{6} d_e^{2\min(\gamma_\corn, \gamma_e) -4}\sigma^{2\ell}A^8,
\quad 
(III)_{\corn}
%  & = \sum_{\alpha_1, \alpha_2\leq 1}\hpar^{2\alpha_1}\hperpone^{2\alpha_2}\hperptwo^{2}\| \dpar^{\alpha_1}\dperpone^{\alpha_2}\dperptwo^{2}v \|^2_{L^2(K_{\corn})} 
%  \\
%  & \leq \sum_{\alpha_1, \alpha_2\leq 1} \sigmaratio^{6} 
%  	d_\corn^{2\alpha_1}d_e^{2\alpha_2}\sigma^{2\ell}\| \dpar^{\alpha_1}\dperpone^{\alpha_2}\dperptwo^{2}v \|^2_{L^2(K_{\corn})}
%  \\
%  & \leq  \sum_{\alpha_1, \alpha_2\leq 1} \sigmaratio^{6} 
%    d_\corn^{2\gamma_\corn-4}\sigma^{2\ell}\| r_\corn^{\alpha_1+\alpha_2+2-\gamma_\corn} \dpar^{\alpha_1}\dperpone^{\alpha_2}\dperptwo^{2}v \|^2_{L^2(K_{\corn})}
%    \\
     \leq \sigmaratio^{6} d_\corn^{2\gamma_\corn-4}\sigma^{2\ell} A^8,
\end{align*}
and if $\alpha_1+\alpha_2+2-\gamma_\corn < 0$, then
\begin{align*}
(III)_{\corn e}
%  & = \sum_{\alpha_1, \alpha_2\leq 1}\hpar^{2\alpha_1}\hperpone^{2\alpha_2}\hperptwo^{2}\| \dpar^{\alpha_1}\dperpone^{\alpha_2}\dperptwo^{2}v \|^2_{L^2(K_{\corn e})} 
%  \\
%  & \leq \sum_{\alpha_1, \alpha_2\leq 1} \sigmaratio^{6} 
%  	d_\corn^{2\alpha_1}d_e^{2\alpha_2}\sigma^{2\ell}\| \dpar^{\alpha_1}\dperpone^{\alpha_2}\dperptwo^{2}v \|^2_{L^2(K_{\corn e})}
%  \\
%  & \leq  \sum_{\alpha_1, \alpha_2\leq 1} \sigmaratio^{6} 
%    d_\corn^{2\alpha_1}d_e^{2\gamma_e-4}\sigma^{2\ell}\| r_\corn^{(\alpha_1+\alpha_2+2-\gamma_\corn)_+} \rho_{\corn e}^{\alpha_2+2-\gamma_e} \dpar^{\alpha_1}\dperpone^{\alpha_2}\dperptwo^{2}v \|^2_{L^2(K_{\corn e})}
%    \\
    & \leq \sigmaratio^{6} d_e^{2\gamma_e-4}\sigma^{2\ell} A^8,
 \quad  (III)_{\corn}
%  & = \sum_{\alpha_1, \alpha_2\leq 1}\hpar^{2\alpha_1}\hperpone^{2\alpha_2}\hperptwo^{2}\| \dpar^{\alpha_1}\dperpone^{\alpha_2}\dperptwo^{2}v \|^2_{L^2(K_{\corn})} 
%  \\
%  & \leq \sum_{\alpha_1, \alpha_2\leq 1} \sigmaratio^{6} 
%  	d_\corn^{2\alpha_1}d_e^{2\alpha_2}\sigma^{2\ell}\| \dpar^{\alpha_1}\dperpone^{\alpha_2}\dperptwo^{2}v \|^2_{L^2(K_{\corn})}
%  \\
%  & \leq  \sum_{\alpha_1, \alpha_2\leq 1} \sigmaratio^{6} 
%    d_\corn^{2\alpha_1}d_e^{2\alpha_2}\sigma^{2\ell}\| r_\corn^{(\alpha_1+\alpha_2+2-\gamma_\corn)_+} \dpar^{\alpha_1}\dperpone^{\alpha_2}\dperptwo^{2}v \|^2_{L^2(K_{\corn})}
%    \\
    \leq \sigmaratio^{6} \sigma^{2\ell} A^8,
\end{align*}
and for all $\alpha_1+\alpha_2+2-\gamma_\corn\in\R$, 
$(III)_{e}$ and $(III)_{0}$ satisfy the bounds that $(III)_{\corn e}$ and $(III)_{\corn}$ satisfy
in case $\alpha_1+\alpha_2+2-\gamma_\corn < 0$,
%\begin{align*}
%(III)_{e}
%  & = \sum_{\alpha_1, \alpha_2\leq 1}\hpar^{2\alpha_1}\hperpone^{2\alpha_2}\hperptwo^{2}\| \dpar^{\alpha_1}\dperpone^{\alpha_2}\dperptwo^{2}v \|^2_{L^2(K_{e})} 
%  \\
%  & \leq \sum_{\alpha_1, \alpha_2\leq 1} \sigmaratio^{6} 
%  	d_\corn^{2\alpha_1}d_e^{2\alpha_2}\sigma^{2\ell}\| \dpar^{\alpha_1}\dperpone^{\alpha_2}\dperptwo^{2}v \|^2_{L^2(K_{e})}
%  \\
%  & \leq  \sum_{\alpha_1, \alpha_2\leq 1} \sigmaratio^{6} 
%    d_\corn^{2\alpha_1}d_e^{2\gamma_e-4}\sigma^{2\ell}\| r_e^{\alpha_2+2-\gamma_e} \dpar^{\alpha_1}\dperpone^{\alpha_2}\dperptwo^{2}v \|^2_{L^2(K_{e})}
%    \\
%    & \leq \sigmaratio^{6} d_\corn^{2\gamma_e-2}d_e^{-2}\sigma^{2\ell} A^8,
%\\
%  (III)_{0}
%  & = \sum_{\alpha_1, \alpha_2\leq 1}\hpar^{2\alpha_1}\hperpone^{2\alpha_2}\hperptwo^{2}\| \dpar^{\alpha_1}\dperpone^{\alpha_2}\dperptwo^{2}v \|^2_{L^2(K_{0})} 
%  \\
%  & \leq \sum_{\alpha_1, \alpha_2\leq 1} \sigmaratio^{6} 
%  	d_\corn^{2\alpha_1}d_e^{2\alpha_2}\sigma^{2\ell}\| \dpar^{\alpha_1}\dperpone^{\alpha_2}\dperptwo^{2}v \|^2_{L^2(K_{0})}
%    \\
%    & \leq \sigmaratio^{6} \sigma^{2\ell} A^8,
%\end{align*}
so that
\begin{align*}
   \| \dperptwo( v - \Pi^K_{pp1} v)\|_{L^2(K)}^2 
	&\leq C
\left( \Psi_{p,s} A^{2s+6}((s+3)!)^2d_\corn^{2(\min(\gamma_\corn, \gamma_e)-1)} 
+
A^{8} d_e^{2(\min(\gamma_\corn, \gamma_e)-2)}\sigma^{2\ell}\right) \;.
% \left((d_\corn^K)^{2(\gamma_\corn-1)} + (d_\corn^K)^{2(\gamma_e-1)}\right).
 \end{align*} 
%jo
Finally, the bound on the $L^2(K)$ norm of the approximation error can be
obtained by a combination of the estimates above.
\end{proof}
  The exponential convergence of the approximation in internal elements (i.e., elements
  not abutting a singular edge or corner) follows, from Lemmas
  \ref{lemma:internal-appx-1} to \ref{lemma:internal-appx-face}.
  \begin{lemma}
    \label{lemma:exp-int}
    Let $d=3$ and $v\in \cJ^\varpi_\ugamma(Q; \Cset, \Eset)$ with $\gamma_\corn>3/2$,
    $\gamma_e>1$. There exists a constant $C_0>0$ such that if $p\geq C_0 \ell$, 
    there exist constants $C, b>0$ such that for every $\ell \in \mathbb{N}$ holds
  \begin{equation*}
        \sum_{K: d_e^K>0} \|v - \Pihpelldim v\|^2_{H^1(K)}
        %\|v - \Pi^K_p v\|^2_{L^2(K)} + \|\nabla \left( v - \Pi^K_p v \right)\|^2_{L^2(K)}
        \leq C e^{-b\ell}.
  \end{equation*}
  \end{lemma}
  \begin{proof}
    We suppose, without loss of generality, that $\gamma_\corn\in (3/2, 5/2)$, and
    $\gamma_e\in (1,2)$. The general case follows from the inclusion
    $\cJ^\varpi_{\ugamma_1}(Q;\Cset, \Eset) \subset \cJ^\varpi_{\ugamma_2}(Q;\Cset, \Eset)$, 
    valid for $\gamma_1 \geq \gamma_2$.
   Fix any $C_0>0$ and choose $p\geq C_0 \ell$.  
   For all $A>0$ there exist $C_1, b_1> 0$ such that 
   (see, e.g., \cite[Lemma 5.9]{SSWII})
   \begin{equation*}
      \forall p\in\N: \quad \min_{1\leq s\leq p} \Psi_{p,s} A^{2s} (s!)^2 \leq C_1 e^{-b_1 p}.
   \end{equation*}
   From \eqref{eq:internal-appx} and \eqref{eq:internal-appx-face}, there holds
   \begin{align*}
     &\sum_{K:d_e^K>0} \|v - \Pihpelldim v\|^2_{H^1(K)}
     \\ & \qquad
          \leq C_2 \left( \sum_{K:d_e^K>0}e^{-b_1\ell}(d_\corn^K)^{2(\min(\gamma_\corn, \gamma_e)-1)}  + \sum_{K:d_e^K>0, d_f^K=0}(d_e^K)^{2(\min(\gamma_e, \gamma_\corn)-2)}\sigma^{2\ell}\right)
     \\ & \qquad
          = C_2\big((I)+ (II)\big),
   \end{align*}
   where $d_f^K$  indicates the distance of an element $K$ to one of the faces
   of $Q$. There holds directly $(I)\leq C\ell^2 e^{-b_1\ell}$. 
%%%   [Why $\ell^2$?]
   Furthermore, because $(\min(\gamma_\corn, \gamma_e)-2)<0$,
   \begin{align*}
     (II) &\leq 6 \sigma^{2\ell}\sum_{k_1=1}^\ell\sum_{k_2=1}^{k_1} \sigma^{2(\ell-k_2)(\min(\gamma_e, \gamma_\corn)-2)}
    % \leq C \sigma^{2\ell}\sum_{k_1=1}^\ell\sigma^{2(\ell-k_1)(\gamma_\corn-\gamma_e)+2(\ell-k_1)(\gamma_e-2)}\\
   % & \leq C \sum_{k_1=1}^\ell\sigma^{2(\ell-k_1)(\gamma_\corn-1)+2k_1}\\
     \\
   & \leq C \sigma^{2\ell}\sum_{k_1=1}^\ell\sigma^{2\ell(\min(\gamma_\corn, \gamma_e)-2)}\\
  & \leq C \ell \sigma^{2(\min(\gamma_\corn, \gamma_e)-1)\ell}.
   \end{align*}
   Adjusting the constants at the exponent to absorb the terms in $\ell$ and $\ell^2$, we obtain
   the desired estimate.
  \end{proof}
A similar statement holds when $d=2$, and the proof follows along the same lines.
  \begin{lemma}
    \label{lemma:exp-int-2d}
    Let $d=2$ and $v\in \cJ^\varpi_\ugamma(Q;\Cset, \Eset)$ with $\gamma_\corn>1$. 
    There exists a constant $C_0>0$ such that if $p\geq C_0 \ell$, 
    there exist constants $C, b>0$ such that 
  \begin{equation*}
        \sum_{K: d_\corn^K>0} \|v - \Pihpelldim v\|^2_{H^1(K)}
        \leq C e^{-b\ell}, \qquad \forall \ell \in \mathbb{N}.
  \end{equation*}
  \end{lemma}
  %
  %%%%%%%%%%%%%%%%%%%%%%%%%%%%%%%%%%%%%%%%%%%%%%%%%%%%%%%%%%%%%%
  \subsection{Estimates on elements along an edge in three dimensions}
  \label{sec:edge-estimates}
%%%%%%%%%%%%%%%%%%%%%%%%%%%%%%%%%%%%%%%%%%%%%%%%%%%%%%%%%%%%%%
   In the following lemma, we consider the elements $K$ along one edge, but separated
   from the singular corner.
  \begin{lemma}
    \label{lemma:edge-elem}
   Let $d=3$, $e\in\Eset$ and let $K\in \cG^\ell_3$  be such that $d_\corn^K >0$ for
   all $\corn\in \Cset$ and $d_e^K=0$. Let $C_v, A_v>0$. Then, if $v\in
   \cJ^\varpi_\ugamma(Q;\Cset, \Eset; C_v, A_v)$ with 
%   $\gamma_\corn>3/2$, $\gamma_e>1$
   $\gamma_\corn\in(3/2,5/2)$, $\gamma_e\in(1,2)$, there exist $C, A>0$ such that
   for all $p\in\N$ and all $1\leq s\leq p$, with $(p_1,p_2,p_3)\in\N^3$ such that
   $p_\parallel = p$, $p_{\perp,1} = 1 = p_{\perp,2}$,
%   with $\{i,j,k\} = \{1,2,3\}$ such that $x_j = x_k = 0$ on $e$, $p_i=p$ and $p_j = 1 = p_k$,
   \begin{equation}
     \label{eq:edge-elem}
   \| v - \Pi^K_{p_1p_2p_3}v\|_{H^1(K)}^2 % + \| \nabla(v - \Pi^K_pv)\|_{L^2(K)}^2 
   \leq C\left(\sigma^{2\{\min\{\gamma_\corn-1,s\}(\ell-k)} \Psi_{p,s} A^{2s}((s+3)!)^2 + \sigma^{ 2(\min(\gamma_e, \gamma_\corn)-1)\ell}\right),
   \end{equation}
   where $k\in \{1, \dots, \ell\}$ is such that $d_\corn^K = \sigma^{\ell-k+1}$.
  \end{lemma}
  \begin{proof}
   We suppose that $K=J^\ell_k\times J^\ell_0\times J^\ell_0$ for some $k\in \{1, \dots, \ell\}$, 
   the elements along other edges follow by symmetry. This implies that the singular edge is parallel to
    the first coordinate direction. Furthermore, we denote
    \begin{equation*}
      \Pi^K_{p11} = \pi^k_p\otimes (\pi^0_1 \otimes \pi^0_1) = \pipar\otimes\piperp.
    \end{equation*}
    For $\alpha = (\alpha_1, \alpha_2, \alpha_3) \in \mathbb{N}_0^3$, we write
  $\alphapar = (\alpha_1, 0, 0)$ and $\alphaperp = (0, \alpha_2, \alpha_3)$.
  Also,
  \begin{equation*}
    h_\parallel = |J^\ell_k| = \sigma^{\ell-k}(1-\sigma)\qquad h_\bot = \sigma^\ell.
  \end{equation*}
  We have
  \begin{equation}
    \label{eq:edge-err-decomp}
    v - \Pi^K_{p11}v = v - \piperp v + \piperp \left( v - \pipar v \right).
  \end{equation}
  We start by considering the first terms at the right-hand side of the above
  equation. We also compute the norms over $K_{\corn e} = K\cap Q_{\corn e}$; the estimate
  on the norms over $K_\corn = K\cap Q_\corn$ and $K_e = K\cap Q_e$ follow by similar or
  simpler arguments.
  By \eqref{eq:2d-square-approx} from Lemma \ref{lemma:2d-square}, we have that 
  if $\gamma_\corn < 2$
  \begin{subequations}\label{eqs:edge-err-perp-1}
  \begin{equation}
%   \label{eq:edge-err-perp-1}
   \begin{aligned}
   % \sum_{\alphaperpm\leq 1}(h^\bot)^{-2\alphaperpm}\|\dalphaperp( v - \piperp v)\|_{L^2(K)}^2
   \sum_{\alphaperpm\leq 1}\hperp^{-2(1-\alphaperpm)}\|\dalphaperp( v - \piperp v)\|_{L^2(K_{\corn e})}^2
    & \lesssim \hperp^{2(\gamma_e-1)}\sum_{\alphaperpm = 2} \| r_e^{2 -\gamma_e} \dalphaperp v\|_{L^2(K_{\corn e})}^2
   \\ 
   & \lesssim \hpar^{2(\gamma_\corn-\gamma_e)}\hperp^{2(\gamma_e-1)}\sum_{\alphaperpm = 2} \| r_\corn^{(2-\gamma_\corn)_+}\rho_{\corn e}^{2 -\gamma_e} \dalphaperp v\|_{L^2(K_{\corn e})}^2
    \\ 
    & \lesssim \sigma^{2k(\gamma_e-1)}\sigma^{2(\ell-k)(\gamma_\corn-1)} A^{4}
    \lesssim \sigma^{2\ell(\min\{\gamma_\corn,\gamma_e\}-1)} A^{4},
   \end{aligned}
  \end{equation}
  whereas for $\gamma_\corn\geq  2$
  \begin{align}
  \nonumber
  \sum_{\alphaperpm\leq 1}\hperp^{-2(1-\alphaperpm)}\|\dalphaperp( v - \piperp v)\|_{L^2(K_{\corn e})}^2
    & \lesssim \hperp^{2(\gamma_e-1)}\sum_{\alphaperpm = 2} \| r_e^{2 -\gamma_e} \dalphaperp v\|_{L^2(K_{\corn e})}^2
   \\ 
    & \lesssim \sigma^{2\ell(\gamma_e-1)} A^{4}.
  \end{align}
  \end{subequations}
  On $K_e$, the same bound holds as on $K_{\corn e}$ for $\gamma_\corn\geq2$, 
  and on $K_\corn$ the same bounds hold as on $K_{\corn e}$ for $\gamma_\corn<2$.
%  \begin{align*}
%  \sum_{\alphaperpm\leq 1}\hperp^{-2(1-\alphaperpm)}\|\dalphaperp( v - \piperp v)\|_{L^2(K_{\corn})}^2
%    & \lesssim \hperp^{2(\gamma_e-1)}\sum_{\alphaperpm = 2} \| r_e^{2 -\gamma_e} \dalphaperp v\|_{L^2(K_{\corn})}^2
%   \\ 
%   & \lesssim \hpar^{2(\gamma_\corn-\gamma_e)}\hperp^{2(\gamma_e-1)}\sum_{\alphaperpm = 2} \| r_\corn^{(2-\gamma_\corn)_+}r_e^{2 -\gamma_e} \dalphaperp v\|_{L^2(K_{\corn})}^2
%    \\ 
%    & \lesssim \sigma^{2k(\gamma_e-1)}\sigma^{2(\ell-k)(\gamma_\corn-1)} A^{4},
%  \end{align*}
  By the same argument, for
  $\alphaparm=1$,
\begin{subequations}\label{eqs:edge-err-perp-2} 
\begin{equation}
%   \label{eq:edge-err-perp-2} 
   \begin{aligned}
   \|\dalphapar ( v - \piperp v)\|_{L^2(K_{\corn e})}^2
    & = 
   \| (\dalphapar v )- \piperp (\dalphapar v)\|_{L^2(K_{\corn e})}^2
   \\ &
   \lesssim \hperp^{2\gamma_e}\sum_{\alphaperpm = 2} \| r_e^{2 -\gamma_e} \dalphaperp \dalphapar v\|_{L^2(K_{\corn e})}^2
   \\ &
   \lesssim \hpar^{2\tgamma-2}\hperp^{2\gamma_e}\sum_{\alphaperpm = 2} \| r_\corn^{3-\gamma_\corn}\rho_{\corn e}^{2 -\gamma_e} \dalpha v\|_{L^2(K_{\corn e})}^2
    \\ & \lesssim \sigma^{2(\ell-k)(\gamma_\corn-1)}\sigma^{2k(\gamma_e-1)} A^{6}
    \lesssim \sigma^{2\ell(\min\{\gamma_\corn,\gamma_e\}-1)} A^{6},
   \end{aligned}
 \end{equation}
and 
\begin{align}
%   \|\dalphapar ( v - \piperp v)\|_{L^2(K_{e})}^2
%    & = 
   \| (\dalphapar v )- \piperp (\dalphapar v)\|_{L^2(K_{e})}^2
%   \\ 
%   & \lesssim \hperp^{2\gamma_e}\sum_{\alphaperpm = 2} \| r_e^{2 -\gamma_e} \dalphaperp \dalphapar v\|_{L^2(K_{e})}^2
%   \\ 
   & \lesssim \sigma^{2\ell\gamma_e} A^{6},
\\
%   \|\dalphapar ( v - \piperp v)\|_{L^2(K_{\corn})}^2
%    & = 
   \| (\dalphapar v )- \piperp (\dalphapar v)\|_{L^2(K_{\corn})}^2
%   \\ 
%   & \lesssim \hperp^{2\gamma_e}\sum_{\alphaperpm = 2} \| r_e^{2 -\gamma_e} \dalphaperp \dalphapar v\|_{L^2(K_{\corn})}^2
%   \\ &
%   \lesssim \hpar^{2\tgamma-2}\hperp^{2\gamma_e}\sum_{\alphaperpm = 2} \| r_\corn^{3-\gamma_\corn}\rho_{\corn e}^{2 -\gamma_e} \dalpha v\|_{L^2(K_{\corn})}^2
%    \\ 
    & \lesssim \sigma^{2(\ell-k)(\gamma_\corn-1)}\sigma^{2k(\gamma_e-1)} A^{6}
    \lesssim \sigma^{2\ell(\min\{\gamma_\corn,\gamma_e\}-1)} A^{6}.
\end{align}
\end{subequations}
 We now turn to the second part of the right-hand side of
 \eqref{eq:edge-err-decomp}. We use \eqref{eq:2d-square-stab} from Lemma
 \ref{lemma:2d-square}
 % and choose in \eqref{eq:2d-square-stab}
 % \begin{equation*}
 %   \eta = \min\left( \frac{\gamma_\corn+1}{2}, \gamma_e \right),
 % \end{equation*}
 so that 
 %for $\alphaparm = s\in \mathbb{N}$,
 \begin{equation}
   \label{eq:prelemmaoned}
   \begin{aligned}
    % \sum_{\alphaperpm\leq 1}(h^\bot)^{-2\alphaperpm}\|\dalphaperp\piperp ( v - \pipar v)\|_{L^2(K)}^2
    &\sum_{\alphaperpm \leq 1}\|\dalphaperp\piperp ( v - \pipar v)\|_{L^2(K)}^2\\
    &\qquad \lesssim
    \sum_{\alphaperpm \leq 1}\| \dalphaperp (v- \pipar v )\|_{L^2(K)}^2 +
    \sum_{\alphaperpm = 2} \hperp^{2(\gamma_e-1)}\| r_e^{2- \gamma_e} \dalphaperp (v- \pipar v)\|_{L^2(K)}^2.
  \end{aligned}
\end{equation}
By Lemma \ref{lemma:oned} we have, %for all
  recalling that $\alphapar=s+1$ and $1\leq s \leq p$, for all $\alphaperpm\leq1$, 
  \begin{align*}
    \| \dalphaperp (v- \pipar v )\|_{L^2(K)}^2
    &=
\| (\dalphaperp v)- \pipar(\dalphaperp v )\|_{L^2(K)}^2
    \\
    &\lesssim\sigmaratio^{2s+2}
    \hpar^{2\min\{\gamma_\corn,s+1\}}\Psi_{p, s}\||x_1|^{(s+1-\gamma_\corn)_+}\dalphapar\dalphaperp v)\|^2_{L^2(K)},
  \end{align*}
  and, for all $\alphaperpm =2$,
  using that $\pipar$ and multiplication by $r_e$ commute, because $r_e$ does not depend on $x_1$,
\begin{align*}
    \| r_e^{2-\gamma_e}\dalphaperp (v- \pipar v )\|_{L^2(K)}^2
    &=
\| (r_e^{2-\gamma_e}\dalphaperp v)- \pipar(r_e^{2-\gamma_e}\dalphaperp v )\|_{L^2(K)}^2\\
    &\lesssim\sigmaratio^{2s+2}
    \hpar^{2\min\{\gamma_\corn,s+1\}}\Psi_{p, s}\||x_1|^{(s+1-\gamma_\corn)_+}r_e^{2-\gamma_e}\dalphapar\dalphaperp v)\|^2_{L^2(K)}.
\end{align*}
Then, remarking that $|x_1| \lesssim r_\corn \lesssim |x_1|$, combining \eqref{eq:prelemmaoned} with
the two inequalities above we obtain
\begin{equation*}
   \begin{aligned}
&\sum_{\alphaperpm \leq  1}\|\dalphaperp\piperp ( v - \pipar v)\|_{L^2(K)}^2
    \\ & \qquad\begin{multlined}[][.95\linewidth]
      \lesssim \sigmaratio^{2s+2}\Psi_{p, s}\hpar^{2\min\{\gamma_\corn-1,s\}}\hpar^2
      \left(\sum_{\alphaperpm \leq 1}\| r_\corn^{(s+1-\gamma_\corn)_+}\dalpha v\|_{L^2(K)}^2
      \right.
      \\
      \left.
        +\sum_{\alphaperpm =2}\hperp^{2(\gamma_e-1)}\| r_\corn^{(s+1-\gamma_\corn)_+} r_e^{2 - \gamma_e} \dalpha v\|_{L^2(K)}^2\right).
    \end{multlined}
  \end{aligned}
\end{equation*}
Adjusting the exponent of the weights, replacing $h_\parallel$ and $h_\bot$ with their
definition, 
we find that there exists $A>0$ depending only on $\sigma$ and $A_v$ such that
%and using \eqref{eq:analytic}, 
\begin{subequations} \label{eqs:edge-err-par-1}
\begin{equation}
%   \label{eq:edge-err-par-1}
   \begin{aligned}
&\sum_{\alphaperpm \leq  1}\|\dalphaperp\piperp ( v - \pipar v)\|_{L^2(K_{\corn e})}^2
    \\ & \qquad
    \begin{multlined}[][.95\linewidth]
    \lesssim\sigmaratio^{2s+2} \Psi_{p, s}\hpar^{2\min\{\gamma_\corn-1,s\}}\hpar^2\left(\sum_{\alphaperpm \leq  1}\hpar^{-2\alphaperpm}\| r_\corn^{(s+1+\alphaperpm-\gamma_\corn)_+}\dalpha v\|_{L^2(K_{\corn e})}^2
     \right.
      \\
      \left.
 +\sum_{\alphaperpm =2}\hperp^{2(\gamma_e-1)}\hpar^{-2\gamma_e}\| r_\corn^{s+3-\gamma_\corn} \rho_{\corn e}^{2 - \gamma_e} \dalpha v\|_{L^2(K_{\corn e})}^2\right)
    \end{multlined}
    \\ & \qquad\lesssim \sigma^{2(\ell-k)\min\{\gamma_\corn-1,s\}}\Psi_{p, s} A^{2s+4}
    %\sigmaratio^{2s} 
    ((s+3)!)^2,
   \end{aligned}
 \end{equation}
and similarly
\begin{align}
& \sum_{\alphaperpm \leq  1}\|\dalphaperp\piperp ( v - \pipar v)\|_{L^2(K_{e})}^2
%    \\ & \qquad
%    \begin{multlined}[][.95\linewidth]
%    \lesssim\sigmaratio^{2s+2} \Psi_{p, s}\hpar^{2\min\{\gamma_\corn,s+1\}}\left(\sum_{\alphaperpm \leq  1}\| r_\corn^{s+1-\gamma_\corn}\dalpha v\|_{L^2(K_{e})}^2
%     \right.
%      \\
%      \left.
% +\sum_{\alphaperpm =2}\hperp^{2(\gamma_e-1)}\| r_\corn^{s+1-\gamma_\corn} r_e^{2 - \gamma_e} \dalpha v\|_{L^2(K_{e})}^2\right)
%    \end{multlined}
%    \\ & \qquad
    \lesssim \sigma^{2(\ell-k)\min\{\gamma_\corn,s+1\}}\Psi_{p, s} A^{2s+4}
    %\sigmaratio^{2s} 
    ((s+3)!)^2,
%\\&\\
%&\sum_{\alphaperpm \leq  1}\|\dalphaperp\piperp ( v - \pipar v)\|_{L^2(K_{\corn})}^2
%    \\ & \qquad
%    \begin{multlined}[][.95\linewidth]
%    \lesssim\sigmaratio^{2s+2} \Psi_{p, s}\hpar^{2(\gamma_\corn-1)}\hpar^2\left(\sum_{\alphaperpm \leq  1}\hpar^{-2\alphaperpm}\| r_\corn^{s+1+\alphaperpm-\gamma_\corn}\dalpha v\|_{L^2(K_{\corn})}^2
%     \right.
%      \\
%      \left.
% +\sum_{\alphaperpm =2}\hperp^{2(\gamma_e-1)}\hpar^{-2\gamma_e}\| r_\corn^{s+3-\gamma_\corn} \rho_{\corn e}^{2 - \gamma_e} \dalpha v\|_{L^2(K_{\corn})}^2\right)
%    \end{multlined}
%    \\ & \qquad\lesssim \sigma^{2(\ell-k)(\gamma_\corn-1)}\Psi_{p, s} A^{2s+4}
%    %\sigmaratio^{2s} 
%    ((s+3)!)^2.
\end{align}
and the estimate on $K_\corn$ is the same as that on $K_{\corn e}$.
\end{subequations}
Similarly to \eqref{eq:prelemmaoned}, %for {$1\leq s\leq p$, % $s\in \mathbb{N}$
using first \eqref{eq:2d-square-stab-L2} from the proof of Lemma \ref{lemma:2d-square},
and then Lemma \ref{lemma:oned}
\begin{align*}
     &\sum_{\alphaparm \leq 1}\|\dalphapar\piperp ( v - \pipar v)\|_{L^2(K)}^2
     \\
     &\qquad \lesssim
\sum_{\alphaparm\leq 1}\left(
    \sum_{\alphaperpm \leq 1}h_\bot^{2\alphaperpm} \| \dalphaperp \dalphapar (v- \pipar v )\|_{L^2(K)}^2 +
     \sum_{\alphaperpm = 2} \hperp^{2\gamma_e}\| r_e^{2 - \gamma_e} \dalphaperp \dalphapar (v- \pipar v)\|_{L^2(K)}^2
  \right)
     \\
     & \qquad\begin{multlined}[][.95\linewidth]
       \lesssim \sigmaratio^{2s+2}\Psi_{p, s}\hpar^{2\min\{\gamma_\corn-1,s\}}
       \left(\sum_{\alphaparm=s+1}\sum_{\alphaperpm \leq 1}h_\bot^{2\alphaperpm}\| r_\corn^{(s+1-\gamma_\corn)_+}\dalphapar \dalphaperp v\|_{L^2(K)}^2
       \right.
       \\
       \left.
         +\sum_{\alphaparm=s+1}\sum_{\alphaperpm =2}h_\bot^{2\gamma_e}\| r_e^{2 - \gamma_e} r_\corn^{(s+1-\gamma_\corn)_+} \dalphapar \dalphaperp v\|_{L^2(K)}^2\right).
     \end{multlined}
\end{align*}
%By Lemma \ref{lemma:oned}, 
%Adjusting the exponent of the weights, and using the bounds from \eqref{eq:analytic},
As before, there exists $A>0$ depending only on $\sigma$ and $A_v$ such that
\begin{subequations}\label{eqs:edge-err-par-2}
\begin{equation}
%   \label{eq:edge-err-par-2}
   \begin{aligned}
     &\sum_{\alphaparm \leq 1}\|\dalphapar\piperp ( v - \pipar v)\|_{L^2(K_{\corn e})}^2
%%% Moved it up, as it holds on all of K
%%     \\ & \qquad\begin{multlined}[][.95\linewidth]
%%       \lesssim \sigmaratio^{2s}\Psi_{p-1, s-1}(h_\parallel)^{2(\gamma_\corn-1)}
%%       \left(\sum_{\alphaparm=s}\sum_{\alphaperpm \leq 1}h_\bot^{2\alphaperpm}\| r_\corn^{s-\gamma_\corn}\dalphaperp \dalphapar v\|_{L^2(K_{\corn e})}^2
%%       \right.
%%       \\
%%       \left.
%%         +\sum_{\alphaparm=s}\sum_{\alphaperpm =2}h_\bot^{2\gamma_e}\| r_e^{2 - \gamma_e} r_\corn^{s-\gamma_\corn}\dalphaperp \dalphapar v\|_{L^2(K_{\corn e})}^2\right)
%%     \end{multlined}
     \\ & \qquad\begin{multlined}[][.95\linewidth]
       \lesssim \sigmaratio^{2s+2}\Psi_{p, s}\hpar^{2\min\{\gamma_\corn-1,s\}}
       \left(\sum_{\alphaparm=s+1}\sum_{\alphaperpm \leq 1}\hperp^{2\alphaperpm }\hpar^{-2\alphaperpm}\| r_\corn^{(s+1+\alphaperpm-\gamma_\corn)_+}\dalpha v\|_{L^2(K_{\corn e})}^2
       \right.
       \\
       \left.
         +\sum_{\alphaparm=s+1}\sum_{\alphaperpm =2}\hperp^{2\gamma_e}\hpar^{-2\gamma_e}\| r_\corn^{s+3-\gamma_\corn} \rho_{\corn e}^{2 - \gamma_e} \dalpha v\|_{L^2(K_{\corn e})}^2\right)
     \end{multlined}
     % \\
     % & \lesssim (h_\bot)^{2} \sum_{\alphaperpm \leq 2}\| r_e^{(\alphaperpm - \gamma_e)_+} \dalphaperp \dalphapar (v- \pipar v)\|_{L^2(K)}^2
    \\ &\qquad \lesssim \sigma^{2(\ell-k)\min\{\gamma_\corn-1,s\}}\Psi_{p, s} A^{2s+4} 
    %\sigmaratio^{2s+2}
    ((s+3)!)^2,
  \end{aligned}
\end{equation}
and 
\begin{equation}
\begin{aligned}
     &\sum_{\alphaparm \leq 1}\|\dalphapar\piperp ( v - \pipar v)\|_{L^2(K_{e})}^2
     \\ & \qquad\begin{multlined}[][.95\linewidth]
       \lesssim \sigmaratio^{2s+2}\Psi_{p, s}\hpar^{2\min\{\gamma_\corn-1,s\}}
       \left(\sum_{\alphaparm=s+1}\sum_{\alphaperpm \leq 1}\hperp^{2\alphaperpm }\| r_\corn^{(s+1-\gamma_\corn)_+}\dalpha v\|_{L^2(K_{e})}^2
       \right.
       \\
       \left.
         +\sum_{\alphaparm=s+1}\sum_{\alphaperpm =2}\hperp^{2\gamma_e}\| r_\corn^{(s+1-\gamma_\corn)_+} r_e^{2 - \gamma_e} \dalpha v\|_{L^2(K_{e})}^2\right)
     \end{multlined}
    \\ &\qquad \lesssim \sigma^{2(\ell-k)\min\{\gamma_\corn-1,s\}}\Psi_{p, s} A^{2s+4} 
    %\sigmaratio^{2s+2}
    ((s+3)!)^2,
%\\&\\
%     &\sum_{\alphaparm \leq 1}\|\dalphapar\piperp ( v - \pipar v)\|_{L^2(K_{\corn})}^2
%     \\ & \qquad\begin{multlined}[][.95\linewidth]
%       \lesssim \sigmaratio^{2s+2}\Psi_{p, s}\hpar^{2\min\{\gamma_\corn-1,s\}}
%       \left(\sum_{\alphaparm=s+1}\sum_{\alphaperpm \leq 1}\hperp^{2\alphaperpm }\hpar^{-2\alphaperpm}\| r_\corn^{(s+1+\alphaperpm-\gamma_\corn)_+}\dalpha v\|_{L^2(K_{\corn})}^2
%       \right.
%       \\
%       \left.
%         +\sum_{\alphaparm=s+1}\sum_{\alphaperpm =2}h_\bot^{2\gamma_e}h_{\parallel}^{-2\gamma_e}\| r_\corn^{s+3-\gamma_\corn} \rho_{\corn e}^{2 - \gamma_e} \dalpha v\|_{L^2(K_{\corn})}^2\right)
%     \end{multlined}
%    \\ &\qquad \lesssim \sigma^{2(\ell-k)\min\{\gamma_\corn-1,s\}}\Psi_{p, s} A^{2s+4} 
%    %\sigmaratio^{2s+2}
%    ((s+3)!)^2,
\end{aligned}
\end{equation}
and the estimate on $K_\corn$ is the same as that on $K_{\corn e}$.
\end{subequations}
The assertion now follows from \eqref{eqs:edge-err-perp-1},
\eqref{eqs:edge-err-perp-2}, \eqref{eqs:edge-err-par-1}, and
\eqref{eqs:edge-err-par-2}, upon possibly adjusting the value of the constant $A$.
\end{proof}
\begin{lemma}
  \label{lemma:exp-edge}
  Let $d=3$ and $v\in\cJ^\varpi_\ugamma(Q;\Cset, \Eset)$ with $\gamma_\corn>3/2$, $\gamma_e>1$. There exists a constant $C_0>0$ such that if $p\geq C_0 \ell$, there exist
   constants $C, b>0$ such that
\begin{equation*}
   \sum_{K:d_\corn^K>0,\atop d_e^K=0}\| v - \Pihpelldim v\|_{H^1(K)} %+ \| \nabla(v - \Pi^K_pv)\|_{L^2(K)} 
   \leq Ce^{-b\ell}, \qquad \forall \ell\in \mathbb{N}.
 \end{equation*}
\end{lemma}
\begin{proof}
  As in the proof of Lemma \ref{lemma:exp-int}, we may assume that 
  $\gamma_\corn\in(3/2,5/2)$ and $\gamma_e\in(1,2)$.
  The proof of this statements follows by summing over the right-hand side of
  \eqref{eq:edge-elem}, i.e.,
  \begin{multline*}
      \begin{aligned}
    \sum_{K:d_\corn^K>0,\atop d_e^K=0}\| v - \Pihpelldim v\|_{H^1(K)}^2 
    %+ \| \nabla(v - \Pihpelldimv)\|_{L^2(K)}^2 \\
    & \leq
    C\left(\sum_{k=1}^\ell \sigma^{2\min\{\gamma_\corn-1,s\}(\ell-k)} \Psi_{p, s} A^{2s}((s+3)!)^2 + \sigma^{ 2(\min(\gamma_\corn,\gamma_e)-1)\ell}\right)
      \\ &
      =C
      %\sum_{k=1}^\ell 
      ((I) + (II)).
    \end{aligned}
  \end{multline*}
  We have $(II) \lesssim \ell \sigma^{2(\min(\gamma_\corn, \gamma_e)-1)\ell}$;
  the observation that for all $A>0$ there exist $C_1, b_1>0$ such that 
  \begin{equation*}
  \min_{1\leq s\leq p}\Psi_{p, s} ((s+3)!)^2 A^{2s} \leq  C_1 e^{-b_1 p},
  \end{equation*}
  (see, e.g., \cite[Lemma 5.9]{SSWII}).
  Combining with $p\geq C_0\ell$ concludes the proof.
\end{proof}
%%%%%%%%%%%%%%%%%%%%%%%%%%%%%%%%%%%%%%%%%%%%%%%%%%%%%%%%%%%%%%
\subsection{Estimates at the corner}
\label{sec:corner-estimates}
%%%%%%%%%%%%%%%%%%%%%%%%%%%%%%%%%%%%%%%%%%%%%%%%%%%%%%%%%%%%%%
The lemma below follows from classic low-order finite element approximation
results and from the embedding $\cJ^2_\ugamma(Q;\Cset, \Eset)\subset H^{1+\theta}(Q)$, valid
for a $\theta>0$ if $ \gamma_\corn-d/2>0$, for all $\corn\in\Cset$, and, when $d=3$, $\gamma_e >1$ for all $e\in\Eset$ 
(see, e.g., \cite[Remark 2.3]{SchSch2018}).
\begin{lemma}
  \label{lemma:exp-corner}
  Let $d \in \{2,3\}$, $K = \bigtimes_{i=1}^dJ_0^{\ell}$. Then, if $v\in\cJ^\varpi_\ugamma(Q;\Cset, \Eset)$ with
  \begin{alignat*}{2}
    &\gamma_\corn>1,\;\text{for all }\corn\in\Cset ,  &&\text{ if } d = 2,\\
    & \gamma_\corn>3/2\text{ and }
    \gamma_e>1,\; \text{for all }\corn\in\Cset\text{ and }e\in\Eset,\quad &&\text{ if } d = 3,
  \end{alignat*}
  % \begin{alignat*}{2}
  %   &\forall\corn\in\Cset: \gamma_\corn>1, &&\text{ if } d = 2,\\
  %   &\forall\corn\in\Cset: \gamma_\corn>3/2 \text{ and }\forall e\in\Eset: \gamma_e>1,\quad &&\text{ if } d = 3,
  % \end{alignat*}
  there exists a constant $C_0>0$ independent of $\ell$ such that if $p\geq C_0\ell$, there exist constants $C,b>0$ such that
  \begin{equation*}
    \| v - \Pihpelldim v\|_{H^1(K)} 
    %+ \| \nabla(v - \Pihpelldimv)\|_{L^2(K)} 
    \leq Ce^{-b\ell}.
  \end{equation*}
\end{lemma}
%%%%%%%%%%%%%%%%%%%%%%%%%%%%%%%%%%%%%%%%%%%%%%%%%%%%%%%%%%%%%%
\subsection{Exponential convergence}
\label{sec:exp-conv}
%%%%%%%%%%%%%%%%%%%%%%%%%%%%%%%%%%%%%%%%%%%%%%%%%%%%%%%%%%%%%%
The exponential convergence of the approximation in the full domain $Q$ follows then from Lemmas
\ref{lemma:exp-int}, \ref{lemma:exp-int-2d}, \ref{lemma:exp-edge}, and \ref{lemma:exp-corner}.

\begin{proposition}
  \label{lemma:exp-conv}
  Let $d \in \{2,3\}$, $v\in\cJ^\varpi_\ugamma(Q;\Cset, \Eset)$ with
  \begin{alignat*}{2}
    &\gamma_\corn>1,\;\text{for all }\corn\in\Cset ,  &&\text{ if } d = 2,\\
    & \gamma_\corn>3/2\text{ and }
    \gamma_e>1,\; \text{for all }\corn\in\Cset\text{ and }e\in\Eset,\quad &&\text{ if } d = 3.
  \end{alignat*}
  % \begin{alignat*}{2}
  %   &\gamma_\corn>1,\; \text{for all }\corn \in \Cset &&\text{ if } d = 2,\\
  %   &\forall\corn\in\Cset: \gamma_\corn>3/2 \text{ and }\forall e\in\Eset: \gamma_e>1,\quad &&\text{ if } d = 3.
  % \end{alignat*}
   Then, there exist constants $c_p>0$ and 
     $C, b>0$ such that, for all $\ell\in \mathbb{N}$, 
  \begin{equation*}
    \| v - \PihpellCpdim v\|_{H^1(Q)}\leq Ce^{-b\ell}.
  \end{equation*}
  With respect to the dimension of the discrete space $\Ndof = \dim(\XhpellCpdim)$, the above bound reads
  \begin{equation*}
    \| v - \PihpellCpdim v\|_{H^1(Q)}\leq C\exp(-b \Ndof^{1/(2d)}).
  \end{equation*}
\end{proposition}
%%%%%%%%%%%%%%%%%%%%%%%%%%%%%%%%%%%%%%%%%%%%%%%%%%%%%%%%%%%%%%%%
\subsection{Explicit representation of the approximant in terms of continuous
  basis functions}
\label{sec:basis}
%%%%%%%%%%%%%%%%%%%%%%%%%%%%%%%%%%%%%%%%%%%%%%%%%%%%%%%%%%%%%%%%
  Let $p\in\N$. Let $\hat{\zeta}_1(x) = (1+x)/2$ and $\hat{\zeta}_2 = (1-x)/2$. Let also
  $\hat{\zeta}_n(x) = \frac{1}{2}\int_{-1}^xL_{n-2}(\xi)d\xi$, for $n=3, \dots, p+1$,
  where $L_{n-2}$ denotes the $L^\infty((-1,1))$-normalized Legendre polynomial of degree $n-2$
  introduced in Section \ref{sec:loc-proj}.
  Then, fix $\ell\in \mathbb{N}$ and write $\zeta^k_n = \hat{\zeta}_n \circ
  \phi_k$, $n=1,\dots, p+1$ and $k=0, \dots, \ell$,
  with the affine map $\phi_k:J_{k}^\ell \to (-1,1)$ introduced in Section \ref{sec:loc-proj}.
  We construct those functions explicitly: denoting $J^\ell_k = (x_k, x_{k+1})$ and
  $h_k = |x_{k+1}-x_k|$, there holds, for $x\in J_k^\ell$,
  \begin{equation}
    \label{eq:zeta-12}
    \zeta^k_1 (x)= \frac{1}{h_k}(x-x_k),\qquad \qquad \zeta^k_2(x)= \frac{1}{h_k}(x_{k+1} -x),
  \end{equation}
  and
  \begin{equation}
    \label{eq:zeta-n}
    \zeta_n^k(x)= \frac{1}{h_k}\int_{x_k}^x L_{n-2}(\phi_k(\eta))d\eta\qquad n=3, \dots, p+1.
  \end{equation}
Then, for any element $K\in \cG^\ell_3$, with $K = J^\ell_{k_1}\times
J^\ell_{k_2}\times J^\ell_{k_3}$, there exist coefficients $c^K_{\is}$
such that
\begin{equation}
  \label{eq:local-proj-c}
  \Pihpelldim u_{|_K} (x_1, x_2, x_3)= \sum_{i_1, i_2, i_3=1}^{p+1} c^K_{\is} \zeta^{k_1}_{i_1}(x_1)\zeta^{k_2}_{i_2}(x_2)\zeta^{k_3}_{i_3}(x_3), \quad \forall (x_1, x_2, x_3)\in K
\end{equation}
by construction.
We remark that, whenever $i_j > 2$ for all $j=1,2,3$, the basis
functions vanish on the boundary of the element:
\begin{equation*}
  \left(\zeta^{k_1}_{i_1}\zeta^{k_2}_{i_2}\zeta^{k_3}_{i_3}  \right)_{|_{\partial K}} = 0 
       \qquad\text{if }i_j \geq 3, \, j=1,2,3.
\end{equation*}
Furthermore, write
\begin{equation*}
  \psi_{\is}^K (x_1, x_2, x_3)= \zeta^{k_1}_{i_1}(x_1)\zeta^{k_2}_{i_2}(x_2)\zeta^{k_3}_{i_3} (x_3)
\end{equation*}
and consider $t_{\is} = \# \{i_j\leq 2,\, j=1,2,3\}$. We have
\begin{itemize}
\item if $t_{\is} = 1$,  then $\psi_{\is}^K$ is not zero only on one face of the boundary of $K$,
\item if $t_{\is} = 2$,  then $\psi_{\is}^K$ is not zero only on one edge and neighboring
  faces of\ the boundary of $K$,
\item if $t_{\is} = 3$,  then $\psi_{\is}^K$ is not zero only on one corner and neighboring
  edges and faces of the boundary of $K$.
\end{itemize}
Similar arguments hold when $d=2$.
%%%%%%%%%%%%%%%%%%%%%%%%%%%%%%%%%%%%%%%%%%%%%%%%%%%%%%%%%%%%%%%%
\subsubsection{Explicit bounds on the coefficients}
\label{sec:c-bounds}
%%%%%%%%%%%%%%%%%%%%%%%%%%%%%%%%%%%%%%%%%%%%%%%%%%%%%%%%%%%%%%%%
We derive here a bound on the coefficients of the local projectors with respect
to the norms of the projected function. 
%We recall the bound
We will use that
\begin{equation}
  \label{eq:L2-L}
\|L_{i}\circ  \phi_{k}\|_{L^2(J^\ell_{k})}
	= \left(\frac{h_k}{2}\right)^{1/2} \|L_{i} \|_{L^2((-1,1))} 
	= \left( \frac{h_{k}}{2i +1} \right)^{1/2} \qquad \forall i\in \mathbb{N}_0,\, \forall k\in\{0, \dots, \ell\}.
\end{equation}
\begin{remark}
\label{rem:hpw11mix}
As mentioned in Remark \ref{rem:bigger-space-hp},
the $hp$-projector $\Pihpelldim$ can be defined for more general functions than $u\in\Hmix^1(Q)$. 
As follows from Equations \eqref{eq:c-int-def}, \eqref{eq:c-face-def}, 
\eqref{eq:c-edge-def} and \eqref{eq:c-node-def} below, 
the projector is also defined for $u\in\Wmix^{1,1}(Q)$.
\end{remark}
%jo
%
\begin{lemma}
  \label{lemma:cbound}
  There exist constants $C_1, C_2$ such that, for all $u\in \Wmix^{1,1}(Q)$, all $\ell
  \in \mathbb{N}$, all $p\in \mathbb{N}$
  \begin{equation}
    \label{eq:c-1}
    |c^K_{\is} |\leq C \left(\prod_{j=1}^di_j  \right) \| u \|_{\Wmix^{1,1}(Q)}\qquad 
    \forall K\in\cG^\ell_d, \, \forall  (\iscomma)\in \{1, \dots, p+1\}^d
  \end{equation}
  and for all $(\iscomma)\in \{1, \dots, p+1\}^d$
\begin{align}
\sum_{K\in \cG^\ell_3}|c^K_{\is} | \leq 
C \| u \|_{\Wmix^{1,1}(Q)}
\begin{cases}
\left( \prod_{j=1}^d i_j  \right)
& \text{ if } t_{\is} = 0, \\
 (\ell+1)\left(\sum_{j_1=1}^d\sum_{j_2=j_1+1}^d i_{j_1}i_{j_2} \right)
& \text{ if } t_{\is} = 1, \\
(\ell+1)^2\left(\sum_{j=1}^di_j  \right) 
& \text{ if } t_{\is} = 2, \\
(\ell+1)^d 
& \text{ if } t_{\is} = 3. 
\end{cases}
\label{eq:c-2}
\end{align}

\end{lemma}
\begin{proof}
  Let $d=3$ and $K = J_{k_1}^\ell \times J_{k_2}^\ell \times J_{k_3}^\ell \in \cG^\ell_3$. 
\\\textbf{Internal modes.}
We start by considering the case of the
coefficients of internal modes, i.e., $c^K_{i_1, i_2, i_3}$ as defined in
\eqref{eq:local-proj-c} for $i_n\geq 3$, $n=1,2, 3$. Let then $i_1, i_2, i_3\in \{3, \dots, p+1\}$ and
write $L_n^k = L_n \circ \phi_k$: there holds
\begin{align}
  c^K_{i_1,i_2, i_3} =&\, (2i_1-3)(2i_2-3)(2i_3-3)\nonumber\\
  & \int_{K} \left({\partial_{x_1}\partial_{ x_2}\partial_{x_3}} u(x_1, x_2, x_3)  \right) L_{i_1-2}^{k_1}(x_1)L_{i_2-2}^{k_2}(x_2)L_{i_3-2}^{k_3}(x_3) dx_1dx_2dx_3.
  \label{eq:c-int-def}
\end{align}
If $u\in \Wmix^{1,1}(K)$, since $\|L_n\|_{L^\infty(-1,1)} = 1$ for all $n$, we have
\begin{equation}
  \label{eq:c-int-1}
 | c^K_{\is} | \leq (2i_1-3)(2i_2-3)(2i_3-3) \| \partial_{x_1}\partial_{ x_2}\partial_{x_3}u \|_{L^1(K)} \qquad i_n\geq 3,\,n=1,2,3,
\end{equation}
hence,
\begin{equation}
  \label{eq:c-int-2}
 \sum_{K\in \cG^\ell_3}| c^K_{\is} | \leq (2i_1-3)(2i_2-3)(2i_3-3) \| \partial_{x_1}\partial_{ x_2}\partial_{x_3}u \|_{L^1(Q)} \qquad i_n\geq 3,\,n=1,2,3.
\end{equation}
\textbf{Face modes.}
We continue with face modes and fix, for ease of notation, $i_1 =1$. 
We also denote $F = J^\ell_{k_2}\times J^\ell_{k_3}$. 
The estimates will then also hold for $i_1=2$ and for any permutation of the indices by symmetry.
We introduce the trace inequality
constant $C^{T,1}$, independent of $K$, such that, 
for all $v\in W^{1,1}(Q)$ and $\hx\in (0,1)$,
\begin{equation}
  \label{eq:trace}
  % \| u \|_{L^p(F)} \leq C^{T,1} \left( h_{k_1}^{-1}\| u\|_{L^1(K)} + \|\partial_{x_1} u \|_{L^1(K)} \right),
  \| v(\hx, \cdot, \cdot) \|_{L^1(F)} 
  \leq 
   \| v(\hx, \cdot, \cdot)\|_{L^1((0,1)^2)}
  \leq 
   C^{T,1} \left( \| v\|_{L^1(Q)} + \|\partial_{x_1} v \|_{L^1(Q)} \right).
\end{equation}
This follows from the trace estimate in \cite[Lemma 4.2]{Schotzau2013a} 
and from the fact that
\begin{multline*}
  \| v(\hx, \cdot, \cdot)\|_{L^1((0,1)^2)}
  \leq C\min
  \bigg\{
  \frac{1}{|1-\hx|}\| v\|_{L^1((\hx,1)\times (0,1)^2)} + \|\partial_{x_1} v \|_{L^1((\hx, 1)\times (0,1)^2)},\\
  \frac{1}{|\hx|}\| v\|_{L^1((0,\hx)\times (0,1)^2)} + \|\partial_{x_1} v \|_{L^1((0,\hx)\times (0,1)^2)}
\bigg\}.
\end{multline*}
 There
holds, for $i_2, i_3\in \{3,
\dots, p+1\}$,
\begin{equation}
\label{eq:c-face-def}
  c^K_{1,i_2, i_3} = (2i_2-3)(2i_3-3)\int_{F} \left({\partial_{ x_2}\partial_{x_3}} u(x_{k_1}^{\ell}, x_2, x_3)  \right) L_{i_2-2}^{k_2}(x_2)L_{i_3-2}^{k_3}(x_3) dx_2dx_3.
\end{equation}
Since the Legendre polynomials %functions
are $L^\infty$ normalized and using the trace
inequality \eqref{eq:trace},
\begin{equation}
  \label{eq:c-face-1}
  |c_{1, i_2, i_3}^K| \leq 
(2i_2-3)(2i_3-3) \| ({\partial_{ x_2}\partial_{x_3}} u )(x_{k_1}^{\ell}, \cdot, \cdot)\|_{L^1(F)}
\leq  C^{T,1}(2i_2-3)(2i_3-3) \| u \|_{\Wmix^{1,1}(Q)}. %\frac{}{4}
\end{equation}
Summing over all internal faces, furthermore,
\begin{equation}
  \label{eq:c-face-2}
  \begin{aligned}
  \sum_{K\in \cG^\ell_3}|c_{1, i_2, i_3}^K| &\leq 
(2i_2-3)(2i_3-3) \sum_{k_1=0}^\ell\| ({\partial_{ x_2}\partial_{x_3}} u )(x_{k_1}^{\ell}, \cdot, \cdot)\|_{L^1((0,1)^2)}\\
&\leq  C^{T,1}(\ell+1)(2i_2-3)(2i_3-3) \| u \|_{\Wmix^{1,1}(Q)}.
  \end{aligned}
\end{equation}
\textbf{Edge modes.}
We now consider edge modes. Fix for ease of notation $i_1 = i_2 = 1$; as
before, the estimates will hold for $(i_1, i_2)\in \{1,2\}^2$ and for any
permutation of the indices. By the same arguments as for \eqref{eq:trace}, there exists
a trace constant $C^{T,2}$ such that, denoting $e = J^\ell_{k_3}$, for all $v\in
W^{1,1}((0,1)^2)$ and for all $\hx\in (0,1)$,
\begin{equation}
  \label{eq:trace-2}
  \| v (\hx, \cdot)\|_{L^1(e)} \leq \| v (\hx, \cdot)\|_{L^1((0,1))} \leq C^{T,2} \left( \| u\|_{L^1((0,1)^2)} + \|\partial_{x_2} u \|_{L^1((0,1)^2)} \right).
\end{equation}
By definition, 
\begin{equation}
\label{eq:c-edge-def}
  c^K_{1,1, i_3} = (2i_3-3)\int_{e} \left({\partial_{x_3}} u(x_{k_1}^{\ell}, x_{k_2}^{\ell}, x_3)  \right) L_{i_3-2}^{k_3}(x_3) dx_3.
\end{equation}
Using \eqref{eq:trace} and \eqref{eq:trace-2}
  \begin{equation}
    \label{eq:c-edge-1}
      |c^K_{1,1, i_3} |
      \leq (2i_3-3) \| (\partial_{x_3} u)(x_{k_1}^{\ell}, x_{k_2}^{\ell}, \cdot) \|_{L^1(e)}
      \leq  C^{T,1} C^{T,2}(2i_3-3) \| u\|_{\Wmix^{1,1}(Q)}.
  \end{equation}
  Summing over edges, in addition,
  \begin{equation}
    \label{eq:c-edge-2}
  \begin{aligned}
      \sum_{K\in \cG^\ell_3}|c^K_{1,1, i_3} |
      &\leq (2i_3-3) \sum_{k_1=0}^\ell\sum_{k_2=0}^\ell\| (\partial_{x_3} u)(x_{k_1}^{\ell}, x_{k_2}^{\ell}, \cdot) \|_{L^1((0,1))}\\
      &\leq  C^{T,1} C^{T,2}(\ell+1)^2(2i_3-3) \| u\|_{\Wmix^{1,1}(Q)}.
  \end{aligned}
\end{equation}
\textbf{Node modes.}
Finally, we consider the coefficients of nodal modes, i.e., $c^K_{i_1,
i_2, i_3}$ for $i_1, i_2, i_3\in \{1,2\}$, 
which by construction equal function values of $u$, e.g. 
\begin{equation}
\label{eq:c-node-def}
c_{111} = u(x_{k_1}^{\ell},x_{k_2}^{\ell},x_{k_3}^{\ell}).
\end{equation}%jo
The Sobolev
imbedding $\Wmix^{1,1}(Q)\hookrightarrow L^{\infty}(Q)$ and
scaling implies the existence of a uniform
constant $C_{\mathrm{imb}}$ such that, for any $v\in
\Wmix^{1,1}(Q)$ %K 
\begin{equation*}
\| v \|_{L^\infty(K)} \leq\| v \|_{L^\infty(Q)} \leq C_{\mathrm{imb}} \| v \|_{\Wmix^{1,1}(Q)}.
\end{equation*}
Then, by construction,
\begin{equation}
  \label{eq:c-node-1}
  |c^K_{i_1, i_2, i_3}| \leq \| u \|_{L^\infty(K)} \leq
    C_{\mathrm{imb}} \| u \|_{\Wmix^{1,1}(Q)}
  \qquad \forall i_1, i_2, i_3\in \{1,2\}.
\end{equation}
  Summing over nodes, it follows directly that
  \begin{equation}
    \label{eq:c-node-2}
  \sum_{{K\in \cG^\ell_3}}|c^K_{i_1, i_2, i_3}| \leq \sum_{K\in \cG^\ell_3}\| u \|_{L^\infty(K)} \leq
    C_{\mathrm{imb}} (\ell+1)^3\| u \|_{\Wmix^{1,1}(Q)}
  \qquad \forall i_1, i_2, i_3\in \{1,2\}.
  \end{equation}
 We obtain \eqref{eq:c-1} from \eqref{eq:c-int-1}, \eqref{eq:c-face-1}, 
 \eqref{eq:c-edge-1}, and \eqref{eq:c-node-1}. Furthermore, \eqref{eq:c-2}
 follows from \eqref{eq:c-int-2}, \eqref{eq:c-face-2}, 
 \eqref{eq:c-edge-2}, and \eqref{eq:c-node-2}.
 The estimates for the case $d=2$ follow from the same argument.
\end{proof}
The following lemma shows the continuous imbedding of $\cJ^{d}_{\ugamma}(Q;\Cset, \Eset)$ 
into $\Wmix^{1,1}(Q)$, 
given sufficiently large weights $\ugamma$.
\begin{lemma}
  \label{lemma:W11J3}
  Let $d\in \{2,3\}$.
  Let $\ugamma$ be such that $ \gamma_\corn>d/2$, for all $\corn\in\Cset$ and
  (if $d=3$) $\gamma_e>1$ for all $e\in\Eset$. There exists a constant $C>0$ such that, for all $u \in \cJ^d_\ugamma(Q;\Cset, \Eset)$, 
  \begin{equation*}
    \| u \|_{\Wmix^{1,1}(Q)} \leq C \| u \|_{\cJ^d_{\ugamma}(Q)}.
  \end{equation*}
\end{lemma}
\begin{proof}
  We recall the decomposition of $Q$  as
  \begin{equation*}
   \overline{Q}= \overline{Q_0}\cup \overline{Q_{\Cset}} \cup \overline{Q_{\Eset}} \cup \overline{Q_{\Cset\Eset}},
  \end{equation*}
  where $Q_{\Eset}= Q_{\Cset\Eset} = \emptyset$ if $d=2$.
  There holds
  \begin{equation}
    \label{eq:W11J3-1}
    \| u \|_{\Wmix^{1,1}(Q_0)} \leq C |Q_0|^{1/2}\| u \|_{H^d(Q_0)}\leq C |Q_0|^{1/2} \| u \|_{\cJ^d_{\ugamma}(Q)}. 
    %%%JO: constant comes from the fact that on the left we have a sum, 
    %%%     on the right the root of the sum of squares
  \end{equation}
  We now consider the subdomain $Q_{\corn}$, for any $\corn\in \Cset$. 
  There holds, 
  with constant $C$ that depends only on $\gamma_\corn$ and on $|Q_\corn|$,
  \begin{equation}
    \label{eq:W11J3-2}
    \begin{aligned}
      \| u \|_{\Wmix^{1,1}(Q_\corn)}
      &= \| u \|_{W^{1,1}(Q_\corn)} + \sum_{\substack{2\leq \alpham \leq d\\ \alphainf\leq 1}} \|\dalpha u \|_{L^1(Q_\corn)}\\
   &\leq C |Q_\corn|^{1/2} \| u \|_{H^{1}(Q_\corn)}  + C \sum_{\substack{2\leq \alpham \leq d\\ \alphainf \leq 1}} \|r_\corn^{-(\alpham-\gamma_\corn)_+}\|_{L^2(Q_\corn)}\|r_\corn^{(\alpham-\gamma_\corn)_+}\dalpha u \|_{L^2(Q_\corn)}\\
   &\leq C \| u \|_{\cJ^d_{\ugamma}(Q)},
    \end{aligned}
  \end{equation}
  where the last inequality follows from the fact that $\gamma_\corn > d/2$, hence
  the norm $\|r_\corn^{-(\alpham-\gamma_\corn)_+}\|_{L^2(Q_\corn)}$ is bounded for all
  $\alpham\leq d$.
  Consider then $d=3$ and any $e\in \Eset$. 
Suppose also, without loss of generality, that $\gamma_\corn-\gamma_e >1/2$ and $\gamma_e<2$
(otherwise, it is sufficient to replace
$\gamma_e$ by a smaller $\tgamma_e$ such that $1<\tgamma_e< \gamma_\corn-1/2$ and $\gamma_e<2$ 
and remark that
$\cJ^d_\ugamma(Q;\Cset, \Eset)\subset \cJ_\utgamma^d(Q;\Cset, \Eset)$ if $\tgamma_e <
\gamma_e$).  
  Since $\gamma_e > 1$, then
  $\|r_e^{-\alphaperpm+\gamma_e}\|_{L^2(Q_e)}$ is bounded by a constant depending
  only on $\gamma_e$ and $|Q_e|$ as long as $\alpha$ is such that
  $\alphaperpm\leq 2$. Hence, denoting by $\dpar$ the derivative in the
  direction parallel to $e$, 
\begin{equation}
    \label{eq:W11J3-3}
    \begin{aligned}
      \| u \|_{\Wmix^{1,1}(Q_e)}
      &= \| u \|_{W^{1,1}(Q_e)} + \sum_{\alphaperpm = 1}\|\dpar\dalphaperp u \|_{L^1(Q_e)}
       +\sum_{\alpha_1 =0,1} \|\dpar^{\alpha_1}\dperpone\dperptwo u \|_{L^1(Q_e)}
   \\
   &\begin{multlined}[][.7\linewidth]
     \leq C |Q_e|^{1/2} \left(\| u \|_{H^{1}(Q_e)}  +\sum_{\alphaperpm = 1}  \|\dpar\dalphaperp v\|_{L^2(Q_e)}\right)
   \\
     +C \sum_{\alpha_1=0,1}  \|r_e^{-2+\gamma_e}\|_{L^2(Q_e)}\|r_e^{2-\gamma_e}\dpar^{\alpha_1}\dperpone\dperptwo u \|_{L^2(Q_e)}
   \end{multlined}
   \\
   &\leq C \| u \|_{\cJ^3_{\ugamma}(Q)}.
   \end{aligned}
  \end{equation}
%Fix any $e\in\Eset$ and let $\xpar$ be the variable in
%direction parallel to $e$.
Since $\xpar\leq r_\corn(x)\leq \hepsilon$ for all $x\in Q_{\corn e}$,  
%and due to the fact that %%% resulted in a too long line
and because
$Q_{\corn e}\subset \left\{\xpar\in(0,\hepsilon), (\xperpone, \xperptwo)\in(0, \hepsilon^2)^2\right\}$,
there holds
\begin{equation*}
 \|r_\corn^{-(\gamma_e+1-\gamma_\corn)_+} r_e^{-2+\gamma_e} \|_{L^2(Q_{\corn e})}
 \leq 
 \|\xpar^{-(\gamma_e+1-\gamma_\corn)_+} \|_{L^2((0, \hepsilon))}\|r_e^{-2+\gamma_e} \|_{L^2((0,\hepsilon^2)^2)}\leq C,
\end{equation*}
for a constant $C$ that depends only on $\hepsilon$, $\gamma_\corn$, and $\gamma_e$. Hence,
\begin{equation}
    \label{eq:W11J3-4}
    \begin{aligned}
      \| u \|_{\Wmix^{1,1}(Q_{\corn e})}
      &= \| u \|_{W^{1,1}(Q_{\corn e})} + \sum_{\alphaperpm = 1}\|\dpar\dalphaperp u \|_{L^1(Q_{\corn e})}+ \sum_{\alpha_1=0,1}\|\dpar^{\alpha_1}\dperpone\dperptwo u \|_{L^1(Q_{\corn e})}\\
   &\begin{multlined}[][.7\textwidth]
     \leq C |Q_{\corn e}|^{1/2} \| u \|_{H^{1}(Q_e)}  + C \sum_{\alphaperpm = 1} \|r_\corn^{-(2-\gamma_\corn)_+}\|_{L^2(Q_{\corn e})} \|r_\corn^{(2-\gamma_\corn)_+}\dpar\dalphaperp u \|_{L^2(Q_{\corn e})}\\
     + C \sum_{\alpha_1 = 0,1}  \|r_\corn^{-(\alpha_1+\gamma_e-\gamma_\corn)_+}r_e^{-2+\gamma_e}\|_{L^2(Q_{\corn e})}\|r_\corn^{(\alpha_1+2-\gamma_\corn)_+}\rho_{\corn e}^{2-\gamma_e}\dpar^{\alpha_1}\dperpone\dperptwo u \|_{L^2(Q_{\corn e})}
   \end{multlined}\\
   &\leq C \| u \|_{\cJ^3_{\ugamma}(Q)},
    \end{aligned}
  \end{equation}
with $C$ independent of $u$. Combining inequalities
\eqref{eq:W11J3-1} to \eqref{eq:W11J3-4} concludes the proof.
\end{proof}
The following statement is a direct consequence of the two lemmas above
and the fact that \linebreak $\normc[L^\infty(K)]{\psi_{\is}^K} \leq 1$ for all $K\in \cG^\ell_3$
and all $\iscomma\in\{1,\ldots,p+1\}$.
%which in turn follows from $\normc[L^\infty((-1,1))]{\zeta^k_n}\leq1$ for all $n=1,\ldots,p+1$.
\begin{corollary}
 Let $\ugamma$ be such that $ \gamma_\corn-d/2>0$, for all $\corn\in\Cset$ and,
 if $d=3$, $ \gamma_e>1$ for all $e\in\Eset$. There exists a constant $C>0$  
 such that for all $\ell,p\in \mathbb{N}$ and for all $u \in \cJ^d_\ugamma(Q;\Cset, \Eset)$,
 \begin{equation*}
   \|\Pihpelldim u \|_{L^\infty(Q)}\leq Cp^{2d}\| u \|_{\cJ^d_\ugamma(Q)}.
 \end{equation*}
\end{corollary}
\subsubsection{Basis of continuous functions with compact support}
\label{sec:compactbasis}
It is possible to construct a basis for
$\Pihpelldim$ in $Q$  such that all basis functions are continuous and have compact
support.
For all $\ell\in \mathbb{N}$ and all $p\in \mathbb{N}$,
extend to zero outside of their domain of definition the functions $\zeta^k_n$
defined in  \eqref{eq:zeta-12} and \eqref{eq:zeta-n}, for $k=0, \dots, \ell$ and
$n=1, \dots, p+1$.
We introduce the univariate functions with compact support $v_j : (0,1)\to
\mathbb{R}$, for $j=1, \dots, (\ell+1)p+1$ so that
  $v_1 = \zeta^0_{2}$, $v_{\ell+2} = \zeta^\ell_{1}$, 
\begin{equation}\label{eq:DefinitionOfvk}
v_{k} = \zeta^{k-2}_{1} + \zeta^{k-1}_{2}, \qquad \text{for all }k =2, \dots, \ell+1
\end{equation}
and
\begin{equation*}
v_{\ell+2 + k(p-1)+n} = \zeta^{k}_{n+2}, \qquad \text{for all }k =0, \dots, \ell\text{ and }n = 1, \dots, p-1.
\end{equation*}
% Finally, denote by $v_{j,i}$, for $j=1, \dots, d$ and $i=1,\dots, p(\ell+2)$ the
% function $v_i$ acting on the $j$th coordinate.

\begin{proposition}
  \label{prop:compactbasis}
Let $\ell\in \mathbb{N}$ and $p\in \mathbb{N}$. Furthermore, let $u\in
\cJ^d_\ugamma(Q;\Cset, \Eset)$ with $\ugamma$ such that $\gamma_\corn-d/2>0$ and, if $d=3$,
$\gamma_e>1$. Let $\Noned = (\ell+1)p+1$. There
exists an array of coefficients
\[ 
	c  = \left\{c_{\is}: (\iscomma) \in \{1, \dots, \Noned\}^d\right\}
 \]
 such that
 \begin{equation}
 \label{eq:compactbasistensor}
 \left(\Pihpelldim u\right)(x_1,\dots,x_d) = \sum_{\iscomma = 1}^{\Noned} c_{\is} \prod_{j=1}^dv_{i_j}(x_j)\qquad \forall (x_1, \dots, x_d)\in Q.
 \end{equation}
Furthermore, there exist constants $C_1, C_2>0$ independent of $\ell$, $p$, and
$u$,  such that 
\begin{equation*}
  |c_{\is} |\leq C_1 (p+1)^d \| u \|_{\cJ^d_\ugamma(Q)}\qquad \forall \iscomma \in\{ 1,\dots, \Noned\}^d
\end{equation*}
and 
\begin{equation*}
\sum_{\iscomma=1}^{\Noned} |c_{\is}|
%	\leq \sum_{t=0}^d \sum_{\substack{\iscomma=1\\ t_{\is}=t}}^{p+1} 
%		\left(\sum_{K\in\cG^\ell_d} |c_{\is}| \right)
	\leq C_2 \left(\sum_{t=0}^d (\ell+1)^{t}(p+1)^{2(d-t)}  \right)\| u \|_{\cJ^d_\ugamma(Q)}.
\end{equation*}
\end{proposition}
\begin{proof}
  The statement follows directly from the construction of the projector, see
  \eqref{eq:local-proj-c}, and from the bounds in Lemmas \ref{lemma:cbound} and \ref{lemma:W11J3}.
  In particular, \eqref{eq:compactbasistensor} holds because the element-wise coefficients 
  related to $\zeta_2^{k-1}$ and to $\zeta_1^{k-2}$ are equal:
  it follows from Equations %\eqref{eq:c-int-def}, 
  \eqref{eq:c-face-def}, 
  \eqref{eq:c-edge-def} and \eqref{eq:c-node-def}
  that $c^{K}_{1i_2\ldots i_d} = c^{K'}_{2i_2\ldots i_d}$
  for all $i_2,\ldots,i_d\in\{1,\ldots,p+1\}$,
  all $K = J_{k_1}^\ell \times J_{k_2}^\ell \times J_{k_3}^\ell \in \cG^\ell_3$ satisfying $k_1<\ell$
  and $K' =  J_{k_1+1}^\ell \times J_{k_2}^\ell \times J_{k_3}^\ell \in \cG^\ell_3$.
  The same holds for permutations of $\iscomma$.
  Because $(v_k)_{k=1}^{(\ell+1)p+1}$ are continuous, 
  this again shows continuity of $\Pihpelldim u$ (Remark \ref{rem:global-continuity}).
  The last estimate is obtained with \eqref{eq:c-2}:
  \begin{equation*}
\sum_{\iscomma=1}^{\Noned} |c_{\is}|
	\leq \sum_{t=0}^d \sum_{\substack{\iscomma=1\\ t_{\is}=t}}^{p+1} 
		\left(\sum_{K\in\cG^\ell_d} |c_{\is}| \right)
	\leq C_2 \left(\sum_{t=0}^d (\ell+1)^{t}(p+1)^{2(d-t)}  \right)\| u \|_{\cJ^d_\ugamma(Q)}.
\end{equation*}%jo
\end{proof}
\subsubsection{Proof of Theorem \ref{thm:Interface}}\label{subsec:ProofOfInterface}
\begin{proof}[Proof of Theorem \ref{thm:Interface}]
  Fix $A_f$, $C_f$, and $\ugamma$ as in the hypotheses. Then, by Proposition \ref{lemma:exp-conv}, there exists
  $c_p$, $C_{\hp}$, $b_{\hp}>0$ such that for every $\ell \in \N$ and for all 
  $v\in \cJ^\varpi_{\ugamma}(Q;\Cset, \Eset; C_f, A_f)$, 
  there exists $v_{\hp}^\ell\in \XhpellCpdim$ such that
  (see Section \ref{sec:hp-prdmshspc} for the definition of the space $\XhpellCpdim$)
    \begin{equation*}
      \| v - v_{\hp}^\ell\|_{H^1(Q)}\leq C_{\hp}e^{-b_{\hp} \ell}.
    \end{equation*}%
  For $\epsilon > 0$, we choose
    \begin{equation}
      \label{eq:L-eps}
      L \coloneqq \left\lceil  \frac{1}{b_{\hp}}\left|\log(\epsilon/C_{\hp}) \right|\right\rceil,
    \end{equation}
    so that 
    \begin{equation*}
      \| v - v_{\hp}^L\|_{H^1(Q)}\leq \epsilon.
    \end{equation*}%
  Furthermore, $v_{\hp}^L = \sum_{\iscomma}^{\Noned} c_{\is} \phi_{\is}$ and, for all $(\iscomma)\in\{1, \dots, \Noned\}^d$, there exists $v_{i_j}$, $j=1, \dots, d$ such that $\phi_{\is} =
  \bigotimes_{j=1}^dv_{i_j}$, see Section
  \ref{sec:compactbasis} and Proposition \ref{prop:compactbasis}. By construction of $v_{i}$ in \eqref{eq:DefinitionOfvk}, and by using 
 \eqref{eq:zeta-12} and \eqref{eq:zeta-n}, we observe that 
  $\|v_{i}\|_{L^\infty(I)}\leq 1$ for all $i=1, \dots, \Noned$. 
In addition, \eqref{eq:L2-L}, demonstrates that 
 \begin{equation*}
    \|v_{i}\|_{H^1(I)} \leq 
    \frac{2}{\left|\supp(v_{i})\right|^{1/2}\deg(v_{i})^{1/2}}
    \leq 2 \sigma^{-L/2}\qquad \forall  i\in\{1, \dots, \Noned\}.
  \end{equation*}
  Then, by \eqref{eq:L-eps},
  \begin{equation*}
    \sigma^{-L} \leq \sigma^{-\frac{1}{b_{\hp}}\log(C_{\hp})} \epsilon^{-\frac{1}{b_{\hp}}\log(1/\sigma)}. 
  \end{equation*}
  This concludes the proof of Items \ref{item:vli} and \ref{item:appx-eps}. Finally, Item
  \ref{item:c} follows from Proposition \ref{prop:compactbasis} and the fact that
  $p\leq C_p\left( 1+\left| \log(\epsilon) \right| \right)$ for a constant $C_p>0$
  independent of $\epsilon$.
\end{proof}
%%%%%%%%%%%%%%%%%%%%%%%%%%%%%%%%%%%%%%%%%%%%%%%%%%%%%%%%%%%%%%%%
\subsection{Combination of multiple patches}
\label{sec:patches}
%%%%%%%%%%%%%%%%%%%%%%%%%%%%%%%%%%%%%%%%%%%%%%%%%%%%%%%%%%%%%%%%
The approximation results in the domain $Q=(0,1)^d$ can be generalized to
include the combination of multiple patches. We give here an example,
relevant for the PDEs considered in Section \ref{sec:applications}. For the
sake of conciseness, we show a single %unique 
construction that takes into account all
singularities of the problems in Section \ref{sec:applications}. 
We will then use this construction to prove expression rate bounds for realizations of NNs.

Let $a>0$ and $\Omega = (-a,a)^d$. 
Denote the set of corners 
\begin{equation}
  \label{eq:Cset-int}
\Cset_\Omega = \bigtimes_{j=1}^d\{-a, 0, a\},
\end{equation}
and the set of edges
%\begin{equation}
%  \label{eq:Eset-int}
%\Eset_\Omega =
%\begin{cases}
%\emptyset &\text{if }d=2,\\
%%  \bigcup_{\corn\in\Cset_{\Omega}} \Eset_\corn 
%  \bigcup_{j=1}^d \bigtimes_{k=1}^{j-1}\{-a, 0, a\} \times \{(-a,-a/2),(-a/2,0),(0,a/2),(a/2,a)\} 
%  	\times \bigtimes_{k=j+1}^d\{-a, 0, a\}
%  &\text{if }d=3.
%\end{cases}
%\end{equation}
$\Eset_\Omega = \emptyset$ if $d=2$, and, if $d=3$,
\begin{equation}
  \label{eq:Eset-int}
\Eset_\Omega =
%\begin{cases}
%\emptyset &\text{if }d=2,\\
%%  \bigcup_{\corn\in\Cset_{\Omega}} \Eset_\corn 
  \bigcup_{j=1}^d \bigtimes_{k=1}^{j-1}\{-a, 0, a\} \times \{(-a,-a/2),(-a/2,0),(0,a/2),(a/2,a)\} 
  	\times \bigtimes_{k=j+1}^d\{-a, 0, a\}.
%  &\text{if }d=3.
%\end{cases}
\end{equation}
We introduce the affine transformations $\psi_{1, +}:(0,1)\to(0, a/2)$,
$\psi_{2,+}:(0,1)\to(a/2,a)$, $\psi_{1, -} :(0,1)\to(-a/2, 0)$, $\psi_{2, -}:(0,1)\to(-a,-a/2)$ such that
\begin{equation*}
  \psi_{1, \pm}(x) = \pm\frac{a}{2}x, \qquad \psi_{2,\pm}(x) = \pm\left(a-\frac{a}{2}x\right).
\end{equation*}
For all $\ell\in \mathbb{N}$, define then
\begin{equation*}
  \tcG^\ell_1 = \bigcup_{i\in \{1,2\}, \star\in \{+.-\}}\psi_{i, \star}(\cG^\ell_1).
\end{equation*}
Consequently, for $d=2,3$, denote 
$\tcG^\ell_d = \{\bigtimes_{i=1}^dK_i: K_1, \dots, K_d\in \tcG^\ell_1\}$, 
see Figure \ref{fig:multipatch}.
\begin{figure}
  \centering
  % \begin{subfigure}{.5\textwidth}
  %   \centering
  %   \includegraphics[width=.8\textwidth]{onedmultipatch.pdf}
  % \end{subfigure}%
  \begin{subfigure}[b]{.45\textwidth}
    \centering
    \includegraphics[width=.7\textwidth]{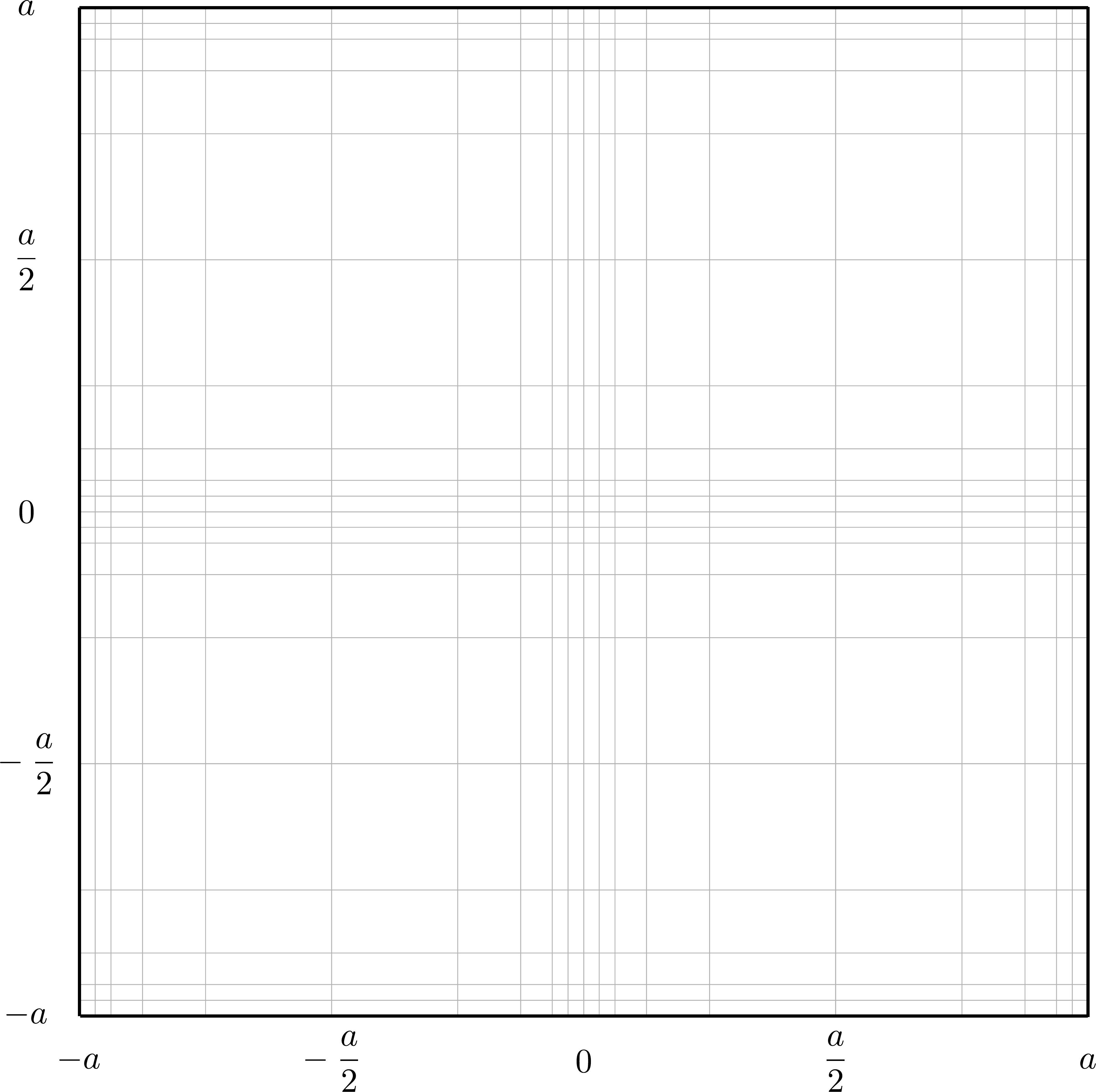}
  \end{subfigure}%
  \begin{subfigure}[b]{.55\textwidth}
    \centering
    \includegraphics[width=.7\textwidth]{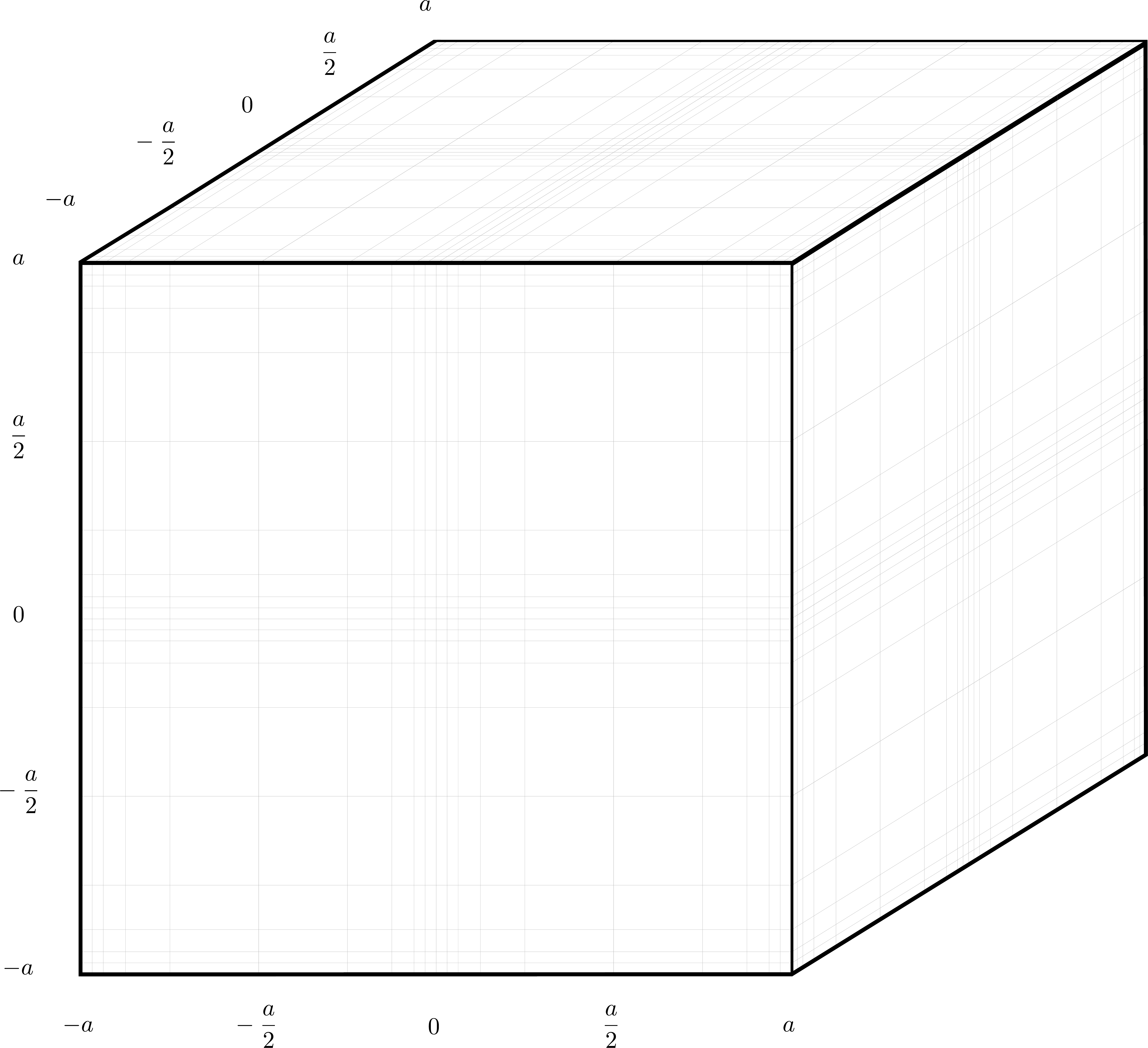}
  \end{subfigure}
  \caption{Multipatch geometric tensor product meshes $\tcG^\ell_d$, for $d=2$ (left) and $d=3$ (right).}
  \label{fig:multipatch}
\end{figure}
The $hp$ space in $\Omega = (-a,a)^d$ is then given by
\begin{equation*}
  \tXhpelldim = \{ v\in H^1(\Omega): v_{|_{K}} \in \mathbb{Q}_{p}(K), \text{ for all } K\in \tcG^\ell_d\}.
\end{equation*}
Finally, recall the definition of $\pihpell$ from
\eqref{eq:pihpell1d} and construct
\begin{equation*}
  % \label{eq:tpihpell}
  \tpihpell  : W^{1,1}((-a,a))\to \tXhpellone
\end{equation*}
such that, for all $v\in W^{1,1}((-a,a))$,
\begin{equation}
\label{eq:tpihpell}
\begin{aligned}
  &\left(\tpihpell v  \right)|_{(0,\frac{a}{2})} = \left(\pihpell (v|_{(0,\frac{a}{2})}\circ\psi_{1,+})\right)\circ\psi_{1,+} ^{-1},\quad 
  \left(\tpihpell v  \right)|_{(\frac{a}{2},a)} = \left(\pihpell (v|_{(\frac{a}{2},a)}\circ\psi_{2,+})\right)\circ\psi_{2,+} ^{-1},\\
  &\left(\tpihpell v  \right)|_{(-\frac{a}{2},0)} = \left(\pihpell (v|_{(-\frac{a}{2},0)}\circ\psi_{1,-}\right)\circ\psi_{1,-} ^{-1},\quad 
  \left(\tpihpell v  \right)|_{(-a, -\frac{a}{2})} = \left(\pihpell (v|_{(-a, -\frac{a}{2})}\circ\psi_{2,-})\right)\circ\psi_{2,-} ^{-1}.
\end{aligned}
\end{equation}
Then, the global $hp$ projection operator $\tPihpelldim :
\Wmix^{1,1} (\Omega)\to \tXhpelldim$ is defined as
\begin{equation*}
 \tPihpelldim = \bigotimes_{i=1}^d \tpihpell.
\end{equation*}
	
\begin{theorem}
\label{prop:internal}
For $a>0$, let $\Omega = (-a,a)^d$, $d=2,3$.
Denote by $\Omega^k$, $k=1, \dots, 4^d$ the patches composing $\Omega$,
i.e., 
the sets $\Omega^k = \bigtimes_{j=1}^d(a^k_j, a^k_j+a/2)$ 
with $a^k_j\in \{-a,-a/2,0,a/2\}$. 
Denote also $\Cset^k =\Cset_\Omega \cap \overline{\Omega}^k$ and
$\Eset^k = \{ e\in \Eset_\Omega: e\subset \overline{\Omega}^k\}$,
which contain one singular corner, 
and three singular edges abutting that corner, as in \eqref{eq:Cset} and \eqref{eq:Eset}.

Let $\mathcal{I}\subset \{1, \dots, 4^d\}$ and
let $v\in \Wmix^{1,1}(\Omega)$ be such that, for all $k\in \mathcal{I}$, 
there holds
$v_{|\Omega^k}\in \cJ^\varpi_{\ugamma^k}(\Omega^k; \Cset^k, \Eset^k)$ 
with 
  \begin{alignat*}{2}
    &\gamma^k_\corn>1,\;\text{for all }\corn\in\Cset^k ,  &&\text{ if } d = 2,\\
    & \gamma^k_\corn>3/2\text{ and }
    \gamma^k_e>1,\; \text{for all }\corn\in\Cset^k\text{ and }e\in\Eset^k,\quad &&\text{ if } d = 3.
  \end{alignat*}
Then,
there exist constants $c_p>0$ and $C, b>0$ such that, for all $\ell\in \mathbb{N}$, 
with $p = c_p \ell$,
\begin{equation}
    \label{eq:hp-Omega}
    \| v - \tPihpelldim v\|_{H^1(\Omega^k)}\leq Ce^{-b\ell} \leq C\exp(-b\sqrt[2d]{\Ndof}).
\end{equation}
Here, $\Ndof = \mathcal{O}(\ell^{2d})$ denotes the overall number of degrees of freedom 
in the piecewise polynomial approximation.
Furthermore, writing $\tNoned = 4(\ell+1)p+1$, there exists an array of coefficients
\[ 
	\tc  = \left\{\tc_{\is}: (\iscomma) \in \{1, \dots, \tNoned\}^d\right\}
 \]
 such that
 \begin{equation*}
 \left(\tPihpelldim v\right)(x_1,\dots,x_d) = \sum_{\iscomma = 1}^{\Noned} 
          \tc_{\is} \prod_{j=1}^d\tv_{ i_j}(x_j)\qquad \forall (x_1, \dots, x_d)\in \Omega,
 \end{equation*}
where for all $j=1, \dots, d$ and $i_j=1, \dots, \tNoned$, $\tv_{ i_j}\in \tXhpellone$ 
with support in at most two, neighboring elements of $\tcG^\ell_1$.
Finally, 
there exist constants $C_1, C_2>0$ independent of $\ell$ %and $p$ 
such that 
\begin{equation}
  \label{eq:v-Omega}
\|v_{i}\|_{H^1((-a,a))} \leq C_1 \sigma^{-\ell/2}\qquad\forall  i=1, \dots, \tNoned,
\end{equation}
and
\begin{equation}
  \label{eq:c-Omega}
\sum_{\iscomma=1}^{\tNoned}|\tc_{\is} |\leq C_2 \sum_{j=0}^d (\ell+1)^{j}(p+1)^{2(d-j)}\|v\|_{\Wmix^{1,1}(\Omega)}.
\end{equation}
\end{theorem}
\begin{proof}
The statement is a direct consequence of 
Propositions \ref{lemma:exp-conv} and \ref{prop:compactbasis}.
We start the proof by showing that for any function $v\in \Wmix^{1,1}(\Omega)$, 
the approximation $\tPihpelldim v$ is continuous; 
the rest of the theorem will then follow from the results in each
sub-patch. Let now $w\in W^{1,1}((-a,a))$. 
Then, it holds that
 $ \left(\tpihpell w  \right)|_{I} \in C(I)$, for all $I\in \{(0, a/2), (a/2,a),
 (-a/2, 0), (-a, -a/2)\}$, by definition \eqref{eq:tpihpell}.
 Furthermore, by the nodal exactness of the local projectors, for $\tilde{x} \in
 \{-a/2, 0, a/2\}$, there holds
 \begin{equation*}
   \lim_{x \to \tilde{x}^-} (\tpihpell w)(x) = w(\tilde{x}) = 
\lim_{x \to \tilde{x}^+} (\tpihpell w)(x),
 \end{equation*}
 implying then that $\tpihpell w$ is continuous. Since $\tPihpelldim =
 \bigotimes_{j=1}^d\tpihpell$, this implies that $\tPihpelldim v$ is continuous
 for all $v\in \Wmix^{1,1}(\Omega)$.
% Let
% \[
%   P=\bigg\{I_1\times \dots\times
%   I_d: I_j\in\{(-a, -a/2),(-a/2, 0),(0, a/2), (a/2,a) \}, \,  j=1, \dots, d\bigg\}
% \]
%denote the set of all patches: each one of those has a singular point of $\Cset_{\Omega}$ at one of
%its corners.
Fix $k\in \{1, \dots, 4^d\}$ such that $v\in \cJ^\varpi_{\ugamma^k}(\Omega^k; \Cset^k,\Eset^k)$.
 There exist then, by Proposition
 \ref{lemma:exp-conv}, constants $C, b, c_p>0$ such that for all $\ell\in\mathbb{N}$
 %there exists $p\leq c_p\ell$, such that
 \begin{equation*}
   \| v -  \tPihpellCpdim \|_{H^1(\Omega^k)}\leq Ce^{-b\ell}. 
%%   \qquad \forall S\in P\text{ such that }S\cap \Omega^k \neq \emptyset.
 \end{equation*}
 Equation \eqref{eq:hp-Omega} follows. The bounds \eqref{eq:v-Omega} and
 \eqref{eq:c-Omega} follow from the construction of the basis functions
 \eqref{eq:zeta-12}--\eqref{eq:zeta-n} and from the application of Lemma \ref{lemma:cbound} 
 in each patch, respectively.
\end{proof}

\section{Proofs of Section \ref{sec:applications}}
\label{sec:proofs-appendix}
\subsection{Proof of Lemma \ref{lemma:exist-PU}} \label{sec:triangulationpolygon}
\begin{proof}[Proof of Lemma \ref{lemma:exist-PU}]
  For any two sets $X,Y\subset\Omega$,
  we denote by $\distomega(X,Y)$ the infimum of Euclidean lengths of paths in $\Omega$ 
  connecting an element of $X$ with one of $Y$.
  We introduce several domain-dependent quantities
  to be used in the construction of 
  the triangulation $\cT$ with the properties stated in the lemma.
 
  Let $\Eset$ denote the set of edges of the polygon $\Omega$.  
  For each corner $c\in \Cset$
  at which the interior angle of $\Omega$ is smaller than $\pi$ 
  (below called \emph{convex corner}), %there exists
  we fix
  a parallelogram $G_c \subset \Omega$ and 
  a bijective, affine transformation $F_c : (0,1)^2\to G_c$ 
  such that
  \begin{itemize}
  \item $F_c((0,0)) = c$,
  \item two edges of $G_c$ coincide partially with the edges of 
        $\Omega$ abutting at the corner $c$.
  \end{itemize}
  If at $c$ the interior angle of $\Omega$ is greater than or equal to $\pi$ 
  (both are referred to by slight abuse of terminology as \emph{nonconvex corner}), 
  the same properties hold, 
  with $F_c : (-1,1)\times(0,1)\to G_c$ if the interior angle equals $\pi$,
  and $F_c : (-1,1)^2\setminus (-1, 0]^{2}\to G_c$ else,
  and with $G_c$ having the corresponding shape.
  Let now
  \begin{equation*}
   d_{c,1} \coloneqq \sup \{r>0: B_r(c) \cap\Omega \subset G_c\},
   \qquad 
   d_{\Cset,1} \coloneqq \min_{c\in\Cset}d_{c,1}.
  \end{equation*}
%%% old location d_Eset
  Then, for each $c\in \Cset$, let $e_1$ and $e_2$ be the edges abutting
    $c$, and define
    \begin{equation*}
      d_{c, 2} \coloneqq \distomega \left(e_1\cap \left(B_{\frac{\sqrt{2}}{\sqrt{2}+1}d_{\Cset, 1}}(c)\right)^c
   , e_2\cap \left(B_{\frac{\sqrt{2}}{\sqrt{2}+1}d_{\Cset, 1}}(c)\right)^c\right), \qquad d_{\Cset, 2} \coloneqq \min_{c\in \Cset} d_{c, 2}.
    \end{equation*}
  Furthermore, for each $e\in\Eset$, denote $d_e := \infty$ if $\Omega$ is a triangle, otherwise
  \begin{equation*}
  d_e \coloneqq \min\left\{ \distomega(e, e_1): e_1\in\Eset \text{ and }\overline{e}\cap\overline{e_1} = \emptyset \right\},
  \qquad %\text{ and } 
  d_{\Eset}  = \min_{e\in\Eset}d_e.
  \end{equation*}
	  % with endpoints $c_1, c_2\in \Cset$, denote
	  % \begin{equation*}
	  %   d_e = \min_{\tilde{e}\in \Eset} \distomega \left(e\cap \left(B_{\frac{\sqrt{2}+1}{\sqrt{2}}d_{\Cset}}(c_1)\right)^c\cap \left(B_{\frac{\sqrt{2}+1}{\sqrt{2}}d_{\Cset}}(c_2)\right)^c,
	  %  \tilde{e}\right),
	  % \end{equation*}
%  and $d_{\Eset}  = \min_{e\in\Eset}d_e$.
	  % \begin{Note}{JO} Comparing with $d_{\Cset} - h({\cTOmega})\geq \frac{\sqrt{2}}{\sqrt{2}+1}d_{\Cset}$,
	  % shouldnt the radius of the balls be $\frac{\sqrt{2}}{\sqrt{2}+1}d_{\Cset}$?
	  % \end{Note}
Finally, for all $x\in\Omega$, let
  \begin{equation*}
    n_e(x) \coloneqq \#\{e_1, e_2, \ldots\in \Eset: \distomega(x,\partial\Omega) = \distomega(x, e_1) = \distomega(x, e_2) = \dots\}.
  \end{equation*}
  Then, in case $\Omega$ is a triangle, let $d_0$ be half of the radius of the inscribed circle,
  else let $d_0 := \tfrac13 d_{\Eset} < \tfrac12 d_{\Eset}$. It holds that
  \begin{equation*}
    \distomega(\{x\in \Omega: n_e(x)\geq 3\}, \partial\Omega) \geq d_0 >0.
  \end{equation*}

   For any shape regular triangulation $\cT$ of
  $\R^2$, such that for all $K\in \cT$, $K\cap \partial \Omega = \emptyset$, 
  denote $\cTOmega = \{K\in \cT: K\subset\Omega\}$ 
  and 
  $h(\cTOmega) = \max_{K\in\cTOmega} h(K)$, 
  where $h(K)$ denotes the diameter of $K$. 
  Denote by $\cNOmega$ the set of nodes of $\cT$ that are in $\overline{\Omega}$. 
  For any $n\in \cNOmega$, define
  \begin{equation*}
    \ptch(n) \coloneqq \interior\left(\bigcup_{K\in \cT: n\in\overline{K}} \overline{K}  \right).
  \end{equation*}
  Let $\mathcal{T}$ be a triangulation of $\R^2$ such that 
  \begin{equation}
    \label{eq:meshsize}
h(\cTOmega) \leq \min\left( \frac{d_0}{\sqrt{2}}, \frac{d_{\Cset,1}}{\sqrt{2}+1} , \frac{d_{\Cset, 2}}{2\sqrt{2}},\frac{d_{\Eset}}{2\sqrt{2}}\right),
  \end{equation}
  and such that for all $K\in\cT$ it holds $K\cap \partial\Omega = \emptyset$.

  The hat-function basis $\{\phi_n\}_{n\in\cNOmega}$ is a basis for 
  $\mathbb{P}_1(\cTOmega)$ such that
  $\supp(\phi_n) \subset \overline{\ptch(n)}$ for all $n\in \cNOmega$, 
  and it is a partition of unity.

  We will show that, for each $n\in \cNOmega$, there exists a subdomain
  $\Omega_n$ with
  the desired properties, such that $\ptch(n)\cap\Omega\subset \Omega_n$.
We point to Figure \ref{fig:patch-corner} for an illustration
  of the patches $\Omega_n$ that will be introduced in the proof, for %the
  different sets of nodes.
  \begin{figure}
    \centering
    \begin{subfigure}[b]{.25\textwidth}
      \centering
      \includegraphics[width=\textwidth]{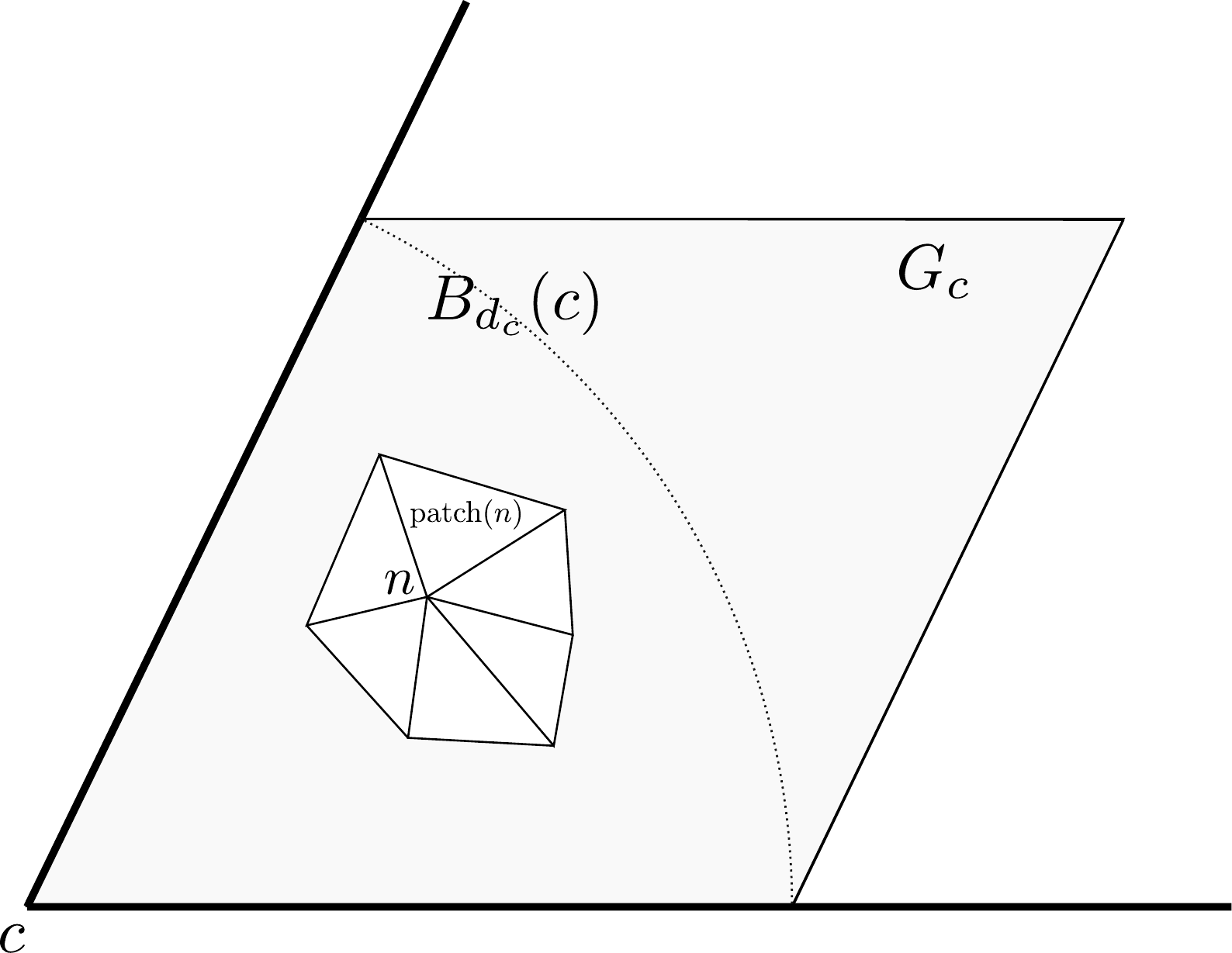}
      \caption{}\label{fig:patch-convex}
    \end{subfigure}%
    \begin{subfigure}[b]{.25\textwidth}
      \centering
      \includegraphics[width=\textwidth]{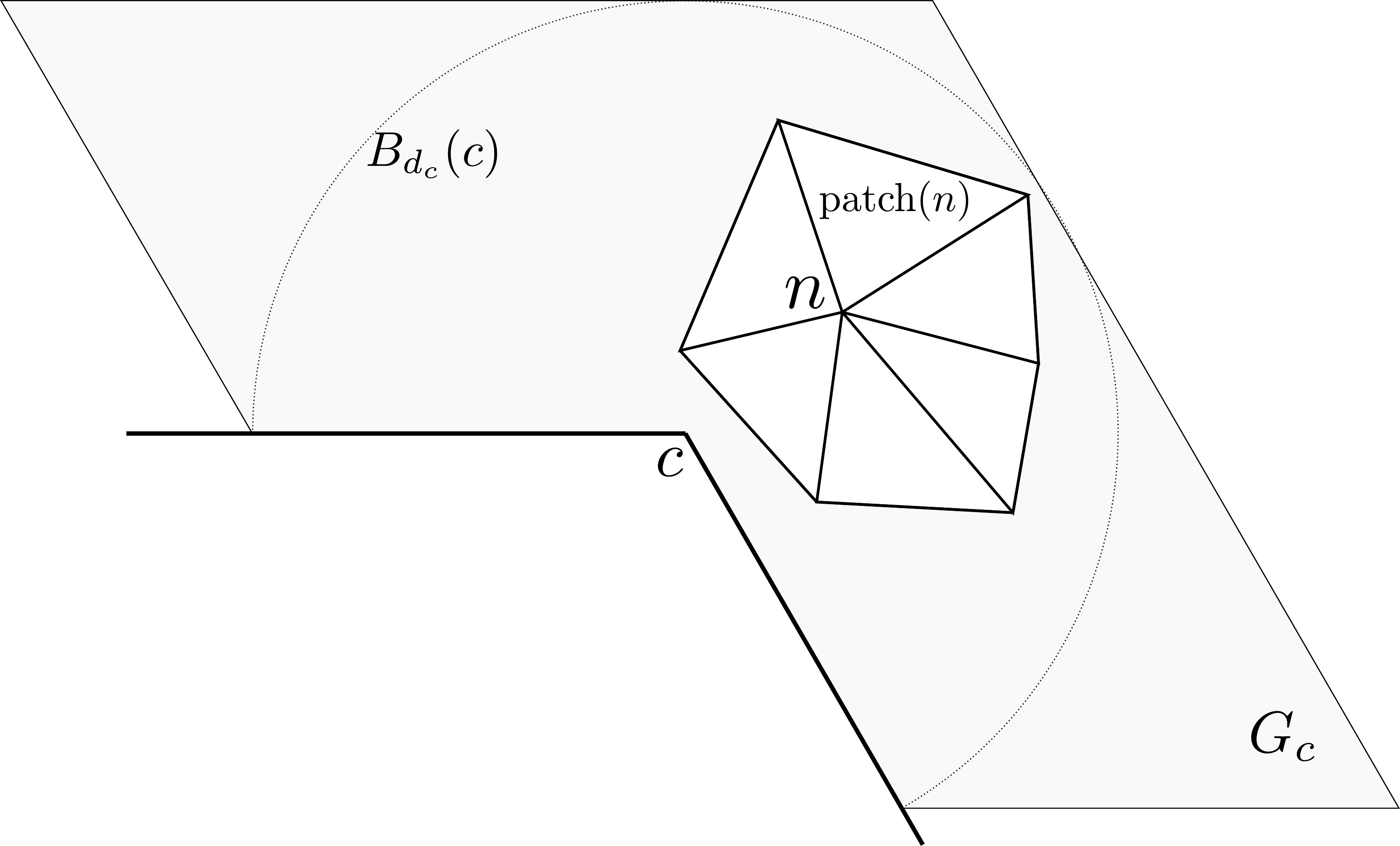}
      \caption{}\label{fig:patch-nonconvex}
    \end{subfigure}%
    \begin{subfigure}[b]{.25\textwidth}
      \centering
      \includegraphics[width=.8\textwidth]{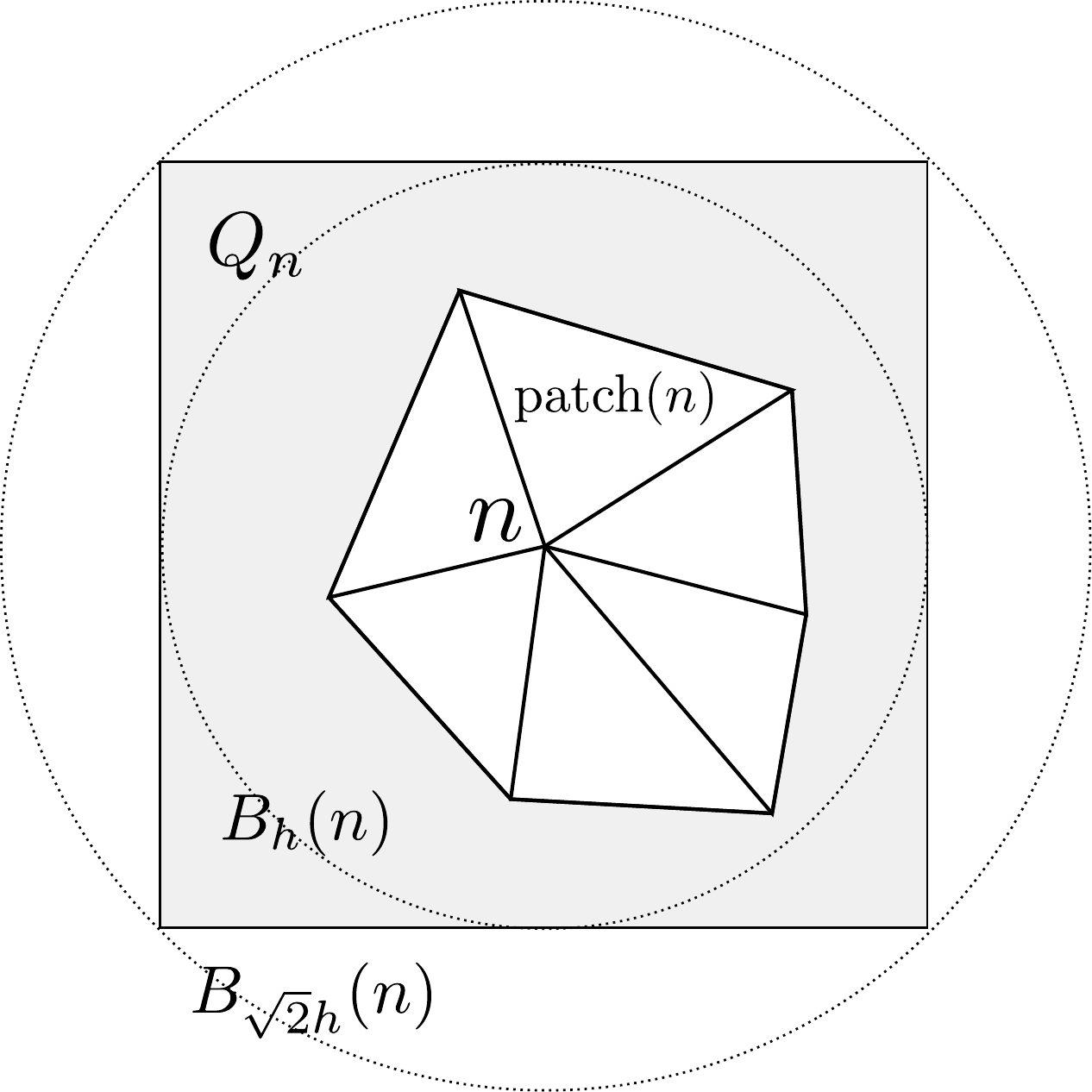}
      \caption{}\label{fig:patch-interior}
    \end{subfigure}%
    \begin{subfigure}[b]{.25\textwidth}
      \centering
      \includegraphics[width=.85\textwidth]{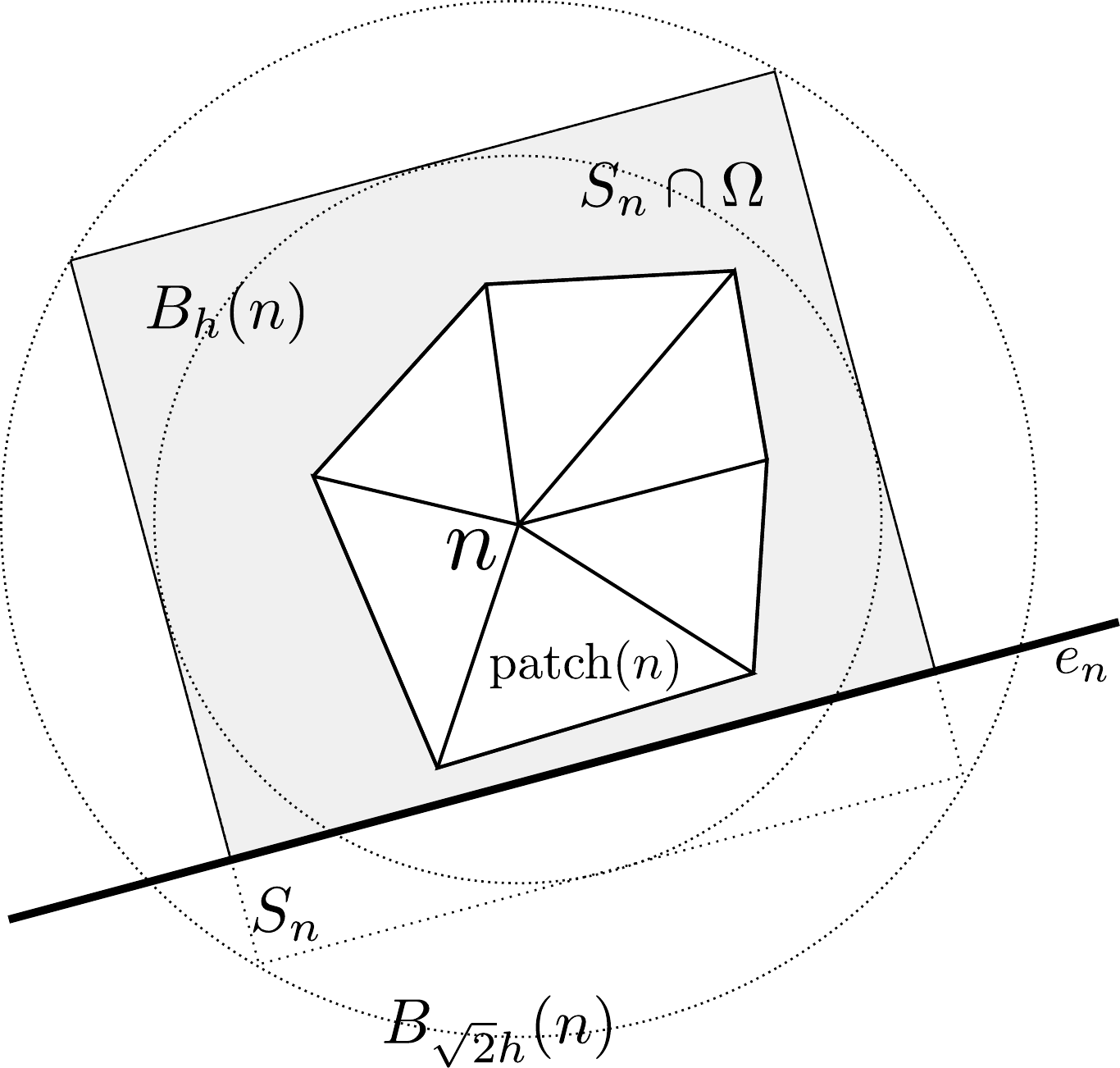}
      \caption{}\label{fig:patch-edge}
    \end{subfigure}%
    \caption{Patches $\Omega_n$ for nodes near a convex (\subref{fig:patch-convex}), nonconvex
      corner (\subref{fig:patch-nonconvex}), for nodes in the interior of $\Omega$ (\subref{fig:patch-interior}),
      and near an edge (\subref{fig:patch-edge}).}
    \label{fig:patch-corner}
  \end{figure}

  For each $c\in\Cset$, let $\hcN_c = \{n\in \cNOmega: \ptch(n)\cap\Omega\subset
  G_c\}$. There holds
  \begin{equation*}
    \cN_c\coloneqq\{n\in \cNOmega: \distomega(n, c)\leq d_{\Cset,1} - h(\cTOmega)\} \subset \hcN_c.
  \end{equation*}
  Therefore, all the nodes $n\in \cN_c$ are such that $\ptch(n)\cap\Omega \subset G_c =: \Omega_n$. 
  Denote then
  \begin{equation*}
\cN_{\Cset}  =\bigcup_{c\in\Cset}\cN_c.
  \end{equation*}
  Note that,
 due to
  \eqref{eq:meshsize}, 
there holds $\sqrt{2}h(\cTOmega) \leq %%% d_{\Cset, 1} - h(\cTOmega)$ and $
  \frac{\sqrt{2}}{\sqrt{2}+1}d_{\Cset, 1} \leq d_{\Cset, 1} - h(\cTOmega)$.
  
  We consider the nodes in $\cN\setminus \cN_{\Cset}$. First, consider the nodes in
  \begin{equation*}
    \cN_0\coloneqq \{n\in \cN\setminus\cN_{\Cset} : \distomega(n, \partial \Omega) \geq \sqrt{2}h(\cTOmega)\}.
  \end{equation*}
  For all $n\in \cN_0$, there exists a square $Q_n$ such that
  \begin{equation*}
    \ptch(n) \subset B_{h(\cTOmega)}(n) \subset Q_n\subset B_{\sqrt{2}h(\cTOmega)}(n)\subset \Omega,
  \end{equation*}
  see Figure \ref{fig:patch-interior}.
  Hence, for all $n\in \cN_0$, we take $\Omega_n := Q_n$. 
  Define 
  \begin{equation*}
 \cN_{\Eset}\coloneqq   \cN \setminus(\cN_0\cup\cN_{\Cset}) = \left\{n\in \mathcal{N}: \distomega(n, c) > d_{\Cset,1}-h(\cTOmega), \forall c\in \Cset,\text{ and }\distomega(n, \partial\Omega) < \sqrt{2}h(\cTOmega)\right\}.
  \end{equation*}
  For all $n\in \cN_{\Eset}$, 
  from \eqref{eq:meshsize} it follows that 
  $\distomega(n, \partial \Omega) <\sqrt{2}h(\cTOmega) \leq
  d_0$, hence $n_e(n) \leq 2$. Furthermore, suppose there exists $n\in
  \cN_{\Eset}$ such that $n_e(n) =2$. Let the two closest edges to $n$ be
  denoted by $e_1$ and $e_2$, so that $\distomega(n, e_1) = \distomega(n, e_2)
%  {= \dist(n, \partial \Omega)} 
  = \distomega(n, \partial\Omega) <\sqrt{2}h(\cTOmega)$. If
    $\overline{e_1}\cap\overline{e_2} = \emptyset$, there must hold $\distomega(n, e_1) + \distomega(n, e_2)\geq d_\Eset$, which is
a contradiction with $\distomega(n, \partial\Omega) < \sqrt{2}h(\cTOmega)\leq
d_{\Eset}/2$. If instead there exists $c\in\Cset$ such that $\overline{e_1}\cap
\overline{e_2} = \{c\}$, then $n$ is on the  bisector of the angle between $e_1$ and
$e_2$. 
Using that $2\sqrt{2} h(\cTOmega)\leq d_{\Cset, 2}$, we now show 
that all such nodes %on the bisector of the angle at the corner of the polygon
belong either to $\cN_{\Cset}$ or to $\cN_0$, which is a contradiction to $n\in\cN_{\Eset}$.
%
%First, for $i=1,2$ define 
%$\Omega_{e_i} := \{x\in\Omega: \dist(x,e_i)<\sqrt{2}h({\cTOmega})\} 
%	\cap \left(B_{\frac{\sqrt{2}}{\sqrt{2}+1}d_{\Cset, 1}}(c)\right)^c$.
Let $x_0\in\Omega$ be the intersection of
$B_{\frac{\sqrt{2}}{\sqrt{2}+1}d_{\Cset, 1}}(c)$ and the bisector.
To show that $n\in\cN_{\Cset}\cup\cN_0$,
it suffices to show that $\dist(x_0,e_i) \geq \sqrt{2}h(\cTOmega)$ for $i=1,2$.
Because $\frac{\sqrt{2}}{\sqrt{2}+1}d_{\Cset, 1} \leq d_{\Cset, 1} - h(\cTOmega)$, 
it a fortiori holds for all points $y$ in $\Omega$ on the bisector
intersected with $\left(B_{d_{\Cset, 1}-h(\cTOmega)}(c)\right)^c$, 
that $\dist(y,e_i)\geq \sqrt{2}h(\cTOmega)$, 
which shows that %for the node $n$ we are considering, 
if $\distomega(n,c)\geq d_{\Cset,1}-h(\cTOmega)$, then $n\in\cN_0$.
If $c$ is a nonconvex corner, then 
$\dist(x_0,e_i) \geq \sqrt{2}h(\cTOmega)$ for $i=1,2$ follows immediately from 
$\dist(x_0,e_i) = \dist(x_0,c) = \frac{\sqrt{2}}{\sqrt{2}+1}d_{\Cset, 1}$ and
\eqref{eq:meshsize}. 
%$\sqrt{2}h({\cTOmega}) \leq \frac{\sqrt{2}}{\sqrt{2}+1}d_{\Cset, 1}$.
%
To show that $\dist(x_0,e_i) \geq \sqrt{2}h(\cTOmega)$, $i=1,2$
in case $c$ is a convex corner, we make the following definitions 
(see Figure \ref{fig:e1ce2}):
\begin{figure}%[t]
    \centering
      \includegraphics[width=.4\textwidth]{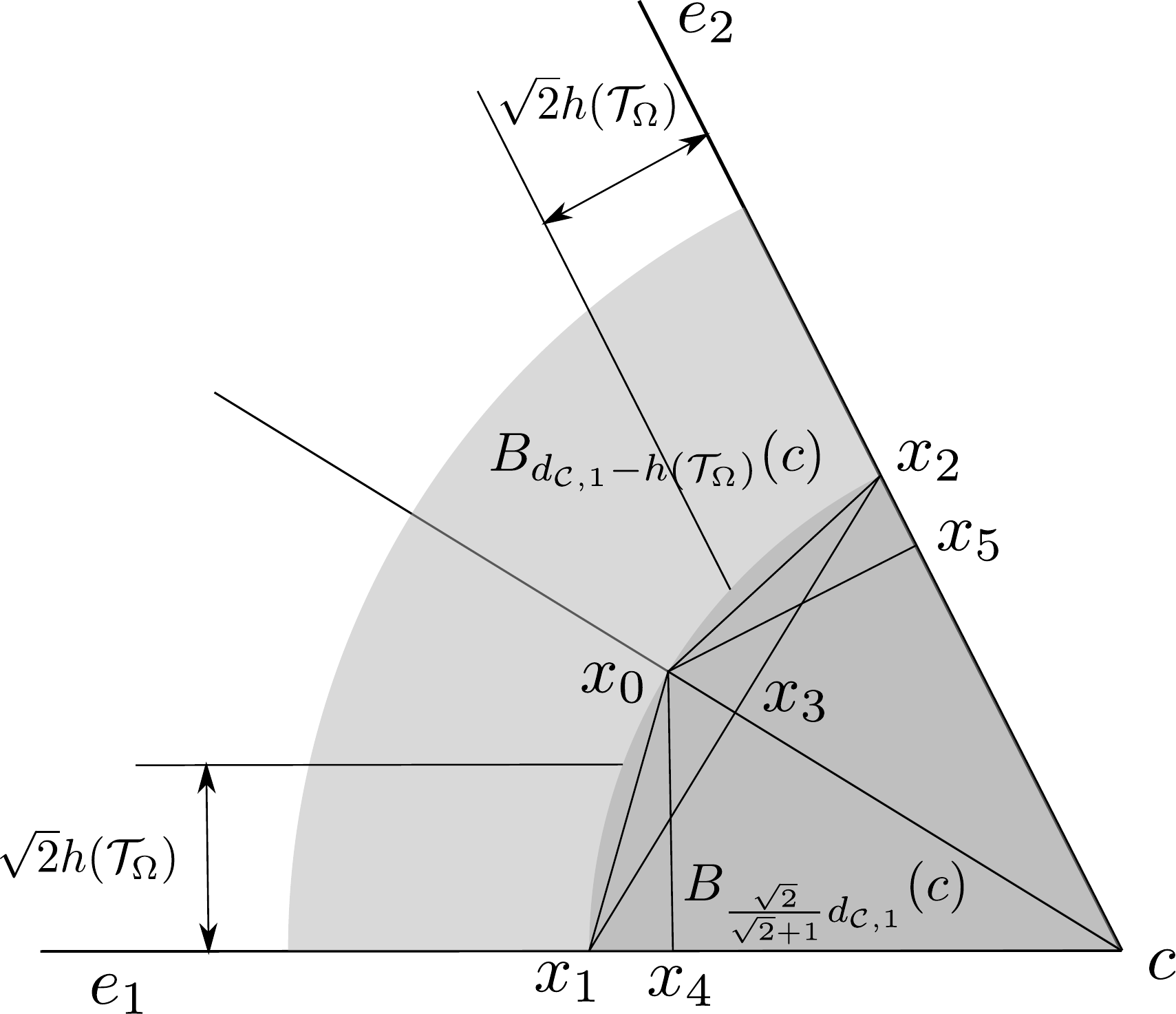}
         \caption{Situation near a convex corner $c$.}
    \label{fig:e1ce2}
  \end{figure}
\begin{itemize}
\item For $i=1,2$, let $x_i$ be the intersection of 
$e_i$ and $B_{\frac{\sqrt{2}}{\sqrt{2}+1}d_{\Cset, 1}}(c)$,
%\item let $x_0\in\Omega$ be the intersection of 
%$B_{\frac{\sqrt{2}}{\sqrt{2}+1}d_{\Cset, 1}}(c)$ and the bisector,
\item let $x_3$ be the intersection of $\overline{x_1x_2}$ with the bisector,
\item and for $i=1,2$, let $x_{i+3}$ be the orthogonal projection of $x_0$ onto $e_i$,
which is an element of $e_i$ because $c$ is a convex corner.
\end{itemize}
Then 
$d_{c,2} = |\overline{x_1x_2}| = |\overline{x_1x_3}| + |\overline{x_3x_2}| = 2 |\overline{x_ix_3}|$.
Because the triangle $cx_0x_{i+3}$ is congruent to $cx_1x_3$, it follows that
$\dist(x_0,e_i) = |\overline{x_0x_{i+3}}| = |\overline{x_ix_3}| = \tfrac12 d_{c,2} \geq \sqrt{2}h(\cTOmega)$.
%As a result, it follows that $\dist(x_0,e_i) \geq \sqrt{2}h({\cTOmega})$ and hence 
%
We can conclude with \eqref{eq:meshsize} that $n_e(n) = 1$ for all $n\in \cN_{\Eset}$
and denote the edge closest to $n$ by $e_n$.
  % for {each} $n\in \cN_{\Eset}$, there exists
  % a unique $e_n\in \Eset$ such that $\distomega(n, e_n) < \sqrt{2}h(\cTOmega)$.
Let then
  $S_n$ be the square with two edges parallel to $e_n$ such that
  \begin{equation*}
   \ptch(n) \subset B_{h(\cTOmega)}(n)\subset S_n\subset B_{\sqrt{2}h(\cTOmega)}(n),
  \end{equation*}
  see Figure \ref{fig:patch-edge}, i.e. $S_n$ has center $n$ and sides of length $2h(\cTOmega)$.
%  \begin{Note}{JO} What do you want to say with these inclusions?
%  Should we explain that $S_n\cap\Omega$ is a rectangle, i.e. why another part of the \end{Note}
 For each $n\in \cN_{\Eset}$, 
 the connected component of $S_n\cap \Omega$ containing $n$ is a rectangle:
\begin{itemize}
\item[(i)] Note that for all edges e such that $\overline{e}\cap\overline{e_n} = \emptyset$, 
it holds that $S_n\cap e \subset B_{\sqrt{2}h(\cTOmega)}(n) \cap e = \emptyset$.
The latter holds because $\sqrt{2}h(\cTOmega) \leq \tfrac12 d_{\Eset}$ 
and $\distomega(n,e_n)<\sqrt{2}h(\cTOmega)$ 
imply $\distomega(n,e)\geq \sqrt{2}h(\cTOmega)$.
\item[(ii)] From the previously given geometric argument
 considering a convex corner $c$ and the two neighboring edges $e_1$ and $e_2$, 
 showing that 
 $\dist(x_0,e_i) \geq \sqrt{2}h(\cTOmega)$ for $i=1,2$, we can additionally conclude 
that there is no $x\in\Omega\setminus B_{\frac{\sqrt{2}}{\sqrt{2}+1}d_{\Cset, 1}}(c)$
for which $\dist(x,e_n)<\sqrt2 h(\cTOmega)$ 
and such that there exists another edge $e$ so that $\overline{e_n}\cap\overline{e}\neq\emptyset$ 
and $\dist(x,e) < \sqrt2 h(\cTOmega)$.
This implies that $S_n\cap\partial\Omega \subset e_n$ or $S_n\cap\partial\Omega = \emptyset$.
\end{itemize}
Thus, the connected component of $S_n\cap\Omega$ containing $n$ is a rectangle,
which we define to be $\Omega_n$.

Setting $N_p := \#\cNOmega$ and $\{\Omega_i\}_{i=1,\ldots,N_p} = \{\Omega_n\}_{n\in\cNOmega}$ concludes the proof. 
  \end{proof}
  \subsection{Proof of Lemma \ref{lemma:W11ext}}
\label{sec:W11ext}
\begin{proof}[Proof of Lemma \ref{lemma:W11ext}]
  Let $d=3$ and denote $R = (-1, 0)^3$. 
  Denote by $O$ the origin, and
  let $E = \{e_1, e_2, e_3\}$ denote the set of edges of $R$ abutting the origin. 
  Let also $F=\{f_1, f_2, f_3\}$ denote 
  the set of faces of $R$ abutting the origin, i.e., the faces of $R$ such
  that $f_i \subset \overline{R}\cap\overline{\Omega}$, $i=1,2,3$.
  Let, finally, for each $f\in F$, $E_f = \{e\in E: e\subset\overline{f}$
  % e\cap\overline{F}\neq\emptyset\}$ 
  denote the subset of $E$ containing the edges neighboring $f$.

  For each $e\in E$, define $u_e$ to be the lifting of $u|_e$ into $R$, i.e.,
  the function such that $u_e|_e = u_e$ and $u_e$ is constant in the two
  coordinate directions perpendicular to $e$. Similarly, let, for each $f\in F$,
  $u_f$ be such that $u_f|_f = u|_f$ and $u_f$ is constant in the direction
  perpendicular to $f$.

  We define $w:R\to \mathbb{R}$ as
  \begin{equation}
    \label{eq:vR}
    w = u_0 + \sum_{e\in E}\big( u_{e}  - u_0\big) 
+ \sum_{f\in F}\big( u_{f}  - u_0 -  \sum_{e\in E_f}(u_{e}-u_0)\big) = u_0 -\sum_{e\in E}u_e + \sum_{f\in F}u_f,
  \end{equation}
  where $u_0 = u(O)$.
  Since $u|_{e}\in W^{1,1}(e)$, $u|_{f}\in\Wmix^{1,1}(f)$ for all
  $e\in E$ and $f\in F$, there holds $u_e\in \Wmix^{1,1}(R)$ and 
  $u_f\in \Wmix^{1,1}(R)$ for all $e\in E$ and $f\in F$ 
  (cf. Equations \eqref{eq:trace} and \eqref{eq:trace-2}), 
  hence $w\in\Wmix^{1,1}(R)$.
  Furthermore, note that
  \begin{equation*}
    \big( u_{e}  - u_0\big)|_{\widetilde{e}} = 0, \quad \text{for all }E\ni\widetilde{e}\neq e
  \end{equation*}
  and that
  \begin{equation*}
    \big(u_{f}  - u_0 -  \sum_{e\in E_f}(u_{e}-u_0)\big)|_{\widetilde{f}} = 0, \quad \text{for all }F\ni\widetilde{f}\neq f.
  \end{equation*}
  From the first equality in \eqref{eq:vR}, then, it follows that, for all $f\in F$,
  \begin{equation*}
    w|_f = u_0 + \sum_{e\in E_f}\big(u_e|_f - u_0\big) + u_f|_f - u_0 - \sum_{e\in E_f}\big(u_e|_f - u_0\big)= u|_f.
  \end{equation*}
  Let the function $v$ be defined as
  \begin{equation}
    \label{eq:vR2}
    v|_R = w, \qquad v|_{\Omega} =u.
  \end{equation}
  Then, $v$ is continuous in $(-1,1)^3$ and $v\in\Wmix^{1,1}((-1,1)^3)$.
  Now, for all $\alpha\in \mathbb{N}^3_0$ such that $|\alpha|_\infty\leq 1$,
  \begin{equation*}
   \|\dalpha u_e\|_{L^1(R)} 
 = \|\partial^{\alpha^e_\parallel} u_e\|_{L^1(R)} 
 = \|\partial^{\alpha^e_\parallel} u\|_{L^1(e)},\qquad \forall e\in E,
  \end{equation*}
  where $\alpha^e_\parallel$ denotes the index in the coordinate direction parallel to $e$, 
  and
  \begin{equation*}
  \|\dalpha u_f\|_{L^1(R)} 
= \|\partial^{\alpha^f_{\parallel,1}} \partial^{\alpha^f_{\parallel,2}} u_f\|_{L^1(R)} 
= \|\partial^{\alpha^f_{\parallel,1}} \partial^{\alpha^f_{\parallel,2}} u\|_{L^1(f)} ,\qquad \forall f\in F,
  \end{equation*}
where $\alpha^f_{\parallel, j}$, $j=1,2$ denote the indices in the coordinate directions
  parallel to $f$. Then, by a trace inequality (see \cite[Lemma 4.2]{Schotzau2013a}), 
  there exists a constant $C>0$ independent of $u$ such that
  \begin{equation*}
   \| u_e \|_{\Wmix^{1,1}(R)} \leq C \| u \|_{\Wmix^{1,1}(\Omega)},\qquad
   \| u_f \|_{\Wmix^{1,1}(R)} \leq C \| u \|_{\Wmix^{1,1}(\Omega)},
  \end{equation*}
  for all $e\in E$, $f\in F$. Then, by \eqref{eq:vR}  and \eqref{eq:vR2},
  \begin{equation*}
    \| v\|_{\Wmix^{1,1}((-1,1)^d)}\leq C \| u \|_{\Wmix^{1,1}(\Omega)},
  \end{equation*}
  for an updated constant $C$ independent of $u$.
  This concludes the proof when $d=3$. 
  The case $d=2$ can be treated by the same argument.
\end{proof}

\FloatBarrier

\bibliographystyle{abbrv}
\bibliography{references}
\end{document}